\documentclass[11pt, reqno]{amsart}
\makeatletter
\usepackage{fullpage}
\usepackage{textcmds}  
\usepackage{amsmath, amssymb, amsfonts, amstext, verbatim, amsthm, mathrsfs, mathtools}
\usepackage[mathcal]{eucal}
\usepackage{microtype}
\usepackage{graphicx}
\usepackage[all]{xy}
\usepackage{stmaryrd}

\usepackage{tikz} 

\usepackage{enumitem}
\usepackage{xspace}
\usepackage{mathdots}

\usepackage[colorlinks=true,linkcolor=blue,citecolor=blue,urlcolor=blue,citebordercolor={0 0 1},urlbordercolor={0 0 1},linkbordercolor={0 0 1}]{hyperref} 
\usepackage[shortalphabetic]{amsrefs} 
\usepackage[nameinlink]{cleveref}

\def\makeCal#1{%
\expandafter\newcommand\csname c#1\endcsname{\mathcal{#1}}}
\def\makeBB#1{%
\expandafter\newcommand\csname b#1\endcsname{\mathbb{#1}}}
\def\makeFrak#1{%
\expandafter\newcommand\csname f#1\endcsname{\mathfrak{#1}}}

\count@=0
\loop
\advance\count@ 1
\edef\y{\@Alph\count@}%
\expandafter\makeCal\y
\expandafter\makeBB\y
\expandafter\makeFrak\y
\ifnum\count@<26
\repeat

\makeatletter
\def\namedlabel#1#2{\begingroup
    #2%
    \def\@currentlabel{#2}%
    \phantomsection\label{#1}\endgroup
}
\makeatother

\theoremstyle{plain}
\newtheorem{thm}{Theorem}[subsection]
\newtheorem{cor}{Corollary}[thm]
\newtheorem{lem}[thm]{Lemma}

\newtheorem{prop}[thm]{Proposition}

\newtheorem{conj}{Conjecture}

\theoremstyle{definition}
\newtheorem{rem}[thm]{Remark}
\newtheorem{hyp}[thm]{Hypotheses}

\newtheorem{defn}[thm]{Definition}
\newtheorem{notn}[thm]{Notation}
\newtheorem{const}[thm]{Construction}
\newtheorem{ex}[thm]{Example}

\newtheorem{thmx}{Theorem}

\DeclareMathOperator{\DCoh}{DCoh}
\DeclareMathOperator{\DSing}{DSing}
\DeclareMathOperator{\APerf}{D^{--}Coh}
\DeclareMathOperator{\AAPerf}{APerf}
\DeclareMathOperator{\Perf}{Perf}
\DeclareMathOperator{\cStab}{{\mathcal S}tab}

\DeclareMathOperator{\sAlg}{sAlg}
\DeclareMathOperator{\Aut}{Aut}

\DeclareMathOperator{\wt}{wt}
\DeclareMathOperator{\Pic}{Pic}
\newcommand{\cl}{{\rm cl}}
\newcommand{\dual}{\vee}
\DeclareMathOperator{\minwt}{minWt}
\DeclareMathOperator{\maxwt}{maxWt}

\DeclareMathOperator{\NS}{NS}
\newcommand{\CDGA}{CDGA\ }
\newcommand{\fg}{\mathfrak{g}}
\DeclareMathOperator{\RHom}{RHom}
\DeclareMathOperator{\Hom}{Hom}
\DeclareMathOperator{\detz}{\det\nolimits^{\bZ}}
\newcommand{\thickslash}{\mathbin{\!\!\pmb{\fatslash}}}

\newcommand{\id}{\mathop{{\rm id}}\nolimits}

\DeclareMathOperator{\GL}{GL}

\newcommand{\Gm}{\mathbb{G}_m}

\newcommand{\radj}[1]{\beta^{\geq #1}}
\newcommand{\ladj}[1]{\beta^{< #1}}
\newcommand{\sod}[1]{\langle #1 \rangle}
\newcommand{\itimes}{\overset{!}{\otimes}}

\DeclareMathOperator{\QC}{QC}
\DeclareMathOperator{\IC}{IndCoh}
\DeclareMathOperator{\Fun}{Fun}
\DeclareMathOperator{\Ho}{Ho}
\DeclareMathOperator{\PrL}{Pr^L}

\newcommand{\inner}[1]{\underline{#1}}
\DeclareMathOperator{\Coh}{Coh}
\DeclareMathOperator{\Map}{Map}
\newcommand{\heart}{\heartsuit}
\DeclareMathOperator{\colim}{colim}

\DeclareMathOperator{\rank}{rank}
\DeclareMathOperator{\cofib}{cofib}
\DeclareMathOperator{\fib}{fib}

\newcommand{\iMap}{\inner{\Map}}

\newcommand{\filt}[2][]{\iMap(\Theta^{#1}, {#2})}
\newcommand{\grad}[2][]{\iMap(B\Gm^{#1}, {#2})}
\DeclareMathOperator{\Filt}{Filt}
\DeclareMathOperator{\Grad}{Grad}
\DeclareMathOperator{\Flag}{Flag}
\DeclareMathOperator{\Spec}{Spec}
\DeclareMathOperator{\Sym}{Sym}
\DeclareMathOperator{\Tot}{Tot}
\DeclareMathOperator{\Rep}{Rep}
\DeclareMathOperator{\image}{image}
\DeclareMathOperator{\Lie}{Lie}
\DeclareMathOperator{\Stab}{Stab}
\DeclareMathOperator{\Crit}{Crit}

\DeclareMathOperator{\ev}{ev}
\newcommand{\us}{{\rm us}}
\DeclareMathOperator{\res}{res}
\DeclareMathOperator{\Mod}{-Mod}

\DeclareMathOperator{\Cech}{Cech}
\DeclareMathOperator{\Cpl}{Cpl}
\DeclareMathOperator{\End}{End}

\begin{document}
\title{Derived $\Theta$-stratifications and the $D$-equivalence conjecture}
\author{Daniel Halpern-Leistner}

\begin{abstract}
The theory of $\Theta$-stratifications generalizes a classical stratification of the moduli of vector bundles on a smooth curve, the Harder-Narasimhan-Shatz stratification, to any moduli problem that can be represented by an algebraic stack. Using derived algebraic geometry, we develop a structure theory, which is a refinement of the theory of local cohomology, for the derived category of quasi-coherent complexes on an algebraic stack equipped with a $\Theta$-stratification. We then apply this to the $D$-equivalence conjecture, which predicts that birationally equivalent Calabi-Yau manifolds have equivalent derived categories of coherent sheaves. We prove that any two projective Calabi-Yau manifolds that are birationally equivalent to a smooth moduli space of Gieseker semistable coherent sheaves on a $K3$ surface have equivalent derived categories. These are the first known cases of the $D$-equivalence conjecture for a birational equivalence class in dimension greater than three.
\end{abstract}

\maketitle

\setcounter{tocdepth}{2}
\tableofcontents

In \cite{bondal1995semiorthogonal}, Bondal and Orlov stated a conjecture which has become one of the motivating problems in a rapidly expanding field studying the relationship between birational geometry and derived categories of coherent sheaves. It was developed further by Kawamata in \cite{kawamata}.
\begin{conj}[$D$-equivalence]\label{conj:D_equivalence}
If $X$ and $Y$ are birationally equivalent projective Calabi-Yau (CY) manifolds, then there is an equivalence of bounded derived categories of coherent sheaves $\DCoh(X) \cong \DCoh(Y)$.
\end{conj}
This conjecture is motivated by homological mirror symmetry: the mirror manifolds of $X$ and $Y$ are expected to be deformation equivalent CY manifolds, and the mirror to this conjecture -- that deformation equivalent polarized CY manifolds have equivalent Fukaya categories --  follows immediately from the fact that any two such manifolds are symplectomorphic.

Although \Cref{conj:D_equivalence} has been established for three-dimensional CY manifolds \cite{bridgeland2002flops}, much less is known in dimension $>3$. There are simple examples of birational transformations, such as Mukai flops, for which a derived equivalence can be constructed from a Fourier-Mukai kernel with a succint geometric description \cite{namikawa}. In general, though, such flops are not sufficient to connect birationally equivalent CY manifolds. As a result, the conjecture is not known to hold for any birational equivalence class of projective CY manifolds in dimension $>3$, with one exception which we discuss below.

One particularly well-studied class of higher dimensional CY manifolds are hyperk\"{a}hler manifolds, and the central examples,\footnote{All known projective hyperk\"{a}hler manifolds are related to these moduli spaces either by deformation or resolution.} whose dimension can be arbitrarily large, are moduli spaces $M^H_S(v)$ of semistable coherent sheaves on a $K3$ surface $S$. More precisely, let $M^H_S(v)$ denote the moduli space of sheaves on $S$ whose characteristic classes are given by a primitive Mukai vector $v \in H^\ast_{\rm alg}(S;\bZ)$, and which are Gieseker semistable \cite{huybrechts2010geometry} with respect to a polarization $H \in {\rm NS}(S)$ that is generic with respect to $v$. The main goal of this paper is to prove of the following:
\begin{thmx}[\Cref{T:D_equivalence}] \label{T:A}
\Cref{conj:D_equivalence} holds for Calabi-Yau manifolds that are birationally equivalent to any moduli space $M^H_S(v)$ of semistable coherent sheaves on a $K3$ surface $S$.
\end{thmx}

More generally, the notion of a Bridgeland stability condition \cite{bridgeland} on $\DCoh(S)$, which we recall in \Cref{S:Bridgeland_moduli}, allows one to define moduli spaces of semistable complexes, generalizing the moduli of semistable coherent sheaves. \Cref{T:D_equivalence} also applies to any CY manifold that is birationally equivalent to one of these moduli spaces. In the paper, we also treat complexes of twisted coherent sheaves, but for simplicity we focus on the non-twisted case in the introduction.

\begin{rem}
In fact, \Cref{T:A} arises from a more general theorem, \Cref{T:magic_windows}, that establishes derived equivalences between smooth varieties that are related by a ``variation of stability" on a derived stack with self-dual cotangent complex. Examples include variation of Bridgeland semistable moduli spaces for objects in any $2$-Calabi-Yau $dg$-category.
\end{rem}

Let us summarize the recent developments that are key ingredients in our approach, which also serves to summarize the proof of \Cref{T:A}:
\begin{enumerate}
\item \textit{MMP:} The minimal model program for CY manifolds that are birationally equivalent to $M^H_S(v)$ for some $K3$-surface $S$ \cite{bayer2014mmp}*{Thm.~1.2} (see \Cref{T:bayer_macri}) allows one to reduce to studying moduli spaces of Bridgeland semistable complexes of coherent sheaves on $S$;\\
\item \textit{Local model for flops:} A recent existence result for good moduli spaces of Bridgeland semistable complexes \cite{AHLH}*{Thm.~7.25} (see \Cref{thm:good_moduli}), combined with the Luna slice theorem for stacks with good moduli space \cite{alper2015luna}*{Thm.~2.9}, allows us to prove a new local structure theorem, \Cref{thm:local_model}, for the moduli stack $\cM_\sigma(v)$ of complexes of coherent sheaves on $S$ that are semistable with respect to a Bridgeland stability condition $\sigma$;\footnote{In fact, our local structure theorem applies to any derived stack with self-dual cotangent complex and whose underlying classical stack admits a good moduli space.}\\
\item \textit{Derived equivalences in the local case:} The local structure theorem allows us to model a variation of Bridgeland stability condition on $\DCoh(S)$ locally as a special kind of variation of GIT quotient, and we use the ``magic windows theorem'' of \cite{halpern2016combinatorial}*{Thm.~1.2} to construct derived equivalences for variations of GIT quotient of this form; and\\
\item \textit{Globalization:} The theory of $\Theta$-stratifications \cite{halpern2014structure} and its derived analog, which we introduce in this paper, allow us to establish a structure theorem for $\DCoh(\cM_\sigma(c))$, \Cref{thm:derived_Kirwan_surjectivity_quasi-smooth}, which we discuss below. Then in \Cref{T:magic_windows}, we use this structure theorem to ``globalize'' the derived equivalences constructed in step (3).\\
\end{enumerate}

It was observed previouly in \cite{addington2016moduli}*{Sect.~4} that for a $K3$ surface $S$ of Picard rank $1$ and degree $2g-2$, the Hilbert scheme of $g$-points ${\rm Hilb}^{g}(S)$ admits only one other birationally equivalent CY manifold, which is also a moduli space of sheaves on $S$. These two moduli spaces are related by a Mukai flop, and hence are derived equivalent. This small birational equivalence class was previously the only known example of the $D$-equivalence conjecture in dimension $>3$. \Cref{T:A} completes this line of reasoning by constructing derived equivalences for any flop that arises from variation of stability condition on $\DCoh(S)$.

At a high-level, our strategy for establishing the derived equivalence in step (3) and (4) is the same as that used to construct derived equivalences for certain smooth variations of GIT quotient, sometimes referred to as ``window categories.'' The mathematical physicists Hori, Herbst, and Page \cite{HHP} introduced the idea in the context of $2D$ gauged linear sigma models, and Ed Segal then formulated it in a mathematical context \cite{segal}. It has subsequently been developed by several authors \cites{halpern2015derived,BFK,DS1,DS2,favero_kelly,halpern2016combinatorial,Rennemo,ballard_wall,addington_donovan_segal}.

The main difference here is that instead of smooth global quotient stacks, we work with algebraic \emph{derived} stacks without a known global quotient presentation. Specifically, we study a ``reduced'' derived moduli stack of $\sigma$-semistable complexes $\cM^{\rm red}_\sigma(v)$ (introduced in \Cref{P:derived_reduction}). We identify a certain subcategory $\fW \subset \DCoh(\cM^{\rm red}_\sigma(v))$ consisting of all complexes satisfying a ``window condition'' (\Cref{D:magic_windows}), and the main theorem states that for any generic Bridgeland stability condition $\sigma'$ in a neighborhood of $\sigma$, the natural restriction functor induces an equivalence
\begin{equation} \label{E:magic_diagram}
\xymatrix@R=5pt{ \fW \ar[r]^-{\cong} \ar@{}[d]|-{\cap} & \DCoh(\cM^{\rm red}_{\sigma'}(v)) \\
\DCoh(\cM^{\rm red}_\sigma(v)) \ar[ur]_-{\res} & }.
\end{equation}
An immediate consequence is that all of the $\cM_{\sigma'}^{\rm red}(v)$ are derived equivalent for generic $\sigma'$ in this neighborhood of $\sigma$.

The main technical contribution of this paper, used in both step (3) and (4) above, is a general structure theorem, \Cref{T:derived_Kirwan_surjectivity}, for the derived category of an algebraic derived stack that is equipped with a certain kind of stratification, called a $\Theta$-stratification. \Cref{T:derived_Kirwan_surjectivity} is much more than what is strictly necessary to prove \Cref{T:A}, and we believe it is of broader interest and applicability, so we have attempted to develop this piece of the story in the greatest possible generality. It is essentially a modification of Grothendieck's theory of local cohomology for a very special kind of closed substack, which tends to arise in moduli theory and equivariant geometry. We discuss this in more detail in the next section.

\subsubsection*{Comparison with other approaches}

In the context of \Cref{T:A}, the flops between moduli spaces that we study can locally be modeled by flops of (algebraic) symplectic resolutions of a singular poisson variety $\Spec(A)$
\begin{equation} \label{E:local_flop}
\xymatrix@C=5pt{X \ar[dr] \ar@{-->}[rr] & & X' \ar[dl] \\ & \Spec(A) & }.
\end{equation}
Therefore, the closest antecedents to \Cref{T:A} are results establishing derived equivalences $\DCoh(X) \cong \DCoh(X')$ for flops of this kind.

We note that the local models we use are closely related to variations of hyperk\"{a}hler quotient. For the specific case of moduli spaces of complexes on a K3 surface they are analytically-locally modelled by flops of Nakajima quiver varieties. Several authors have used local models of this kind to study the birational geometry of these moduli spaces \cites{MR3783425,MR2221132,budur_zhang_2019,bandiera_manetti_meazzini_2021}.

Besides the method of window conditions, which this paper builds on, there are two other main approaches to constructing derived equivalences for flops of the kind \eqref{E:local_flop}:
\begin{enumerate}[label=(\roman*)]
\item Using quantization in characteristic $p$ \cites{BK,K} to construct tilting bundles on $X$ and $X'$ whose algebra of endomorphisms agree, and
\item Using categorical representations of Lie algebras \cite{CKL} to explicitly construct Fourier-Mukai kernels.
\end{enumerate}
There are challenges in extending both of these methods from the local to the global setting, where $\Spec(A)$ is replaced by a singular projective variety $Y$. In (i), there are cohomological obstructions both to the existence of the tilting bundles and to extending them from the local to the global setting.

Likewise, extending the categorical representations of Lie algebras in (ii) from the local to the global setting is challenging due to the explicit nature of the complexes involved and the higher coherence data needed to glue complexes. Recently, \cite{addington2020categorical} has made significant progress by establishing a categorical $\mathfrak{sl}_2$ action on the derived category of certain moduli spaces of sheaves, categorifying \cite{Nakajima2}. However, at the moment this method can only construct derived equivalences for a limited class of flops, called stratified Mukai flops.

By contrast, in our approach the objects are global to begin with. The general theory of derived categories of derived algebraic stacks provides the globally defined diagram of $\infty$-categories \eqref{E:magic_diagram}. The main insight of our proof of \Cref{T:A} is that is suffices to verify that the restriction functor $\fW \to \DCoh(\cM_{\sigma'}(v))$ is an equivalence of $\infty$-categories \'etale-locally over the good moduli space of $\cM^{\rm red}_\sigma(v)$.

\subsection*{Derived \texorpdfstring{$\Theta$}{Theta}-stratifications}

The notion of a $\Theta$-stratification of an algebraic stack $\cX$ was introduced in \cite{halpern2014structure} as a generalization of the Harder-Narasimhan stratification of the moduli stack of vector bundles on a smooth curve. It also generalizes the Hesselink-Kempf-Kirwan-Ness stratification of the unstable locus in geometric invariant theory.

The starting point in \cite{halpern2014structure} is the observation that if we consider $\Theta := \bA^1/\bG_m$, then a map $f : \Theta \to \cX$, i.e., an equivariant map $\bA^1 \to \cX$, is analogous to a filtration of the point $f(1) \in \cX$, with associated graded point $f(0) \in \cX$. We refer to such a map $f$ as a filtration in $\cX$. We can then consider the mapping stack
\[
\Filt(\cX) := \iMap(\Theta,\cX)
\]
that parameterizes algebraic families of filtrations in $\cX$, and refer to this as the stack of filtrations. Under favorable hypotheses, $\Filt(\cX)$ is also an algebraic stack (see \Cref{T:mapping_stack}). The assignment $(f : \Theta \to \cX) \mapsto f(1)$ defines an ``evaluation at $1$'' morphism of algebraic stacks
\[
\ev_1 : \Filt(\cX) \to \cX.
\]

A \emph{$\Theta$-stratum} in $\cX$ is a union of connected components $\cS \subset \Filt(\cX)$ such that $\ev_1 : \cS \to \cX$ is a closed immersion. Informally, a $\Theta$-stratum is a closed substack that parameterizes a point in $\cX$ along with a canonical filtration of that point. The mapping stack $\Filt(\cX)$ and this notion of a $\Theta$-stratum can be formulated in both the classical and derived context, but specifying a derived $\Theta$-stratum is equivalent to specifying a $\Theta$-stratum in the underlying classical stack, by \Cref{L:classical_stratum}.

A \emph{$\Theta$-stratification} of $\cX$ is an ascending union of open substacks $\cX_{\leq \alpha} \subset \cX$ indexed by $\alpha$ in a totally ordered set $I$, such that $\bigcup_{\alpha'<\alpha} \cX_{\leq \alpha'} \subset \cX_{\leq \alpha}$ is the complement of $\ev_1(\cS_\alpha)$ for some $\Theta$-stratum $\cS_\alpha \subset \Filt(\cX_{\leq \alpha})$ (see \Cref{D:theta_stratification} for a precise definition).

\begin{ex}[Derived Bia\l{}ynicki-Birula stratum] \label{E:stratum_example}
If $X = \Spec(A) / \bG_m$ for some simplicial commutative algebra $A$ over a field $k$, with $\bG_m$-action encoded by a $\bZ$-grading on $A$, then one may present $A$ as a level-wise polynomial simplicial commutative algebra $A_n = k[U_n]$, where $U_n$ is a $\bZ$-graded vector space. If $U_n^{>0}$ denotes the subspace spanned by positive degree generators, $I^n_+ := A_n \cdot U^{>0}_n$ is a simplicial ideal in $A$. The closed derived substack $\Spec(A/I_+) / \bG_m \hookrightarrow \Spec(A) / \bG_m$ can be given the structure of a $\Theta$-stratum (see \Cref{lem:filt_abelian_quotient}). In fact, we will see that all derived $\Theta$-strata are locally of this form (\Cref{T:local_structure_stratum}).
\end{ex}

\subsubsection*{A new perspective on $\Theta$-strata}

We will observe that the derived category of quasi-coherent complexes on any $\Theta$-stratum $\cS$, which we denote $\QC(\cS)$, has a canonical ``weight filtration,'' encoded by the following notion:

\begin{defn} \cite{baric}
A \emph{baric structure} on a stable $\infty$-category $\cC$ is a collection of semiorthogonal decompositions $\cC = \sod{\cC^{<w},\cC^{\geq w}}$ for each $w \in \bZ$, with $\cC^{<w} \subset \cC^{<w+1}$ and $\cC^{\geq w} \subset \cC^{\geq w-1}$. We let $\radj{w}$ and $\ladj{w}$ denote the right adjoint and left adjoint of the inclusions $\cC^{\geq w} \subset \cC$ and $\cC^{< w} \subset \cC$ respectively, and we call them the \emph{baric truncation functors}.
\end{defn}

For a $\Theta$-stratum $\cS$ that is a global quotient stack, we identified a baric structure on $\QC(\cS)$ previously in \cites{halpern2015derived} via explicit computation, but here we will give a different, intrinsic, explanation for this structure.

The stack $\Theta$ is a monoid in the homotopy category of (derived) algebraic stacks, i.e., the $1$-category whose objects are stacks and whose morphisms are $2$-isomorphism classes of $1$-morphisms of stacks. The multiplication map $\Theta \times \Theta \to \Theta$ is given by the map $\bA^1 \times \bA^1 \to \bA^1$ mapping $(t_1,t_2) \mapsto t_1 t_2$, which is equivariant for the group homomorphism $\Gm^2 \to \Gm$ given by the same formula. The identity point is the point $1 \in \Theta$. We refer to an action of $\Theta$ on a (derived) stack $\cS$ in the homotopy category of (derived) stacks as a \emph{weak action of $\Theta$ on $\cS$} (see \Cref{D:weak_theta_action}).

\begin{thmx}
A weak action of $\Theta$ on a derived stack $\cS$ induces a baric structure $\QC(\cS) = \langle \QC(\cS)^{<w},\QC(\cS)^{\geq w}\rangle$ whose truncation functors are defined geometrically (\Cref{P:baric_stratum}). If $\cX$ is an algebraic derived stack locally almost of finite presentation and with affine diagonal (over a fixed base stack), then equipping a closed substack $\cS \hookrightarrow \cX$ with the structure of a $\Theta$-stratum is equivalent to specifying a weak action of $\Theta$ on $\cS$ such that the relative cotangent complex $\bL_{\cS/\cX}$ lies in $\QC(\cS)^{\geq 1}$ (\Cref{P:theta_action_is_filt}).
\end{thmx}

This baric structure is key to our main result, the structure theory for the derived category of a derived algebraic stack with a $\Theta$-stratum. The theorem below generalizes the main theorem of \cite{halpern2015derived} in several directions: it extends the context to derived stacks, it works over an arbitrary noetherian base (e.g., characteristic $p$), it applies to stacks which are not necessarily quotient stacks, and it removes the technical hypotheses $(A)$ and $(L)$ of \cite{halpern2015derived}, which were quite restrictive for non-smooth stacks and difficult to check in practice.

\begin{thmx}[\Cref{P:baric_decomp_supports},\Cref{T:derived_Kirwan_surjectivity}] \label{T:C}
Given an algebraic derived stack $\cX$ locally almost of finite presentation and with affine diagonal (over a fixed locally noetherian base stack), along with a $\Theta$-stratum $i : \cS \hookrightarrow \cX$, there is a unique baric structure on the derived category of quasi-coherent complexes,
\[
\QC(\cX) = \langle \QC(\cX)^{<w}, \QC_\cS(\cX)^{\geq w} \rangle,
\]
such that: i) complexes in the right factor are set-theoretically supported on $\cS$; ii) the baric truncation functors commute with filtered colimits and have locally bounded homological amplitude; and iii) $i_\ast : \QC(\cS) \to \QC(\cX)$ commutes with the baric truncation functors. The baric truncation functors induce a baric structure $\APerf_\cS(\cX) = \langle \APerf_\cS(\cX)^{<w}, \APerf_\cS(\cX)^{\geq w} \rangle$ as well, and for any $w \in \bZ$, these two subcategories form part of a semiorthogonal decomposition
\[
\APerf(\cX) = \langle \APerf_\cS(\cX)^{<w}, \cG^w, \APerf_\cS(\cX)^{\geq w} \rangle,
\]
such that the functor of restriction to the open complement $\cX \setminus \cS$ induces an equivalence
\[
\res : \cG^w \xrightarrow{\cong} \APerf(\cX \setminus \cS).
\]
\end{thmx}

\Cref{T:C} is just a summary of the full statements of \Cref{P:baric_decomp_supports} and \Cref{T:derived_Kirwan_surjectivity}, which are more precise and include much more, such as further decomposition of the semiorthogonal factors, and the interaction of the baric truncation functors with the $t$-structure. We also formulate an immediate extension to $\Theta$-stratifications with multiple strata, \Cref{T:derived_Kirwan_surjectivity_full}.

One consequence of this structure theory is a general version of the ``quantization commutes with reduction'' theorem for the geometric invariant theory quotient of a smooth variety \cite{quantization}.

\begin{prop}[\Cref{P:quantization_commutes_with_reduction}]
Under the hypotheses of \Cref{T:C}, let $F \in \APerf(\cX)$ be such that $i^\ast(F) \in \QC(\cS)^{\geq w}$, and let $G \in \QC(\cX)^{<w}$. Then the restriction map is an equivalence
\[
\RHom_\cX(F,G) \to \RHom_{\cX \setminus \cS}(F|_{\cX\setminus \cS},G|_{\cX \setminus \cS}).
\]
\end{prop}

We have applied this proposition elsewhere to prove a version of the Verlinde formula for the moduli of Higgs bundles on a curve \cite{halpern2016equivariant}.

Although \Cref{T:C} is very general, it only applies to $\APerf(\cX)$ and therefore is not strong enough for our application in \Cref{T:A}. So in \Cref{S:quasi-smooth}, we specialize to the context of algebraic derived stacks over a field of characteristic $0$, so that we can use the theory of ind-coherent sheaves and Grothendieck duality developed in \cites{GR1,GR2}. We establish two variants of the structure theorem under more specialized hypotheses. If the inclusion of the $\Theta$-stratum $\cS \hookrightarrow \cX$ is a regular closed immersion, then a version of the theorem, \Cref{prop:DKS_perfect}, holds for the category of perfect complexes $\Perf(\cX)$.

If $\cX$ is quasi-smooth, meaning locally almost finitely presented and with perfect cotangent complex with Tor-amplitude in $[-1,1]$ (\Cref{D:quasi_smooth}), then a version of the theorem, \Cref{thm:derived_Kirwan_surjectivity_quasi-smooth}, holds for $\DCoh(\cX)$ under an additional hypothesis on the obstruction spaces of points in the center of $\cS$. It is this last variant that we actually use in step (3) and (4) of the proof of \Cref{T:A}.

\subsection*{A comment on homotopical methods}

Our proof of \Cref{T:A} makes essential use of derived algebraic geometry, even though the statement only involves classical varieties. There are two main ways in which homotopical methods are essential to our arguments:\\

First, we make use of the very flexible theory of descent for $\infty$-categories in order to reduce ``global'' statements to local ones. We have already discussed the role that local models play in the proof of \Cref{T:A}. In addition, in \Cref{T:C} we use a local structure theorem for a derived $\Theta$-stratum, \Cref{T:local_structure_stratum}, to reduce the proof of several general claims to a direct verification in the simple case of \Cref{E:stratum_example}.\\

Second, the proof of \Cref{T:C} depends on two key facts about $\cS$: the relative cotangent complex of a $\Theta$-stratum, $\bL_{\cS/\cX}$, has strictly positive weights along $\cS$, and $\cS$ itself is ``contracted'' by a $\Gm$-equivariant action of the monoid $\bA^1$. As \Cref{P:theta_action_is_filt} shows, these criteria uniquely determine the structure of a derived $\Theta$-stratum on $\cS$. In the simple case of \Cref{E:stratum_example}, where $\cX = \Spec(A) / \Gm$ for a classical graded ring $A$ in characteristic $0$, this corresponds to finding a quotient algebra of $A$ that is non-positively graded, and that admits a semi-free resolution as a commutative differential graded $A$-algebra using homogeneous generators of positive weight only. It is not hard to find examples, such as \Cref{E:stratum_derived_structure}, where this is not possible without allowing $B$ to be a commutative differential graded algebra  with non-vanishing higher homology groups.

This observation, that in order to get the ``correct'' deformation theory on the stratum $\cS$, one must leave the world of classical stacks, was the starting point of this paper. As a result, even though the statement of \Cref{T:C} is new in the case where $\cX$ is a \emph{classical} stack -- in fact it is new for stacks of the form $\Spec(A)/\Gm$ as above -- the proof requires working with derived stacks.

\subsection*{Notes and notation}

\subsubsection*{Author's note}

The paper \cite{ballard_wall} serves as a proof-of-concept for using variation of stability to study the derived categories of moduli spaces directly, and it inspired this work. In order to study moduli of sheaves on a $K3$, we needed to build on that method in several ways: by extending the method from smooth stacks to more general derived stacks and to stacks without a known global quotient presentation, and by handling the more complicated wall crossings that arise in this setting.

This paper unifies, generalizes, and replaces two previous unpublished notes: the preprint \cite{halpern2015remarks}, which contains the results of \Cref{S:theta_strat} and \Cref{S:quasi-smooth} for the special case of local quotient stacks over a field of characteristic $0$, and the preprint \cite{halpern2017K3}, which has been publicly available on my website since February 2017. The latter contains a proof of \Cref{T:magic_windows}, on derived equivalences from variation of stability on a derived stack with self-dual cotangent complex, and sketches its application to the $D$-equivalence conjecture. Most of the results of this paper were also announced in \cite{halpern2015theta}. There are several simplifications and corrections in this final version compared to the preliminary versions \cites{halpern2015remarks,halpern2017K3}. Also, I have moved the discussion of the ``virtual non-abelian localization formula'' of \cite{halpern2015remarks}*{Sect.~5} to a separate short paper.

The final version of this paper uses quite a bit of the ``beyond geometric invariant theory'' program, which has grown in unexpected directions and whose development has ultimately delayed this paper. To those who have expressed interest over the years, I apologize for the long lag between the original announcements of these results and the release of this manuscript. I would also like to acknowledge several interesting related developments since the preliminary versions of this work were made public:

\begin{itemize}
\item Yukinobu Toda's theory of $d$-critical flips \cites{toda2019semiorthogonal,toda2019birational} also uses derived algebraic geometry to study how variants of the derived category of certain moduli spaces changes under wall-crossing. Although the approach is closely related to ours, the moduli stacks there are $d$-critical stacks (e.g., sheaves on a $3$-fold), whereas our main application is to quasi-smooth stacks (e.g., sheaves on a surface).

More recently, \cite{halpern2017K3} has inspired work by Koseki-Toda \cite{koseki2020derived} and Toda \cite{toda2020window}, which adapts the methods of this paper to study variation of stability on certain smooth stacks and quasi-smooth derived stacks that do not necessarily have self-dual cotangent complexes. Their application is to categorify wall-crossing formulas that arise in Donaldson-Thomas theory for local Calabi-Yau $3$-folds.\\

\item For a classical scheme $X$ with a $\Gm$-action such that $X/\Gm$ admits a good moduli space, Wai-kit Yeung has developed in \cites{Yeung1, Yeung2, Yeung3} a very nice approach to structure theorems for $\QC(X/\Gm)$ and $\APerf(X/\Gm)$ that makes use of non-commutative geometry rather than derived algebraic geometry. This approach also manages to remove some of the technical hypotheses of \cite{halpern2015derived}. It is expected that in the case of such $X/\Gm$, the semi-orthogonal decompositions obtained via this approach agree with those described in \cite{halpern2015remarks}, and it is possible that much of the general theory of \Cref{S:theta_strat} can be developed along these lines.\\
\end{itemize} 

I would like to thank many colleagues for helpful and encouraging conversations about the content of this paper: Nick Addington, Jarod Alper, Dima Arinkin, Matthew Ballard, Arend Bayer, Bhargav Bhatt, Ben Davison, Dennis Gaitsgory, Daniel Huybrechts, Dmitry Kubrak, Diletta Martinelli, Akhil Mathew, Davesh Maulik, Daniel Pomerleano, Georg Oberdieck, Andrei Okounkov, Alex Perry, David Rydh, Giulia Sacc\`a, Constantin Teleman, Chris Woodward, Yukinobu Toda, Wai-kit Yeung, Xiaolei Zhao, the participants of the Winter 2018/19 SFB seminar on this work at the Universit{\"a}t Bonn, and many people who I have forgotten to mention.

This work was supported by Columbia University, the Institute for Advanced Study, the Mathematical Sciences Research Institute, and Cornell University, as well as the NSF grants DMS-1303960, DMS-1601976, and DMS-1945478. \\

\subsubsection*{Notation}

The term \emph{derived stack} will denote a presheaf of $\infty$-groupoids that satisfies \'etale descent on the $\infty$-site of simplicial commutative algebras, $\sAlg_R^{\rm op}$, over a fixed base ring $R$. In \Cref{S:theta_strat}, the base ring is $\bZ$, and in \Cref{S:quasi-smooth} and \Cref{S:D_equivalence} the base ring is a field $k$ of characteristic $0$.

For any $A \in \sAlg_R$, $\QC(R):= A\Mod$ denotes the $\infty$-category of $dg$-modules over the normalized chain algebra of $A$. For any derived stack $\cX$, we define the $\infty$-category $\QC(\cX)$ of unbounded quasi-coherent complexes on $\cX$ via right Kan extension from the affine case, i.e., $\QC(\cX)$ is the limit under pullback of the $\infty$-categories $\QC(A)$ over all maps $\Spec(A) \to \cX$. $\Perf(\cX)$ and $\AAPerf(\cX)$ denote the full subcategories of $\QC(\cX)$ of perfect and almost perfect complexes \cite{HA}*{Defn.~7.2.5.10} respectively, where the latter is the analogous notion in derived algebraic geometry to pseudo-coherent complexes. For a closed substack $\cS \subset \cX$, we add a subscript, as in $\QC_\cS(\cX)$, to denote the full $\infty$-subcategory of complexes that are set-theoretically supported on $\cS$, i.e., those whose restriction to $\cX \setminus \cS$ vanishes. We typically denote objects of $\QC(\cX)$ in Roman font.

We use homological grading conventions throughout, so for the usual $t$-structure on $\QC(\cX)$, $\tau_{\leq n} F$ is an object that is homologically bounded above ($H_i(F) = 0$ for $i>n$) and $\tau_{\geq n} F$ is homologically bounded below. Subcategories defined by the $t$-structure will be denoted by subscripts, so $\QC(\cX)_{<\infty}$ and $\QC(\cX)_{>-\infty}$ are the full subcategories of complexes that are homologically bounded above and below respectively, and $\QC(\cX)_\heart$ is the category of quasi-coherent sheaves.

When $\cX$ is locally noetherian, meaning it admits a smooth surjection from a disjoint union of noetherian affine derived schemes, we let $\DCoh(\cX)$ (respectively $\APerf(\cX)$) denote the full $\infty$-subcategory of $\QC(\cX)$ consisting of complexes whose homology sheaves are coherent and bounded (respectively coherent and homologically bounded below). We have $\AAPerf(\cX)=\APerf(\cX)$ in this case. Despite the classical notation, these will typically not agree with the derived category of the abelian category of coherent sheaves $\Coh(\cX) := \DCoh(\cX)_\heart$ -- in fact $\Coh(\cX) \cong \Coh(\cX^{\rm cl})$, where the latter denotes the underlying classical stack of $\cX$.

For any morphism of algebraic stacks $f : \cX \to \cY$, there is a cocontinuous pullback functor $f^\ast : \QC(\cY) \to \QC(\cX)$, whose right adjoint we denote $f_\ast$ -- it exists by the adjoint functor theorem, because $\QC(\cX)$ is presentable for any algebraic stack. If $f$ is quasi-compact and quasi-separated and $f_\ast$ has universally bounded cohomological amplitude, then $f_\ast$ is cocontinuous, and the base change and projection formulas hold. Although it is certainly not the original treatment of quasi-coherent sheaves on derived stacks, we will use the same terminology as in \cite{halpern2014mapping}*{App.~A}, and we refer the reader there for further discussion.

We will also use some more exotic constructions: for a closed immersion $i : \cS \to \cX$, we let $i^{\QC,!}$ denote the right adjoint of $i_\ast$, and we let $\inner{\RHom}_\cX^\otimes(-,-)$ denote the inner Hom for the symmetric monoidal $\infty$-category $\QC(\cX)^\otimes$. We caution that the formation of $\inner{\RHom}_\cX^\otimes(F,G)$ and $i^{\QC,!}(G)$ are only smooth-local over $\cX$ if $F \in \APerf(\cX)$ and $G \in \QC(\cX)_{\leq d}$ for some $d$ \cite{preygel2011thomsebastiani}*{Lem.~A.1.1}.

Once we specialize to working over a field of characteristic $0$, in \Cref{S:quasi-smooth} and \Cref{S:D_equivalence}, we will use the theory of ind-coherent sheaves developed in \cites{gaitsgory2013ind,drinfeld2013some,GR1,GR2,arinkin2015singular}. In this case $\IC(\cX)$ denotes the $\infty$-category of Ind-coherent sheaves, and if $f : \cX \to \cY$ is a morphism of derived algebraic stacks, then $f^!$ denotes the shriek pullback $\IC(\cY) \to \IC(\cX)$.

When considering representations of the group $(\Gm)_R$, i.e., $\bZ$-graded $R$-modules, $R\langle w \rangle$ will denote the free graded $R$-module generated in weight $-w$.

Finally, a \emph{semi-orthogonal decomposition} of a stable $\infty$-category $\cC$, which we denote $\cC = \langle \cA, \cB \rangle$, consists of two full stable $\infty$-subcategories $\cA,\cB \subset \cC$ such that $H_0(\RHom(B,A)) = 0$ for any $A\in \cA$ and $B \in \cB$, and objects of $\cA$ and $\cB$ generate the homotopy category of $\cC$ as a triangulated category. Note that we do not assume that $\cA$ or $\cB$ is an admissible subcategory.


\section{Derived \texorpdfstring{$\Theta$}{Theta}-strata and local cohomology}
\label{S:theta_strat}

Throughout this section, we work over a fixed base, an algebraic locally Noetherian derived stack $\cB$, which we often suppress from our notation. All products are formed relative to $\cB$ unless otherwise stated. We develop a structure theory for the derived category of quasi-coherent complexes on a stack with a derived $\Theta$-stratum $\cS \hookrightarrow \cX$. The main result is \Cref{P:baric_decomp_supports}, which defines a functor $\radj{w}_{\cS}(-) : \QC(\cX) \to \QC(\cX)$ that can be regarded as the ``weight $\geq w$'' part of the usual local cohomology functor $R\Gamma_\cS(-)$.

\subsection{Baric structures and weak \texorpdfstring{$\Theta$}{Theta}-actions}

Note that $\Theta$ is a monoidal object in the $\infty$-category of $\cB$-stacks, where the product map $\mu: \Theta \times \Theta \to \Theta$ is induced by the map $\bA^2 \to \bA^1$ taking $(t_1,t_2) \to t_1 t_2$. This is equivariant with respect to the group homomorphism $\Gm^2 \to \Gm$ taking $(z_1,z_2) \mapsto z_1 z_2$. The identity section $\cB \to \Theta$ corresponds to the point $1$ in $\Theta$.

\begin{defn} \label{D:weak_theta_action}
A \emph{weak action} of the monoid $\Theta$ on a derived $\cB$-stack $\cS$ is an action of $\Theta$ on $\cS$ in the homotopy category of $\cB$-stacks. It consists of an action map $a : \Theta \times \cS \to \cS$ along with isomorphisms $a|_{\{1\} \times \cS} \simeq \id_\cS$ and $a \circ (\id_\Theta \times a) \simeq a \circ (\mu \times \id_\cS) : \Theta \times \Theta \times \cS \to \cS$.
\end{defn}

Recall that for any derived stack $\cX$, the stable $\infty$-category of quasi-coherent complexes on $B\Gm \times \cX$ splits as a direct sum $\QC(B\Gm \times \cX) \simeq \oplus_w \QC(\cX)^w$ of copies of $\QC(\cX)$. This splitting can be described geometrically by considering the projection $\pi : B\Gm \times \cX \to \cX$. The inclusion $\QC(\cX)^w \subset \QC(B\Gm \times \cX)$ is equivalent to the functor $\cO_{B\Gm}\langle -w \rangle \otimes \pi^\ast(-)$, and the projection $\QC(B\Gm \times \cX) \to \QC(\cX)^w$ is given by $\pi_\ast(\cO_{B\Gm}\langle w \rangle \otimes (-))$. We say that an object $E \in \QC(B\Gm \times \cX)$ is \emph{concentrated in weight $\geq w$ (respectively $<w$)} if under the direct sum decomposition $E = \oplus_i E_i$ we have $E_i = 0$ for $i<w$ (respectively $i \geq w$).

We shall give an analogous construction of a baric stucture on $\QC(\cS)$ for any stack with a weak $\Theta$-action. Consider the fiber sequence in $\QC(\Theta)$
\[
\cO_\Theta\langle w \rangle \xrightarrow{t^w} \cO_\Theta[t^{-1}] \to \cO_\Theta[t^{-1}] / \cO_\Theta \cdot t^w,
\]
where $t$ is the coordinate of weight $-1$ on $\bA^1$, and $\cO_\Theta[t^{-1}]$ corresponds to the graded $\cO_\cB[t]$-module $\cO_\cB[t^\pm]$. By abuse of notation, we will use the same notation to denote the pullback of this fiber sequence along the projection $\Theta \times \cS \to \Theta$ for any derived $\cB$-stack $\cS$.

\begin{prop} \label{P:baric_stratum}
A weak $\Theta$-action $a : \Theta \times \cS \to \cS$ on a derived $\cB$-stack $\cS$ induces a baric structure $\QC(\cS) = \langle \QC(\cS)^{\geq w}, \QC(\cS)^{<w} \rangle$. The subcategories $\QC(\cS)^{\geq w}$ and $\QC(\cS)^{< w}$ are the essential image of the baric truncation functors $\radj{w}$ and $\ladj{w}$ defined by the fiber sequence
\[
\xymatrix{ \pi_\ast (\cO_\Theta\langle w \rangle \otimes a^\ast(F)) \ar[r]^{t^w} \ar@{=}[d] & \pi_\ast(\cO_\Theta[t^{-1}] \otimes a^\ast(F)) \ar[r] \ar[d]^\simeq & \pi_\ast((\cO_\Theta[t^{-1}] / \cO_\Theta \cdot t^w) \otimes a^\ast(F)) \ar@{=}[d] \\ \radj{w}(F) \ar[r] & F \ar[r] & \ladj{w}(F) },
\]
where $\pi : \Theta \times \cS \to \cS$ is the projection, and the isomorphism $\pi_\ast(\cO_\Theta[t^{-1}] \otimes a^\ast(F)) \to F$ is obtained from the canonical isomorphism $\pi_\ast(\cO_\Theta[t^{-1}] \otimes a^\ast(F)) \simeq a^\ast(F) |_{\{1\} \times \cS}$ along with the given isomorphism $a|_{\{1\}\times \cS} \simeq \id_\cS$. Furthermore, we have:
\begin{enumerate}
\item The baric truncation functors commute with filtered colimits and are right $t$-exact;
\item The canonical map $\colim_{w \to -\infty} \radj{w}(F) \to F$ is an equivalence;
\item $\radj{w}$ and $\ladj{w}$ preserve $\Perf(\cS)$ and $\AAPerf(\cS)$, and in particular they induce baric structures on these subcategories as well;
\item This baric structure is natural with respect to $\Theta$-equivariant maps in the sense that if $\cS$ and $\cS'$ are derived stacks with weak $\Theta$-actions and $f : \cS' \to \cS$ is a morphism such that the diagram
\[
\xymatrix{\Theta \times \cS' \ar[r]^{a'} \ar[d]_{\id_\Theta \times f} & \cS' \ar[d]^f \\ \Theta \times \cS \ar[r]^a & \cS }
\]
is commutative up to isomorphism, then $f^\ast (\QC(\cS)^{\geq w}) \subset \QC(\cS')^{\geq w}$, $f^\ast(\QC(\cS)^{<w}) \subset \QC(\cS')^{<w}$, and $f_\ast(\QC(\cS')^{<w}) \subset \QC(\cS)^{<w}$;
\item $F \in \QC(\cS)$ lies in $\QC(\cS)^{\geq w}$ if and only if $a^\ast(F)|_{(\{0\}/\Gm) \times \cS} \in \QC(B\Gm \times \cS)$ is concentrated in weight $\geq w$, and if $F \in \AAPerf(\cS)$, then $F \in \AAPerf(\cS)^{<w}$ if and only if $a^\ast(F)|_{(\{0\}/\Gm) \times \cS} \in \QC(B\Gm \times \cS)$ is concentrated in weight $<w$;
\item The baric structure is multiplicative in the sense that $\QC(\cS)^{\geq v} \otimes \QC(\cS)^{\geq w} \subset \QC(\cS)^{\geq v+w}$ and $\AAPerf(\cS)^{< v} \otimes \AAPerf(\cS)^{< w} \subset \AAPerf(\cS)^{< v+w-1}$.
\end{enumerate}
\end{prop}

The proof makes use of the classical Reese equivalence, which extends to our $\infty$-categorical context as follows: Let $\underline{\bZ}$ be the free $\infty$-category generated by objects $[w]$ for $w \in \bZ$ and a single arrow $[w+1] \to [w]$ for each $w$. $\underline{\bZ}$ is the nerve of $\bZ$ regarded as a partially ordered set. Then the Reese construction gives an equivalence $\QC(\Theta \times \cS) \simeq \Fun(\underline{\bZ},\QC(\cS))$. The equivalence takes a complex $E \in \QC(\Theta \times \cS)$ to the diagram $[\cdots \to \pi_\ast(\cO_\Theta\langle w+1 \rangle \otimes E) \to \pi_\ast(\cO_\Theta\langle w \rangle \otimes E) \to \cdots]$, and it takes a diagram $[\cdots \to E_{w+1} \to E_w \to \cdots]$ to the graded $\cO_\cS[t]$-module $\bigoplus_w E_w$ with $t$ acting by the given maps $E_{w+1} \to E_w$.

Note that by definition of the baric truncation functors in \Cref{P:baric_stratum}, for $F \in \QC(\cS)$ the object $a^\ast(F) \in \QC(\Theta \times \cS)$ corresponds to
\[
a^\ast(F) \simeq [\cdots \to \radj{w+1}(F) \to \radj{w}(F) \to \cdots]
\]
under the Reese equivalence.

\begin{defn}
If $\cX$ is a derived stack, the \emph{tautological weak action} of $\Theta$ on $\Theta \times \cS$ is given by the multiplication map $a := \mu \times \id_\cX : \Theta \times \Theta \times \cX \to \Theta \times \cX$ along with the identity isomorphism $\mu|_{\{1\} \times \Theta} \simeq \id_\Theta$.
\end{defn}

\begin{lem} \label{L:tautological_baric_structure}
For any derived $\cB$-stack $\cX$ and the tautological weak $\Theta$-action on $\Theta \times \cX$, the subcategories defined in \Cref{P:baric_stratum} define a baric structure on $\QC(\Theta \times \cX)$. Under the Reese equivalence, we have
\[
\radj{w}([\cdots \to F_{w+1} \to F_w \to F_{w-1} \to \cdots]) \simeq [\cdots F_{w+1} \to F_w \to F_w \to \cdots],
\]
where all arrows after the vertex $w$ have been replaced with the identity map, and
\begin{align*}
\QC(\Theta \times \cX)^{\geq w} &= \left\{ \left. [\cdots \to F_{n+1} \to F_n \to \cdots ] \right| F_{n+1} \to F_n \text{ is an isomorphism for } n < w \right\} \\
\QC(\Theta \times \cX)^{< w} &= \left\{ \left. [\cdots \to F_{n+1} \to F_n \to \cdots ] \right| F_n = 0 \text{ for } n\geq w \right\} 
\end{align*}
\end{lem}

\begin{proof}

Given $E \in \QC(\Theta \times \cX)$ corresponding to a diagram $[\cdots \to E_{w+1} \to E_w \to \cdots]$, we would like to identify $(\mu \times \id_\cX)^\ast(E) \in \QC(\Theta \times \Theta \times \cX)$, which by the Reese equivalence corresponds to a diagram indexed by $\underline{\bZ} \times \underline{\bZ}$. Using subscripts to name coordinate functions, we shall factor the multiplication map $\mu : \Theta^2 \to \Theta$ as
\[
\bA^1_{x_1} \times \bA^1_{x_2} / \Gm^2 \xrightarrow{\mu_2} \bA^1_t / \Gm^2 \xrightarrow{\mu_1} \bA^1_t / \Gm.
\]
$\mu_2 : t \mapsto x_1 x_2$ with $\bA^1_t$ regarded as a $\Gm^2$-scheme where $t$ has bidegree $(-1,-1)$. $\mu_1$ is the identity on $\bA^1_t$, which is equivariant with respect to the group homomorphism $(z_1,z_2) \mapsto z_1 z_2$.

Because $\cO_{\bA^1_t}$ is the free graded $\cO_\cB$-algebra on a single generator of degree $(-1,-1)$, we may identify $\QC(\bA^1_t / \Gm^2)$ with the category of bigraded objects $[F_{i,j} \in \QC(\cX)]$ along with a map $t : F_{i+1,j+1} \to F_{i,j}$ for all $i,j$. The pullback $\mu_1^\ast(E)$ corresponds to the diagram with $F_{i,j}=E_{i}$ if $i=j$ and $0$ otherwise, and the map $t$ is the given map $t: E_{i+1} \to E_i$ along the diagonal and $0$ otherwise.

The map $\mu_2$ is an affine morphism corresponding to a homomorphism of $\cO_\cB$-algebras $\cO_{\bA^1_t} \to \cO_{\bA_{x_1}^1 \times \bA_{x_2}^1}$ under which $\cO_{\bA_{x_1}^1 \times \bA_{x_2}^1}$ is a free $\cO_{\bA^1_t}$-module:
\begin{align*}
\cO_{\bA_{x_1}^1 \times \bA_{x_2}^1} &\simeq \bigoplus_{n > 0} \cO_{\bA_t^1} \cdot x_1^n \oplus \cO_{\bA_t^1} \oplus \bigoplus_{m>0} \cO_{\bA_t^1} \cdot x_2^{m} \\
&\simeq \bigoplus_{n > 0} \cO_{\bA_t^1} \langle n,0 \rangle \oplus \cO_{\bA_t^1} \langle 0,0 \rangle \oplus \bigoplus_{m>0} \cO_{\bA_t^1} \langle 0,m \rangle
\end{align*}
Multiplication by $x_1$ corresponds to the identity map $\cO_{\bA_t^1} \langle n,0\rangle \to \cO_{\bA^1_t} \langle n+1,0\rangle$ on the summands $\cO_{\bA_t^1} \langle n,0\rangle$ for $n \geq 0$, and it corresponds to the multiplication by $t$ map $\cO_{\bA_t^1} \langle 0 , m \rangle \to \cO_{\bA_t^1} \langle 0,m-1 \rangle$ on the summands $\cO_{\bA_t^1} \langle 0 , m \rangle$ for $m>0$. Likewise multiplication by $x_2$ corresponds to the multiplication by $t$ map $\cO_{\bA_t^1}\langle n,0 \rangle \to \cO_{\bA_t^1}\langle n-1,0 \rangle$ for $n>0$ and the identity map $\cO_{\bA_t^1}\langle 0,m \rangle \to \cO_{\bA_t^1} \langle 0,m+1 \rangle$ for $m \geq 0$.

From this we deduce that under the Reese equivalence $\QC(\Theta \times \Theta \times \cX) \simeq \Fun(\underline{\bZ} \times \underline{\bZ}, \QC(\cX))$ we have
\[
(\mu \times \id_\cX)^\ast(E) \simeq \left[ \vcenter{\xymatrix{
\ddots & \vdots & \vdots & \vdots & \iddots \\
\cdots \ar[r]^t & E_{w+1} \ar[r]^t \ar[u]^{\id} & E_w \ar[r]^t \ar[u]^{\id} & E_{w-1} \ar[r]^{\id} \ar[u]^{\id} & \cdots \\
\cdots \ar[r]^t & E_{w+1} \ar[r]^t \ar[u]^{\id} & E_w \ar[r]^{\id} \ar[u]^{\id} & E_{w} \ar[r]^{\id} \ar[u]^t & \cdots \\
\cdots \ar[r]^t & E_{w+1} \ar[r]^{\id} \ar[u]^{\id} & E_{w+1} \ar[r]^{\id} \ar[u]^t & E_{w+1} \ar[r]^{\id} \ar[u]^t & \cdots \\
\iddots & \vdots \ar[u]^{t} & \vdots \ar[u]^{t} & \vdots \ar[u]^{t} & \ddots
}} \right],
\]
where the objects shown in the diagram shown range from bidegree $(w-1,w+1)$ in the upper-left corner to bidegree $(w+1,w-1)$ in the lower-right corner. Then $\cO_{\Theta_{x_1}}\langle w \rangle \otimes (\mu \times \id_{\cX})^\ast(E)$ corresponds to the same diagram shifted up by $w$, and
\[
\radj{w}(E) := (\pi_{\Theta_{x_1}})_\ast \left( \cO_{\Theta_{x_1}}\langle w \rangle \otimes (\mu \times \id_{\cX})^\ast(E) \right) \in \QC(\Theta \times \cX)
\]
corresponds to the $0^{th}$ row of this shifted diagram, or equivalently taking row $w$ of the diagram representing $(\mu \times \id)^\ast(E)$. This is the first claim of the lemma.

Now given an object $E \simeq [E_\bullet] \in \QC(\Theta \times \cX)$, let $[E'_\bullet]$ be the diagram where $E'_{v+1} \to E'_{v}$ is the map $E_{v+1} \to E_v$ for $v > w$ and $E_v = 0$ for $v \leq w$. Then we have a cofiber sequence
\begin{equation} \label{E:cofib_seq_1}
[E'_\bullet] \to \radj{w}(E) \to \cO_\Theta \langle -w \rangle \otimes \pi^\ast (E_w),
\end{equation}
where $\pi : \Theta \times \cX \to \cX$ is the projection. For any $G \in \QC(\Theta \times \cX)$, $\ladj{w}(G) = \cofib( \radj{w}(G) \to G)$ corresponds to a diagram indexed by $\underline{\bZ}$ whose terms vanish in weight $\geq w$. It follows that $\pi_\ast(\cO\langle w \rangle \otimes \ladj{w}(G)) = 0$ and thus
\[
\RHom_{\Theta \times \cX}(\cO_\Theta \langle -w \rangle \otimes \pi^\ast (E_w),\ladj{w}(G)) \simeq \RHom_\cX(E_w,\pi_\ast(\cO_\Theta\langle w \rangle \otimes \ladj{w}(G)))= 0.
\]
We also have $\RHom_{\Theta \times \cX}([E'_\bullet],a^\ast(\ladj{w}(G))) = 0$ because the diagram $[E'_\bullet]$ only has non-zero entries in weight $> w$, but $\ladj{w}(G)$ only has non-zero entries in weight $<w$. It follows by applying $\RHom_{\Theta \times \cX}(-,\ladj{w}(G))$ to the fiber sequence \eqref{E:cofib_seq_1} that
\[
\RHom_{\Theta \times \cX} (\radj{w}(E),\ladj{w}(G)) = 0
\]
for any $E,G \in \QC(\Theta \times \cX)$.

It follows from this semiorthogonality and the existence of a cofiber sequence $\radj{w}(F) \to F \to \ladj{w}(F)$ that $\radj{w}$ and $\ladj{w}$ are the projection functors for a semiorthogonal decomposition. Now if $[\cdots \to F_{w+1} \to F_w \to \cdots]$ is a diagram, then under the equivalence $\radj{w}([F_\bullet]) \simeq [\cdots \to F_{w+1} \to F_w \to F_w \to \cdots]$ the canonical map $\radj{w}([F_\bullet]) \to [F_\bullet]$ is homotopic to the identity in weight $\geq w$ and homotopic to the given map $t^{w-n} \colon F_w \to F_n$ in weight $n < w$. It follows that if $t \colon F_{n} \to F_{n-1}$ is an isomorphism for all $n \leq w$, then $[F_\bullet] \in \QC(\Theta \times \cX)^{\geq w}$, and the converse is immediate. Then $[F_\bullet] \in \QC(\Theta \times \cX)^{<w}$ if and only if $\radj{w}([F_\bullet]) = 0$, which happens if and only if $F_n = 0$ for $n \geq w$. These descriptions of the semiorthogonal factors immediately imply $\QC(\Theta \times \cX)^{\geq w+1} \subset \QC(\Theta \times \cX)^{\geq w}$ and $\QC(\Theta \times \cX)^{<w} \subset \QC(\Theta \times \cX)^{<w+1}$, so we have our baric structure.
\end{proof}

\begin{lem} \label{L:theta_equivariant_map}
Let $\cS$ and $\cS'$ be two derived $\cB$-stacks with a weak $\Theta$-action, and let $f : \cS' \to \cS$ be a morphism which commutes with the $\Theta$-action up to isomorphism in the sense of part (4) of \Cref{P:baric_stratum}. Then there are natural isomorphisms $f^\ast(\radj{w}(F)) \simeq \radj{w}(f^\ast(F))$ and $f^\ast(\ladj{w}(F)) \simeq \ladj{w}(f^\ast(F))$ for $F \in \QC(\cS)$.
\end{lem}
\begin{proof}
The flat base change formula for the base change of $\pi : \Theta \times \cS \to \cS$ along the map $\cS' \to \cS$ gives a canonical isomorphism $\pi'_\ast \circ f^\ast \simeq f^\ast \circ \pi_\ast$, where $\pi' : \Theta \times \cS' \to \cS'$ is the projection. The existence of an isomorphism $f^\ast \circ a^\ast \simeq (a')^\ast \circ f^\ast$, where $a' \colon \Theta \times \cS' \to \cS'$ is the action map, follows from the commutativity of the diagram in part (4) of \Cref{P:baric_stratum}. These two observations and the definition of $\radj{w}$ and $\ladj{w}$ in \Cref{P:baric_stratum} imply the claim.
\end{proof}

\begin{proof}[Proof of \Cref{P:baric_stratum}]
Let $i_{\{1\}} : \{1\} \times \cS \to \Theta \times \cS$ denote the inclusion. Under the equivalence $a \circ i_{\{1\}} \simeq \id_\cS$, which is part of the weak $\Theta$-action on $\cS$, the composition
\[
\RHom_\cS(F,G) \xrightarrow{a^\ast} \RHom_{\Theta \times \cS}(a^\ast(F),a^\ast(G)) \to \RHom_\cS (i_{\{1\}}^\ast a^\ast(F),i_{\{1\}}^\ast a^\ast(G)) \simeq \RHom_\cS(F,G)
\]
is homotopic to the identity. Thus $\RHom_{\cS}(F,G)$ is a retract of $\RHom_{\Theta \times \cS}(a^\ast(F),a^\ast(G))$. In particular, the former vanishes if the latter does.

The axioms of a weak $\Theta$-action imply that the map $a : \Theta \times \cS \to \cS$ is compatible with the tautological $\Theta$-action on $\Theta \times \cS$ and the given $\Theta$-action on $\cS$, hence \Cref{L:theta_equivariant_map} implies that $\radj{w}(a^\ast(F)) \simeq a^\ast(\radj{w}(F))$ and $\ladj{w}(a^\ast(F)) \simeq a^\ast(\ladj{w}(F))$, so \Cref{L:tautological_baric_structure} implies that $\RHom_{\Theta \times \cS}(a^\ast(\radj{w}(F)),a^\ast(\ladj{w}(G)))=0$ for all $F,G \in \QC(\cS)$, and thus $\RHom_{\cS}(\radj{w}(F),\ladj{w}(G)) = 0$ as well. This combined with the fiber sequence $\radj{w}(F) \to F \to \ladj{w}(F)$ for any $F$ implies that $\radj{w}$ and $\ladj{w}$ are the projection functors for a semiorthogonal decomposition.

This semiorthogonal decomposition implies that $F \in \QC(\cS)$ lies in $\QC(\cS)^{\geq w}$ if and only if $\RHom_{\cS}(F,G), \forall G \in \QC(\cS)^{<w}$. Therefore the facts that $a^\ast(\QC(\cS)^{<w}) \subset \QC(\Theta \times \cS)^{<w}$ and $\RHom_{\cS}(F,G)$ is a retract of $\RHom_{\Theta \times \cS}(a^\ast(F),a^\ast(G))$ imply that $F \in \QC(\cS)^{\geq w}$ if and only if $a^\ast(F) \in \QC(\Theta \times \cS)^{\geq w}$. The same argument implies that $F \in \QC(\cS)^{<w}$ if and only if $a^\ast(F) \in \QC(\Theta \times \cS)^{<w}$. It follows that $\QC(\cS)^{\geq w+1} \subset \QC(\cS)^{\geq w}$, because the same is true on $\Theta \times \cS$ by \Cref{L:tautological_baric_structure}. Thus we have established the baric structure on $\QC(\cS)$.

\medskip
\noindent{\textit{Claims (1), (2), and (3):}}
\medskip

All three functors $a^\ast(-)$, $\cO_\Theta \langle w \rangle \otimes (-)$, and $\pi_\ast(-)$ commute with filtered colimits, are right $t$-exact, and preserve $\Perf$ and $\AAPerf$, and it follows from the definition that the same holds for $\radj{w}$. The fiber sequence in the statement of the proposition implies that $\ladj{w}$ preserves $\Perf$ and $\AAPerf$, and the fact that $\cO_\Theta [t^{-1}] / \cO_\Theta \cdot t^w \otimes (-)$ commutes with filtered colimits and is right $t$-exact implies the same for $\ladj{w}$. Finally, the isomorphism $\cO_{\Theta}[t_{\pm}] \simeq \colim_{w \to -\infty} \cO_\Theta \langle w \rangle$, where $\cO_\Theta \langle w \rangle$ denotes the submodule $\cO_\Theta \cdot t^{w} \subset \cO_\Theta [t^\pm]$, shows that the canonical map $\colim_{w \to -\infty} \radj{w}(F) \to F$ is an isomorphism

\medskip
\noindent \textit{Claim (4):}
\medskip

The first two statements, that $f^\ast$ preserves the baric structure, follow from \Cref{L:theta_equivariant_map}. The third statement follows from the observation that if $E \in \QC(\cS')^{<w}$, then for any $F \in \QC(\cS)^{\geq w}$, $\RHom_\cS(F,f_\ast(E)) \simeq \RHom_{\cS'}(f^\ast(F), E) = 0$, because $f^\ast(F) \in \QC(\cS')^{\geq w}$.

\medskip
\noindent \textit{Claim (5):}
\medskip

Because $F \in \QC(\cS)^{\geq w}$ if and only if $a^\ast(F) \in \QC(\Theta \times \cS)^{\geq w}$, and likewise for $\QC(\cS)^{<w}$, it suffices to show that $E \in \QC(\Theta \times \cS)$ lies in $\QC(\Theta \times \cS)^{\geq w}$ if and only if $E|_{(\{0\}/\Gm) \times \cS}$ is concentrated in weight $\geq w$, and to show the same for $\QC(\cS)^{<w}$ when $E \in \AAPerf(\Theta \times \cS)$.

If $E \in \QC(\Theta \times \cS)$ corresponds to $[\cdots E_{w+1} \to E_w \to \cdots]$ under the Reese equivalence, then
\[
E|_{(\{0\}/\Gm) \times \cS} \simeq \bigoplus_{i \in \bZ} \cofib(E_{i+1} \to E_i).
\]
It follows from \Cref{L:tautological_baric_structure} that $E \in \QC(\Theta \times \cS)^{\geq w}$ if and only if $E|_{(\{0\}/\Gm) \times \cS}$ is concentrated in weight $\geq w$.

On the other hand $E \in \AAPerf(\Theta \times \cS)$ lies in $\AAPerf(\Theta \times \cS)^{<w}$ if and only if $\radj{w}(E)=0$. By Nakayama's lemma,\footnote{Nakayama's lemma implies that the support of a coherent sheaf is a closed substack of $\Theta \times \cS$. If $E$ were a non-zero almost perfect complex such that $E|_{(\{0\}/\Gm) \times \cS} = 0$, then the support of the first non-vanishing homology sheaf of $E$ would have to be an non-empty closed substack of $\Theta \times \cS$ which does not meet $\{0\}/\Gm \times \cS$. There is no such closed substack, so $E|_{(\{0\}/\Gm) \times \cS} =0$ implies $E=0$.} this is equivalent to
\[
0 = \radj{w}(E)|_{(\{0\}/\Gm) \times \cS} = \oplus_{i \geq w} \cofib(E_{i+1} \to E_i).
\]
It follows that $E \in\AAPerf(\Theta \times \cS)^{<w}$ if $a^\ast(F)|_{(\{0\}/\Gm) \times \cS}$ is concentrated in weight $<w$.

\medskip
\noindent \textit{Claim (6):}
\medskip

As discussed above, any $E \in \QC(B\Gm \times \cS)$ is canonically isomorphic to $\bigoplus_{w \in \bZ} \cO_{B\Gm}\langle -w \rangle \otimes \pi^\ast(E^w)$, where $\pi : B\Gm \times \cS \to \cS$ is the projection, and $E^w \in \QC(\cS)$. It follows from this observation that if $E, F \in \QC(B\Gm \times \cS)$ is such that $E$ is concentrated in weight $\geq w$ (resp. $<w$) and $F$ is concentrated in weight $\geq v$ (resp. $<v$), then $E \otimes F$ is concentrated in weight $\geq v+w$ (resp. $<v+w-1$). This observation combined with claim (5) and the fact that $a^\ast(-)$ and $(-)|_{(\{0\}/\Gm) \times \cS}$ are symmetric monoidal functors implies the claim.
\end{proof}

\begin{rem}
Note that the monoidal structure on $\Theta$ equips the $\infty$-category $\QC(\Theta)$ with the structure of a co-monoidal object in the homotopy category of stable presentable $\infty$-categories. Then the proof of \Cref{P:baric_stratum} can be adapted to show that for any stable presentable $\infty$-category $\cC$, such as $\cC = \QC(\cS)$ for an algebraic derived stack $\cS$, a co-action of $\QC(\Theta)$ on $\cC$ in the homotopy category of stable presentable $\infty$-categories induces a baric structure on $\cC$ whose truncation functors are given by the formula of \autoref{P:baric_stratum}. We speculate that a baric decomposition on $\cC$ satisfying suitable hypotheses is equivalent to a coaction of $\QC(\Theta)$ on $\cC$ in $\Ho(\PrL_{St})$.
\end{rem}

\subsection{The stack of filtered objects and derived \texorpdfstring{$\Theta$}{Theta}-strata}

First let us recall the following

\begin{thm}[\cite{halpern2014mapping}*{Thm.~5.1.1, Rem.~5.1.3, Rem.~5.1.4}] \label{T:mapping_stack}
Let $\cX$ be an algebraic derived stack locally almost of finite presentation over a derived algebraic base stack $\cB$, and assume points of $\cX$ have affine automorphism groups relative to $\cB$. Then $\iMap_\cB(\Theta,\cX)$ and $\iMap_\cB(B\Gm,\cX)$ are algebraic derived stacks locally almost of finite presentation with affine automorphism groups relative to $\cB$. If $\cX$ has quasi-affine or affine diagonal over $\cB$, then so does $\iMap_\cB(\Theta,\cX)$.
\end{thm}

As in \cite{halpern2014structure} we will denote $\Grad(\cX) = \grad{\cX}$ and $\Filt(\cX) = \filt{\cX}$, and refer to these as the stacks of graded and filtered objects in $\cX$, respectively. We will study several canonical maps:
\begin{equation} \label{E:canonical_maps}
\xymatrix{ \Grad(\cX) \ar@/^/[r]^{\sigma} & \Filt(\cX) \ar@/^/[l]^{\ev_0}  \ar[r]^{\ev_1} & \cX }
\end{equation}
$\ev_1$ is the restriction of the tautological map $\ev : \Theta \times \filt{\cX} \to \cX$ to $\{1\} \times \filt{\cX}$, and $\ev_0$ is the map classified by the restriction of $\ev$ to $(\{0\} / \Gm) \times \filt{\cX}$. Finally, $\sigma$ is induced by the projection $\Theta \to B\Gm$, and it is classified by the composition
\[
\Theta \times \grad{\cX} \to B\Gm \times \grad{\cX} \xrightarrow{\ev} \cX.
\]
There is a canonical isomorphism $\ev_0 \circ \sigma \simeq \id_{\grad{\cX}}$.

\begin{defn}
Let $\cX$ be a derived algebraic stack locally almost of finite presentation and with affine automorphism groups relative to $\cB$. A derived $\Theta$-stratum in $\cX$ is union of connected components $\cS \subset \Filt(\cX)$ such that $\ev_1|_{\cS} : \cS \to \cX$ is a closed immersion.
\end{defn}

The underlying classical stack $\iMap(\Theta,\cX)^{\rm{cl}}$ is by definition the restriction of the functor $\filt{\cX}$ to the full sub $\infty$-category of discrete simplicial commutative algebras. This agrees with the functor-of-points definition of the classical mapping stack, i.e.
\[
\iMap(\Theta,\cX)^{\rm{cl}} \simeq \iMap^{\rm{cl}}(\Theta,\cX^{\rm{cl}}).
\]
We studied the latter object in detail in \cite{halpern2014structure}*{Sect.~1}, and many of the results established there for the underlying classical stack immediately imply the corresponding result for the derived mapping stack. For instance, we have:

\begin{lem} \label{L:classical_stratum}
If $\cX$ is a derived stack satisfying the hypotheses of \Cref{T:mapping_stack}, then the map that takes $\cS \subset \Filt(\cX)$ to $\cS^{\rm cl} \subset \iMap^{\rm cl}(\Theta,\cX^{\rm cl})$ induces a bijection between derived $\Theta$-strata in $\cX$ and classical $\Theta$-strata in $\cX^{\rm cl}$.
\end{lem}
\begin{proof}
This follows from the observation that an algebraic derived stack has the same topological space of points as its underlying classical stack, along with the observation that a map of derived stacks $\cS \to \cX$ is a closed immersion if and only if the induced map on underlying classical stacks $\cS^{\rm cl} \to \cX^{\rm cl}$ is a closed immersion \cite{DAGIX}*{Theorem 4.4}.

\end{proof}

\subsubsection{Weak $\Theta$-action on the stack of filtered objects}

For any derived stack $\cX$, pre-composition $(t',f(t)) \mapsto f(t' \cdot t)$ defines a weak action of $\Theta$ on the stack $\filt{\cX}$. More formally:

\begin{prop} \label{P:theta_action_strata}
For any derived stack $\cX$, $\Filt(\cX) = \iMap(\Theta,\cX)$ admits a canonical weak $\Theta$-action whose action map $a : \Theta \times \Filt(\cX) \to \Filt(\cX)$ is classified by the composition
\begin{equation} \label{E:theta_action_filt}
\Theta \times \Theta \times \Filt(\cX) \xrightarrow{\mu \times \id_{\Filt(\cX)}} \Theta \times \Filt(\cX) \xrightarrow{\ev} \cX,
\end{equation}
and this weak $\Theta$-action restricts to any open and closed substack $\cS \subset \Filt(\cX)$.
\end{prop}

\begin{proof}
For the construction of a weak $\Theta$-action on $\Filt(\cX)$, note that $a|_{\{1\} \times \Filt(\cX)} : \Filt(\cX) \to \Filt(\cX)$ is classified by the map $\ev \circ (\mu \times \id_{\Filt(\cX)}) : \Theta \times \{1\} \times \Filt(\cX) \to \cX$, which is canonically isomorphic to the map $\ev$. Hence we have our equivalence $a|_{\{1\} \times \Filt(\cX)} \simeq \id_{\Filt(\cX)}$. Similarly, $a \circ a : \Theta_1 \times \Theta_2 \times \Filt(\cX) \to \Filt(\cX)$ is classified by the map
\[
\Theta \times \Theta_1 \times \Theta_2 \times \Filt(\cX) \xrightarrow{\mu \times \id_{\Theta_2} \times \id_{\Filt(\cX)}} \Theta \times \Theta_2 \times \Filt(\cX) \xrightarrow{\mu \times \id_{\Filt(\cX)}} \Theta \times \Filt(\cX) \xrightarrow{\ev} \cX,
\]
so the isomorphism $a \circ a \simeq a \circ \mu$ follows from the associativity law for the monoidal structure on $\Theta$.

Finally, let $\cS \subset \Filt(\cX)$ be an open and closed substack. Because $a^{-1}(\cS) \subset \Theta \times \Filt(\cX)$ is a closed substack which contains the open substack $\{1\} \times \cS$, its underlying set of points must contain that of $\Theta \times \cS$. Then we must have $\Theta \times \cS \subset a^{-1}(\cS)$, because inclusions of open substacks are determined by underlying sets of points. One can check that the structure of the weak $\Theta$-action on $\Filt(\cX)$ induce associativity and identity isomorphisms for the resulting map $a : \Theta \times \cS \to \cS$.
\end{proof}

\subsubsection{Induced $\Theta$-strata}

The following notion will play an important role beginning in \Cref{P:baric_decomp_supports} below.

\begin{defn} \label{D:induced_stratum}
Let $p : \cX' \to \cX$ be a morphism of algebraic derived stacks that satisfy the hypotheses of \Cref{T:mapping_stack}. We say that a $\Theta$-stratum $\cS \subset \Filt(\cX)$ \emph{induces} a $\Theta$-stratum in $\cX'$ if its preimage $\cS' \subset \Filt(\cX')$ under the canonical map $\Filt(\cX') \to \Filt(\cX)$ is a $\Theta$-stratum, and $\ev_1(|\cS'|) = p^{-1}(\ev_1(|\cS|))$ as subsets of $|\cX'|$.
\end{defn}

\begin{lem} \label{L:induced_stratum_base_change}
Let $\cX$ be a derived algebraic stack locally almost of finite presentation and with affine automorphisms over a base stack, and let $\cS \subset \Filt(\cX)$ be a $\Theta$-stratum. Given two morphisms $\cX' \to \cX$ and $\cX'' \to \cX$ such that $\cS$ induces $\Theta$-strata in $\cX'$ and $\cX''$, $\cS$ also induces a $\Theta$-stratum in the fiber product $\cX' \times_\cX \cX'' \to \cX$.
\end{lem}
\begin{proof}
Let $\cS' \subset \Filt(\cX')$ and $\cS'' \subset \Filt(\cX'')$ denote the preimage of $\cS \subset \Filt(\cX)$. Then the condition that $\cS$ induces a $\Theta$-stratum in each stack $\cX'$ and $\cX''$ is equivalent to the condition that the canonical morphisms $\cS' \to \cS \times_\cX \cX'$ and $\cS'' \to \cS \times_\cX \cX''$ are surjective closed immersions. It follows that $\cS' \times_\cS \cS'' \to \cS \times_{\cX} (\cX' \times_\cX \cX'')$ is a surjective closed immersion. The claim now follows from the fact that $\iMap(\Theta,-)$ commutes with fiber products, so $\cS' \times_\cS \cS''$ can be identified with the preimage in $\Filt(\cX' \times_\cX \cX'')$ of $\cS$.
\end{proof}

\begin{lem} \label{L:induced_stratum_smooth}
Let $\cX$ be a derived algebraic stack locally almost of finite presentation and with affine automorphisms over a base stack, and let $\cS \subset \Filt(\cX)$ be a $\Theta$-stratum. If $p  : \cX' \to \cX$ is a smooth representable morphism that induces a $\Theta$-stratum $\cS' \subset \Filt(\cX')$, then the canonical map is an isomorphism
\[
\cS' \xrightarrow{\cong} \cX' \times_\cX \cS.
\]
\end{lem}
\begin{proof}
By \Cref{C:stratum_smooth} below, whose proof is independent of this lemma, $\cS'$ is smooth over $\cS$. Thus the canonical map $\phi : \cS' \to \cX' \times_\cX \cS$ is a surjective closed immersion of algebraic derived stacks which are smooth and representable over $\cS$. It follows that the relative cotangent complex $\bL_\phi$ is perfect, and passing to fibers over $\cS$, one can see that it is $0$, i.e., the morphism $\phi$ is \'etale. An \'etale surjective closed immersion is an isomorphism.
\end{proof}

\subsection{The cotangent complex of the stack of filtered objects}

For any morphism of finite Tor-amplitude between algebraic derived stacks $\pi : \cY \to \cB$ such that $\pi^\ast : \APerf(\cB) \to \APerf(\cY)$ admits a left adjoint $\pi_+$ and the same holds for the base change of $\pi$ along any morphism $\Spec(R) \to \cB$, the mapping stack $\iMap_\cB(\cY,\cX)$ has a cotangent complex for any derived stack $\cX$ which has a cotangent complex \cite{halpern2014mapping}*{Prop.~5.1.10}. If $\ev : \cY \times_\cB \iMap_\cB(\cY,\cX) \to \cX$ is the tautological evaluation map, then
\begin{equation} \label{E:cotangent_mapping_stack}
\bL_{\iMap_\cB(\cY,\cX)/\cB} \simeq \pi_+(\ev^\ast(\bL_{\cX/\cB})) \in \APerf(\iMap_\cB(\cY,\cX)).
\end{equation}

\begin{lem} \label{L:left_adjoint_pullback}
For any derived stack $\cX$ the functor of pullback along the projection $\pi: \Theta \times \cX \to \cX$ admits a left adjoint $\pi_+ : \QC(\Theta \times \cX) \to \QC(\cX)$ given by $\pi_+(F) := \pi_\ast( (\cO_\Theta[t^{-1}] / \cO_\Theta \cdot t) \otimes F)$.
\end{lem}
\begin{proof}
Under the Reese equivalence $\pi^\ast(G) = [\cdots\to 0 \to 0 \to G \to G \to G\to \cdots]$ where all maps $G \to G$ are the identity and $G$ first appears in weight $0$. The explicit description of the baric structure on $\QC(\Theta \times \cX)$ in \Cref{L:tautological_baric_structure} shows that $\pi^\ast(G) \in \QC(\Theta \times \cX)^{<1}$, and so any map $F \to G$ factors uniquely through $F \to \ladj{1}(F)$. Note however that $\QC(\Theta \times \cX)^{<1}$ is equivalent to the $\infty$-category of diagrams indexed by integers $\leq 0$, and $\pi^\ast(G)$ is the constant object in this diagram category, so if $F \simeq [ \cdots \to F_{n+1} \to F_n \to \cdots ]$, then
\[
\ladj{1}(F) \simeq [\cdots \to 0 \to \cofib(F_1 \to F_0) \to \cofib(F_1\to F_{-1}) \to \cdots]
\]
and we have
\[
\RHom_{\Theta \times \cX}(F,\pi^\ast(G)) \simeq \RHom_\cX(\colim\limits_{n \to -\infty}(\cofib(F_1 \to F_n)), G).
\]
This means that
\begin{align*}
\pi_+(F) &\simeq \colim\limits_{n \to -\infty}(\cofib(F_1 \to F_n)) \\
&\simeq \cofib(F_1 \to \colim\limits_{n \to -\infty}(F_n)) \\
&\simeq \cofib \left(\pi_\ast(\cO_\Theta \langle1\rangle \otimes F) \xrightarrow{t} \pi_\ast(\cO_\Theta[t^{-1}] \otimes F) \right),
\end{align*}
and the claim follows.
\end{proof}

\begin{lem} \label{L:relative_cotangent_complex}
Let $\cX$ be a derived stack that admits a cotangent complex. Then there is an equivalence of canonical fiber sequences $\QC(\Filt(\cX))$
\begin{equation} \label{E:equivalence_fiber_sequences}
\xymatrix{ \radj{1} (\ev_1^\ast \bL_\cX) \ar[r] \ar[d]^\simeq & \ev_1^\ast \bL_\cX \ar[r] \ar@{=}[d] & \ladj{1}( \ev_1^\ast \bL_\cX) \ar[d]^\simeq \\ \bL_{\Filt(\cX)/\cX}[-1] \ar[r] & \ev_1^\ast \bL_\cX \ar[r] & \bL_{\Filt{\cX}} }
\end{equation}
In particular, $\bL_{\Filt(\cX) / \cX} \in \QC(\Filt(\cX))^{\geq 1}$.
\end{lem}

\begin{proof}
Note that composing the action map $a : \Theta \times \Filt(\cX) \to \Filt(\cX)$ with $\ev_1 : \Filt(\cX) \to \cX$ corresponds to restricting the map $ \Theta \times \Theta \times \Filt(\cX) \to \cX$ from \eqref{E:theta_action_filt} classified by $a$ to the substack $\{1 \} \times \Theta \times \Filt(\cX)$. We therefore have a canonical equivalence $\ev_1 \circ a \simeq \ev$ as maps $\Theta \times \Filt(\cX) \to \cX$. Using this observation and \Cref{L:left_adjoint_pullback} gives
\begin{align*}
\ladj{1}(\ev_1^\ast(\bL_\cX)) &:= \pi_\ast \left((\cO_\Theta[t^{-1}] / \cO_\Theta \cdot t) \otimes a^\ast(\ev_1^\ast(\bL_\cX))\right) \\
&\simeq \pi_\ast\left((\cO_\Theta[t^{-1}] / \cO_\Theta \cdot t) \otimes \ev^\ast(\bL_\cX)\right) \simeq \bL_{\Filt(\cX)}.
\end{align*}
To establish the lemma it suffices to show that under this isomorphism the canonical map $\ev_1^\ast(\bL_\cX) \to \ladj{1}(\ev_1^\ast(\bL_\cX))$ is homotopic to the canonical map $D\ev_1 : \ev_1^\ast(\bL_\cX) \to \bL_{\Filt(\cX)}$.

Under the equivalence \eqref{E:cotangent_mapping_stack}, the canonical morphism on $\Theta \times \Filt(\cX)$
\[
D\ev : \ev^\ast(\bL_\cX) \to \bL_{\Theta \times \Filt(\cX)} \cong \bL_\Theta \oplus \pi^\ast(\pi_+(\ev^\ast(\bL_\cX)))
\]
is the $0$ on the first summand and the unit of adjunction $\ev^\ast(\bL_\cX) \to \pi^\ast(\pi_+(\ev^\ast(\bL_\cX)))$ on the second summand. $D\ev_1$ is obtained from $D\ev$ by restricting to $\{1\} \times \Filt(\cX)$ and projecting $\bL_{\Theta \times\Filt(\cX)}$ onto the second summand $\pi^\ast(\pi_+(\ev^\ast(\bL_\cX)))|_{\{1\} \times \Filt(\cX)} \cong \pi_+(\ev^\ast(\cX))$. Under the construction of $\pi_+$ in \Cref{L:left_adjoint_pullback}, for any $F \in \APerf(\Theta \times \Filt(\cX))$ one can identify the canonical map $F|_{\{1\} \times \Filt(\cX)} \to \pi^\ast(\pi_+(F))|_{\{1\} \times \Filt(\cX)} \cong \pi_+(F)$ with the map
\[
\pi_\ast(\cO_\Theta[t^{-1}] \otimes F) \to \pi_\ast((\cO_\Theta[t^{-1}]/\cO_\Theta \cdot t) \otimes F).
\]
$D \ev_1$ is obtained by applying this to $F = \ev^\ast(\bL_\cX)$, which establishes the claim.

\end{proof}

\begin{cor} \label{C:stratum_smooth}
Let $\cX' \to \cX$ be a smooth morphism of algebraic derived stacks that are locally almost of finite presentation and have affine automorphism groups relative to the base stack. Then the canonical map $\Filt(\cX') \to \Filt(\cX)$ is smooth as well. In particular, if $\cX$ is smooth over a classical base, then $\Filt(\cX)$ is smooth as well, and hence classical.
\end{cor}
\begin{proof}
By \Cref{L:relative_cotangent_complex}, and using the fact from \Cref{P:baric_stratum} part (4) that pullback along $\Filt(\cX') \to \Filt(\cX)$ commutes with baric truncation, it suffices to show that $\ladj{1}(\ev_1^\ast(\bL_{\cX'/\cX}))$ is perfect of Tor-amplitude in $[0,1]$. $\ev_1^\ast(\bL_{\cX'/\cX})$ is perfect of Tor-amplitude in $[0,1]$ by hypothesis, and thus $\ladj{1}(\ev_1^\ast(\bL_{\cX'/\cX}))$ is perfect by \Cref{P:baric_stratum} part (3), so it suffices to show that its fibers have homology supported in degrees $[0,1]$.

Any filtration $f : \Theta_k \to \cX'$ specializes, as a point in $\Filt(\cX')$, to its associated split filtration $\sigma(\ev_0(f)) : \Theta_k \to \cX'$, so it suffices to bound the fiber homology at points in the image of $\sigma : \Grad(\cX') \to \Filt(\cX')$. Any such point has a canonical $(\Gm)_{k}$-action, and the resulting map $B(\Gm)_k \to \Filt(\cX')$ is equivariant with respect to the weak $\Theta$-action on each stack. Therefore, by \Cref{P:baric_stratum} part (4), the claim on fiber homology follows from the observation that if $E \in \QC((\Gm)_{k})$ has homology in degree $0$ and $1$ only, then the same is true for $\ladj{1}(E)$.
\end{proof}

\subsection{An intrinsic characterization of a \texorpdfstring{$\Theta$}{Theta}-stratum}

Say that $i : \cS \hookrightarrow \cX$ is the closed substack underlying a $\Theta$-stratum, i.e. we have specified an open and closed immersion $\cS \subset \Filt(\cX)$ under which $i$ is identified with $\ev_1$. \Cref{P:theta_action_strata} shows that $\cS$ carries a canonical weak $\Theta$-action from its embedding in $\Filt(\cX)$, and \Cref{L:relative_cotangent_complex} shows that $\bL_{\cS/\cX} \in \QC(\cS)^{\geq 1}$ with respect to the baric structure defined by this canonical weak $\Theta$-action. We establish a converse to these results under suitable hypotheses on $\cX$. It shows, in particular, that specifying the structure of a $\Theta$-stratum on a closed substack is equivalent to specifying a suitable weak $\Theta$-action.

\begin{prop} \label{P:theta_action_is_filt}
Let $\cX$ be a derived algebraic stack locally almost of finite presentation and with affine diagonal relative to a base stack $\cB$,\footnote{All that is required for the proof is that $\Filt(\cX)$ is algebraic and $\ev_1 : \Filt(\cX) \to \cX$ is separated.} and let $i : \cS \to \cX$ be a closed immersion. Then for any weak $\Theta$-action on $\cS$, the action map $a : \Theta \times \cS \to \cS$ composed with $i$ is classified by a morphism $$\phi : \cS \to \iMap(\Theta,\cX)$$ which is a closed immersion and $\Theta$-equivariant for the weak $\Theta$-action on $\Filt(\cX)$ described in \Cref{P:theta_action_strata}. If $\bL_{\cS/\cX} \in \QC(\cS)^{\geq 1}$, then $\phi$ is an open immersion as well and hence identifies $\cS$ with a $\Theta$-stratum in $\cX$.
\end{prop}

\begin{proof}

\noindent \textit{Step 1: $f$ is $\Theta$-equivariant.}
\medskip

The composition $a_{\Filt(\cX)} \circ (\id \times \phi) \colon \Theta \times \cS \to \Theta \times \Filt(\cX) \to \Filt(\cX)$ classifies the map $i \circ a_{\cS} \circ (\mu \times \id_{\cS}) : \Theta \times \Theta \times \cS \to \cX$, whereas the composition $\phi \circ a_{\cS} : \Theta \times \cS \to \Filt(\cX)$ classifies the map $i \circ a_{\cS} \circ (\id_\Theta \times a_{\cS}) : \Theta \times \Theta \times \cS \to \cX$. These two maps are isomorphic by the definition of a weak $\Theta$-action.

\medskip
\noindent \textit{Step 2: $\phi$ is a closed immersion.}
\medskip

Note that the isomorphism $a|_{\{1\} \times \cS} \simeq \id_{\cS}$ gives an isomorphism $\ev_1 \circ \phi \simeq i$. It follows that for any simplicial commutative ring $R$ the map $\phi_R : \cS(R) \to \Filt(\cX)(R)$ is injective on $\pi_0$ and an isomorpism on $\pi_1$ for any base point in $\cS(R)$, because the same holds after composition with the map $\ev_1 : \Filt(\cX) \to \cX$, which is representable \cite{halpern2014structure}*{Lem.~1.1.13}. This implies that the induced map of underlying classical stacks $\phi^{\rm cl} \colon \cS^{\rm{cl}} \to \Filt(\cX)^{\rm{cl}}$ is a finitely presented representable monomorphism, and it suffices to show that this is a closed immersion.

This can be checked locally over $\cX$, so for any morphism $\xi : \Spec(R) \to \cX$ we can consider $\Flag(\xi) := \Filt(\cX) \times_{\cX} \Spec(R)$, and we have a commutative diagram
\[
\xymatrix{\cS_R \ar[rr]^{\phi_R} \ar[dr]^{i_R} & & \Flag(\xi) \ar[dl]^{(\ev_1)_R} \\ & \Spec(R) & },
\]
in which $i_R$ is a closed immersion. $(\ev_1)_R$ is separated because $\cX$ has affine diagonal over the base \cite{halpern2014structure}*{Prop.~1.5.3}, so $\phi_R$ is a closed immersion as well.

\medskip
\noindent \textit{Step 3: $f$ is an open immersion if $\bL_{\cS/\cX} \in \QC(\cS)^{\geq 1}$.}
\medskip

It suffices to show that $D\phi : \phi^\ast(\bL_{\Filt(\cX)}) \to \bL_\cS$ is an isomorphism. First observe the following:

\begin{lem}\label{L:cotangent_complex_theta_action}
Let $\cS$ be a derived stack which admits a cotangent complex. If $\cS$ has a weak $\Theta$-action, then $\bL_{\cS} \in \QC(\cS)^{<1}$ for the baric structure of \Cref{P:baric_stratum}.
\end{lem}
\begin{proof}
For any map $Y \to \Filt(\cS)$ classifying a map $f : \Theta \times Y \to \cS$, we have $\bL_{\Filt(\cS)}|_{Y} \simeq \pi_+(f^\ast(\bL_\cS))$ by \eqref{E:cotangent_mapping_stack}. Applying this to the map $\phi : \cS \to \Filt(\cS)$ associated to the action map $a : \Theta \times \cS \to \cS$, we have
\[
\phi^\ast (\bL_{\Filt(\cS)}) \simeq \pi_+(a^\ast(\bL_\cS)) \simeq \ladj{0}(\bL_\cS)
\]
by \Cref{L:left_adjoint_pullback}. The fact that $\ev_1 \circ \phi \simeq \id_{\cS}$ implies that $\bL_\cS$ is a retract of $\phi^\ast (\bL_{\Filt(\cS)})$. But $\phi$ is $\Theta$-equivariant, and $\bL_{\Filt(\cS)} \in \QC(\Filt(\cS))^{<1}$ by \Cref{L:relative_cotangent_complex}, hence $\bL_{\cS} \in \QC(\cS)^{<1}$ by part (4) of \Cref{P:baric_stratum}.
\end{proof}

To complete the proof of \Cref{P:theta_action_is_filt}, note that the isomorphism $\ev_1 \circ \phi \simeq i$ implies that the composition
\[
i^\ast(\bL_\cX) \simeq \phi^\ast \ev_1^\ast(\bL_\cX) \xrightarrow{\phi^\ast(D\ev_1)} \phi^\ast(\bL_{\Filt(\cX)}) \xrightarrow{D\phi} \bL_\cS
\]
is isomorphic to the canonical map $Di : i^\ast (\bL_\cX) \to \bL_\cS$. If we assume that $\bL_{\cS/\cX} \in \QC(\cS)^{\geq 1}$, then $Di$ becomes an isomorphism after applying the functor $\ladj{1}(-)$. It follows that the composition above becomes an isomorphism after applying $\ladj{1}(-)$.

\Cref{L:relative_cotangent_complex} and part (4) of \Cref{P:baric_stratum} imply that $\phi^\ast(D\ev_1)$ is isomorphic to the canonical map $\phi^\ast \ev_1^\ast (\bL_\cS) \to\ladj{1}(\phi^\ast \ev_1^\ast(\bL_\cS))$, and in particular $\ladj{1}(\phi^\ast(D\ev_1))$ is an isomorphism. It follows that $\ladj{1}(D\phi):\ladj{1}(\phi^\ast(\bL_{\Filt(\cX)})) \to \ladj{1}(\bL_\cS)$ is an isomorphism. But as we have already observed, $\phi^\ast(\bL_{\Filt(\cX)}) \in \QC(\cS)^{<1}$ already, and $\bL_\cS \in \QC(\cS)^{<1}$ by \Cref{L:cotangent_complex_theta_action}. It follows that $D\phi$ is an isomorphism.

\end{proof}

In the case where $i : \cS \to \cX$ is the identity, \Cref{P:theta_action_is_filt} states that any weak $\Theta$-action on $\cS$ comes from a $\Theta$-equivariant open and closed embedding $\cS \subset \Filt(\cS)$. Although we have not introduced a ``homotopy coherent'' notion of a $\Theta$-action, the $\Theta$-action on $\Filt(\cS)$ would certainly qualify as homotopy coherent for any sensible definition. Therefore our result shows that there is no difference between a homotopy coherent $\Theta$-action and the notion introduced in \Cref{D:weak_theta_action}.

Furthermore, the data of a weak $\Theta$-action on $\cS$ is ``discrete'' in the following sense:

\begin{cor} \label{C:intrinsic_theta_action}
If $\cS$ is an algebraic derived stack locally of finite presentation and with affine diagonal over a base stack, then specifying a weak $\Theta$-action on $\cS$ is equivalent to specifying an open and closed substack $\cY \subset \Filt(\cS)$ such that $\ev_1 : \cY \to \cS$ is an isomorphism.
\end{cor}

\begin{proof}
Applying \Cref{P:theta_action_is_filt} to the identity map $i : \cS \to \cS$, we see that $\ev_1$ is an isomorphism between $\phi(\cS) \subset \Filt(\cS)$ and $\cS$. Conversely if $\cY \subset \Filt(\cS)$ is an open an closed substack such that $\ev_1$ induces an isomorphism $\ev_1 : \cY \simeq \cS$, then the canonical weak $\Theta$-action on $\Filt(\cS)$ induces a weak $\Theta$-action on $\cY$ and hence on $\cS$. In this case the embedding $\phi: \cS \to \Filt(\cS)$ constructed in \Cref{P:theta_action_is_filt} is the inverse of the isomorphism $\ev_1$.
\end{proof}

\subsection{The center of a derived \texorpdfstring{$\Theta$}{Theta}-stratum}

Recall the stack of graded objects in $\cX$, $\Grad(\cX) := \grad{\cX}$, along with the canonical morphism $\sigma : \Grad(\cX) \to \Filt(\cX)$ associated to the projection $\Theta \to B\Gm$ from \eqref{E:canonical_maps}. $\sigma$ is the map which takes a graded object to the corresponding split filtration in $\cX$.

\begin{defn} \label{D:center}
If $\cX$ is an algebraic stack locally almost of finite presentation and with affine stabilizers over $\cB$, and $\cS \subset \Filt(\cX) \to \cX$ is a $\Theta$-stratum, the \emph{center} of $\cS$ is the open and closed substack $\cZ := \sigma^{-1}(\cS) \subset \grad{\cX}$.
\end{defn}

By construction the morphism $\sigma : \Grad(\cX) \to \Filt(\cX)$ restricts to a morphism $\cZ \to \cS$, and the projection $\ev_0 : \Filt(\cX) \to \Grad(\cX)$ restricts to a morphism $\ev_0 : \cS \to \cZ$ as well. We will also see that the center is associated intrinsically to the weak $\Theta$-action on $\cS$:

\begin{lem}
If $\cX$ has affine diagonal and $\cS \subset \Filt(\cX) \to \cX$ is a $\Theta$-stratum, and $\phi : \cS \subset \Filt(\cS)$ is the open and closed embedding of \Cref{C:intrinsic_theta_action} induced by the canonical weak $\Theta$-action on $\cS$, then the center of $\cS$ is canonically isomorphic to $\cZ \cong \sigma^{-1}(\phi(\cS)) \subset \Grad(\cS)$.
\end{lem}

\begin{proof}

Even though $\cZ$ is defined as an open and closed substack of $\Grad(\cX)$, the embedding $\cZ \subset \Grad(\cX)$ factors through the closed embedding $\Grad(\cS) \hookrightarrow \Grad(\cX)$ \cite{halpern2014structure}*{Cor.~1.1.8}, so we can canonically regard $\cZ$ as an open and closed substack of $\Grad(\cS)$. The claim follows from the fact that the embedding $\cS \subset \Filt(\cX)$ is the composition of the canonical embedding $\cS \subset \Filt(\cS)$ associated to the weak $\Theta$-action with the closed embedding $\Filt(\cS) \hookrightarrow \Filt(\cX)$.

\end{proof}

The stack $B\Gm$ is also a monoidal object in the homotopy category of the $\infty$-category of stacks. As in \Cref{P:theta_action_strata}, the stack of graded objects $\Grad(\cX) := \grad{\cX}$ admits a canonical weak action of the monoid $B\Gm$, where the action map $B\Gm \times \Grad(\cX) \to \Grad(\cX)$ classifies the composition
\[
B\Gm \times B\Gm \times \Grad(\cX) \xrightarrow{\mu \times \id_{\Grad(\cX)}} B\Gm \times \Grad(\cX) \xrightarrow{\ev} \cX.
\]
This weak action of $B\Gm$ restricts to a weak action on any open and closed substack $\cZ \subset \Grad(\cX)$.

The projection map $\Theta=\bA^1 /\Gm \to B\Gm$ is a map of monoids, and thus we can equip $\Grad(\cX)$ with a weak action of $\Theta$, where the action map is the composition $\Theta \times \Grad(\cX) \to B\Gm \times \Grad(\cX) \to \Grad(\cX)$.

\begin{lem} \label{L:center_grading}
If $\cZ$ is a derived stack with a weak $B \Gm$-action $a : B\Gm \times \cZ \to \cZ$, then there is a canonical direct sum decomposition $\QC(\cZ) = \bigoplus_{w \in \bZ} \QC(\cZ)^w$, where the projection onto $\QC(\cZ)^w$ is given by $F \mapsto F_w := \pi_\ast(\cO_{B\Gm}\langle w \rangle \otimes a^\ast(F))$. The baric structure corresponding to the weak $\Theta$-action induced by the projection $\Theta \to B\Gm$ has $\QC(\cZ)^{\geq w} = \bigoplus_{n \geq w} \QC(\cZ)^w$ and $\QC(\cZ)^{<w} = \bigoplus_{n < w} \QC(\cZ)^w$.
\end{lem}
\begin{proof}
One could immitate the proof of \Cref{P:baric_stratum}, but instead we will use that result to give a short argument. Let $p : \Theta \to B\Gm$ be the projection, which leads to an action map via the composition
\[
\xymatrix{\Theta \times \cZ \ar[r]^-p & B\Gm \times \cZ \ar[r]^-a & \cZ}
\]
Then the baric truncation functors of \Cref{P:baric_stratum} are given by
\begin{align*}
\radj{w}(F) &=\pi_\ast((p_\ast \cO_\Theta) \otimes \cO_{B\Gm}\langle w \rangle \otimes a^\ast(F)) = \oplus_{n \geq w} F_w, \text{ and} \\
\ladj{w}(F) &=\pi_\ast( p_\ast(\cO_\Theta[t^{-1}] / \cO_\Theta) \otimes \cO_{B\Gm}\langle w \rangle \otimes a^\ast(F)) =  \oplus_{n<w} F_w.
\end{align*}
Note in particular that by part (2) of \Cref{P:baric_stratum} the canonical map $\bigoplus_{w \in \bZ} F_w \to F$ is an isomorphism.

If we let $p^- : \Theta \to B \Gm$ denote the projection map followed by the automorphism of $B\Gm$ given by the group homomorphism $z \mapsto z^{-1}$, then one finds that the resulting baric decomposition has $\radj{w}(F) = \bigoplus_{n<-w+1} F_w$ and $\ladj{w}(F) = \bigoplus_{n \geq -w+1} F_w$. Using this one can show that $\RHom(F_w,G_v)=0$ for any $F,G \in \QC(\cZ)$ and any $v \neq w$. This implies that $\QC(\cZ)$ is a direct sum of the categories $\QC(\cZ)^w$, which are defined as the essential image of the projection functors $F \mapsto F_w$.

\end{proof}

\begin{lem} \label{L:aperf_baric_criterion}
Let $\cS$ be a $\Theta$-stratum with center $\cZ$. A complex $F \in \APerf(\cS)$ lies in $\APerf(\cS)^{\geq w}$ (respectively $\APerf(\cS)^{<w}$) if and only if $\sigma^\ast(F) \in \APerf(\cZ)^{\geq w}$ (respectively $\sigma^\ast(F) \in \APerf(\cZ)^{<w}$).
\end{lem}
\begin{proof}
Both $\ev_0 : \Filt(\cX) \to \Grad(\cX)$ and $\sigma : \Grad(\cX) \to \Filt(\cX)$ are $\Theta$-equivariant, so they preserve the baric structure by \Cref{P:baric_stratum} part (4). For the weak $\Theta$-action on $\Filt(\cX)$ defined in \Cref{P:theta_action_strata} and for the corresponding weak action of $B\Gm$ on $\Grad(\cX)$, we have a commutative diagram
\[
\xymatrix{B\Gm \times \Filt(\cX) \ar[r]^{\id_{B\Gm} \times \ev_0} \ar[d] & B\Gm \times \Grad(\cX) \ar[r]^a & \Grad(\cX) \ar[d]^\sigma \\ \Theta \times \Filt(\cX) \ar[rr]^a & & \Filt(\cX) }.
\]
It follows that for $F \in \APerf(\cS)$, $a^\ast(F)|_{B\Gm \times \cS}$ is concentrated in weight $\geq w$ (respectively $<w$) if and only if $(\id_{B\Gm} \times \ev_0)^\ast a^\ast \sigma^\ast(F)$ is concentrated in weight $\geq w$ (respectively $<w$). Because the map $\ev_0 : \Filt(\cX) \to \Grad(\cX)$ has a section $\sigma$, this happens if and only if $a^\ast(\sigma^\ast(F))$ has weights concentrated in the corresponding range. By definition of the graded structure on $\QC(\cZ)$, this is equivalent to the condition that $\sigma^\ast(F) \in \QC(\cZ)^{\geq w}$ (respectively $\sigma^\ast(F) \in \QC(\cZ)^{<w}$). The claim then follows from the characterization of $\APerf(\cS)^{<w}$ and $\APerf(\cS)^{\geq w}$ in \Cref{P:baric_stratum} part (5).
\end{proof}

We can also describe the cotangent complex of the center:

\begin{lem} \label{L:center_cotangent}
Let $\cX$ be a derived stack which admits a cotangent complex, and consider the canonical map $\sigma : \Grad(\cX) \to \Filt(\cX)$. Then there is a canonical equivalence of fiber sequences in $\QC(\Grad(\cX))$
\[
\xymatrix{\ladj{0}(\sigma^\ast \bL_{\Filt(\cX)}) \ar[r] \ar[d] & \sigma^\ast \bL_{\Filt(\cX)} \ar[r] \ar@{=}[d] & \radj{0}(\sigma^\ast \bL_{\Filt(\cX)}) \ar[r] \ar[d] & \\ \bL_{\Grad(\cX)/\Filt(\cX)}[-1] \ar[r] & \sigma^\ast \bL_{\Filt(\cX)} \ar[r] & \bL_{\Grad(\cX)} \ar[r] & }
\]
\end{lem}
\begin{proof}
The pullback along the projection map $\pi^\ast : \QC(\Grad(\cX)) \to \QC(B\Gm \times \Grad(\cX))$ has a left adjoint $\pi_+$ which can be identified with $\pi_\ast$. From the general formula \eqref{E:cotangent_mapping_stack} for the cotangent complex of a mapping stack $\bL_{\iMap(B\Gm,\cX)} \simeq \pi_\ast(\ev^\ast(\bL_\cX))$. If $u : \Grad(\cX) \to \cX$ is the forgetful morphism, then
\[
u \circ a \cong \ev : B\Gm \times \Grad(\cX) \to \cX,
\]
so $\bL_{\Grad(\cX)}$ is the projection of $u^\ast(\bL_\cX)$ onto the weight $0$ piece of the direct sum decomposition of $\Grad(\cX)$ defined in \Cref{L:center_grading}.

Note that by considering the opposite of the canonical weak $B\Gm$-action on $\Grad(\cX)$ we see that $\radj{0}(-)$ is also a \emph{left} adjoint of the inclusion $\QC(\Grad(\cX))^{\geq 0} \subset \QC(\Grad(\cX))$, so the canonical map $D\sigma : \sigma^\ast(\bL_{\Filt(\cX)}) \to \bL_{\Grad(\cX)}$ factors uniquely through a morphism
\begin{equation} \label{E:factorization_cotangent_grad}
\radj{0}(\sigma^\ast(\bL_{\Filt(\cX)})) \to \bL_{\Grad(\cX)}.
\end{equation}
We note that $\ev_1 \circ \sigma \cong u : \Grad(\cX) \to \cX$, and this implies that $D\sigma \circ \sigma^\ast(D\ev_1) \cong Du : u^\ast(\bL_\cX) \to \bL_{\Grad(\cX)}$. When we restrict to the direct summand of weight $0$, $\sigma^\ast(D\ev_1)$ and $Du$ are isomorphisms, and it follows that so is $D\sigma$. \Cref{L:relative_cotangent_complex} implies $\sigma^\ast(\bL_{\Filt(\cX)}) \in \QC(\Grad(\cX))^{<1}$, and therefore \eqref{E:factorization_cotangent_grad} is an isomorphism.

\end{proof}

\begin{lem} \label{L:center_pullback}
Let $\cX$ be an algebraic derived stack locally almost of finite presentation and with affine diagonal over a noetherian base stack, let $\cS \hookrightarrow \cX$ be a $\Theta$-stratum with center $\cZ$, and let $\APerf(\cS)^w := \APerf(\cS)^{\geq w} \cap \APerf(\cS)^{<w+1}$. Then the pullback functor $\ev_0^\ast : \QC(\cZ) \to \QC(\cS)$ induces an equivalence $\QC(\cZ)^w \simeq \QC(\cS)^w$, whose inverse is given by $\radj{w}((\ev_0)_\ast(-))$.
\end{lem}

We will postpone the proof of this last lemma until \Cref{S:local_structure}, because it relies on the local structure theorem \Cref{T:local_structure_stratum}. It is not used until \Cref{S:structure_theorem}.


\subsection{Example: Filtered objects in a quotient stack}

Fix a ground field $k$, and let $G$ be an affine algebraic group acting on an affine derived scheme $\Spec(A)$. Let $\lambda : \Gm \to G$ be a one parameter subgroup, and let $P_\lambda$ be the associated block upper triangular group, i.e. the subgroup of $g \in G$ such that $\lim_{t\to 0} \lambda(t) g \lambda(t)^{-1}$ exists. Then by \cite{halpern2014structure}*{Thm.~1.4.8}, the stack $\Spec(\pi_0(A)/I_+) / P_\lambda$ is a union of connected components of $\Filt(\Spec(\pi_0(A))/G)$, and if $G$ has a split maximal torus then all connected components arise in this way for a unique conjugacy class of $\lambda$.

Now choose a quasi-isomorphism of simplicial commutative algebras $A_\bullet \simeq A'_\bullet$ where each $A'_n = k[U_n]$ is a polynomial ring generated by some $G$-representation $U_n$. We let $U_n^{<1}$ be the quotient of $U_n$ by the subspace of positive $\lambda$-weight $U_n^{\geq 1}$. We have $k[U_n^{<1}] = k[U_n] / (I_+)_n$, where $(I_+)_n := k[U_n] \cdot U^{\geq 1}$. These quotient rings form a simplical commutative algebra $B_\bullet = k[U_\bullet^{<1}] = k[U_\bullet] / (I_+)_\bullet$ which inherits an action of $P_\lambda$.

\begin{lem} \label{L:filt_quotient_stack}
There is a canonical map $\phi : \Spec(B_\bullet) / P_\lambda \to \Filt(\Spec(A_\bullet)/G)$ which is an open and closed immersion, and if $G$ is split then every connected component lies in a substack of this form for a unique conjugacy class of $\lambda$.
\end{lem}

To construct this map, we will construct a $\Gm \times P_\lambda$-equivariant $G$-bundle on $\bA^1 \times \Spec(B_\bullet)$ along with a $G$-equivariant map to $\Spec(A_\bullet)$, following \cite{halpern2014structure}*{Sect.~1.4.3}. For each $n$, we equip $\Spec(k[t] \otimes k[U_n^{<1}] \otimes k[G])$ with a $\Gm \times P_\lambda$-action by
\[
(t,p) \cdot (z,x,g) = (tx, p\cdot x, \lambda(tz) p \lambda(tz)^{-1} g),
\]
which commutes with the action of $G$ by right multiplication. In addition, the map $\Spec(k[t] \otimes k[U_n^{<1}] \otimes k[G]) \to \Spec(k[U_n])$ defined by
\[
(z,x,g) \mapsto g^{-1} \lambda(z) \cdot x
\]
is $G$ equivariant and $\Gm \times P_\lambda$ invariant. This equivariant structure and these level-wise maps are compatible with the face and degeneracy maps in $B_\bullet$ and $A_\bullet$, so we have a map of simplicial commutative algebra $A_\bullet \to k[t] \otimes B_\bullet \otimes k[G]$ which is $G$-equivariant and invariant for the $\Gm \times P_\lambda$ action on the target. Our canonical map $\phi : \Spec(B_\bullet)/P_\lambda \to \Filt(\Spec(A_\bullet) / G)$ by definition classifies the resulting map
\[
\bA^1 \times \Spec(B_\bullet) / \Gm \times P_\lambda \simeq \bA^1 \times \Spec(B_\bullet) \times G / \Gm \times P_\lambda \times G \to \Spec(A_\bullet) / G.
\]

\begin{proof}[Proof of \Cref{L:filt_quotient_stack}]
$\pi_0(B_\bullet) \simeq \pi_0(A) / I_+$, and the canonical map above restricts to the canonical map $\Spec(\pi_0(A)/I_+)/P_\lambda \to \Filt(\Spec(\pi_0(A))/G)$ which is shown to be an open and closed immersion in \cite{halpern2014structure}*{Theorem 1.4.8}. It follows that the canonical map in the derived setting is a closed immersion, and it suffices to show that $\phi$ induces an isomorphism $\bL_{\Spec(B_\bullet)/P_\lambda} \to \phi^\ast \bL_{\Filt(\Spec(A_\bullet)/G)}$. Because both complexes are almost perfect, it suffices by Nakayama's lemma to verify that the fiber of this map is an isomorphism at every $\lambda(\Gm)$-fixed point $p \in \Spec(\pi_0(B_\bullet))$.

Consider the map $\psi : \Spec(B_\bullet)/P_\lambda \to \Spec(A_\bullet / G)$ induced by the $(P_\lambda \to G)$-equivariant map $\Spec(B_\bullet) \to \Spec(A_\bullet)$. The isomorphism $\psi \simeq \ev_1 \circ \phi$ gives a factorization
\[
\xymatrix{ \psi^\ast \bL_{\Spec(A_\bullet)/G} \ar[rr]^{D\psi} \ar[dr]_{\phi^\ast(D\ev_1)} & & \bL_{\Spec(B_\bullet)/P_\lambda} \\ & \phi^\ast \bL_{\Filt(\Spec(A_\bullet)/G)} \ar[ur]_{D\phi} & }.
\]
From the explicit description of $A_\bullet$ and $B_\bullet$ as simplicial polynomial algebras, $\psi^\ast ( \bL_{\Spec(A_\bullet)/G}) \otimes k(p)$ has the form $\cdots \to U_2\otimes k(p) \to U_1 \otimes k(p) \to U_0 \otimes k(p) \to (\fg)^{\dual} \otimes k(p)$, and the map $D\psi$ is precisely the inclusion of the subcomplex of weight $\leq 0$, $\cdots \to U_2^{<1} \to U_1^{<1} \to U_0^{<1} \to (\mathfrak{p}_\lambda)^\dual$. The same is true for the map $D\ev_1$ by \Cref{L:relative_cotangent_complex}, so it follows that $D \phi$ is an isomorphism.
\end{proof}

\begin{rem}
A similar argument deduces that $\Grad(\Spec(A_\bullet)/G)$ is a disjoint union of stacks of the form $\Spec(A_\bullet / (I_+)_\bullet + (I_-)_\bullet) / L_\lambda$, where $L_\lambda$ is the centralizer of $\lambda$, using the corresponding classical statement in \cite{halpern2014structure}*{Thm.~1.4.8}.
\end{rem}

More generally, let $R_\bullet$ be a simplicial commutative algebra, and let $A_\bullet$ is a simplicial commutative $R_\bullet$ algebra with a $(\Gm^n)_{R_\bullet}$-action, i.e., $A_\bullet$ is a commutative algbera in $\bZ^n$-graded simplicial $R_\bullet$-modules. As above, we choose a quasi-isomorphism with a semi-free $R_\bullet$-algebra $A_\bullet \cong R_\bullet[U_\bullet]$. For any cocharacter $\lambda : \Gm \to \Gm^n$ we can define $B_\bullet = R_\bullet [U_\bullet^{<1}] = R_\bullet[U_\bullet] / (I_+)_\bullet$. Note that $P_\lambda = \Gm^n$ because the group is abelian. In this case the map $\phi$ constructed above has a simpler description:

For $f \in U_n$ let $\sum_{w \in \bZ} f_w$ be the weight decomposition of the projection of $f$ under the action of $\lambda(\Gm)$. Then the homomorphism $R_n[U_n] \to R_n[t] \otimes R_n[U_n^{<0}]$ taking $f \mapsto \sum_{w \leq 0} t^{-w} f_w$ is compatible with the simplicial structure and equivariant for the group homomorphism $(\lambda, \id_{\Gm^n}) : \Gm \times \Gm^n \to \Gm^n$. It thus defines a map $\Theta \times \Spec(B_\bullet)/\Gm^n \to \Spec(A_\bullet)/\Gm^n$, which is classified by a map
\begin{equation} \label{E:abelian_comparison}
\phi_\lambda : \Spec(B_\bullet) / \Gm^n \to \Filt(\Spec(A_\bullet)/\Gm^n).
\end{equation}

In addition if we define $C_\bullet := R_\bullet[U_\bullet^0] = R_\bullet[U_\bullet] / ((I_+)_\bullet + (I_-)_\bullet)$, then the canonical closed embedding $\Spec(C_\bullet) \to \Spec(A_\bullet)$ is equivariant for the group homomorphism $(\lambda,\id_{\Gm^n}) : \Gm \times \Gm^n \to \Gm^n$, where the first factor of $\Gm$ acts trivially on $C_\bullet$. We can thus regard this inclusion as a morphism $(B\Gm) \times \Spec(C_\bullet)/\Gm^n \to \Spec(A_\bullet)/\Gm^n$, which is classified by a morphism
\begin{equation} \label{E:abelian_comparison_2}
\sigma_\lambda : \Spec(C_\bullet) / \Gm^n \to \Grad(\Spec(A_\bullet)/\Gm^n).
\end{equation}

\begin{lem} \label{lem:filt_abelian_quotient}
The map \eqref{E:abelian_comparison} is an open and closed immersion, and the image of $\phi_\lambda$ covers $\Filt(\Spec(A_\bullet)/\Gm^n)$ as $\lambda$ ranges over all cocharacters of $\Gm^n$. Likewise the map \eqref{E:abelian_comparison_2} is an open and closed immersion, and the image of $\sigma_\lambda$ covers $\Grad(\Spec(A_\bullet)/\Gm^n)$ as $\lambda$ ranges over all cocharacters of $\Gm^n$.
\end{lem}
\begin{proof}
The argument is the same as the proof of \Cref{L:filt_quotient_stack}: one uses the computation of the cotangent complex in \Cref{L:relative_cotangent_complex} and \Cref{L:center_cotangent}, combined with the analogous statement for the underlying classical stack \cite{halpern2014structure}*{Thm.~1.47}.
\end{proof}

\begin{ex} \label{E:stratum_derived_structure}
Even if the stack $\cX$ is classical (i.e. $\pi_n(\cO_\cX)=0$ for $n>0$), the stack $\Filt(\cX)$ need not be classical. For instance, consider a $\Gm$-action with on $\bC[x_1,\ldots,x_n,y_1,\ldots,y_m]$ in which each $x_i$ has positive weight and each $y_j$ has negative weight. Let $X = \Spec(\bC[x_1,\ldots,x_n,y_1,\ldots,y_m] / (f))$ where $f$ is a non-zero polynomial in the ideal $(x_1,\ldots,x_n)$ that is homogeneous of weight $d \in \bZ$. The classical attracting locus for the tautological one parameter subgroup $\lambda(t)=t$ is $Y^{\rm{cl}} = \Spec(\bC[x_i,y_j] / (f,x_i)) = \Spec (\bC[y_j])$, which is an affine space with a contracting action of $\Gm$. The cotangent complex of $X/\Gm$ is the three term complex given by $\cO_X \cdot df \to \bigoplus_i \cO_X \cdot dx_i \oplus \bigoplus_j \cO_X \cdot dy_j \to \cO_X$, where the second map takes $dx_i \mapsto a_i x_i$ where $a_i$ is the weight of $x_i$, and likewise for each $dy_j$. If $f$ has weight $d<0$, then the cotangent complex of the derived $\Theta$-stratum $Y/\Gm$ is
\[
\bL_{Y/\Gm} = [ \cO_Y \cdot df \to \bigoplus_j \cO_Y \cdot dy_j \to \cO_Y ]
\]
by \Cref{L:relative_cotangent_complex}, where the first differential is $0$ because $f \in (x_1,\ldots,x_n)$. This does not agree with the cotangent complex of the underlying classical $\Theta$-stratum, which is missing the summand of $\cO_Y \cdot df$ in homological degree $1$.
\end{ex}


\subsection{Quasi-coherent complexes supported on a \texorpdfstring{$\Theta$}{Theta}-stratum}

\begin{defn} \label{D:subcategories_qcoh}
Let $i : \cS \hookrightarrow \cX$ be a $\Theta$-stratum. We will let $\QC_\cS(\cX)^{\geq w}, \QC_\cS(\cX)^{<w} \subset \QC(\cX)$ denote the smallest full stable subcategories which contain the essential image $i_\ast(\QC(\cS)^{\geq w})$ and $i_\ast(\QC(\cS)^{<w})$ respectively and are closed under extensions, filtered colimits, and limits of towers $\cdots \to F_2 \to F_1 \to F_0$ for which $\tau_{\leq k}(F_i)$ is eventually constant for any $k$. We let $\QC(\cX)^{<w} \in \QC(\cX)$ denote the full subcategory of complexes $F$ for which $R\Gamma_\cS(F) \in \QC_\cS(\cX)^{<w}$.
\end{defn}

\begin{prop} \label{P:baric_decomp_supports}
Let $\cX$ be an algebraic stack locally almost of finite presentation and with affine automorphism groups over a locally noetherian algebraic base stack, and let $i : \cS \hookrightarrow \cX$ be a $\Theta$-stratum. Then we have a baric decomposition
\[
\QC(\cX) = \langle \QC(\cX)^{<w}, \QC_\cS(\cX)^{\geq w} \rangle,
\]
whose truncation functors we denote $\radj{w}_\cS$ and $\ladj{w}$ (See \Cref{D:truncation_supports}). This baric decomposition satisfies the following properties, and is uniquely characterized by (1), (2), (3), and the fact that all objects of $\QC_\cS(\cX)^{\geq w}$ are set-theoretically supported on $\cS$:
\begin{enumerate}
\item $\radj{w}_\cS$ and $\ladj{w}$ commute with filtered colimits;
\item $i_\ast : \QC(\cS) \to \QC_\cS(\cX)$ intertwines $\radj{w}_\cS$ and $\ladj{w}$ with the baric truncation functors on $\QC(\cS)$ induced by \Cref{P:baric_stratum};
\item $\radj{w}_\cS$ and $\ladj{w}$ are locally uniformly bounded below in homological amplitude, i.e., for any map $p : \Spec(A) \to \cX$, there is some $d \leq 0$ such that $p^\ast \circ \radj{w}$ and $p^\ast \circ \ladj{w}$ map $\QC(\cX)_{\geq \ast}$ to $\QC(A)_{\geq \ast+d}$;
\item $\radj{w}_\cS$ and $\ladj{w}$ commute with the functor $R\Gamma_\cS(-)$ and therefore preserve $\QC_\cS(\cX)$. They also preserve the subcategory $\APerf_\cS(\cX)$, and they are right $t$-exact on $\QC_\cS(\cX)$.
\item If $p : \cX' \to \cX$ is a morphism such that $\cS$ induces a $\Theta$-stratum $\cS' \hookrightarrow \cX'$ (\Cref{D:induced_stratum}) and the canonical map $\cS' \to \cS \times_\cX \cX'$ is an isomorphism, then $p^\ast : \QC(\cX) \to \QC(\cX')$ canonically commutes with the baric truncation functors.
\item If $\cX$ has affine diagonal over the base, then a complex $F \in \QC(\cX)$ lies in $\QC(\cX)^{<w}$ if and only if $H_n(R\Gamma_\cS(F)) \in \QC(\cX)^{<w}$ for all $n$.
\end{enumerate}
\end{prop}

\begin{rem} \label{R:left_category}
If $F \in \QC(\cX)_{<\infty}$ and $\cdots \to E_1 \to E_0$ is a tower whose truncations are eventually constant, then $\Map(\lim_i E_i, F) \cong \Map(E_k,F)$ for $k \gg 0$. It follows that for $F \in \QC(\cX)_{<\infty}$,  $F \in \QC(\cX)^{<w}$ if and only if it is right orthogonal to $i_\ast(\QC(\cS)^{\geq w})$, which is equivalent to $i^{\QC,!}(F) \in \QC(\cS)^{<w}$. We will see in \Cref{S:ev_coconn} that if $\cX$ is eventually coconnective, this condition completely characterizes $\QC(\cX)^{<w}$.
\end{rem}

In the context of the proposition above, we first make some more elementary observations.

\begin{lem} \label{L:semi_orthogonal_pushforwards}
Let $i : \cS \hookrightarrow \cX$ be a $\Theta$-stratum with $\cX$ quasi-compact, and let $F \in \QC(\cS)^{\geq w}$ and $G \in \QC(\cS)^{<w}$. Then $\RHom_\cX(i_\ast(F),i_\ast(G)) = 0$.
\end{lem}
\begin{proof}
It is equivalent to show that $\RHom_\cS(i^\ast(i_\ast(F)),G)=0$. Because $F = \colim_{d \to -\infty} \radj{w}(\tau_{\geq d}(F))$ and $\radj{w}$ is right $t$-exact, it suffices to assume that $F \in \QC(\cS)_{>-\infty}$. \Cref{C:filtration_push_pull} implies that $i^\ast(i_\ast(F))$ is the limit of a tower $\to \cdots \to F_{n+1} \to F_n \to \cdots \to F_0$ such that $\forall d, \tau_{\leq d}(F_n)$ is eventually constant in $n$, and $\fib(F_{n} \to F_{n-1}) \cong \Sym^n(\bL_{\cS/\cX}) \otimes F$. The latter lies in $\QC(\cS)^{\geq w}$ by \Cref{L:relative_cotangent_complex} and \Cref{P:baric_stratum} part (6), so we have $F_n \in \QC(\cS)^{\geq w}$ as well. Finally, because $\radj{w}$ is right $t$-exact, it commutes with limits of towers whose truncations are eventually constant. It follows that $i^\ast(i_\ast(F)) \cong \lim_n F_n$ lies in $\QC(\cS)^{\geq w}$, which establishes the claim.
\end{proof}

\begin{defn} \label{D:cat_C}
Under the hypotheses of \Cref{P:baric_decomp_supports} and assuming $\cX$ is quasi-compact, let $\cC \subset \QC(\cX)$ denote the smallest full stable subcategory containing the essential image $i_\ast(\QC(\cS))$.
\end{defn}

\begin{lem} \label{L:prop_category_C}
For any $F \in \cC$ and $G \in \QC(\cX)$, the complexes $\tau_{\leq d}(F)$, $\tau_{\geq d}(F)$, $F \otimes G$, and $\inner{\RHom}^\otimes_{\cX}(F,G)$ all lie in $\cC$, where the last denotes the inner Hom in the symmetric monoidal $\infty$-category $\QC(\cX)^\otimes$. In particular, $\cC$ inherits a $t$-structure from $\QC(\cX)$.
\end{lem}

\begin{proof}
Any $F \in \cC$ can be constructed as a sequence of extensions of objects in $i_\ast(\QC(\cS))$. More precisely, we let $\cC_0 := i_\ast(\QC(\cS))$ and let $\cC_i \subset \cC$ be the full sub category consisting of objects that are cofibers of a morphism in $\cC_{i-1}$ for any $i>0$. Then each $\cC_i$ is closed under shifts, $\cC_0 \subset \cC_1 \subset \cdots$, and $\cC = \bigcup_i \cC_i$.

First we prove by induction on $i$ that for $F \in \cC_i$, $\tau_{\leq d}(F)$ and $\tau_{\geq d}(F)$ lie in $\cC$ for any $d$, and thus $\cC$ is closed under truncation. Because $i_\ast$ is exact, $\cC_0$ is preserved by the functors $\tau_{\leq d}(-)$ and $\tau_{\geq d}(-)$. For any map $A \to B$ between objects in a stable $\infty$-category with a left and right complete $t$-structure, there is a fiber sequence
\[
\cofib(\tau_{\geq d}(A) \to \tau_{\geq d}(B)) \to \tau_{\geq d} \cofib(A\to B) \to Q[d],
\]
where $Q$ is the cokernel of the map $H_d(B) \to H_d(\cofib(A\to B))$ in the heart of the $t$-structure. Dually, we have a fiber sequence
\[
K[d] \to \tau_{\leq d} \fib(A\to B) \to \fib(\tau_{\leq d}(A) \to \tau_{\leq d}(B)),
\]
where $K = \ker(H_d(\fib(A\to B)) \to H_d(A))$. The subcategory $i_\ast(\QC(\cS)_\heart) \subset \QC(\cX)_\heart$ is closed under subobjects and quotients (this reduces to the statement for underlying classical stacks, which can be checked locally), so $Q$ and $K$ lie in $i_\ast(\QC(\cS)_\heart)$ as well. This shows that under the inductive hypothesis, if $A,B \in \cC_{i-1}$, then $\tau_{\leq d}(\cofib(A\to B))$ and $\tau_{\geq d}(\cofib(A \to B))$ lie in $\cC$, which proves the claim.

By a similar inductive argument, the claim that $F \otimes G$ and $\inner{\RHom}_X^\otimes(F,G)$ lie in $\cC$ can be reduced to the case where $F = i_\ast(E)$ for some $E \in \QC(\cS)$. The claim then follows from the base change formula $i_\ast(E) \otimes G \cong i_\ast(E \otimes i^\ast(G))$ and the formula $\inner{\RHom}_{\cX}^\otimes(i_\ast(E),G) \cong i_\ast(\inner{\RHom}_\cS^\otimes(E,i^{\QC,!}(G)))$, where $i^{\QC,!}$ denotes the right adjoint to $i_\ast$.
\end{proof}

\begin{lem} \label{L:supports_truncation_simple}
There is a unique semiorthogonal decomposition $\cC = \langle \cC^{<w},\cC^{\geq w}\rangle$ such that the truncation functors $\radj{w}_\cC$ and $\ladj{w}_\cC$ are right $t$-exact, $i_\ast(\QC(\cS)^{<w}) \subset \cC^{<w}$, and $i_\ast(\QC(\cS)^{\geq w}) \subset \cC^{\geq w}$.
\end{lem}

\begin{proof}
This is essentially a formal consequence of \Cref{L:semi_orthogonal_pushforwards}, which implies that if $\cC^{\geq w}$ and $\cC^{<w}$ are the smallest full stable subcategories containing $i_\ast(\QC(\cS)^{\geq w})$ and $i_\ast(\QC(\cS)^{<w})$ respectively, then $\cC^{\geq w}$ is left semiorthogonal to $\cC^{<w}$. The sequences $i_\ast(\radj{w}(F)) \to i_\ast(F) \to i_\ast(\ladj{w}(F))$ give the canonical fiber sequences for any $F \in i_\ast(\QC(\cS))$, and the functoriality of this factorization under cones and shifts implies the existence of such fiber sequences for any object of $\cC$.
\end{proof}

Assuming $\cX$ is quasi-compact, we recall from \Cref{T:deformation_to_normal_cone} the canonical presentation of the local cohomology functor
\[
R\Gamma_\cS(F) \cong \lim{}_d \colim_n \inner{\RHom}_\cX^\otimes(\cO_{\cS^{(n)}},\tau_{\leq d}(F)),
\]
where $\cO_{\cS^{(n)}} \in \APerf_\cS(\cX)$ is the structure sheaf of the $n^{th}$ infinitesimal neighborhood of $\cS \hookrightarrow \cX$. Note that $\cO_{\cS^{(n)}}$ has a finite filtration whose associated graded complexes are $i_\ast \Sym^n(\bL_{\cS/\cX}[-1])$, so the complexes $\inner{\RHom}_\cX^\otimes(\cO_{\cS^{(n)}},\tau_{\leq d}(F))$ lie in $\cC$ by \Cref{L:prop_category_C}.

\begin{defn} \label{D:truncation_supports}
We \emph{define} the baric truncation functors $\radj{w}_\cS$ and $\ladj{w}$ explicitly when $\cX$ is quasi-compact by the formulas
\begin{align} \label{E:def_truncation_supports}
\radj{w}_\cS(F) &= \lim {}_d \left( \colim_n \left(\radj{w}_{\cC}\left(\inner{\RHom}_\cX^\otimes(\cO_{\cS^{(n)}}, \tau_{\leq d}F) \right) \right) \right), \text{ and}\\
\ladj{w}(F) &= \cofib(\radj{w}_\cS(F) \to F),
\end{align}
making use of \Cref{L:supports_truncation_simple} to define the baric truncation functor $\radj{w}_\cC$ on $\cC$.
\end{defn}

\begin{lem} \label{L:construct_trunc1}
$\radj{w}_\cC \inner{\RHom}_\cX^\otimes(\cO_{\cS^{(n)}}, \tau_{\leq d}(-)) : \QC(\cX) \to \QC(\cX)$ commutes with filtered colimits.
\end{lem}
\begin{proof}
Because $\cO_{\cS^{(n)}}$ is a finite sequence of extensions of objects of the form $i_\ast(E)$ for $E \in \QC(\cS)$, it suffices to show that
\begin{align*}
\radj{w}_\cC \inner{\RHom}_\cX^\otimes(i_\ast(E),\tau_{\leq d}(-)) &\cong \radj{w}_\cC i_\ast(\inner{\RHom}_\cS^\otimes(E,i^{\QC,!}(\tau_{\leq d}(-)))) \\
&\cong i_\ast(\radj{w} \inner{\RHom}_\cS^\otimes(E,i^{\QC,!}(\tau_{\leq d}(-))))
\end{align*}
commutes with filtered colimits. Because $i_\ast : \QC(\cS) \to \QC(\cX)$ is $t$-exact, its right adjoint $i^{\QC,!}$ is left $t$-exact, and $i^{\QC,!} : \QC(\cX)_{\leq d} \to \QC(\cS)_{\leq d}$ is right adjoint to $i_\ast : \QC(\cS)_{\leq d} \to \QC(\cX)_{\leq d}$. Both categories are compactly generated by $\DCoh(-)_{\leq d}$ \cite{halpern2014mapping}*{Thm.~A.2.1}, and thus $i^{\QC,!}$ commutes with filtered colimits in $\QC(\cX)_{\leq d}$ because $i_\ast$ preserves compact objects. Likewise the fact that $E \in \APerf(\cX)$ implies that $\inner{\RHom}_\cX^\otimes(E,-)$ commutes with filtered colimits in $\QC(\cS)_{\leq d}$, and the lemma follows from this and the fact that $i_\ast$ and $\radj{w}$ commute with filtered colimits in $\QC(\cS)$ (the latter is part (1) of \Cref{P:baric_stratum}).
\end{proof}

\begin{lem} \label{L:construct_trunc2}
The functors $\radj{w}_\cS$ and $\ladj{w}$ of \Cref{D:truncation_supports} are canonically isomorphic to the truncation functors $\radj{w}_\cC$ and $\ladj{w}_\cC$ of \Cref{L:supports_truncation_simple} when restricted to $\cC$.
\end{lem}
\begin{proof}
Precomposing with the map $\cO_\cX \to \cO_{\cS^{(n)}}$ defines a natural transformation of functors $\colim_n \radj{w}_\cC \inner{\RHom}_\cX^\otimes(\cO_{\cS^{(n)}},-) \to \radj{w}_\cC(-)$, and the right exactness of $\radj{w}_\cC$ from \Cref{L:supports_truncation_simple} allows us to pre-compose with $\tau_{\leq d}(-)$ and pass to the limit to define a natural transformation of functors on $\cC$
\[
\radj{w}_\cS(-) \to \varprojlim {}_d \left( \radj{w}_\cC(\tau_{\leq d}(-)) \right) \cong \radj{w}_\cC(-),
\]
where we have used the right $t$-exactness of $\radj{w}_\cC$ to deduce that the inverse limit on the right agrees with $\radj{w}_\cC(-)$ after truncation at any degree.

To show that this is an isomorphism, it suffices to show that for any $E \in \QC(\cS)_{<\infty}$, the natural transformation
\[
\colim_n \radj{w}_\cC \inner{\RHom}_\cX^\otimes(\cO_{\cS^{(n)}},i_\ast(E)) \to \radj{w}_\cC(i_\ast(E))
\]
is an isomorphism. By adjunction, the isomorphism $\radj{w}_\cC(i_\ast(-)) \cong i_\ast(\radj{w}(-))$, and the fact that $\radj{w}$ commutes with filtered colimits in $\QC(\cS)$, it suffices to show that
\[
\colim_n \inner{\RHom}_\cS^\otimes(i^\ast (\cO_{\cS^{(n)}}),E) \to E
\]
is an isomorphism. It suffices to verify this map is an isomorphism after pushing forward to $\cX$ once more, in which case it follows from the formula for $R\Gamma_\cS(i_\ast(E))$ in \Cref{T:deformation_to_normal_cone} and the fact that $R\Gamma_\cS(i_\ast(E)) \cong i_\ast(E)$.

\end{proof}

\begin{proof}[Proof of \Cref{P:baric_decomp_supports}]

We will assume that $\cX$ is quasi-compact. To deduce the general case from the quasi-compact case, one can use \cite{AHLH}*{Lem.~6.8} to choose a Zariski-cover of $\cX$ by quasi-compact open substacks $\cU \subset \cX$ such that $\cS$ induces a $\Theta$-stratum in $\cU$. Then, we use the theory of $\infty$-categorical descent -- we discuss this in more detail in the proof of \Cref{T:derived_Kirwan_surjectivity} below (see \Cref{L:sod_limits}) -- to deduce the statements for $\cS \hookrightarrow \cX$ from the corresponding statements for $\cS \cap \cU \hookrightarrow \cU$. In addition, we will postpone the proof of part (6) until the next section, because it relies on the local structure theorem \Cref{T:local_structure_stratum}.

Note that if $F \in \QC(\cX)_{<\infty}$, then the inverse limit in \eqref{E:def_truncation_supports} is eventually constant and thus is given by the simpler formula
\[
\radj{w}_\cS(F) = \colim_n(\radj{w}_\cC(\inner{\RHom}_\cX^\otimes(\cO_{\cS^{(n)}},F))).
\]
We first observe that the resulting functor $\radj{w}_\cS : \QC(\cX)_{<\infty} \to \QC(\cX)$ satisfies parts (1) through (4) of the proposition, and from the defining cofiber sequence for $\ladj{w}$ it suffices to verify the claims for $\radj{w}_\cS$ only, which we do in order:
\begin{enumerate}
\item \Cref{L:construct_trunc1} implies that $\radj{w}_\cS$ commutes with filtered colimits in $\QC(\cX)_{\leq k}$ for any $k$, which is a modified version of (1).
\item This follows from \Cref{L:construct_trunc2} and the definition of $\radj{w}_\cC$ in \Cref{L:supports_truncation_simple}. \item This is an immediate consequence of part (4), and the fact that there is some $d \leq 0$ such that for $F \in \QC(\cX)_{[0,k]}$, $R\Gamma_\cS(F) \in \QC_\cS(\cX)_{[d,k]}$.
\item It is clear from the formula that $\radj{w}_\cS(R\Gamma_\cS(F)) \to \radj{w}_\cS(F)$ is an isomorphism, and $\radj{w}_\cS(F) \in \QC_\cS(\cX)$ for any $F \in \QC(\cX)$, which implies the first part of (4). To show that $\radj{w}_\cS$ is right $t$-exact on $\QC_\cS(\cX)_{<\infty}$, we observe that any complex in $F \in \QC_\cS(\cX)_{<\infty}$ is a uniformly homologically bounded filtered colimit of objects in $\DCoh_\cS(\cX) \subset \cC$, so the right $t$-exactness follows from the right $t$-exactness of $\radj{w}_\cC$ in \Cref{L:supports_truncation_simple} and \Cref{L:construct_trunc2}. Finally, to see that $\radj{w}_\cS$ maps $\DCoh_\cS(\cX)$ to $\APerf_\cS(\cX)$, it suffices, by iterated extension, to show this for objects of the form $i_\ast(E)$ for $E \in \Coh(\cS)$, so the claim follows from (2) and part (3) of \Cref{P:baric_stratum}.
\end{enumerate}

We now use the fact that $\QC(\cX)$ is the left $t$-completion of $\QC(\cX)_{<\infty}$. It follows from this that for any stable $\infty$-category $\cC$ with a left complete $t$-structure and any functor $F : \QC(\cX)_{<\infty} \to \cC$ that is right exact up to a shift and commutes with uniformly homologically bounded above filtered colimits extends uniquely to a cocontinuous functor $F : \QC(\cX) \to \cC$ that is right $t$-exact up to a shift, given by the formula $\lim_d F(\tau_{\leq d}(-))$. We have verified that the functor $\colim_n \radj{w}_\cC \inner{\RHom}_\cX^\otimes(\cO_{\cS^{(n)}},-)$ is right $t$-exact up to a shift and commutes with uniformly homologically bounded filtered colimits, and the formula for $\radj{w}_\cS(-)$ in \Cref{D:truncation_supports} is precisely the unique extension of this functor from $\QC(\cX)_{<\infty}$ to $\QC(\cX)$. The properties of the functor $\radj{w}_\cS$ in parts (1) through (4) are straightforward consequences of this observation and the corresponding claims for the functor on $\QC(\cX)_{<\infty}$.

To establish the semi-orthogonal decomposition, it suffices to show that the essential image of $\radj{w}_\cS$ and $\ladj{w}$ are semi-orthogonal. One can then use the properties of $\radj{w}_\cS$ that we have already established to verify that $\QC_\cS(\cX)^{\geq w}$ and $\QC(\cX)^{<w}$ are the essential image of $\radj{w}_\cS$ and $\ladj{w}$ respectively.  We will verify the following two things for any $E \in \QC(\cX)$:
\begin{enumerate}[label=\roman*)]
\item the canonical map $\radj{w}_\cS (E) \to E$ becomes an isomorphism after applying $\radj{w}_\cS(-)$; and
\item for $F \cong \radj{w}_\cS(E)$, the canonical map $\radj{w}_\cS(F) \to F$ is an isomorphism.
\end{enumerate}
This suffices so show the semi-orthogonality of the essential images of $\radj{w}_\cS$ and $\ladj{w}$ because (i) implies that $\radj{w}_\cS(\ladj{w}(E)) = 0$, and then (ii) implies that any map $F \to G$ with $F$ in the essential image of $\radj{w}_\cS$ and $G$ in the essential image of $\ladj{w}$ must factor through $\radj{w}(G) = 0$.

To prove claim (i) of the previous paragraph, we observe that the canonical map $\radj{w}_\cS(E) \to E$ factors
\[
\xymatrix@C=5pt{\lim_d \colim_n \radj{w}_\cC \inner{\RHom}^\otimes_\cX(\cO_{\cS^{(n)}},\tau_{\leq d} E) \ar[rr] \ar[rd] & & \lim_d \tau_{\leq d}(E) \\
& \lim_d \colim_n \inner{\RHom}^\otimes_\cX(\cO_{\cS^{(n)}},\tau_{\leq d} E) \cong R\Gamma_\cS(E) \ar[ur] & },
\]
where the first map is induced from the canonical map on each factor arising from the semi-orthogonal decomposition of $\cC$. Observing that $\radj{w}_\cS(-)$ commutes with both the limit, because it is right $t$-exact up to a shift, and the colimit, the fact that the first map becomes an isomorphism after applying $\radj{w}_\cS(-)$ follows from \Cref{L:construct_trunc2}. The fact that the second map becomes an isomorphism follows from part (4) above.

To prove claim (ii), we observe that the essential image of $\radj{w}_\cS$ is generated by $i_\ast(\QC(\cS)^{\geq w})$ under filtered colimits and limits of systems whose truncations are eventually constant. Therefore, because $\radj{w}_\cS$ preserves filtered colimits and is right exact up to a shift, it suffices to prove the claim for these objects, which follows from \Cref{L:supports_truncation_simple}.

\medskip
\noindent \textit{Proof of uniqueness of the baric structure:}
\medskip

Any semiorthogonal decomposition $\QC(\cX) = \langle \cB, \cA \rangle$ with $\cA \subset \QC_\cS(\cX)$ induces a semiorthogonal decomposition $\QC_\cS(\cX) = \langle \QC_\cS(\cX) \cap \cB, \cA \rangle$. Conversely the fact that $\QC_\cS(\cX) \subset \QC(\cX)$ admits a right adjoint $R\Gamma_\cS$ implies that any semiorthogonal decomposition $\QC_\cS(\cX) = \langle \cB, \cA \rangle$ induces a semiorthogonal decomposition $\QC(\cX) = \langle \cA^{\perp},\cA \rangle$. So the baric decomposition of $\QC(\cX)$ is uniquely determined by its restriction to $\QC_\cS(\cX)$. We have already observed in \Cref{L:supports_truncation_simple} that part (2) of the proposition uniquely determines the baric structure on $\cC \subset \QC_\cS(\cX)$, and the hypothesis that $\radj{w}$ commutes with filtered colimits and has uniformly bounded homological amplitude forces it to be given by the formula of \Cref{D:truncation_supports}.

\medskip
\noindent \textit{Proof of (5):}
\medskip

First note that pullback along $\cS' \to \cS$ commutes with the baric truncation functors by part (4) of \Cref{P:baric_stratum}. The hypothesis that $\cS' \cong \cS \times_\cX \cX'$ allows us to apply the base change formula to deduce that
\[
p^\ast(i_\ast(\QC(\cS)^{\geq w})) = i'_\ast((p')^\ast(\QC(\cS)^{\geq w})) \subset \QC_{\cS'}(\cX')^{\geq w},
\]
and likewise for $\QC(\cS)^{<w}$, where $p' : \cS' \to \cS$ and $i' : \cS' \to \cX$ are the base change of $p$ and $i$ respectively. Because $\QC_\cS(\cX)^{\geq w}$ is generated by $i_\ast(\QC(\cS)^{\geq w})$ under extensions, filtered colimits, and limits of towers with eventually constant truncations, and $p^\ast$ preserves these constructions, it follows that $p^\ast(\QC_\cS(\cX)^{\geq w}) \subset \QC_{\cS'}(\cX')^{\geq w}$. The same argument shows that $p^\ast(\QC_\cS(\cX)^{<w}) \subset \QC_{\cS'}(\cX')^{<w}$. 

This shows that for any $F \in \QC_\cS(\cX)$, the pullback of the fiber sequence $\radj{w}(F) \to F \to \ladj{w}(F)$ is the canonical fiber sequence coming from the baric structure on $\QC_{\cS'}(\cX')$. In particular $p^\ast \circ \radj{w}_\cS \cong \radj{w}_{\cS'} \circ p^\ast$ as functors $\QC_\cS(\cX) \to \QC_{\cS'}(\cX')$. For a general $F \in \QC(\cX)$ we have a canonical isomorphism $\radj{w}_\cS(F) \simeq \radj{w}_\cS(R\Gamma_\cS(F))$, and the same is true for $p^\ast(F)$, so in fact this shows that $p^\ast \circ \radj{w}_\cS \cong \radj{w}_{\cS'} \circ p^\ast$ as functors $\QC(\cX) \to \QC(\cX')$ as well. This in turn implies that $p^\ast \circ \ladj{w} \cong \ladj{w} \circ p^\ast$.

\end{proof}

\subsubsection{Simplifications in the eventually coconnective case}
\label{S:ev_coconn}

If $\cX$ is eventually coconnective, i.e., $\cO_\cX \in \QC(\cX)_{\leq d}$ for some $d$, then there is an alternative formula for $\radj{w}_\cS$ that is simpler than the one in \Cref{D:truncation_supports}. In this case it is possible to find a directed system $P_0 \to P_1 \to \cdots$ of objects in the category $\cC$ of \Cref{D:cat_C} with $\colim_{n \geq 0} P_n \simeq R\Gamma_\cS(\cO_\cX)$. For instance, one can use the canonical complexes $P_n = \inner{\RHom}_\cX^{\otimes}(\cO_{\cS^{(n)}},\cO_\cX)$. Alternatively, because $R\Gamma_\cS(\cO_\cX) \in \QC(\cX)_{\leq d}$ as well, one can write $R\Gamma_\cS(\cO_\cX)$ as a filtered colimit of complexes $P_n \in \DCoh_\cS(\cX)$.

Given such a choice of directed system $\{P_n\}$, it follows a posteriori from the properties of $\radj{w}_\cS$ and \Cref{L:prop_category_C} that for any $F \in \QC(\cX)$,
\begin{equation} \label{E:alternate_radj}
\radj{w}_\cS (F) \cong \radj{w}_\cS(R\Gamma_\cS(\cO_\cX) \otimes F) \cong \colim_n \radj{w}_\cC( P_n \otimes F ).
\end{equation}
In fact, if we were willing to restrict to the case where $\cX$ was eventually coconnective, we could have taken this formula as our definition and derived all of the properties from it.

One consequence of the formula \eqref{E:alternate_radj} is that $\QC_\cS(\cX)^{\geq w}$ is already generated by $i_\ast(\QC(\cS)^{\geq w})$ under extensions and filtered colimits -- it is not necessary to add limits of any towers in \Cref{D:subcategories_qcoh}. It follows that we have an alternative characterization
\[
\QC(\cX)^{<w} = \{F \in \QC(\cX)| i^{\QC,!}(F) \in \QC(\cS)^{<w}\}.
\]


\section{The main theorem: general case}

In this section, we refine the local cohomology theory of the previous section to give a semi-orthogonal decomposition of the category of almost perfect complexes $\APerf(\cX)$ for a stack with a derived $\Theta$-stratification. We continue to work over a fixed base, an algebraic locally Noetherian derived stack $\cB$.

In the previous section, we used purely ``intrinsic'' methods wherever possible, the stronger results of this section will use a local structure theorem, \Cref{T:local_structure_stratum}, to reduce certain claims to explicit calculations in the case of a quotient stack.

\subsection{The local structure theorem, and vanishing of local cohomology}
\label{S:local_structure}

To build on \Cref{P:baric_decomp_supports} and to complete the proof of part (6), we will use a result on the local-quotient structure of a $\Theta$-stratum.

\begin{thm} \label{T:local_structure_stratum}
Let $\cX$ be an algebraic derived stack almost of finite presentation and with affine diagonal over a noetherian derived algebraic space, and let $\ev_1 : \cS \subset \Filt(\cX) \to \cX$ be a $\Theta$-stratum. Then there is an affine derived scheme $X = \Spec(A)$ with a $\Gm$-action and a smooth surjective affine morphism $p : X/\Gm \to \cX$ that induces (\Cref{D:induced_stratum}) the tautological $\Theta$-stratum, i.e., such that
\[
\cS' := \Filt(p)^{-1}(\cS) \subset \Filt(X/\Gm) \to X/\Gm
\]
can be identified with the closed immersion $\Spec(A/I_+)/\Gm \to \Spec(A)/\Gm$ of \eqref{E:abelian_comparison}, and $\ev_1(|\cS'|) = p^{-1}(\ev_1(|\cS|)) \subset |X/\Gm|$.

\end{thm}
\begin{proof}The analogous claim in the classical context was established in \cite{AHLH}*{Lem.~6.11}. In the derived context, it suffices, by \Cref{L:classical_stratum}, to find a smooth surjective morphism $\Spec(A)/\Gm \to \cX$ whose base change along the closed immersion $\cX^{\rm cl} \hookrightarrow \cX$ induces the tautological stratum in $\Spec(\pi_0(A)/\bG_m)$. For each $n$, we will construct a smooth morphism $X_n / \Gm \to \tau_{\leq n} \cX$ satisfying this condition, where $\tau_{\leq n} \cX$ is the $n^{th}$ truncation of $\cX$, along with $(n+1)$-connective closed embeddings $X_n / \Gm \hookrightarrow X_{n+1}/\Gm$ which induce isomorphisms
\[
X_n / \Gm \cong (X_{n+1} / \Gm) \times_{\tau_{\leq n+1} \cX} \tau_{\leq n}\cX.
\]
It follows that the map $X/\Gm := \colim_n X_n / \Gm \to \cX$ will satisfy the conditions of the theorem. Note also that the stated properties of the algebraic derived stack $X_n/\Gm$ imply that it is cohomologically affine, i.e., quasi-coherent sheaves have vanishing higher cohomology (and thus so do connective quasi-coherent complexes). The base case $\tau_{\leq 0} \cX = \cX^{\rm cl}$ is \cite{AHLH}*{Lem.~6.11}. The stacks $\cX_n$ can now be constructed inductively using the following, which is a generalization of the argument of \cite{HAGII}*{Lem.~C.0.11}:

\medskip
\noindent \textit{Claim: If $\pi : \cY \to \cX$ is a smooth morphism of algebraic stacks with $\cY$ cohomologically affine, and $\cX \hookrightarrow \cX'$ is a derived square-zero extension by $M \in \APerf(\cX)_{\leq 0}$, then there is a square-zero extension $\cY \hookrightarrow \cY'$ over $\cX'$ by $\pi^\ast(M)$ such that $\cY \to \cY' \times_{\cX'} \cX$ is an isomorphism.}
\medskip

The square-zero extension is classified by a morphism $\eta : \bL_\cX \to M[1]$, and the claim amounts to verifying that $\pi^\ast(\eta) : \pi^\ast \bL_\cX \to \pi^\ast(M)[1]$ factors through $D\pi : \pi^\ast \bL_\cX \to \bL_\cY$. From the exact triangle $\pi^\ast \bL_\cX \to \bL_\cY \to \bL_{\cY/\cX}$ we have an exact sequence of groups
\[
\Hom_\cY(\bL_\cY,\pi^\ast(M)[1]) \to \Hom_\cY(\pi^\ast(\bL_\cX),\pi^\ast(M)[1]) \to \Hom_\cY(\bL_{\cY/\cX}[-1],\pi^\ast(M)[1]).
\]
The third term in this sequence vanishes because $\bL_{\cY/\cX}$ is a locally free sheaf, and $\cY$ is cohomologically affine. The choice of lift of $\pi^\ast(\eta) \in \Hom_\cY(\pi^\ast(\bL_\cX),\pi^\ast(M)[1])$ to $\Hom_\cY(\bL_\cY,\pi^\ast(M)[1])$ classifies a square-zero extension $\cY \hookrightarrow \cY'$ over $\cX'$. The pullback of the canonical fiber sequence $M \to \cO_{\cX'} \to \cO_\cX$ is the corresponding fiber sequence $\pi^\ast(M) \to \cO_{\cY'} \to \cO_\cY$, which implies that $\cY \cong \cY' \times_{\cX'} \cX$.
\end{proof}

\begin{rem}
\Cref{T:local_structure_stratum} roughly says that any $\Theta$-stratum in a finite type algebraic stack with affine diagonal is a ``stacky BB-stratum'' in the sense of \cite{ballard_wall}.
\end{rem}

For the remainder of this section, we fix a derived algebraic stack $\cX$ almost of finite presentation and with affine diagonal over a noetherian base algebraic space, and fix a $\Theta$-stratum $\cS \hookrightarrow \cX$. We consider an affine derived scheme $X = \Spec(A)$ with $\Gm$-action which admits a smooth surjective map 
\[
p : X / \Gm \to \cX
\]
as in \Cref{T:local_structure_stratum}. Our first applications of \Cref{T:local_structure_stratum} will be to complete the proofs of two earlier claims for which we could not find an ``intrinsic'' argument.

\begin{proof}[Proof of part (6) of \Cref{P:baric_decomp_supports}]
We may replace the condition that $F \in \QC(\cX)^{<w}$ with the equivalent condition that $\radj{w}_\cS(F) = 0$. Because $\radj{w}_\cS(F) \simeq \radj{w}_\cS(R\Gamma_\cS(F))$, it suffices to consider only $F \in \QC_\cS(\cX)$, and the claim is that $F \in \QC_\cS(\cX)^{<w}$ if and only if $H_n(F) \in \QC_\cS(\cX)^{<w}$ for all $n$. Both conditions are equivalent to the corresponding condition after applying $p^\ast$ by \Cref{P:baric_decomp_supports} part (5), so it suffices to prove the claim when $\cX = X/\Gm$. Note that an object lies in $\QC(B\Gm)^{<w}$ if and only if its homology sheaves do, and the pushforward along the map $\pi : X/\Gm \to B\Gm$ is $t$-exact, so the statement of the lemma immediately follows from the stronger claim:

\medskip
\noindent \textit{Claim: A complex $F \in \QC_{S/\Gm} (X/\Gm)$ lies in $\QC(X/\Gm)^{<w}$ if and only if $\pi_\ast(F) \in \QC(B\Gm)^{<w}$.} \label{C:claim_criterion}
\medskip

First assume $\pi_\ast(F) \in \QC(B\Gm)^{<w}$. Then $\pi_\ast(\tau_{\leq d}(F)) \in \QC(B\Gm)^{<w}$ as well, and it suffices to show that $\tau_{\leq d} F \in \QC(X/\Gm)^{<w}$ for all $d$, so we may assume $F \in \QC(X/\Gm)_{\leq d}$. In this case it suffices to show that $F$ is right orthogonal to $i_\ast(\QC(S/\Gm)^{\geq w})$ (see \Cref{R:left_category}).

The category $\QC(S/\Gm)^{\geq w}$ is compactly generated by the twists of the structure sheaf $\cO_{S/\Gm} \langle n \rangle$ by a character $n$ of $\Gm$ with $n \leq w$ (so that the fiber weight along the center $X^{\Gm}$ is $\geq w$). It therefore suffices to show that
\[
\RHom_{X/\Gm}(i_\ast(\cO_{S/\Gm}\langle n \rangle), F) = 0
\]
for $n \leq w$. The complex $\cO_{S/\Gm} \langle n \rangle$ has an explicit Koszul resolution whose terms are direct sums of objects of the form $\cO_{X/\Gm} \langle m \rangle$ with $m \leq n$. It therefore suffices to show that
\[
0 = \RHom_{X/\Gm}(\cO_{X/\Gm}\langle n \rangle , F) \simeq \RHom_{X/\Gm}(\pi^\ast(\cO_{B\Gm}\langle n \rangle), F)
\]
for $n \leq w$. By adjunction this is equivalent to $\pi_\ast(F) \in \QC(B\Gm)^{<w}$.

Conversely, assume that $F \in \QC_\cS(\cX)^{<w}$. Then $F$ lies in the smallest stable subcategory of $\QC_\cS(\cX)$ containing $i_\ast(\QC(\cS)^{<w})$ and closed under extensions, filtered colimits, and limits of towers whose truncations are eventually constant. It therefore suffices to show that for $E \in \QC(\cS)^{<w}$, $\pi_\ast(E) \in \QC(B\Gm)^{<w}$. This follows from part (4) of \Cref{P:baric_stratum} and the fact that $\pi : S/\Gm \to B\Gm$ is equivariant for the natural weak $\Theta$-actions on both stacks.
\end{proof}

\begin{proof}[Proof of \Cref{L:center_pullback}]

Part (4) of \Cref{P:baric_stratum} implies that $\ev_0^\ast$ maps $\QC(\cZ)^{w}$ to $\QC(\cS)^{w}$ and $\radj{w}((\ev_0)_\ast(-))$ maps $\QC(\cS)^w$ to $\QC(\cZ)^w$. Because both $\radj{w}(-)$ and $(\ev_0)_\ast(-)$ are right adjoints, one can check that $\radj{w}((\ev_0)_\ast(-))$ is a right adjoint to $\ev_0^\ast : \QC(\cZ)^w \to \QC(\cS)^w$. So it suffices to show that the counit and unit of adjunction,
\[
F \to \radj{w}((\ev_0)_\ast(\ev_0^\ast(F))) \text{ and } \ev_0^\ast(\radj{w}((\ev_0)_\ast(G))) \to G,
\]
are isomorphisms for any $F \in \QC(\cZ)^w$ and $G \in \QC(\cS)^w$.

We now apply \Cref{T:local_structure_stratum} to the special situation where $\cS = \cX$, which provides a smooth surjective map $\cS_0 := \Spec(A_0) / \Gm \to \cS$ such that $A$ is a simplicial commutative algebra with $\Gm$-action encoded by a \emph{non-positive} grading, and such that the lift $\cS \subset \Filt(\cS)$ encoding the weak $\Theta$-action pulls back to the tautological $\Theta$-stratum in $\cS_0$.

Consider the Cech nerve of the map $\cS_0 \to \cS$, which is a simplicial derived stack whose $n^{th}$ level has the form $\cS_n = \Spec(A_n)/\Gm$. Furthermore, \Cref{L:induced_stratum_base_change} implies that each of these stacks is identical to its tautological $\Theta$-stratum, so each $A_n$ is non-positively graded. Note also that $\iMap(B\Gm,-)$ commutes with fiber products, which implies that if $\cZ_0 = \Spec(A^{\Gm})/\Gm$ is the center of $\cS_0$, then the Cech nerve of the smooth surjective map $\cZ_0 \to \cZ$ is level-wise identified with the center of each $\Theta$-stratum $\cZ_n \to \cS_n$. By \Cref{P:baric_stratum}, the baric truncation functors commute with pullback along $\Theta$-equivariant maps, so using the base change formula and smooth descent it suffices to very that the unit and counit of adjunction above are isomorphisms on each level.

In the case where $\cS = \Spec(A)/\Gm$ for a non-positively graded simplicial commutative algebra $A$ and $\cZ = \Spec(A/I_-)/\Gm$, $\QC(\cZ)^w \cong \QC(A/I_-)$ is compactly generated by $\cO_\cZ \langle -w \rangle$. Likewise $\QC(\cS)$ is compactly generated by objects of the form $\cO_{\cS}\langle a \rangle$ for $a \in \bZ$. It follows from Nakayama's lemma and part (4) of \Cref{P:baric_stratum} that $\radj{w}(\ladj{w+1}(\cO_\cS\langle a \rangle)) = 0$ for any $a \neq -w$, so $\cO_\cS\langle -w \rangle$ generates $\QC(\cS)^w$.

We have therefore reduced the claim to showing that the canonical maps are isomorphisms
\[
A \otimes_{A/I_-} \radj{w}(A\langle -w \rangle) \to A \langle -w \rangle \text{ and } (A/I_-)\langle -w \rangle  \to \radj{w}(A\langle -w \rangle),
\]
where in both cases $\radj{w}$ denotes the baric truncation of the underlying object in $\QC(\cZ)$. If $A = \bigoplus_{d \leq 0} A^d$ is the weight decomposition of $A$, then the map $A^0 \to A / I_-$ is a weak-equivalence of simplicial commutative algebras by construction. This implies that the two maps above are isomorphisms.

\end{proof}

Finally, we establish a result on local cohomology generalizing \cite{quantization}*{Prop.~2.6}.

\begin{prop} [Quantization commutes with reduction] \label{P:quantization_commutes_with_reduction}
Let $\cX$ be an algebraic derived stack that is locally almost finitely presented with affine diagonal over a locally noetherian base stack, and let $\cS \hookrightarrow \cX$ be a $\Theta$-stratum. If $G \in \QC(\cX)^{<w}$ and $F \in \APerf(\cX)$ is such that $i^\ast(F) \in \QC(\cS)^{\geq w}$, then the restriction map
$$\RHom_\cX(F,G) \to \RHom_{\cX^{\rm ss}}(F|_{\cX^{\rm ss}},G|_{\cX^{\rm ss}})$$
is an equivalence.
\end{prop}

\begin{proof}

The claim is equivalent to the claim that $\RHom(F,R\Gamma_\cS(G)) = 0$. Part (4) of \Cref{P:baric_decomp_supports} implies that $R\Gamma_\cS(G) \simeq R\Gamma_\cS(\ladj{w}(G)) \simeq \ladj{w}(R\Gamma_\cS(G))$. So it suffices to show that $\RHom(F,G) = 0$ for $G \in \QC_\cS(\cX)^{<w}$. Then part (6) of \Cref{P:baric_decomp_supports} implies that $\tau_{\leq p}(G) \in \QC_\cS(\cX)^{<w}$, and
\[
\RHom(F,G) \simeq \varprojlim_p \RHom(F,\tau_{\leq p}(G)),
\]
so it suffices to consider $G \in \QC_\cS(\cX)^{<w}_{\leq p}$. The $\infty$-category $\QC_\cS(\cX)_{\leq p}$ is compactly generated by $\DCoh_\cS(\cX)_{\leq p}$, so we may express $\tau_{\leq p}(G)$ as a filtered colimit
\[
\tau_{\leq p}(G) \simeq \colim_{\alpha} G_\alpha,
\]
with $G_\alpha \in \DCoh_\cS(\cX)_{\leq p}$. The functor $\tau_{\leq p} (\ladj{w}(-))$ commutes with filtered colimits and does not affect the left-hand-side, so we have
\[
\tau_{\leq p}(G) \simeq \colim_{\alpha} \tau_{\leq p} (\ladj{w}(G_\alpha)),
\]
where $\tau_{\leq p} (\ladj{w}(G_\alpha)) \in \DCoh_\cS(\cX)_{\leq p}^{<w}$ by part (4) of \Cref{P:baric_decomp_supports} and part (6) of \Cref{P:baric_decomp_supports}.

Because $F \in \APerf(\cX)$, $\RHom(F,-)$ commutes with this filtered colimit, and it therefore suffices to show that $\RHom_\cX(F,G_\alpha) = 0$ for $G_\alpha \in \DCoh(\cX)_{\leq p}^{<w}$. Such a complex $G_\alpha$ can be constructed as a finite sequence of extensions of objects of the form $i_\ast(E)$ for $E \in \APerf(\cS)^{<w}$, so it therefore suffices to show that $\RHom_\cX(F,i_\ast(E))=0$ for such $E$. This follows by adjunction and the condition that $i^\ast(F) \in \APerf(\cS)^{\geq w}$.

\end{proof}

\subsection{Structure theorem for \texorpdfstring{$\APerf$}{D-Coh}}
\label{S:structure_theorem}

Let $\cX$ be an algebraic derived stack locally almost of finite presentation and with affine diagonal over a locally noetherian algebraic base stack $\cB$, and consider a derived $\Theta$-stratum $i : \cS \subset \Filt(\cX) \hookrightarrow \cX$. Let $\pi : \cS \to \cZ$ denote the projection onto the center of the stratum. In mapping stack notation, $i = \ev_1$ and $\pi = \ev_0$.

\begin{defn} \label{D:subcategories_aperf}
We will consider several subcategories of $\APerf(\cX)$ indexed by an integer $w \in \bZ$. For any full stable subcategory $\cA \subset \APerf(\cX)$, we denote the following full stable subcategories of $\cA$:
\begin{align*}
\cA^{\geq w} &:= \left\{ F \in \cA \left| i^\ast(F) \in \APerf(\cS)^{\geq w} \right. \right\} \\
\cA^{<w} &:= \left\{ F \in \cA \left| i^{\QC,!}(F) \in \QC(\cS)^{<w} \right. \right\} \\
\cA^{w} &:= \cA^{\geq w} \cap \cA^{<w+1}
\end{align*}
Note that $\cA^{\geq w+1} \subset \cA^{\geq w}$ and $\cA^{<w} \subset \cA^{<w+1}$ by definition. We also introduce special notation for the following full $\infty$-subcategory of ``grade restricted'' complexes
\[
\cG_\cS^w = \APerf(\cX)^{\geq w} \cap \APerf(\cX)^{<w} \subset \APerf(\cX).
\]
\end{defn}

Our main structure theorem for $\APerf(\cX)$ is the following:

\begin{thm}\label{T:derived_Kirwan_surjectivity}
Let $\cX$ and $\cS$ be as above, and let $w \in \bZ$. The subcategories of $\APerf(\cX)$ of \Cref{D:subcategories_aperf} define a semiorthogonal decomposition
\begin{equation} \label{E:main_SOD}
\APerf(\cX) = \left \langle \lefteqn{\overbrace{\phantom{\APerf_\cS(\cX)^{<w}, \cG_\cS^w}}^{\APerf(\cX)^{<w}}} \APerf_\cS(\cX)^{<w}, \lefteqn{\underbrace{\phantom{\cG_\cS^w, \APerf_\cS(\cX)^{\geq w}}}_{\APerf(\cX)^{\geq w}}} \cG_\cS^w, \APerf_\cS(\cX)^{\geq w} \right\rangle.
\end{equation}
The projection functors $\ladj{w}$ and $\radj{w}_\cS$ onto the respective semiorthogonal factors $\APerf(\cX)^{<w}$ and $\APerf_\cS(\cX)^{\geq w}$ are right $t$-exact on $\APerf_\cS(\cX)$, and they have uniformly bounded below homological amplitude on $\APerf(\cX)$ if $\cX$ is quasi-compact. The projection functors $\ladj{w}_\cS$ and $\radj{w}$ onto the respective semiorthogonal factors $\APerf_\cS(\cX)^{<w}$ and $\APerf(\cX)^{\geq w}$ are right $t$-exact. The semiorthogonal decomposition \eqref{E:main_SOD} has the following properties:
\begin{enumerate}
\item The restriction functor induces an equivalence of $\infty$-categories $\cG_\cS^w \to \APerf(\cX \setminus \cS)$. \\
\item $\APerf_\cS(\cX)^{<w}$ and $\APerf_\cS(\cX)^{\geq w}$ generate $\APerf_\cS(\cX)$, giving a semiorthogonal decomposition
\begin{equation} \label{E:supports_SOD}
\APerf_\cS(\cX) = \sod{\APerf_\cS(\cX)^{<w}, \APerf_\cS(\cX)^{\geq w}}.
\end{equation}
\item The functor $i_\ast \pi^\ast : \APerf(\cZ)^w \to \APerf(\cX)$ is fully faithful with essential image $\APerf_\cS(\cX)^{w}$.\\
\item If we let $\cG_\cS^{[u,w)}:= \APerf(\cX)^{<w} \cap \APerf(\cX)^{\geq u}$, then for any $u \leq v \leq w$ there is a semiorthogonal decomposition
\begin{equation} \label{E:window_SOD}
\cG_\cS^{[u,w)} = \left \langle \APerf_\cS(\cX)^{u}, \ldots, \APerf_\cS(\cX)^{v-1}, \cG_\cS^v, \APerf_\cS(\cX)^{v}, \ldots, \APerf_\cS(\cX)^{w-1}\right\rangle.
\end{equation}
\item If $\cX$ is quasi-compact, then for any $F \in \APerf(\cX)$ and any $p \in \bZ$, both maps
\[
\tau_{\leq p}(F) \to \tau_{\leq p}(\ladj{w}(F)) \quad \text{and} \quad \tau_{\leq p}(\radj{-w}(F)) \to \tau_{\leq p}(F)
\]
are isomorphisms for $w \gg 0$. It follows that the canonical maps induce isomorphisms for any $F \in \APerf(\cX)$:
\[
F \xrightarrow{\simeq} \lim \limits_{w\to \infty} \ladj{w}(F) \quad \text{ and } \quad \colim \limits_{w\to -\infty} \radj{w}(F) \xrightarrow{\simeq} F.
\]
\item If $p : \cX' \to \cX$ is a morphism such that $\cS$ induces a $\Theta$-stratum $\cS' \hookrightarrow \cX'$ (\Cref{D:induced_stratum}) and the canonical map $\cS' \to \cS \times_\cX \cX'$ is an isomorphism, then $p^\ast : \APerf(\cX) \to \APerf(\cX')$ preserves all of the subcategories in the semiorthogonal decomposition \eqref{E:main_SOD} and thus canonically commutes with the projection functors.
\end{enumerate}
\end{thm}

First note the possible conflict introduced by the differing definitions of $\APerf_\cS(\cX)^{\geq w}$ in \Cref{D:subcategories_aperf} and $\QC_\cS(\cX)^{\geq w}$ \Cref{D:subcategories_qcoh}. This is resolved by the following:

\begin{lem} \label{L:aperf_vs_qcoh_weights}
A complex $F \in \APerf_\cS(\cX)$ lies in $\QC_\cS(\cX)^{\geq w}$ if and only if $i^\ast(F) \in \QC(\cS)^{\geq w}$.
\end{lem}

\begin{proof}
We can use \cite{AHLH}*{Lem.~6.8} to cover $\cX$ by Zariski open quasi-compact substacks $\cU \subset \cX$ in which $\cS$ induces a $\Theta$-stratum, and it suffices to verify the claim for each $\cS \cap \cU \hookrightarrow \cU$, so we may assume $\cX$ is quasi-compact.

Because $\radj{w}_\cS$ and $\ladj{w}_\cS$ induce a baric decomposition of $\APerf_\cS(\cX)$, we know that $F$ lies in $\APerf_\cS(\cX)^{\geq w}$ if and only if it is left orthogonal to every object of the form $\ladj{w}_\cS(G)$ for some $G \in \APerf_\cS(\cX)$. The fact that $\ladj{w}_\cS$ has uniformly bounded below homological amplitude implies that the canonical map $\ladj{w}_\cS(G) \to \varprojlim_p \ladj{w}_\cS(\tau_{\leq p}(G))$ is an equivalence, so it suffices to check left orthogonality for $G \in \DCoh_\cS(\cX)$. For any such $G$, $\ladj{w}_\cS(G)$ can be constructed as a finite sequence of extensions of objects of the form $\ladj{w}_\cS(i_\ast(E)) \cong i_\ast(\ladj{w}(E))$ for some $E \in \DCoh(\cS)$, so if $i^\ast(F) \in \QC(\cS)^{\geq w}$ then $\RHom(F,\ladj{w}_\cS(G)) = 0$, and the claim follows.
\end{proof}

We will prove \Cref{T:derived_Kirwan_surjectivity} after some preliminary observations.

\subsubsection{Key statements in the local setting} \label{sect:koszul}
Fix a noetherian simplicial commutative base ring $R$, and let $X$ be an affine derived scheme almost of finite presentation over $R$ with a $(\Gm)_R$-action, and let $S \subset X$ be the $\Theta$-stratum induced by the tautological cocharacter of $\Gm$, as in \Cref{lem:filt_abelian_quotient}. We will make use of an alternative explicit description of $R\Gamma_S(\cO_X)$. Choose a morphism of graded simplicial commutative rings $R[x_1,\ldots,x_k] \to \cO_X$, where $x_i$ has degree $d_i>0$ for $i=1,\ldots,k$, and $x_1,\ldots,x_k$ map to homogeneous generators of the ideal $I_+ \subset \pi_0(\cO_X)$ generated by positive weight elements. In particular, $x_1,\ldots,x_k$ cut out $S$ set-theoretically.

We have the Koszul complex
\begin{equation} \label{E:koszul_complex}
K_{X}(x_1,\ldots,x_k) = (\cO_{X} \xrightarrow{x_1} \cO_{X}\langle d_1 \rangle) \otimes \cdots \otimes (\cO_{X} \xrightarrow{x_k} \cO_{X} \langle d_k \rangle),
\end{equation}
which is the pullback of the linear dual of the standard Koszul complex resolving $\cO_{\{0\}} \in \QC(\bA^k)$ under the map $X \to \bA^k$ defined above. Furthermore, we have canonical maps $K_X(x_1^n,\ldots,x_k^n) \to K_X(x_1^{n+1},\ldots,x_k^{n+1})$, and because $S$ is the set-theoretic preimage of $\{0\} \in \bA^k$, we have
\begin{equation} \label{E:local_cohomology_koszul}
R\Gamma_{S} (\cO_{X}) \simeq \colim_{n\geq 1} K_{X}(x^n_1,\ldots,x^n_k).
\end{equation}
Because $\radj{w}_\cS(F) = \radj{w}_\cS(R\Gamma_S (\cO_X) \otimes F)$ and $\radj{w}_\cS$ commutes with filtered colimits, we have
\begin{equation} \label{E:local_radj}
\radj{w}_\cS(F) = \colim_{n} \radj{w}_\cS(K_X(x_1^n,\ldots,x_k^n) \otimes F).
\end{equation}

\begin{lem} \label{L:finiteness_radj}
For any $F \in \APerf(X/\Gm)$ and any $w \in \bZ$, each homology sheaf in the colimit \eqref{E:local_radj} is eventually constant for $n\gg 0$. It follows that $\radj{w}_\cS(F) \in \APerf_\cS(X/\Gm)^{\geq w}$.
\end{lem}

\begin{proof}

Consider the complexes
\[
C_n:= \cofib(K_{X}(x^n_1,\ldots,x^n_k) \to K_{X}(x^{n+1}_1,\ldots,x^{n+1}_k)).
\]
We must show that for any $p$, $\radj{w}_\cS(C_n \otimes F) \in \QC(X/\Gm)_{\geq p}$ for $n \gg 0$. By the right $t$-exactness of $\radj{w}_\cS$ on $\QC_\cS(X/\Gm)$ (part (4) of \Cref{P:baric_decomp_supports}), suffices to show that
\[
\tau_{\leq p}(C_n \otimes F) \in \QC(X/\Gm)^{<w} \text{ for } n \gg 0.
\]
Because $\tau_{\leq p}(C_n \otimes F)$ is set theoretically supported on $S/\Gm$, the main claim in the proof of part (6) of \Cref{P:baric_decomp_supports} shows that lying in $\QC_\cS(X/\Gm)^{<w}$ is equivalent to $R\Gamma(X,\tau_{\leq p}(C_n \otimes F)) \in \QC(B\Gm)^{<w}$.

We now extend the map of graded simplicial commutative algebras $R[x_1,\ldots,x_k] \to \cO_X$, which we used to define the Koszul complex above, to a map $R[x_1,\ldots,x_N] \to \cO_X$ which induces a surjection on $\pi_0(\cO_X)$, where $x_{k+1},\ldots,x_N$ are homogeneous of non-positive degree. This defines a $\Gm$-equivariant embedding $i : X \hookrightarrow \bA_R^N$ such that $S$ is the set-theoretic preimage of the $\Theta$-stratum $\cS' := (\bA_R^N)_+/\Gm \subset \bA_R^N/\Gm$, and each $C_n$ is the restriction of the analogous complex on $\bA_R^N/\Gm$, which we also denote $C_n$. It therefore suffices by the projection formula to show that $R\Gamma(\bA^N_R,\tau_{\leq p} (i_\ast(F) \otimes C_n)) \in \QC(B\Gm)^{<w}$, i.e., we have reduced the claim to the special case where $X=\bA^N_R$ with a linear $\Gm$-action.

For any $F \in \APerf(\bA^N_R/\Gm)$, one may find a perfect complex $P$ with a map to $F$ such that $\fib(P \to F)$ is arbitrarily connective. So by part (3) of \Cref{P:baric_decomp_supports} it suffices to consider only $F \in \Perf(\bA^N_R/\Gm)$. Because $\Perf(\bA^N_R/\Gm)$ is split-generated by $\cO_{\bA^N_R} \langle d \rangle$ for all $d \in \bZ$, it suffices to prove the claim for these objects, i.e., to show that $\tau_{\leq p}(C_n \langle d \rangle) \in \QC(\bA^N_R/\Gm)^{<w}$ for $n \gg 0$. $\QC_{\cS'}(\bA^N_R / \Gm)^{<w}$ is closed under homological truncation by part (6) of \Cref{P:baric_decomp_supports}, so it suffices to show that $C_n \langle d \rangle \in \QC_{\cS'}(\bA^N_R / \Gm)^{<w}$ for $n\gg 0$.

We can now complete the proof by an explicit computation: $C_n \langle d \rangle$ lies in the stable subcategory generated by $\cO_{\bA_R^N}\langle q \rangle$ for $q \geq n \min(d_1,\ldots, d_k) + d$. If $i' : \cS' \to \bA^N_R / \Gm$ denotes the inclusion, then
\[
(i')^{\QC,!}(\cO_{\bA^N_R}\langle q \rangle) \simeq \cO_{\cS'} \langle q - d_1 - \cdots - d_k \rangle [-k],
\]
so $(i')^{\QC,!}(C_n \langle d \rangle) \in \QC(\cS')^{<w}$ as long as $d_1+\cdots + d_k -d-n \min(d_1,\ldots,d_k) < w$. This is satisfied for $n\gg0$.

The first claim of the lemma implies that $\radj{w}_\cS(F) \in \APerf(\cX)$, because it implies that for any $p$, $\tau_{\leq p}(\radj{w}_\cS(F)) \cong \tau_{\leq p}(\radj{w}_\cS(K_{X}(x^n_1,\ldots,x^n_k) \otimes F))$ for some $n\gg0$. $K_{X}(x^n_1,\ldots,x^n_k)$ is supported on $S$, so part (4) of \Cref{P:baric_decomp_supports} implies $\radj{w}_\cS(K_{X}(x^n_1,\ldots,x^n_k) \otimes F) \in \APerf(\cX)$.

\end{proof}

\begin{lem} \label{L:characterize_sod}
For any derived $\Theta$-stratum $\cS \hookrightarrow \cX$, $F \in \APerf(\cX)$ is left orthogonal to $\APerf_\cS(\cX)^{<w}$ if and only if $F \in \APerf(\cX)^{\geq w}$.
\end{lem}
\begin{proof}
Again it suffices to assume $\cX$ is quasi-compact. By part (6) of \Cref{P:baric_decomp_supports} and the left-completeness of the $t$-structure on $\APerf_\cS(\cX)$, $F$ is left orthogonal to $\APerf_\cS(\cX)^{<w}$ if and only if it is left orthogonal to objects of $\DCoh_\cS(\cX)^{<w}$. Every such object is a finite extension of objects of the form $i_\ast(E)$ for $E \in \APerf(\cS)^{<w}$. The claim follows by adjunction.
\end{proof}

\begin{lem} \label{L:local_sod}
In the setting above, where $X$ is an affine derived $R$-scheme with a $(\Gm)_R$-action, and $i : \cS = S / \Gm \hookrightarrow X / \Gm$ the tautological $\Theta$-stratum, we have a semiorthogonal decomposition for any $w \in \bZ$ 
\[
\APerf(X/\Gm) = \langle \APerf_\cS(X/\Gm)^{<w}, \APerf(X/\Gm)^{\geq w} \rangle,
\]
and the projection functors are right $t$-exact.
\end{lem}
\begin{proof}
After \Cref{L:characterize_sod}, it suffices to show that every connective complex $F \in \APerf(X/\Gm)_{\geq 0}$ admits a fiber sequence $F' \to F \to F''$ with $F'' \in \APerf_{\cS}(X/\Gm)^{<w}$ and $i^\ast(F') \in\QC(\cS)^{\geq w}$ and with both $F',F''$ connective.

First consider $F = \cO_X \langle n \rangle$, and choose an $m$ such that $m \min(d_1,\ldots,d_k) \geq w + n$. Then the Koszul complex $K_X(x_1^m,\ldots,x_k^m)$ of \eqref{E:koszul_complex} is a perfect complex set theoretically supported on $S$, and there is a canonical map $K_X(x_1^m,\ldots,x_n^m)\langle n\rangle \to \cO_X\langle n \rangle$ whose fiber can be constructed as a sequence of extension of $\cO_X\langle d \rangle [p]$ with $d - n \geq w$ and $p<0$. We now let
\[
C = \fib(\cO_X\langle n \rangle \to K_X(x^m_1,\ldots,x^m_k)^\dual \langle n \rangle).
\]
Then $C$ is connective, $C|_{\cS} \in \QC(\cS)^{\geq w}$ by \Cref{L:aperf_baric_criterion}, and $K_X(x_1^m,\ldots,x_k^m)^\dual \langle n  \rangle$ is connective and set theoretically supported on $\cS$. We now define $F'' := \ladj{w}(K_X(x^m_1,\ldots,x^m_k)^\dual \langle n \rangle)$ and $F' = \fib(\cO_X \langle n \rangle \to F'')$. Then $F''$ lies in $\QC_{\cS}(X/\Gm)^{<w}$ and is connective by part (4) of \Cref{P:baric_decomp_supports}. Likewise $F'$ is naturally an extension
\[
C \to F' \to \radj{w}_\cS(K_X(x^m_1,\ldots,x^m_k)^\dual \langle n \rangle)
\]
where both $C$ and $\radj{w}_\cS(K_X(x^m_1,\ldots,x^m_k)^\dual \langle n \rangle)$ lie in $\APerf(X/\Gm)^{\geq w}$ and are connective (again by part (4) of \Cref{P:baric_decomp_supports}). It follows that $F' \in \APerf(X/\Gm)^{\geq w}$ and is connective.

Now let $F \in \APerf(X/\Gm)$ be an arbitrary connective complex. Express $F$ as a filtered colimit $F = \colim P_\alpha$ where $P_\alpha$ is a connective perfect complex constructed as a sequences of extensions of the objects $\cO_X \langle n \rangle [p]$ for $p \geq 0$, and $\tau_{\leq p}(P_\alpha)$ is eventually constant for any $p$. If $\cA \subset \APerf(X/\Gm)_{\geq 0}$ denotes the smallest full $\infty$-subcategory consisting of $F$ which admit a factorization of the desired form, then $\cA$ is closed under suspension and closed under extensions and cofibers by the semiorthogonality of $\APerf(X/\Gm)^{\geq w}$ and $\APerf_{\cS}(X/\Gm)^{<w}$. In particular $P_\alpha \in \cA$ for all $\alpha$, and if we let $P_\alpha' \to P_\alpha \to P''_\alpha$ denote the resulting factorization for each $\alpha$, then $\{P'_\alpha\}$ and $\{P''_\alpha\}$ is a filtered system of connective perfect complexes which is eventually constant after applying $\tau_{\leq p}(-)$ for any $p$. It follows that $\colim P'_\alpha \to \colim P_\alpha \to \colim P''_\alpha$ is a fiber sequence of connective complexes with $\colim P'_\alpha \in \APerf(X/\Gm)^{\geq w}$ and $\colim P''_\alpha \in \APerf_\cS(X/\Gm)^{<w}$.
\end{proof}

\subsubsection{The proof of \Cref{T:derived_Kirwan_surjectivity}}

\begin{lem} \label{L:sod_limits}
Let $\cC_i$ be a diagram of stable $\infty$-categories and exact functors between them indexed by a category $I$. Assume that each $\cC_i$ has a semiorthogonal decomposition $\cC_i = \sod{\cA_i,\cB_i}$ such that for any morphism $f:i \to j$ the functor $f_\ast : \cC_i \to \cC_j$ maps $\cA_i$ to $\cA_j$ and maps $\cB_i$ to $\cB_j$. Then $\cC := \varprojlim_{i\in I} \cC_i$ admits a semiorthogonal decomposition $\cC = \sod{\cA,\cB}$, where $\cA = \varprojlim \cA_i$ and $\cB = \varprojlim \cB_i$.
\end{lem}

\begin{proof}
For each $i$, let $\tilde{\cC}_i \subset \Fun(\Delta^1 \times \Delta^1, \cC_i)$ be the full sub-$\infty$-category consisting of commutative diagrams
\begin{equation} \label{E:triangle_square}
\xymatrix{B \ar[r] \ar[d] & C \ar[d] \\ 0 \ar[r] & A}
\end{equation}
that are cartesian and such that $B \in \cB_i \subset \cC_i$ and $A \in \cA_i \subset \cC_i$. The hypothesis that for any $f : i \to j$, $f_\ast$ is exact and maps $\cA_i$ to $\cA_j$ and $\cB_i$ to $\cB_j$ implies that $\Fun(\Delta^1 \times \Delta^1, f_\ast)$ maps the subcategory $\tilde{\cC}_i$ to $\tilde{\cC}_j$. Thus we have a diagram of $\infty$-categories taking $i \mapsto \tilde{\cC}_i$ along with a map of diagrams $\{\tilde{\cC}_i\}_{i \in I} \to \{\cC_i\}_{i \in I}$, where the functor $\tilde{\cC}_i \to \cC_i$ is restriction along the inclusion of the vertex $\{0\} \times \{1\} \in \Delta^1 \times \Delta^1$, i.e., the functor taking the square \eqref{E:triangle_square} to $C$. The fact that we have a semiorthogonal decomposition $\cC_i = \langle \cA_i, \cB_i \rangle$ implies that each functor of $\infty$-categories $\tilde{\cC}_i \to \cC_i$ is a categorical equivalence, and hence
\[
\varprojlim_{i \in I} \tilde{\cC}_i \xrightarrow{\cong} \varprojlim_{i \in I} \cC_i.
\]

Similarly, we can replace $\cA_i$ and $\cB_i$ with equivalent sub-$\infty$-categories of $\Fun(\Delta^1,\cC_i)$ of the form $0 \to A$ and $B \to 0$ respectively. We then have explicit maps of diagrams
\[
\xymatrix{ \{\tilde{\cA}_i\}_{i \in I} \ar@/^/[r]^-{\iota_\cA} & \{\tilde{\cC}_i \}_{i \in I} \ar@/^/[r]^{\pi_\cB} \ar@/^/[l]^{\pi_\cA} & \{\tilde{\cB}_i\}_{i \in I} \ar@/^/[l]^-{\iota_\cB}},
\]
where the functors out of $\tilde{\cC}_i$ take a square \eqref{E:triangle_square} to $(B \to 0) \in \tilde{\cB}_i$ and $(0 \to A) \in \tilde{\cA}_i$, and the functors into $\tilde{\cC}_i$ are given by
\[
\iota_\cB : B \mapsto \begin{array}{c}\xymatrix@1{B \ar[r]^1 \ar[d] & B \ar[d] \\ 0 \ar[r] & 0}\end{array} \text{ and }  \iota_\cA : A \mapsto \begin{array}{c}\xymatrix@1{0 \ar[r] \ar[d] & A \ar[d]^1 \\ 0 \ar[r] & A}\end{array}.
\]
More precisely, all four functors arise from maps between the simplicial sets $\Delta^1 \times \Delta^1$ and $\Delta^1$. For instance, the composition $\iota_\cB \circ \pi_\cB$ is induced by the map of posets $[1] \times [1] \to [1] \times [1]$ taking $(i,j) \mapsto (0,j)$. Furthermore, under the isomorphism
\[
\Fun(\Delta^1 \times \Delta^1 \times \Delta^1, \cC_i) \cong \Fun(\Delta^1, \Fun(\Delta^1 \times \Delta^1,\cC_i)),
\]
pulling back along the map $\Delta^1 \times \Delta^1 \times \Delta^1 \to \Delta^1 \times \Delta^1$ induced by the map of posets $[1] \times [1] \times [1] \to [1] \times [1]$ that takes $(0,i,j) \mapsto (0,0,j)$ and $(1,i,j) \mapsto (1,i,j)$ gives a natural transformation $\eta_\cB : \iota_\cB \circ \pi_\cB \to \id_{\tilde{\cC}_i}$ of functors between diagram categories. In a similar way, one constructs natural transformations $\epsilon_\cB : \id_{\tilde{\cB}_i} \to \pi_\cB \circ \iota_\cB$, and for each individual $i\in I$, it is clear that these correspond to the counit and unit of an adjunction $\iota_\cB \dashv \pi_\cB$. Therefore, after passing to the limit $\eta_\cB$ and $\iota_\cB$ will again satisfy the unit/counit identities for an adjunction $\varprojlim(\iota_\cB) \dashv \varprojlim(\pi_\cB)$. A similar argument constructs an adjunction $\varprojlim(\pi_\cA) \dashv \varprojlim(\iota_\cA)$.

From here it is straightforward to see that the functors above define a semiorthogonal decomposition
\[
\varprojlim \nolimits_{i \in I} \tilde{\cC}_i = \langle \varprojlim \nolimits_{i \in I} \tilde{\cA}_i, \varprojlim \nolimits_{i \in I} \tilde{\cB}_i \rangle.
\]
The facts that the unit $\epsilon_\cB : \id_{\tilde{\cB}_i} \to \pi_\cB \circ \iota_\cB$ and counit $\eta_\cA : \pi_\cA \circ \iota_\cA \to \id_{\tilde{\cA}_i}$ are categorical equivalences, as well as the facts that $\pi_\cA \circ \iota_\cB \sim 0$ and $\pi_\cB \circ \iota_\cA \sim 0$, can be checked for each $i \in I$ individually and therefore pass to the limit. The fact that the essential image of $\varprojlim(\iota_\cA)$ and $\varprojlim(\iota_\cB)$ are semiorthogonal follows from the formula
\[
\RHom_{\varprojlim \tilde{\cC}_i} (\{F_i\},\{G_i\}) \simeq \varprojlim_i \RHom_{\tilde{\cC}_i}(F_i,G_i),
\]
and the semiorthogonality of $\tilde{\cB}_i \dashv \tilde{\cA}_i$ in each $\tilde{\cC}_i$.
\end{proof}

Finally we note the following, which establishes claim (6):

\begin{lem} \label{L:pullback_preservation}
If $p : \cX' \to \cX$ is a morphism such that $\cS$ induces a $\Theta$-stratum $\cS' \hookrightarrow \cX'$ for which the canonical map $\cS' \to \cS \times_\cX \cX'$ is an isomorphism, then $p^\ast : \APerf(\cX) \to \APerf(\cX')$ preserves all of the subcategories of $\APerf$ appearing in \eqref{E:main_SOD}.
\end{lem}
\begin{proof}
$p^\ast$ automatically preserves the subcategories $\APerf(-)$ and $\APerf_\cS(-)$. By \Cref{P:baric_decomp_supports} part (5) the pullback functor $p^\ast$ maps $\QC(\cX)^{<w}$ to $\QC(\cX')^{<w}$. The morphism $\cS' \to \cS$ is $\Theta$-equivariant, so \Cref{P:baric_stratum} part (4) implies that $p^\ast$ maps $\APerf(\cX)^{\geq w}$ to $\APerf(\cX')^{\geq w}$.
\end{proof}

Let $p : \cX' \to \cX$ be a smooth representable morphism of algebraic derived stacks that are locally almost of finite presentation and with affine automorphism groups over a noetherian base stack, and let $\cS \hookrightarrow \cX$ be a $\Theta$-stratum that induces a $\Theta$-stratum $\cS' \hookrightarrow \cX'$. Then the Cech nerve $\cX'_\bullet$ of $p$ is the simplicial stack with $\cX'_n := \cX'^{\times_\cX^{n+1}} = \cX' \times_{\cX} \cX' \times_{\cX} \cdots \times_{\cX} \cX'$, and by \Cref{L:induced_stratum_base_change}, $\cS$ induces a $\Theta$-stratum $\cS'_n \hookrightarrow \cX'_n$ for every $n$.

\begin{lem} \label{L:cech_stratum}
If in the situation above, the semiorthogonal decomposition \eqref{E:main_SOD} and the claims of \Cref{T:derived_Kirwan_surjectivity} about the homological amplitude of the truncation functors are known for $\APerf(\cX'_n)$ for all $n\geq 0$, then they hold for $\APerf(\cX)$ as well.
\end{lem}
\begin{proof}
Because $\Filt(-)$ is functorial and commutes with limits, $\Filt(\cX'_\bullet)$ is a simplicial stack, and it is canonically identified with the Cech nerve of the map $\Filt(\cX') \to \Filt(\cX)$, which is smooth by \Cref{C:stratum_smooth}. By definition $\cS'_n \subset \Filt(\cX'_n)$ is the preimage of $\cS \subset \Filt(\cX)$ under the augmentation map $\Filt(\cX'_n) \to \Filt(\cX)$, and it follows that $\cS'_\bullet$ is a simplicial stack for which all maps are weakly $\Theta$-equivariant. \Cref{L:induced_stratum_smooth} implies that the canonical map $\cS'_n \cong \cS \times_{\cX} \cX'_n$ is an isomorphism for every $n$.

It follows formally that for any map $[m] \to [n]$, the inclusion $\cS'_n \hookrightarrow \cX'_n$ is the $\Theta$-stratum induced by $\cS'_m \hookrightarrow \cX'_m$, and the canonical map $\cS'_n \to \cS'_m \times_{\cX'_m} \cX'_n$ is an isomorphism. We now apply \Cref{L:pullback_preservation} to conclude that the pullback $\APerf(\cX'_m) \to \APerf(\cX'_n)$ preserves all of the categories in the semiorthogonal decomposition \eqref{E:main_SOD}. It follows from \Cref{L:sod_limits} that the limit of the cosimplicial $\infty$-category
\[
\APerf(\cX) \cong \Tot\{\APerf(\cX'_\bullet)\}
\]
has a semiorthogonal decomposition whose factors are, by construction, the subcategories of $\APerf(\cX)$ appearing in \eqref{E:main_SOD}. All of the claims about the homological amplitude of the truncation functors can be checked locally, and $p^\ast$ commutes with the truncation functors by \Cref{L:pullback_preservation}, so it suffices to verify them on $\cX'$.
\end{proof}

We now make the following successive simplifications to proving \Cref{T:derived_Kirwan_surjectivity}:
\begin{enumerate}
\item By \Cref{L:sod_limits} and faithfully flat descent for $\APerf(\cX)$, it suffices to establish the semiorthogonal decomposition after base change to a smooth cover of the base stack $\cB$, so we may assume $\cB = \Spec(R)$.
\item By \cite{AHLH}*{Lem.~6.8} one can find a Zariski cover of $\cX$ by a union of quasi-compact open substacks $\cU$ such that $\cS$ induces a $\Theta$-stratum in $\cU$. By \Cref{L:cech_stratum} it suffices to prove the existence of the semiorthogonal decomposition and the homological amplitude bounds for each $\cU$. We may therefore assume $\cX$ is quasi-compact.
\item Assuming $\cX$ is quasi-compact, we can choose a smooth affine surjective map $u_0 : X_0 / \Gm \to \cX$ as in \Cref{T:local_structure_stratum}, where $X_0$ is an affine derived $\Gm$-scheme, and $\cS$ induces the tautological $\Theta$-stratum in $X_0 / \Gm$. In fact every level of the Cech nerve has this form, so by \Cref{L:cech_stratum} it suffices to prove the first part of the theorem for the tautological $\Theta$-stratum in $X/\Gm$ for an affine derived $\Gm$-scheme $X$.
\end{enumerate}

In the local case $X/\Gm$, \Cref{L:finiteness_radj} implies that the truncation functors $\radj{w}_\cS$ and $\ladj{w}$ of the semiorthogonal decomposition from \Cref{P:baric_decomp_supports} preserve the subcategory $\APerf(\cX)$, which establishes a semiorthogonal decomposition
\[
\APerf(X/\Gm) = \langle \APerf(X/\Gm)^{<w},\APerf_\cS(X/\Gm)^{\geq w} \rangle.
\]
The bounds on homological amplitude of the truncation functors follow from parts (3) and (4) of \Cref{P:baric_decomp_supports}. \Cref{L:local_sod} then establishes a semiorthogonal decomposition
\[
\APerf(X/\Gm) = \langle \APerf_\cS(X/\Gm)^{<w},\APerf(X/\Gm)^{\geq w} \rangle,
\]
whose truncation functors are right $t$-exact. This establishes the first part of \Cref{T:derived_Kirwan_surjectivity}.

\medskip
\noindent \textit{Proof of (1):}
\medskip

The restriction functor $\cG_\cS^w \to \APerf(\cX \setminus \cS)$ is fully faithful by \Cref{P:quantization_commutes_with_reduction}, so it suffices to show essential surjectivity. Any $F \in \Coh(\cX \setminus \cS)$ admits a lift to an object $\tilde{F} \in \Coh(\cX)$. Then $\radj{w}(\ladj{w}(\tilde{F})) \in \cG_\cS^w$ is also a lift, and it lies in $\APerf(\cX)_{\geq -d}$, where $d$ is the homological amplitude of the truncation functor $\ladj{w}$. By induction, any $F \in \DCoh(\cX)_{\geq q}$ is the restriction of a unique object in $\cG_\cS^w \cap \APerf(\cX)_{\geq q-d}$. Finally, consider a general $F \in \APerf(\cX)_{\geq q}$. Fully faithfulness implies that the inverse system $\{ \tau_{\leq p} (F)\}_p$ lifts to an inverse system $\{\tilde{F}_p\}$ in $\cG_\cS^w$, and furthermore the fiber of the map $\tilde{F}_{p+1} \to \tilde{F}_p$ is $(p-d)$-connective for all $p$. It follows that $\varprojlim_{p \to \infty} \tilde{F}_p$ is an object of $\cG_\cS^w$ lifting $F$.

\medskip
\noindent \textit{Proof of (2):}
\medskip

This follows from part (4) of \Cref{P:baric_decomp_supports}, which provides for any $F \in \APerf_\cS(\cX)$ a fiber sequence $\radj{w}_\cS(F) \to F \to \ladj{w}(F)$ with $\ladj{w}(F) \in \APerf_\cS(\cX)^{<w}$ and $\radj{w}_\cS(F) \in \APerf_\cS(\cX)^{\geq w}$ by \Cref{L:aperf_vs_qcoh_weights}.

\medskip
\noindent \textit{Proof of (3):}
\medskip

We have seen in \Cref{L:center_pullback} that $\pi^\ast : \APerf(\cZ)^{w} \to \APerf(\cS)^{w}$ is an equivalence. To show that $i_\ast : \APerf(\cS)^{w} \to \APerf(\cX)$ is fully faithful, it suffices by adjunction to show that the canonical map
\[
\RHom_\cS(F,G) \to \RHom_\cS(i^\ast (i_\ast(F)),G)
\]
is an isomorphism. As we argue in the proof of \Cref{L:semi_orthogonal_pushforwards}, $\cofib(i^\ast(i_\ast(F)) \to F)$ has a nonnegative filtration whose associated graded is $\Sym^{>0}(\bL_{\cS/\cX}) \otimes F$, and it therefore lies in $\QC(\cS)^{\geq w+1}$. Because $G \in \QC(\cS)^{<w+1}$ by hypothesis, $\RHom_\cS(\cofib(i^\ast(i_\ast(F)) \to F), G)=0$, and the claim follows.

\medskip
\noindent \textit{Proof of (4):}
\medskip

This is a consequence of applying (2) at different values of $w$ to arrive at refined semiorthogonal decompositions for $\APerf_\cS(\cX)^{\geq v}$ and $\APerf_\cS(\cX)^{<v}$, and then combining these with \eqref{E:main_SOD}. We leave the details to the reader.

\medskip
\noindent \textit{Proof of (5):}
\medskip

Using the main semiorthogonal decomposition for $\APerf(\cX)$ at a fixed weight $w_0$, we see that the projection of $\radj{w}(F)$ onto $\APerf(\cX)^{\geq w_0}$ is independent of $w$ for $w\leq w_0$, so it suffices to prove the first claim for $F \in \APerf_\cS(\cX)^{<w_0}$. By the same reasoning it suffices to prove the second claim for $F \in \APerf_\cS(\cX)^{\geq w_0}$. We may therefore assume that $F \in \APerf_\cS(\cX)$.

Because $\radj{w}$ and $\ladj{w}$ are right $t$-exact on $\APerf_\cS(\cX)$, we can replace $F$ with $\tau_{\leq p}(F) \in \DCoh_\cS(\cX)$. If the claim holds for two complexes, then it holds for any extension of those complexes, so we may reduce to the case where $F = i_\ast(E)$ for some $E \in \DCoh(\cS)_{\heart}$. Because $i_\ast$ is $t$-exact and commutes with $\radj{w}$ and $\ladj{w}$, it suffices to show the corresponding claim for the baric truncation functors on $\APerf(\cS)$.

The baric truncation functors commute with restriction to the center $\sigma^\ast \colon \APerf(\cS) \to \APerf(\cZ)$ by part (4) of \Cref{P:baric_stratum}, and by Nakayama's lemma the homological degree of the lowest nonvanishing homology sheaf of $F \in \APerf(\cS)$ will agree with that of $\sigma^\ast(F)$. It follows that to show that $\radj{-w}(F) \to F$ is $p$-connected for $w \gg 0$, it suffices to show that $\radj{-w}(\sigma^\ast(F)) \to \sigma^\ast(F)$ is $p$-connected  for $w \gg 0$. This is manifest, because the baric truncation functors on $\cZ$ are $t$-exact. The same argument shows that $F \to \ladj{w}(F)$ is $p$-connected for $w \gg 0$.

\qed

\subsection{Extensions to multiple strata}
\label{sect:multiple_strata}

Now let $\cX$ be a derived algebraic stack locally almost of finite presentation and with affine diagonal over a locally noetherian algebraic base stack $\cB$.
\begin{defn} \label{D:theta_stratification} \cite{halpern2014structure}*{Defn.~2.0.1.2}
A \emph{$\Theta$-stratification} of $\cX$ indexed by a totally ordered set $I$ with minimal element $0 \in I$ consists of: 1) a collection of open substacks $\cX_{\leq \alpha} \subset \cX$ for $\alpha \in I$ such that $\cX_{\leq \alpha} \subset \cX_{\leq \alpha'}$ when $\alpha < \alpha'$; 2) a $\Theta$-stratum $\cS_\alpha \subset \Filt(\cX_{\leq \alpha})$ for all $\alpha \in I$ such that $\cX_{\leq \alpha} \setminus \ev_1(|\cS_\alpha|) = \bigcup_{\alpha'<\alpha} \cX_{\leq \alpha'}$; and 3) for every $x \in |\cX|$, there is a minimal $\alpha \in I$ such that $x \in \cX_{\leq \alpha}$. We define the \emph{semistable locus} to be $\cX^{\rm ss} := \cX_{\leq 0} \subset \cX$.
\end{defn}
%

Given a $\Theta$-stratification, we will let $i_\alpha := \ev_1 |_{\cS_{\alpha}} : \cS_\alpha \subset \Filt(\cX_{\leq \alpha}) \to \cX_{\leq \alpha}$ denote the corresponding closed immersion. We will sometimes regard $\cS_{\alpha}$ as a locally closed substack of $\cX$, and write $\cX = \cX^{{\rm ss}} \bigcup_{\alpha>0} \cS_\alpha$. 

Note that $\cX$ need not be quasi-compact, and the stratification need not be finite. As in the case of a single $\Theta$-stratum, \Cref{D:theta_stratification} makes sense as written when $\cX$ is a classical stack and $\Filt(\cX)$ is interpreted as the classical mapping stack, and the restriction of $\cS$ to $\Filt(\cX^{\rm cl})^{\rm cl}$ induces a bijection between derived $\Theta$-stratifications and $\Theta$-stratifications of the underlying classical stack $\cX^{\rm cl}$.


\begin{ex}[Bia\l{}ynicki-Birula stratification]
The simplest example is when $\cX = X /\Gm$ for some projective $\Gm$-scheme $X$ with equivariant ample line bundle $\cO_X(1)$. Let $Y_i \hookrightarrow X$ denote the locally closed subscheme of points $x$ such that $\lim_{t \to 0} t \cdot x$ lies on the open and closed subscheme $Z_i \subset X^{\Gm}$ on which $\cO_X(1)$ has weight $i$. Then $\bigcup_{j > i} Y_j$ is closed for any $i$, $Y_i / \Gm \hookrightarrow (X \setminus \bigcup_{j>i} Y_j)/\Gm$ equipped with a canonical derived structure is a $\Theta$-stratum, and these define a $\Theta$-stratification of $X/\Gm$ with $\cX^{\rm ss} = \emptyset$.
\end{ex}

\begin{defn} \label{D:categories_multiple_strata}
Let $\bS = \{(\cX_{\leq \alpha},\cS_\alpha)\}_{\alpha \in I}$ be a $\Theta$-stratification of $\cX$. Choose $w_\alpha \in \bZ$ for each $\alpha \in I$. We define
\begin{gather*}
\APerf(\cX)^{\geq w} := \left\{ F \in \APerf(\cX) \left| \forall \alpha, i_\alpha^\ast (F|_{\cX_{\leq \alpha}}) \in \APerf(\cS_\alpha)^{\geq w_\alpha} \right. \right\}, \text{ and} \\
\APerf(\cX)^{<w} := \left\{ F \in \APerf(\cX) \left| \forall \alpha, i_\alpha^{\QC,!}(F|_{\cX_{\leq \alpha}}) \in \QC(\cS_\alpha)^{< w_\alpha} \right. \right\}.
\end{gather*}
We also define $\cG_\bS^w := \APerf(\cX)^{\geq w} \cap \APerf(\cX)^{<w}$, $\APerf_{\cX^{\us}}(\cX)^{\geq w} := \APerf_{\cX^{\us}}(\cX) \cap \APerf(\cX)^{\geq w}$, and $\APerf_{\cX^{\us}}(\cX)^{< w} := \APerf_{\cX^{\us}}(\cX) \cap \APerf(\cX)^{< w}$.
\end{defn}
In other words, each of these categories consist of complexes whose restriction to each open substack $\cX_{\leq \alpha}$ lies in the corresponding category of \Cref{D:subcategories_aperf} with respect to the $\Theta$-stratum $\cS_\alpha \hookrightarrow \cX_{\leq \alpha}$.

\begin{thm} \label{T:derived_Kirwan_surjectivity_full}
Let $\cX$ be an algebraic derived stack locally a.f.p. and with affine diagonal over a locally noetherian base stack, and let $\bS = \{(\cX_{\leq \alpha},\cS_\alpha)\}_{\alpha \in I}$ be a $\Theta$-stratification of $\cX$. Then for any choice $w = \{w_\alpha \in \bZ\}_{\alpha \in I}$ we have a semiorthogonal decomposition
\begin{equation} \label{E:main_SOD_full}
\APerf(\cX) = \left \langle \lefteqn{\overbrace{\phantom{\APerf_{\cX^{\us}}(\cX)^{<w}, \cG_\bS^w}}^{\APerf(\cX)^{<w}}} \APerf_{\cX^{\us}}(\cX)^{<w}, \lefteqn{\underbrace{\phantom{\cG_\bS^w, \APerf_{\cX^{\us}}(\cX)^{\geq w}}}_{\APerf(\cX)^{\geq w}}} \cG_\bS^w, \APerf_{\cX^{\us}}(\cX)^{\geq w} \right\rangle.
\end{equation}
This semiorthogonal decomposition has the following properties:
\begin{enumerate}
\item The restriction functor $\cG_\bS^{w} \to \APerf(\cX^{\rm ss})$ is an equivalence.
\item We have a semiorthogonal decomposition
\[
\APerf_{\cX^{\us}}(\cX) = \langle \APerf_{\cX^{\us}}(\cX)^{<w}, \APerf_{\cX^{\us}}(\cX)^{\geq w} \rangle.
\]
\end{enumerate}
\end{thm}

\begin{rem}
\Cref{T:derived_Kirwan_surjectivity_full} follows from the main theorem, \Cref{T:derived_Kirwan_surjectivity}, via a straightforward inductive argument. Parts (3)-(6) of that theorem immediately lead to analogous elaborations in the context of multiple strata, but for brevity we have omitted these details.
\end{rem}

\begin{proof}[Proof of \autoref{T:derived_Kirwan_surjectivity_full}]
Note that the $\Theta$-stratification on $\cX$ induces a $\Theta$-stratification on each open substack $\cX_{\leq \alpha}$, where the strata are $\cS_\beta$ for $\beta \leq \alpha$. We prove that the statement holds for $\cX_{\leq \alpha}$ for all $\alpha \in I \cup \{\infty\}$ using transfinite induction. For the minimal element $\cX_{\leq 0} = \cX^{\rm ss}$ there is nothing to prove.

Now assume that the statement holds for all $\cX_{\leq \beta}$ with $\beta < \alpha$, and consider the open substack
\[
\cX_{<\alpha} := \bigcup_{\beta < \alpha} \cX_{\leq \beta} = \cX_{\leq \alpha} \setminus \cS_\alpha.
\]
The $\Theta$-stratification of $\cX$ induces a $\Theta$-stratification of $\cX_{<\alpha}$, which we denote $\bS_{<\alpha} \subset \Filt(\cX_{<\alpha})$, and where $\cX_{<\alpha}^{us} := \bigcup_{\beta < \alpha} |\cS_\alpha|$. The inductive hypothesis combined with \Cref{L:sod_limits} gives a semiorthogonal decomposition
\begin{align*}
\APerf(\cX_{< \alpha}) &= \varprojlim_{\beta < \alpha} \APerf(\cX_{\leq \beta}) \\
&\cong \left \langle \APerf_{\cX_{<\alpha}^{\us}}(\cX_{<\alpha})^{<w}, \cG_{\bS_{<\alpha}}^w, \APerf_{\cX_{<\alpha}^{\us}}(\cX_{<\alpha})^{\geq w} \right \rangle.
\end{align*}
In addition, \Cref{T:derived_Kirwan_surjectivity} applied to the $\Theta$-stratum $\cS_\alpha \hookrightarrow \cX_{\leq \alpha}$ gives a semi-orthogonal decomposition
\begin{equation} \label{E:sod_outer_stratum}
\APerf(\cX_{\leq \alpha}) = \langle \APerf_{\cS_\alpha}(\cX_{\leq \alpha})^{<w_\alpha}, \cG_{\cS_{\alpha}}^{w_\alpha},\APerf_{\cS_\alpha}(\cX_{\leq \alpha})^{\geq w_\alpha}\rangle.
\end{equation}
Under the equivalence $\cG_{\cS_\alpha}^{w_\alpha} \cong \APerf(\cX_{<\alpha})$ given by restriction to $\cX_{<\alpha}$, we can combine \eqref{E:sod_outer_stratum} with the previous semiorthogonal decomposition to obtain a five-term semiorthogonal decomposition of $\APerf(\cX_{\leq \alpha})$.

Unraveling the definitions, the two left terms generate $\APerf_{\cX_{\leq \alpha}^{\us}}(\cX_{\leq \alpha})^{<w}$, the right two terms generate $\APerf_{\cX_{\leq \alpha}^{\us}}(\cX_{\leq \alpha})^{\geq w}$, and all of these categories together generate $\APerf_{\cX_{\leq \alpha}^{\us}}(\cX_{\leq \alpha})$ by \Cref{T:derived_Kirwan_surjectivity} part (2) and the inductive hypothesis. The middle term is precisely $\cG_{\bS_{\leq \alpha}}^w$, and restriction induces an equivalence $\cG_{\bS_{\leq \alpha}}^w \cong \APerf(\cX_{\leq 0})$.

\end{proof}

\subsection{Structure theorem for \texorpdfstring{$\Perf$}{Perf}}

We now prove a version of \Cref{T:derived_Kirwan_surjectivity} for the $\infty$-category $\Perf(\cX)$ under an additional hypotheses. One can regard $\Perf(-)$ as a cohomology theory valued in $\infty$-categories, and in this sense the theorem is a categorification of Kirwan surjectivity \cite{MR766741}.

\begin{defn}
A closed immersion of derived algebraic stacks $i : \cS \hookrightarrow \cX$ is a regular embedding if it is locally a.f.p. and $\bL_{\cS/\cX}[-1] \in \QC(\cS)$ is locally free, i.e., locally isomorphic to $\cO_\cS^{n}$ for some $n$.
\end{defn}

If $i$ is a regular embedding, then $i_\ast(\cO_\cS) \in \Perf(\cX)$, which is also equivalent to $i$ having locally finite Tor-amplitude, $i_\ast$ preserving perfect complexes, or to $i$ being locally eventually coconnective (see \cite{gaitsgory2013ind}*{Lem.~3.6.3}). 

If $\cX$ is a locally a.f.p. algebraic derived stack with affine automorphism groups and $i : \cS \hookrightarrow \cX$ is a $\Theta$-stratum with center $\cZ \to \cS$, then \Cref{L:relative_cotangent_complex} implies that $i$ is a regular embedding if and only if the positive weight summand $(\bL_\cX|_{\cZ})^{>0} \subset \QC(\cZ)$ is locally free. This in turn is equivalent to the condition that for any finite type point $\xi : \Spec(k') \to \cX$ and homomorphism $\lambda : (\Gm)_{k'} \to \Aut_\cX(\xi)$ that define a point of $\cZ \subset \Grad(\cX)$, the summand of $H_i(\xi^\ast(\bL_\cX))$ with positive $\lambda$-weights vanishes for $i > 0$. We recall the notation $\Perf(\cX)^{\geq w}$, $\Perf(\cX)^{<w}$, etc., from \Cref{D:subcategories_aperf}.


\begin{prop} \label{prop:DKS_perfect}
Let $\cX$ be an algebraic derived stack locally a.f.p. and with affine diagonal over a locally noetherian algebraic derived stack $\cB$, and let $i : \cS \hookrightarrow \cX$ be a $\Theta$-stratum with center $\sigma : \cZ \to \cS$. If $i$ is a regular embedding, then for any $w \in \bZ$ the intersection of $\Perf(\cX) \subset \APerf(\cX)$ with the semiorthogonal factors in \eqref{E:main_SOD}, \eqref{E:supports_SOD}, and \eqref{E:window_SOD} in \Cref{T:derived_Kirwan_surjectivity} define semiorthogonal decompositions of the corresponding subcategories of $\Perf(\cX)$. Furthermore: 
\begin{enumerate}
\item If we let $\eta$ be the weight of the line bundle $\det(\radj{1}(\bL_{\cX}|_\cZ))$, then
\begin{equation} \label{E:perf_factors}
\begin{array}{rl}
\Perf(\cX)^{\geq w} &= \left\{ F \in \Perf(\cX) | \sigma^\ast i^\ast(F) \in \Perf(\cZ)^{\geq w} \right \}, and \\
\Perf(\cX)^{< w} &= \left\{ F \in \Perf(\cX) | \sigma^\ast i^\ast(F) \in \Perf(\cZ)^{< w+\eta} \right \}.
\end{array}
\end{equation}
\item If we denote $\cG^w_{\cS,{\rm perf}} := \cG_\cS^w \cap \Perf(\cX) = \Perf(\cX)^{\geq w} \cap \Perf(\cX)^{<w}$, then restriction to the semistable locus $\cX^{\rm ss} := \cX \setminus \cS$ induces an equivalence $\cG_{\cS,{\rm perf}}^w \cong \Perf(\cX^{\rm ss})$.
\item If $\cX$ is quasi-compact,\footnote{The quasi-compactness is necessary to guarantee that every perfect complex lies in the subcategory generated by finitely many of the semiorthogonal factors.} then we have an infinite semiorthogonal decomposition
\[
\Perf(\cX) = \left \langle \lefteqn{\overbrace{\phantom{\ldots, \Perf_\cS(\cX)^{w-2}, \Perf_\cS(\cX)^{w-1}, \cG_{\cS,{\rm perf}}^w}}^{\Perf(\cX)^{<w}}} \ldots, \Perf_\cS(\cX)^{w-2}, \Perf_\cS(\cX)^{w-1}, \lefteqn{\underbrace{\phantom{\cG_{\cS,{\rm perf}}^w, \Perf_\cS(\cX)^{w}, \Perf_\cS(\cX)^{w+1}, \ldots}}_{\Perf(\cX)^{\geq w}}} \cG_{\cS,{\rm perf}}^w, \Perf_\cS(\cX)^{w}, \Perf_\cS(\cX)^{w+1}, \ldots \right\rangle.
\]
\item If $\pi = \ev_0 : \cS \to \cZ$ is the projection, then $i_\ast \pi^\ast : \Perf(\cZ)^{w} \to \Perf_\cS (\cX)^{w}$ is an equivalence.
\end{enumerate}
\end{prop}

We prove this proposition at the end of the subsection, after establishing some preliminary results.

\subsubsection{General properties of $i^{\QC,!}$}

\begin{lem} \label{L:pullback_rhom}
Let $f : \cY' \to \cY$ be a morphism of finite Tor-amplitude between algebraic derived stacks. Then for any $E \in \AAPerf(\cY)$ and $F \in \QC(\cY)_{<\infty}$, the pullback of the evaluation map $f^\ast(E) \otimes f^\ast(\inner{\RHom}_{\cY}^\otimes(E,F)) \to f^\ast(F)$ classifies a canonical isomorphism
\begin{equation} \label{E:pullback_rhom}
f^\ast(\inner{\RHom}_{\cY}^\otimes(E,F)) \to \inner{\RHom}_{\cY'}^\otimes(f^\ast E, f^\ast F).
\end{equation}
\end{lem}
\begin{proof}
The proof of \cite{preygel2011thomsebastiani}*{Lem.~A.1.1}, which is stated for fppf morphisms, works essentially verbatim. First one uses faithfully flat descent to reduce to the case where $\cY' = \Spec(B)$ and $\cY = \Spec(A)$. In this case one can write $E$ as a filtered colimit $E = \colim_{k} P_k$ indexed by $k \in \bN$, where $P_k$ is perfect and $\tau_{\leq n} P_k$ is eventually constant in $k$ for any $n$. Then $\inner{\RHom}_A^\otimes(P_k,F) \cong P_k^\dual \otimes F$, and one can identify the canonical homomorphism \eqref{E:pullback_rhom} with the canonical homomorphism
\[
B \otimes_A \left( \varprojlim_k P_k^\dual \otimes_A F \right) \to \varprojlim_k B \otimes_A (P_k^\dual \otimes_A F).
\]
By shifting $F$ we may assume $F \in \QC(\cY)_{\leq 0}$. If $E \in \QC(\cY)_{\geq n}$, then for any $E' \in \QC(\cY)_{>-n}$ the mapping space $\Map(E',\inner{\RHom}_A^\otimes(E,F)) \cong \Map(E' \otimes E, F)$ is contractible, so $\inner{\RHom}_A^\otimes(E,F) \in \QC(\cY)_{\leq -n}$. It follows that
\[
\tau_{\geq -n} (P_k^\dual \otimes F) \cong \tau_{\geq -n} \inner{\RHom}_A^\otimes(\tau_{\leq n} P_k, F)
\]
is eventually constant in $k$ for any fixed $n$. The claim of the lemma now follows from the general observation that because the formation of $\bN^{\rm op}$-indexed limits of $A$-modules has bounded homological amplitude, tensoring with an $A$-module of finite Tor-amplitude commutes with the formation of the limit of any $\bN^{\rm op}$-indexed inverse system of $A$-modules whose homology is eventually constant in every degree.
\end{proof}

Now consider a closed immersion of derived algebraic stacks $i : \cS \hookrightarrow \cX$. Using the projection formula, one can construct a natural isomorphism $\inner{\RHom}^\otimes_\cX(i_\ast(\cO_\cS),E) \cong i_\ast(i^{\QC,!}(E))$ for $E \in \QC(\cX)$, and under this isomorphism the map $\inner{\RHom}^\otimes_\cX(i_\ast(\cO_\cS),E) \to E$ induced by composition with $\cO_\cX \to i_\ast(\cO_\cS)$ naturally homotopic to the counit of adjunction.

\begin{lem} \label{L:pullback_qcshriek}
Consider a cartesian diagram of algebraic derived stacks
\[
\xymatrix{ \cS' \ar[r]^{i'} \ar[d]^{\pi'} & \cX' \ar[d]^\pi \\ \cS \ar[r]^i & \cX },
\]
such that $i$ is an almost finitely presented closed immersion and $\pi$ is of finite Tor-amplitude. Then the composition of natural maps
\[
\xymatrix{ (\pi')^\ast i^{\QC,!}(E) \ar[r]_-{\rm{unit}} & (i')^{\QC,!} i'_\ast (\pi')^\ast i^{\QC,!}(E) \ar[r]^{\cong}_-{\rm{base\ change}} & (i')^{\QC,!} \pi^\ast i_\ast i^{\QC,!}(E) \ar[r]_-{(i')^{\QC,!} \pi^\ast(\rm{counit})} & (i')^{\QC,!} \pi^\ast (E) }
\]
is an isomorphism for any $E \in \QC(\cX)$.
\end{lem}

\begin{proof}
Consider the commutative algebra objects $\cA = i_\ast(\cO_\cS)$ in $\QC(\cX)$ and $\cA' = i'_\ast(\cO_\cS')$ in $\QC(\cX')$. By the base change formula we have $\pi^\ast(\cA) \cong \cA'$, and by Barr-Beck-Lurie $i_\ast$ induces an equivalence $\QC(\cS) \cong \cA\Mod(\QC(\cX))$, and likewise for $\cS'$. By hypothesis the underlying quasi-coherent complex of $\cA$ is almost perfect, and likewise for $\cA'$.

Under these identifications, $i_\ast$ and $i'_\ast$ are the forgetful functors $\cA\Mod(\QC(\cX)) \to \QC(\cX)$ and $\cA'\Mod(\QC(\cX')) \to \QC(\cX')$. The right adjoint of these functors on $E \in \QC(\cX)$ is given by
\[
i^{\QC,!}(E) = \inner{\RHom}^\otimes_{\cX}(\cA,E),
\]
which we regard as an $\cA$-module object via the canonical action of the associative algebra object $\End(\cA)^{\rm op}$ in $\QC(\cX)$ along with the canonical map of associative algebra objects $\cA \to \End(\cA)^{\rm op}$. The analogous formula holds for $(i')^{\QC,!}(-)$. \Cref{L:pullback_rhom} gives an isomorphism
\begin{equation} \label{E:pullback_rhom_2}
\pi^\ast(\inner{\RHom}^\otimes_\cX(\cA,E)) \cong \inner{\RHom}^\otimes_{\cX'}(\pi^\ast(\cA),\pi^{\ast}(E))
\end{equation}
that is compatible with the module structure under the canonical map of associative algebra objects
\[
\pi^\ast(\End(\cA)^{\rm op}) \to \End(\cA')^{\rm op}.\footnote{\Cref{L:pullback_rhom} implies this is actually an isomorphism if $\cA \in \QC(\cX)_{<\infty}$.}
\]
It follows that the map \eqref{E:pullback_rhom_2} is an isomorphism of $\pi^\ast(\cA) \cong \cA'$ modules under the map of associative algebra objects $\pi^\ast(\cA) \to \pi^\ast(\End(\cA)^{\rm op})$.
\end{proof}

For any $E \in \QC(\cX)$, we can tensor with the counit of adjunction $i_\ast(i^{\QC,!}(\cO_\cX)) \to \cO_\cX$ with $E$ and apply the projection formula to obtain a natural map $i_\ast(i^\ast(E) \otimes i^{\QC,!}(\cO_\cX) ) \to E$. By adjunction, this defines a natural transformation of functors $\QC(\cX) \to \QC(\cS)$
\begin{equation} \label{E:regular_shriek_pullback_qc}
i^\ast(-) \otimes i^{\QC,!}(\cO_\cX) \to i^{\QC,!}(-).
\end{equation}

The following was establishing in \cite{arinkin2015singular}*{Cor.~2.2.7} over a field of characteristic $0$.

\begin{lem}\label{L:regular_closed_immersion_general}
If $i : \cS \hookrightarrow \cX$ is a closed immersion of finite Tor-amplitude between algebraic derived stacks, the natural transformation \eqref{E:regular_shriek_pullback_qc} is an isomorphism, and furthermore if $i$ is a regular closed immersion, then $i^{\QC,!}(\cO_\cX) \cong \det(\bL_{\cS/\cX}[-1])[\rank(\bL_{\cS/\cX})]$.
\end{lem}
\begin{proof}
As discussed above, $\inner{\RHom}^\otimes_\cX(i_\ast(\cO_\cS),E) \cong i_\ast(i^{\QC,!}(E))$ for $E \in \QC(\cX)$. If $i : \cS \to \cX$ has finite Tor-amplitude, then $i_\ast(\cO_\cS) \in \QC(\cX)$ is a perfect complex, so this gives a natural isomorphism
\[
i_\ast(i^{\QC,!}(E)) \cong (i_\ast \cO_\cS)^\dual \otimes E,
\]
where $(i_\ast \cO_\cS)^\dual$ denotes the linear dual of $i_\ast \cO_\cS \in \Perf(\cX)$. This formula implies that $i^{\QC,!}(-)$ commutes with filtered colimits, and that $i^{\QC,!}(\cO_\cX) \in \QC(\cX)_{\geq m}$ for some $m$.

The fact that $i^{\QC,!}(\cO_\cX)$ is homologically bounded below implies that both sides of \eqref{E:regular_shriek_pullback_qc} commute with the limit $E \cong \varprojlim_n \tau_{\leq n}(E)$ for any $E \in \QC(\cX)$. In particular, it suffices to verify that \eqref{E:regular_shriek_pullback_qc} is an isomorphism for $E \in \QC(\cX)_{\leq d}$ for all $d$. By \Cref{L:pullback_qcshriek}, the formation of $i^{\QC,!}(E)$ is smooth local on $\cX$ when $E \in \QC(\cX)_{\leq d}$, so it further suffices to prove that \eqref{E:regular_shriek_pullback_qc} is an isomorphism when $\cX = \Spec(A)$ is affine. Note that \eqref{E:regular_shriek_pullback_qc} is an isomorphism when $E=\cO_\cX$, so because both sides commute with filtered colimits and $\cO_\cX$ generates $\QC(\cX)$ under filtered colimits when $\cX$ is affine, we conclude that \eqref{E:regular_shriek_pullback_qc} is an isomorphism in this case.

Next we show that if $i$ is a regular closed immersion, then $i^{\QC,!}(\cO_\cX)$ is a shift of an invertible sheaf. The claim is smooth local by \Cref{L:pullback_qcshriek}, so we may assume $\cX = \Spec(A)$ is affine. In a Zariski open affine neighborhood of any point $x \in \cS$, one can construct \cite[Prop.~2.3.8]{KhanRydh} a cartesian diagram
\[
\xymatrix{ \cS \ar@{^{(}->}[r]^i \ar[d]^{\phi'} & \Spec(A) \ar[d]^\phi \\
\{0\}/\GL_n \ar@{^{(}->}[r]^{i'} & \bA^n / \GL_n }.
\]
The map $\phi$ has finite Tor-amplitude because $\bA^n/\GL_n$ is regular. \Cref{L:pullback_qcshriek} then shows that $i^{\QC,!}(\cO_\cX) \cong (\phi')^\ast ((i')^{\QC,!}(\cO_{\bA^n/\GL_n}))$, so it suffices to prove the claim for the example of the inclusion $i : \{0\} / \GL_n \hookrightarrow \bA^n / \GL_n$. One can make use of the cosection $p : \bA^n / \GL_n \to B\GL_n$ to compute
\[
i^{\QC,!}(\cO_{\bA^n}) \cong p_\ast(i_\ast(i^{\QC,!}(\cO_{\bA^n}))) \cong R\Gamma(\bA^n, \inner{\RHom}^\otimes_{\bA^n}(\cO_{\{0\}},\cO_{\bA^n})).
\]
Using the Koszul resolution of $\cO_{\{0\}} \in \QC(\bA^n/\GL_n)$, one computes $i^{\QC,!}(\cO_{\bA^n}) \cong \det[-n] \in \QC(B\GL_n)$, where $\det$ is the determinant of the standard representation.


We now prove the formula $i^{\QC,!}(\cO_\cX) \cong \det(\bL_{\cS/\cX}[-1])[\rank(\bL_{\cS/\cX})]$ when $i$ is a regular closed immersion using the method of \cite{GR2}*{Thm.~II.9.7.2.2}, which can be simplified in our context by using the more geometric construction of deformation to the normal cone in \cite{KhanRydh}. By \cite[Thm.~4.1.13]{KhanRydh} we can factor the closed immersion $\id_{\bA^1} \times i : \bA^1 \times \cS \to \bA^1\times \cX$ over $\bA^1$ as
\[
\xymatrix{\bA^1 \times \cS \ar[r]^{j} & \cD_{\cS/\cX} \ar[r]^\pi & \bA^1 \times \cX },
\]
with the following properties: 1) $j$ is a regular closed immersion; 2) $\pi$ is an isomorphism after base change to $\bA^1 \setminus \{0\}$; and 3) after base change to $\{0 \} \in \bA^1$, $(\cD_{\cS/\cX})_0 \cong \Spec_\cS(\Sym_\cS(\bL_{\cS/\cX}[-1]))$ and $j_0$ is isomorphic to the zero section. Futhermore, it is not stated explicitly, but is implicit in the construction that this family over $\bA^1$ is $\bG_m$-equivariant, so we can replace $\bA^1$ with $\Theta$ above.

We have already shown that $L := j^{\QC,!}(\cO_{\cD_{\cS/\cX}})$ is a shift of an invertible sheaf on $\Theta \times \cS$. If $p : \Theta \times \cS \to B\bG_m \times \cS$ is the map induced by the equivariant projection $\bA^1 \to \{0\}$, then the pullback functor induces an equivalence of $\infty$-groupoids $\Pic(B\bG_m \times \cS) \to \Pic(\Theta \times \cS)$ for any derived stack $\cS$, and the inverse functor is given by restriction along the inclusion $B\bG_m \times \cS \hookrightarrow \Theta \times \cS$. This claim is local on $\cS$, and it can be checked easily in the affine case. Thus if $L_0 := L|_{B\bG_m \times \cS}$, one has a canonical isomorphism $p^\ast(L_0) \to L$ in $\QC(\Theta \times \cS)$.

The inclusion $\{0\}/\bG_m \to \bA^1/\bG_m$ has finite Tor-amplitude, so \Cref{L:pullback_qcshriek} implies that
\[
L_0 \cong j_0^{\QC,!}(\Spec_\cS(\Sym_\cS(\bL_{\cS/\cX}[-1]))).
\]
Because $j_0$ is the inclusion of the zero section of a locally free sheaf of rank $n$ on $\cS$, it is the base change of the universal example $\{0\}/\GL_n \hookrightarrow \bA^n / \GL_n$, where $n = -\rank(\bL_{\cS/\cX})$. It follows from the explicit computation of this example above that
\[
i^{\QC,!}(\cO_\cX) \cong (p^\ast L_0)|_{\{1\}\times \cS} \cong L_0 \cong \det(\bL_{\cS/\cX}[-1])[\rank \bL_{\cS/\cX}],
\]
which completes the proof.
\end{proof}

\subsubsection{Proof of \Cref{prop:DKS_perfect}}


First we show that the alternate description of the categories in \eqref{E:perf_factors} holds. In light of \Cref{D:subcategories_aperf} and \Cref{L:aperf_baric_criterion}, it suffices to show that for $F \in \APerf(\cX)$, $i^{\QC,!}(F) \in \QC(\cS)^{<w}$ if and only if $i^\ast(F) \in \APerf(\cS)^{<w+\eta}$. This follows immediately from the canonical equivalence of \Cref{L:regular_closed_immersion_general} and \Cref{L:relative_cotangent_complex}, which identifies the weight of $\det(\bL_{\cS/\cX})$ with the weight of $\det(\radj{1}(\bL_\cX|_{\cZ}[1]))$, which is $-\eta$.

In order to prove that the semiorthogonal decompositions \eqref{E:main_SOD}, \eqref{E:supports_SOD}, and \eqref{E:window_SOD} induce semiorthogonal decompositions of $\Perf(\cX)$, $\Perf_\cS(\cX)$ and $\Perf(\cX)^{<w} \cap \Perf(\cX)^{\geq u}$ respectively, it suffices to show that the functors $\radj{w}_\cS$ and $\ladj{w}_\cS$ on $\APerf(\cX)$ preserve $\Perf(\cX)$. As in the proof of \Cref{T:derived_Kirwan_surjectivity}, we can use \Cref{L:induced_stratum_base_change}, \Cref{L:sod_limits}, and \Cref{T:local_structure_stratum} to reduce first to the case where $\cX$ is quasi-compact, and then to the case where $\cX = X/\Gm$ for a noetherian affine derived $\Gm$-scheme $X$ with the tautological $\Theta$-stratum $S / \Gm \hookrightarrow X/\Gm$.

Because $i$ is a regular embedding, $i_\ast$ maps $\Perf(\cS)$ to $\Perf(\cX)$. It follows from \Cref{P:baric_decomp_supports} part (2) and \Cref{P:baric_stratum} part (3) that $\radj{w}_{\cS}$ and $\ladj{w}$ preserve the essential image $i_\ast(\Perf(\cS))$. Because $i_\ast(\Perf(\cS))$ generates $\Perf_\cS(\cX)$ under shifts, cones, and retracts, $\radj{w}_\cS$ and $\ladj{w}$ preserve $\Perf_\cS(\cX)$.

Now consider $F \in \Perf(\cX)$, where $\cX = X/\Gm$, and let $K_{X}(x_1^n,\ldots,x_k^n)$ be the Koszul complexes of \eqref{E:koszul_complex} in \Cref{sect:koszul}. If we define $C_n := \cofib(K_X(x_1^n,\ldots,x_k^n) \to \cO_X)$, then as $n \to \infty$ the all of the weights of $C_n|_Z$ tend to $-\infty$. It follows from the description of $\Perf(X/\Gm)^{<w}$ in \eqref{E:perf_factors} that $C_n \otimes F \in \Perf(X/\Gm)^{<w}$ for $n \gg 0$, and thus the canonical homomorphism $K_X(x_1^n,\ldots,x_k^n) \to \cO_X$ induces an equivalence
\[
\radj{w}_\cS(F \otimes K_X(x_1^n,\ldots,x_k^n)) \cong \radj{w}_\cS(F).
\]
The same argument implies that the dual morphism $\cO_X \to K_X(x_1^n,\ldots,x_k^n)^\dual$ induces an equivalence
\[
\ladj{w}_\cS(F) \cong \ladj{w}_\cS(F \otimes K_X(x_1^n,\ldots,x_k^n)^\dual),
\]
for $n \gg 0$. It follows that $\radj{w}_\cS(F)$ and $\ladj{w}_\cS(F)$ lie in $\Perf(X/\Gm)$, and thus we have our semiorthogonal decomposition
\[
\Perf(\cX) = \langle \Perf_\cS(\cX)^{\geq w}, \cG^w_{\cS,{\rm perf}}, \Perf_\cS(\cX)^{<w} \rangle.
\]

\medskip
\noindent{\textit{Proof of (2):}}
\medskip

It suffices again to assume $\cX = X/\Gm$. We know that $\cG_{\cS,{\rm perf}}^w \to \Perf(X^{\rm ss} / \Gm)$ is fully faithful, so it suffices to show it is essentially surjective. The image of the restriction functor $\Perf(X/\Gm) \to \Perf(X^{\rm ss} / \Gm)$ generates the latter up to retracts. The fact that any $F \in \Perf(X/\Gm)$ has the same restriction to $X^{\rm ss}$ as $\radj{w}(\ladj{w}(F)) \in \cG^w_{\cS,{\rm perf}}$ implies that the essential image of $\cG_{\cS,{\rm perf}}^w$ also generates $\Perf(X^{\rm ss}/\Gm)$ up to retracts. $\cG_{\cS,{\rm perf}}^w$ is idempotent complete by its definition, so the fully faithfulness guarantees that any retract of an object in the essential image of $\cG_{\cS,{\rm perf}}^w \to \Perf(X^{\rm ss} /\Gm)$ also lies in the essential image.

\medskip
\noindent{\textit{Proof of (3):}}
\medskip

It follows from the fact that $\radj{w}_\cS$ and $\ladj{w}_\cS$ preserve $\Perf_\cS(\cX)$ and part (4) of \Cref{T:derived_Kirwan_surjectivity}, that for any $u \leq v \leq w$ we have a semiorthogonal decomposition
\[
\cG_{\cS,{\rm perf}}^{[u,w)} = \langle \Perf_\cS(\cX)^u,\ldots,\Perf_\cS(\cX)^{v-1},\cG_{\cS,{\rm perf}}^v,\Perf_\cS(\cX)^v,\ldots,\Perf_\cS(\cX)^{w-1} \rangle.
\]
So the infinite semiorthogonal decomposition of $\Perf(\cX)$ follows from \eqref{E:perf_factors}, which implies that $\Perf(\cX) = \bigcup_{u,w} \cG_{\cS,{\rm perf}}^{[u,w)}$.

\medskip
\noindent{\textit{Proof of (4):}}
\medskip

This follows from \Cref{T:derived_Kirwan_surjectivity} part (3), and the fact that both $i_\ast \pi^\ast(-)$ and the inverse functor $\ladj{w+1}(\sigma^\ast i^\ast(-))$ preserve perfect complexes.

\section{The main theorem: quasi-smooth case} \label{S:quasi-smooth}

For the rest of the paper, we work over a field $k$ of characteristic $0$. This allows us to make use of the theory of Ind-coherent sheaves on derived stacks developed in \cites{gaitsgory2013ind, drinfeld2013some, GR1, GR2}. For an algebraic derived stack almost of finite type and with affine automorphism groups over $k$, the category $\IC(\cX)$ is the ind-completion of the $\infty$-category $\DCoh(\cX)$ of bounded quasi-coherent complexes with coherent cohomology groups \cite{drinfeld2013some}*{Thm.~0.4.5}.

\subsection{Quasi-smooth derived stacks}

\begin{defn} \label{D:quasi_smooth}
A morphim of algebraic derived stacks $\cX \to \cY$ is \emph{quasi-smooth} if it is locally almost of finite presentation and the relative cotangent complex $\bL_{\cX/\cY}$ is perfect and has Tor-amplitude in $[-1,1]$. For a $k$-stack $\cX$, we say that $\cX$ is quasi-smooth if $\cX \to \Spec(k)$ is quasi-smooth.
\end{defn}

\begin{ex}
If $\cX_0,\cX_1 \to \cY$ are maps between smooth stacks over a base $\cB$, then the derived fiber product $\cX' := \cX_0 \times_{\cY} \cX_1$ is quasi-smooth over $\cB$.

A special case is the derived zero locus of a section $s$ of a locally free sheaf $V$ on a smooth stack $\cX$, which is by definition the derived intersection of $s$ with the zero section of $\Tot_\cX(V)$. If $\cU_\bullet$ denotes the simplicial $\cO_\cX$-module associated under the Dold-Kan correspondence to the two-term complex $\cV^\dual \to \cO_\cX$ defined by the section $s$, then the derived $0$ locus of $s$ is the relative $\Spec$ over $\cX$ of the free simplicial commutative $\cO_\cX$-algebra $\Sym_{\cO_\cX}(\cU_\bullet)$, whose $n^{th}$ level is just $\Sym_{\cO_\cX}(\cU_n)$.
\end{ex}

\begin{ex}
If $S$ is a smooth surface, then $\cX = \inner{\Coh}(S)$ is a quasi-smooth derived stack. By \cite{toen_vaquie}*{Cor.~3.17} the fiber of the cotangent complex $\bL_\cX$ at a point $[E] \in \inner{\Coh}(S)$ is
\[
\bL_{\cX,[E]} \simeq \RHom_S(E,E[1])^\dual \simeq \RHom_S(E,E)^\dual [-1].
\]
$\RHom$ between coherent sheaves has homology in positive cohomological degree only, so this combined with Serre duality $\RHom_S(E,E)^\dual \simeq \RHom(E,E\otimes \omega_S[2])$ implies that $\bL_{\cX,[E]}$ has homology in degree $-1,0,$ and $1$. Because $\bL_\cX \in \APerf(\inner{\Coh}(S))$, this implies that $\bL_{\cX}$ is perfect with Tor-amplitude in $[-1,1]$, hence $\inner{\Coh}(S)$ is quasi-smooth.
\end{ex}

One key observation about quasi-smooth stacks is the following:

\begin{lem} \label{lem:finite_tor_amplitude}
Let $\cX$ be an algebraic derived stack locally a.f.p. and with affine automorphism groups over a base algebraic derived $k$-stack $\cB$, and let $i : \cS \hookrightarrow \cX$ be a $\Theta$-stratum with center $\cZ$. Then if $\cX$ is quasi-smooth over $\cB$ or $k$, then so are $\cS$ and $\cZ$, and the projection $\pi : \cS \to \cZ$ is quasi-smooth as well.
\end{lem}
\begin{proof}
Working relative to $\cB$, \Cref{L:relative_cotangent_complex} and \Cref{P:baric_stratum} part (4) imply
\[
\bL_{\cS/\cB}|_\cZ \simeq \ladj{1}(\bL_{\cX/\cB}|_{\cS})|_{\cZ} \simeq (\bL_{\cX/\cB}|_\cZ)^{\leq 0}
\]
is perfect and concentrated in degrees $-1$,$0$, and $1$. Because every point in $\cS$ specializes to one in the image of $\sigma : \cZ \to \cS$, $\bL_{\cS/\cB}$ is perfect with fiber homology concentrated in those degrees as well. It follows from \Cref{L:center_cotangent} that $\bL_{\cZ/\cB} \cong (\bL_{\cX/\cB}|_{\cZ})^0$ is perfect with fiber homology concentrated in degree $-1,0,$ and $1$, and thus $\cZ$ is quasi-smooth as well. The pullback of $\bL_{\cB / \Spec(k)}$ to $\cX$ has weight $0$, so the previous formulas remain true with $\cB$ replaced by $\Spec(k)$, and thus $\cS$ and $\cZ$ are quasi-smooth over $\Spec(k)$ when $\cX$ is.

Finally, considering the fiber sequence of cotangent complexes coming from the composition $\cZ \xrightarrow{\sigma} \cS \xrightarrow{\pi} \cZ$, which is isomorphic to the identity, we have an isomorphism $\sigma^\ast \bL_{(\pi : \cS \to \cZ)} \cong \bL_{\cZ/\cS}[-1] \cong (\bL_{\cX}|_\cZ)^{<0}$, so this is a perfect complex with fiber homology in degree $-1,0,$ and $1$ as well.
\end{proof}

\subsubsection{Serre duality for quasi-smooth stacks}

We recall the key constructions of Serre duality in the derived setting \cite{drinfeld2013some}*{Sect.~4.4}. For any algebraic derived stack $\cX$ which is a.f.p. and has affine automorphism groups over $\Spec(k)$, there is a canonical complex $\omega_\cX := \pi^!(k) \in \IC(\cX)$, where $\pi$ is the projection to $\Spec(k)$ and $\pi^! : \IC(\Spec k) \to \IC(\cX)$ is the shriek pullback functor on ind-coherent sheaves. The inner Hom for the $\QC(\cX)^\otimes$-module category $\IC(\cX)$ defines a functor
\[
\bD_\cX(\bullet) := \inner{\RHom}_{\cX}^\otimes(\bullet,\omega_\cX) : \IC(\cX) \to \QC(\cX)^{\rm op}
\]
that induces an equivalence $\DCoh(\cX) \to \DCoh(\cX)^{\rm op}$. Ind-coherent sheaves have a canonical symmetric monoidal structure given by
\[
F \itimes G := \Delta^!(F \boxtimes G),
\]
where $\Delta : \cX \to \cX \times \cX$ is the diagonal. $\omega_\cX$ is uniquely characterized as the being the unit for this monoidal $\infty$-category \cite{GR1}*{Sect.~I.5.4.3.4} \cite{HA}*{Cor.~5.4.4.7}.

When $\cX$ is eventually coconnective, $\omega_\cX \in \DCoh(\cX) \subset \IC(\cX)$, and we regard $\omega_\cX$ as an element of $\QC(\cX)$ under the canonical isomorphism $\Psi_\cX : \IC(\cX)_{<\infty} \to \QC(\cX)_{<\infty}$. Furthermore, if $\cX$ is derived Gorenstein, which is the case when $\cX$ is quasi-smooth \cite{arinkin2015singular}*{Cor.~2.2.7}, then $\omega_\cX$ is an invertible complex.


\subsubsection{Computing the canonical complex}

For a smooth separated $dg$-scheme $X$, one has a canonical isomorphism \cite{GR2}*{Prop.~II.9.7.3.4}
\[
\omega_X := (X \to \Spec(k))^!(k) \cong \det(\bL_\cX)[\rank \bL_X],
\]
where the right hand side is the ``graded determinant'' (see \Cref{A:grothendieck}) regarded as an Ind-coherent sheaf via the embeddings $\Perf(X) \subset \DCoh(X) \subset \IC(X)$. For any quasi-smooth stack that admits a locally closed embedding into $\bP^n/G$, one can deduce the same formula using Grothendieck's formula \cite{GR2}*{Cor.~II.9.7.3.2} and an explicit computation on $\bP^n/G$, and we expect it to hold for any quasi-smooth algebraic derived stack.

To our knowledge, however, a general result of this form does not appear in the literature, even for smooth stacks, and we were unable to find a simple general argument. In private communication, Dima Arinkin has sketched a proof using deformation to the normal cone, but we will content ourselves here with a weaker claim that covers our applications.

\begin{prop} \label{P:canonical_complex_classical}
Let $f : \cX \to \cY$ be a morphism between quasi-smooth locally a.f.p. algebraic derived $k$-stacks, and let $i : \cX^{\rm cl} \hookrightarrow \cX$ be the inclusion of the underlying classical stack. Then regarding $\omega_{\cX/\cY} := f^{!}(\cO_\cY) \in \IC(\cX)_{<\infty}$ as a quasi-coherent complex, there is a canonical isomorphism
\[
i^\ast(\omega_{\cX/\cY})\cong i^\ast(\det(\bL_{\cX/\cY})[\rank \bL_{\cX/\cY}]).
\]
\end{prop}

The proof is given in \Cref{A:grothendieck}, but we note the following two useful corollaries.

\begin{cor} \label{C:canonical_weight}
Let $\cX$ be a quasi-smooth locally a.f.p. algebraic derived $k$-stack. For any field extension $k \subset k'$, any $p : \Spec(k') \to \in \cX$, and any homomorphism $(\Gm)_{k'} \to \Aut_\cX(p)$, the fibers $p^\ast \omega_\cX$ and $p^\ast (\det(\bL_\cX)[\rank \bL_\cX])$ have the same weight with respect to the induced $(\Gm)_{k'}$-action on each.
\end{cor}
\begin{proof}
This follows from \Cref{P:canonical_complex_classical} and the fact that the map $(B\Gm)_{k'} \to \cX$ factors through $\cX^{\rm cl}$.
\end{proof}

\begin{cor} \label{C:canonical_class}
Let $\cX$ be a quasi-smooth a.f.p. algebraic derived $k$-stack. Then in the Grothendieck $K$-group $K_0(\DCoh(\cX))$,
\[
[\omega_\cX] = (-1)^{\rank \bL_\cX} [\det(\bL_\cX)].
\]
\end{cor}
\begin{proof}
Note that $[E] = \sum (-1)^n [H_n(E)]$, where the sum is finite because $\cX$ is quasi-compact. We have an isomorphism of $\cO_{\cX}$-modules $H_n(E) \cong i_\ast(H_n(E))$, where the latter denotes the canonical $\cO_{\cX^{\rm cl}}$-module structure on $H_n(E)$ and $i : \cX^{\rm cl} \to \cX$ is the inclusion. When $E$ is locally free, hence flat, we have
\[
H_n(E) \cong H_n(\cO_\cX) \otimes E \cong i_\ast(H_n(\cO_\cX) \otimes_{\cO_{\cX^{\rm cl}}} i^\ast(E)),
\]
where the last isomorphism is the projection formula. The claim now follows from \Cref{P:canonical_complex_classical}.
\end{proof}

\subsection{A reformulation of the grade-restriction rules}
\label{S:alternative_grr_condition}

Given an extension $k'/k$ we regard a $\bZ$-graded vector space $U$ as a $(B\bG_m)_{k'}$-representation. We let $k'[U[1]]$ denote the free simplicial commutative graded $k'$-algebra on the simplicial graded vector space that corresponds, under the Dold-Kan correspondence, to the single-term complex of graded vector spaces $U[1]$.

\begin{defn} \label{D:regularization}
A \emph{regular graded point} of a quasi-smooth algebraic derived $k$ stack $\cX$ consists of a $(\Gm)_{k'}$-representation $U$ and a map
\begin{equation} \label{E:regular_graded_point}
f_{\rm reg} : \Spec(k'[U[1]]) / (\Gm)_{k'} \to \cX,
\end{equation}
such that if $0 \in \Spec(k'[U[1]])$ is the unique $k'$ point and $x = f_{\rm reg}(0) \in \cX(k')$, then fiber over $0$ of the map of cotangent complexes induced by $f_{\rm reg}$
\begin{equation}\label{E:regular_graded_point_extension}
\bL_{\cX,x} \to U[1] \cong \bL_{\Spec(k'[U[1]])}|_0
\end{equation}
induces an isomorphism $H_1(\bL_{\cX,x}) \cong U$.
\end{defn}

The map \eqref{E:regular_graded_point} is a square-zero extension of the underlying graded point $f : (B\Gm)_{k'} \to \cX$, and is thus determined up to isomorphism by the element \eqref{E:regular_graded_point_extension} in $\Hom_{(B\Gm)_{k'}}(f^\ast(\bL_\cX),U[1])$.

\begin{const}
Given a graded point $f : (B\Gm)_{k'} \to \cX$ one can construct a \emph{regularization}, a regular graded point $f_{\rm reg} : \Spec(k'[U[1]]) / (\Gm)_{k'} \to \cX$ whose underlying graded point is $f$. $f_{\rm reg}$ is the square-zero extension associated to a choice of $\Gm$-equivariant splitting of the canonical map $H_1(\bL_{\cX,x})[1] \to \bL_{\cX,x}$. $f_{\rm reg}$ is unique up to $2$-isomorphism and composition with an automorphism of $U$.
\end{const}

Given a regular graded point $f_{\rm reg}$, the relative cotangent complex of the induced map $\Spec(k'[U[1]]) \to \cX$ is $\tau_{\geq 0}(\bL_{\cX,x}) [1]$. Because this is concentrated in homological degree $0$ and $1$, $f_{\rm reg}$ is a quasi-smooth morphism. It follows that $f_{\rm reg}$ is of finite Tor-amplitude, so the pullback functor $f_{\rm reg}^\ast$ maps $\DCoh(\cX)$ to $\DCoh(\Spec(k'[U[1]])/\Gm)$ \cite{arinkin2015singular}*{Cor.~2.2.4}. For concreteness, we identify the latter category with $\DCoh^{\Gm}(k'[U[1]])$,\footnote{Here and elsewhere, we use a group in the superscript to denote the category of complexes of equivariant sheaves.} the derived $dg$-category of graded $k'[U[1]]$-modules with bounded finite dimensional cohomology.

Any $F \in \DCoh^{\Gm}(k'[U[1]])$ can be canonically constructed by a finite sequence of extensions from its homology modules. Because $U$ is in homological degree 1, it acts trivially on $H_\ast(F)$, so each homology module is a direct sum of graded modules of the form $k'\langle n \rangle$, i.e., the one dimensional module with weight $-n \in \bZ$.

\begin{defn}
If $U$ is a $\Gm$-representation, for $F \in \DCoh^{\Gm}(k'[U[1]])$, we define $\minwt(F)$ and $\maxwt(F)$ to be respectively the smallest and largest $\Gm$ weight appearing in the $\Gm$-representation $H_\ast(F)$.
\end{defn}

\begin{lem} \label{L:weight_duality}
For any $(\Gm)_{k'}$-representation $U$ and any $F \in \DCoh^{\Gm}(k'[U[1]])$, we have
\[
\minwt(\bD_{k'[U[1]]}(F)) = - \maxwt(F).
\]
If $U$ has nonnegative weights, and $F_0 = F \otimes_{k'[U[1]]} k' \in \APerf((B\Gm)_{k'})$, then we have $\minwt(F_0) = \minwt(F) > -\infty$, where the former denotes the smallest non-vanishing weight space in $H_\ast(F_0)$.
\end{lem}
\begin{proof}
For the first equality, observe that $\minwt(F)$ and $\maxwt(F)$ are the lower and upper bounds of the smallest interval $[a,b]$ such that $F$ lies in the subcategory of $\DCoh^{\Gm}(k'[U[1]])$ generated under shifts and cones by $k\langle -n \rangle$ for $n \in [a,b]$. Because
\[
\bD_{k'[U[1]]}(k'\langle n \rangle) = \RHom_{k'[U[1]]}(k' \langle n \rangle,\omega_{k'[U[1]]}) \cong k'\langle -n \rangle
\]
the corresponding interval for $\bD_{k'[U[1]]}(F)$ is $[-a,-b]$.

For the second equality, note that $\minwt(F) \leq \minwt(F_0)$ as an immediate consequence of the fact that $k' \cong k'[U[2],U[1]; d = \id_U]$ is a semi-free presentation of $k'$ with generators of nonnegative weight only. The inequality $\minwt(F_0) \leq \minwt(F)$ follows from considering a quasi-isomorphism $F \cong (k'[U[1]] \otimes_k V, d)$ with a graded semi-free $k'[U[1]]$-module, where $V$ is a vector space that is graded both homologically and as a $\Gm$-representation, and such that $d \otimes_{k'[U[1]]} k' = 0$. The construction of such a quasi-isomorphism is the same as the usual construction of a minimal presentation of a complex of graded modules over the polynomial ring $k[U]$. If $H_\ast(F)$ is non-zero in weight $w$ and $U$ has nonnegative weights, then $F_0 \cong V$ must be non-vanishing in some weight $\leq w$, and hence $\minwt(F_0) \leq w$.
\end{proof}

\begin{lem} \label{L:bounded_weights}
Let $\cX$ be a quasi-smooth algebraic derived $k$-stack, let $f : (B\Gm)_{k'} \to \cX$ be a finite type graded point such that $H_1(f^\ast(\bL_\cX))$ has nonnegative weights, and let $f_{\rm reg} : \Spec(k'[U[1]])/\Gm \to \cX$ be a regularization of $f$. For any $F \in \DCoh(\cX)$, the weights of $f^\ast(F) \in \APerf((B\Gm)_{k'})$ are bounded below, and
\begin{align*}
\minwt(f^\ast(F)) &= \minwt(f_{\rm reg}^\ast(F)) \\
\minwt(f^\ast(\bD_\cX(F))) &= \wt(\det(\tau_{\leq 0}(f^\ast \bL_\cX))) - \maxwt(f_{\rm reg}^\ast(F))
\end{align*}
\end{lem}
\begin{proof}
If $(\Gm)_{k'} \to \Aut_\cX(f(0))$ is trivial, then all of the claims are tautological. Otherwise, $f$ factors through the graded point $f' : (B\Gm)_{k'} \to \cX$ associated to the image of this homomorphism $(\Gm)_{k'} \subset \Aut_\cX(f(0))$. Pullback along the map $(B\Gm)_{k'} \to (B\Gm)_{k'}$ has the effect of scaling weights by a non-zero constant, so it suffices to prove the claims for $f'$, i.e., we may assume that $f$ is a representable morphism.

\Cref{L:weight_duality} and the fact that $f_{\rm reg}$ has finite Tor-amplitude immediately imply the first equality, and in particular that the weights of $f^\ast(F)$ are bounded below. The isomorphism
\[
\bD_{k'[U[1]]}(f_{\rm reg}^\ast(\bD_\cX(F))) \cong f_{\rm reg}^!(F)
\]
of \cite{drinfeld2013some}*{Prop.~4.4.11} combined with \Cref{L:weight_duality} implies that the second equality of this lemma is equivalent to
\begin{equation} \label{E:maxwt_double_dual}
\maxwt(f_{\rm reg}^!(F)) = \maxwt(f_{\rm reg}^\ast(F)) - \wt(\det(\tau_{\leq 0} (f^\ast(\bL_\cX)))).
\end{equation}
We claim that this inequality follows from the formula
\begin{equation} \label{E:small_grothendieck_formula}
f_{\rm reg}^!(F) \cong f_{\rm reg}^\ast(F) \otimes \det(\bL_{f_{\rm reg}}) [\rank(\bL_{f_{\rm reg}})].
\end{equation}
Indeed, tensoring by the determinant of the relative cotangent complex $\det(\bL_{f_{\rm reg}})$ has the effect of shifting the homology weights of $f^\ast(F)$ by the weight of the fiber $\det(\bL_{f_{\rm reg}}|_0)$, and we have already observed that $\bL_{f_{\rm reg}}|_0 \cong \tau_{\leq 0}(\bL_{\cX,x})[1]$.

The Grothendieck formula \eqref{E:small_grothendieck_formula} has been established for representable smooth morphisms and regular closed immersions in \cite{GR2}*{Sect.~II.9.7.2}. If two quasi-smooth morphisms satsify \eqref{E:small_grothendieck_formula}, then their composition does as well, and using this one can show that the formula holds for any quasi-smooth map of affine $\Gm$-schemes. Using a variant of \Cref{T:local_structure_stratum}, whose proof uses the same inductive deformation theory argument but takes the classical result \cite{halpern2014structure}*{Lem.~4.0.4.6} as the base case, one may lift the graded point $f$ along a smooth cover of the form $X/\Gm \to \cX$ where $X$ is an affine quasi-smooth $\Gm$-scheme. Because this cover is smooth one can also lift $f_{\rm reg}$ to a map $\Spec(k'[U[1]])/\Gm \to X/\Gm$. Thus $f_{\rm reg}$ can be factored as the composition of map induced by a quasi-smooth equivariant map of affine $\Gm$-schemes followed by a smooth morphism, so the formula \eqref{E:small_grothendieck_formula} follows.
\end{proof}

\subsubsection{Grade restriction rules}

We now use the notion of regular graded points to describe when $F \in \DCoh(\cX)$ lies in $\APerf(\cX)^{\geq w}$ or $\APerf(\cX)^{<w}$. We refer to these conditions as \emph{grade restriction rules}

\begin{prop} \label{P:alternate_characterization}
Let $\cX$ be an algebraic derived stack locally a.f.p. and with affine diagonal over a locally a.f.p. algebraic derived $k$-stack $\cB$. Assume that $\cX$ is quasi-smooth over $k$. Let $i : \cS \subset \Filt(\cX) \to \cX$ be a $\Theta$-stratum with center $\sigma : \cZ \to \cS$, and assume that
\begin{itemize}
\item[\namedlabel{eqn:obstruction_weights}{$(\dagger)$}] for every finite type point of $\cZ$, which classifies a finite type point $\xi$ of $\cX$ along with a homomorphism $\Gm \to \Aut_\cX(\xi)$, the weights of $H_1(\xi^\ast \bL_{\cX})$ are $\geq 0$ with respect to this $\Gm$-action.
\end{itemize}
Let $a$ denote the weight of the invertible sheaf $\det(\ladj{1}(\bL_{\cX}|_\cZ))$. Then for $F \in \DCoh(\cX)$, the following are equivalent:
\begin{enumerate}
\item $F$ lies in $\APerf(\cX)^{\geq w}$ \hfill (respectively $F$ lies in $\APerf(\cX)^{<w}$);
\item $\sigma^\ast i^\ast(F) \in \APerf(\cZ)^{\geq w}$ \hfill (respectively $\sigma^\ast i^\ast(\bD_{\cX}(F)) \in \APerf(\cZ)^{\geq  a+1-w})$
\item For every finite type point of $\cZ$, which classifies a map $f : (B\bG_m)_{k'} \to \cX$, the weights of $H_\ast(f_{\rm reg}^\ast(F))$ are $\geq w$ \hfill (respectively $< w+\wt(\det((\tau_{\leq 0}(f^\ast(\bL_\cX))^{>0}))$).
\end{enumerate}
\end{prop}

\begin{proof}
As in the proof of \Cref{T:derived_Kirwan_surjectivity}, we can use \Cref{L:induced_stratum_base_change} and \Cref{L:sod_limits} to reduce to the case where $\cX$ is quasi-compact and $\cB = \Spec(R)$ for some a.f.p. simplicial commutative $k$-algebra $R$. For any $T \to \cB = \Spec(R)$, we can regard a morphism $\Theta_T \to \cX$ over $\cB$ as a morphism over $\Spec(k)$, and this defines a canonical $\Theta_k$-equivariant morphism of derived $k$-stacks
\[
\Filt_\cB(\cX) := \iMap_\cB(\Theta_\cB,\cX) \to \Filt_k(\cX) := \iMap(\Theta_k, \cX).
\]
One can deduce that this morphism is an isomorphism from the fact that the map of underlying classical stacks is an isomorphism, which is \cite{halpern2014structure}*{Cor.~1.3.16}, combined with the computation of the cotangent complex of $\Filt_\cB(\cX)$ and $\Filt_k(\cX)$ in \Cref{L:relative_cotangent_complex}. It follows that it suffices to prove the theorem in the case where $R = k$. So, for the remainder of the proof we will assume:

\noindent
\begin{center}
\parbox[t]{.9\textwidth}{\emph{$\cX$ is an algebraic derived stack a.f.p., quasi-smooth, and with affine diagonal over $\Spec(k)$, $\cS \subset \Filt_k(\cX)$ is a $\Theta$-stratum with center $\cZ$, and the hypothesis \ref{eqn:obstruction_weights} holds.}}
\end{center}

\begin{lem} \label{L:sigma_quasi_smooth}
$\sigma : \cZ \to \cS$ is relatively quasi-smooth and thus has finite Tor-amplitude.
\end{lem}
\begin{proof}
\Cref{L:center_cotangent} and \Cref{L:relative_cotangent_complex} imply that $\bL_{\cZ/\cS} \simeq (\sigma^\ast \bL_\cS)^{<0}[1] \simeq (\sigma^\ast \bL_\cX)^{<0}[1]$, and \ref{eqn:obstruction_weights} implies this is perfect with fiber homology in degrees $1,0,-1$.
\end{proof}

\begin{lem} \label{L:duality_stratum}
$\DCoh(\cX)^{< w} = \bD_\cX(\DCoh(\cX)^{\geq a+1-w})$, where $a$ is the weight of $\omega_\cS|_\cZ$.
\end{lem}

\begin{proof}
For any $F \in \DCoh(\cS)$, \Cref{L:sigma_quasi_smooth} and \Cref{L:pullback_rhom} give a canonical isomorphism \[\sigma^\ast \bD_\cS(F) \simeq \inner{\RHom}_{\cZ}^\otimes(\sigma^\ast F, \sigma^\ast \omega_\cS).\] Because $\sigma^\ast(\omega_\cS)$ is an invertible complex concentrated in weight $a$, it follows from the characterization of \Cref{L:aperf_baric_criterion} that $\bD_\cS$ exchanges $\DCoh(\cS)^{\geq w}$ and $\DCoh(\cS)^{<a+1-w}$.

More generally, for $F \in \DCoh(\cX)$, the defining adjunction and the projection formula imply that for any $E \in \DCoh(\cS)$
\[
\RHom_\cX(E,i^{\QC,!}(\bD_\cX(F))) \cong \RHom_\cS(i^\ast(F),\bD_\cS(E)).
\]
We have already observed that $\bD_\cS$ interchanges $\DCoh(\cS)^{\geq a+1-w}$ and $\DCoh(\cS)^{<w}$, so to complete the proof, it suffices to show that $i^\ast(F) \in \APerf(\cS)^{\geq a+1-w}$ if and only if $\RHom_\cS(i^\ast(F),E) = 0$ for all $E \in \DCoh(\cS)^{<a+1-w}$. This follows from the fact that $\APerf(\cS)^{<a+1-w}$ is closed under homological truncation, which is established in part (6) of \Cref{P:baric_decomp_supports}.

\end{proof}

The fact that $F \in \DCoh(\cX)$ lies in $\DCoh(\cX)^{\geq w}$ if and only if $\sigma^\ast i^\ast(F) \in \APerf(\cZ)$ follows from \Cref{D:subcategories_aperf} and \Cref{L:aperf_baric_criterion}. This is equivalent to (3) by Nakayama's lemma and the first equality in \Cref{L:bounded_weights}.

The characterization of $\DCoh(\cX)^{\geq w}$ and \Cref{L:duality_stratum} imply that $F \in \DCoh(\cX)$ lies in $\APerf(\cX)^{<w}$ if and only if $\bD_\cX(F) \in \APerf(\cX)^{\geq a+1-w}$, where $a$ is the weight of $\omega_\cS|_\cZ$, which has the same weight as $\det(\ladj{1}(\bL_{\cX}|_{\cZ})) [ \rank(\ladj{1}(\bL_{\cX}|_{\cZ}))]$ by \Cref{C:canonical_weight} and \Cref{L:relative_cotangent_complex}. The equivalence of (2) and (3) in this case follows from Nakayama's lemma and the second equality in \Cref{L:bounded_weights}.

\end{proof}


\subsection{Structure theorem for \texorpdfstring{$\DCoh$}{DCoh}}

We now state our main theorem in the quasi-smooth case.

\begin{thm} \label{thm:derived_Kirwan_surjectivity_quasi-smooth}
Let $\cX$ be an algebraic derived stack locally a.f.p. and with affine diagonal over a locally a.f.p. algebraic derived $k$-stack $\cB$. Assume that $\cX$ is quasi-smooth over $k$. Let $i : \cS \subset \Filt(\cX) \to \cX$ be a $\Theta$-stratum with center $\sigma : \cZ \to \cS$ that satisfies condition \ref{eqn:obstruction_weights} of \Cref{P:alternate_characterization}.

Then the intersection of the semiorthogonal factors in \eqref{E:main_SOD} and \eqref{E:window_SOD} of \Cref{T:derived_Kirwan_surjectivity} with $\DCoh(\cX) \subset \APerf(\cX)$ induce semiorthogonal decompositions of $\DCoh(\cX)$ and $\cG_{\cS,{\rm coh}}^{[u,w)} := \cG_{\cS}^{[u,w)} \cap \DCoh(\cX)$ respectively. Furthermore:

\begin{enumerate}
\item The restriction functor $\cG_{\cS,{\rm coh}}^w \to \DCoh(\cX \setminus \cS)$ is an equivalence;
\item If $\cX$ is quasi-compact, then for any $w \in \bZ$ we have an infinite semiorthogonal decomposition\footnote{This is an infinite semiorthogonal decomposition in the usual sense, so that every $F \in \DCoh(\cX)$ lies in a subcategory generated by finitely many semiorthogonal factors.}
\begin{equation} \label{eqn:main_SOD_quasi_smooth}
\DCoh(\cX) = \langle \lefteqn{\overbrace{\phantom{\ldots, \DCoh_\cS(\cX)^{w-1}, \cG_{\cS,{\rm coh}}^w}}^{\DCoh(\cX)^{<w}}} \underbrace{\ldots, \DCoh_\cS(\cX)^{w-1}}_{\DCoh_\cS(\cX)^{<w}},\lefteqn{\underbrace{\phantom{\cG_{\cS,{\rm coh}}^w,\DCoh_\cS(\cX)^{w},\DCoh_\cS(\cX)^{w+1},\ldots}}_{\DCoh(\cX)^{\geq w}}} \cG_{\cS,{\rm coh}}^w, \overbrace{\DCoh_\cS(\cX)^{w},\DCoh_\cS(\cX)^{w+1},\ldots}^{\DCoh_\cS(\cX)^{\geq w}} \rangle,
\end{equation}
\item $i_\ast \pi^\ast : \DCoh(\cZ)^{w} \to \DCoh_\cS (\cX)^{w}$ is an equivalence.
\end{enumerate}
\end{thm}

Before proving this theorem in the next section, we discuss some examples and applications.

\begin{ex}
Consider the derived cotangent stack of a smooth global quotient stack $\cX = T^\ast (X/G)$ over a field of characteristic $0$. For simplicity, we work with commutative differential graded algebras instead of simplicial commutative rings. If we define
\[
Y:=\Spec_X(\Sym(\cO_X \otimes \fg \to T_X)),
\]
where $T_X$ is in degree $0$, and $\cO_X \otimes \fg \to T_X$ is the two-term complex defined by the derivative of the $G$-action on $X$, then $\cX \simeq Y/G$. For any point $y \in Y$, whose image under the projection $Y \to X$ we denote by $x \in X$, we have
\[
\bL_\cX|_y \simeq \left[\fg_{k(y)} \to T_{X,x} \oplus \Omega^1_{X,x} \to \fg_{k(y)}\right].
\]
This complex is self-dual, and in particular $H_1(\bL_\cX|_y)$ is the lie algebra of $\Stab_{G}(y)$. If $\cX$ has a $\Theta$-stratification and $y$ corresponds to an unstable point, then we will automatically have $\Stab_G(y) \subset P_\lambda$, where $\lambda$ is the one-parameter subgroup associated to the stratum containing $y$, and $P_\lambda$ is the corresponding parabolic subgroup (see \Cref{L:filt_quotient_stack}). By construction the lie algebra of $P_\lambda$ has nonnegative weights with respect to $\lambda$, and hence so does the lie algebra of $\Stab_G(y)$, so the hypotheses \ref{eqn:obstruction_weights} is satisfied.
\end{ex}

Extending the example above, we observe:
\begin{lem} \label{L:verify_quasi-smooth_hypotheses}
If $\cX$ is an algebraic derived $k$-stack with affine automorphism groups for which $\bL_\cX \simeq (\bL_{\cX})^\dual$, then $\cX$ is quasi-smooth and any $\Theta$-stratum in $\cX$ satisfies the hypothesis \ref{eqn:obstruction_weights}, and hence the conclusions of \Cref{thm:derived_Kirwan_surjectivity_quasi-smooth}.
\end{lem}
\begin{proof}
The isomorphism $\bL_\cX \simeq (\bL_\cX)^\dual$ implies that at any point $x \in |\cX|$, the fiber $\bL_{\cX,x}$ is $1$-truncated, so $\cX$ is quasi-smooth. Furthemore, the isomorphism implies that
\[
H_1(\bL_{\cX,x}) \simeq H_{-1}(\bL_{\cX,x}) \simeq \Lie(\Aut_\cX(x))^\ast.
\]
Now let $\cS \subset \Filt(\cX)$ be a $\Theta$-stratum, and let $x \in \cX(k')$ and $\lambda : \Gm \to \Aut_\cX(x)$ correspond to a point in the center of $\cS$. The fact that $\cS \to \cX$ is a closed immersion implies that $\Aut_{\cS}(x) \simeq \Aut_\cX(x)$ as group schemes, and hence $H_{-1}(\bL_{\cS,x}) \simeq H_{-1}(\bL_{\cX,x})$. It follows from these observations and \Cref{L:relative_cotangent_complex} that $H_1(\bL_{\cX,x})$ has nonnegative weights with respect to $\lambda$.
\end{proof}

\begin{ex}
\Cref{L:verify_quasi-smooth_hypotheses} implies that \Cref{thm:derived_Kirwan_surjectivity_quasi-smooth} applies to any $\Theta$-stratification of a $0$-shifted derived algebraic symplectic stack in the sense of \cite{pantev2013shifted}. Simple examples of $0$-shifted derived symplectic algebraic stacks arise as the derived Marsden-Weinstein quotient of an algebraic symplectic variety by a hamiltonian action of a reductive group \cite{pecharich2012derived}. Another important class of $0$-shifted derived symplectic stacks are the moduli stacks $\cX = \inner{\Coh}(S)$ of coherent sheaves on a $K3$ surface $S$ \cite{pantev2013shifted}. The equivalence $\bL_\cX = (\bL_\cX)^\dual$ comes from the Serre duality equivalence $\RHom_S(E,E[1]) \simeq \RHom_S(E[1],E[2])^\dual \simeq \RHom_S(E,E[1])^\dual$ for $E \in \Coh(S)$, or more precisely from a version of this equivalence in families.
\end{ex}

\begin{ex}[Hypothesis \ref{eqn:obstruction_weights} is necessary] We observe that the semiorthogonal decomposition of \Cref{T:derived_Kirwan_surjectivity} can fail to preserve $\DCoh(\cX)$ for quasi-smooth $\cX$. Consider a linear $\Gm$-action on $\bA^n$ having positive and negative weights, and choose $\lambda(t) = t$ so that the unstable subspace $\bA^n_+$ is the subspace defined by the vanishing of the coordinate functions of positive weight. Let $X \subset \bA^n$ be the hypersurface defined by a non-zero homogeneous polynomial $f$ of weight $-d<0$, and assume that $f(0)=0$. Then the tautological $\Theta$-stratum in $X/\Gm$ is $S / \Gm$, where $S := X \cap \bA^n_+$.

Assume that we can choose an $F \in \DCoh(X^{\rm ss}/\Gm)$ that fails to be perfect in any neighborhood of $S \subset X$ -- such a complex exists as long as the closure of the singular locus of $X^{\rm ss}$ meets $S$. It follows that for any extension of $F$ to $\DCoh(X/\Gm)$, the derived restriction to $\{0\}$ must be homologically unbounded, because $\{0\}$ lies in the orbit closure of every point of $S$. \Cref{T:derived_Kirwan_surjectivity} implies there is a unique extension of $F$ to a complex $\tilde{F} \in \cG_\cS^w \subset \APerf(X/\Gm)$. If $\tilde{F} \in \DCoh(X/\Gm)$, then it would admit a presentation as a right-bounded complex of free graded $\cO_X$-modules $F_\bullet$ which was eventually $2$-periodic up to a shift by a character of $\Gm$. Specifically, in high homological degree $F_\bullet$ is the restriction to $X$ of a graded matrix factorization \cite{segal} on $\bA^n$, so $F_{n+2} = F_{n}(d)$ for $n \gg 0$. This implies that the weights of $F_n|_{\{0\}}$ tend to $-\infty$ as $n \to \infty$, which contradicts the fact that $\tilde{F}|_{\{0\}}$ has weight $\geq w$ and is homologically unbounded. It follows that $\tilde{F} \notin \DCoh(\cX)$.
\end{ex}

For our first application of \Cref{thm:derived_Kirwan_surjectivity_quasi-smooth}, consider a derived algebraic stack $\cX$ containing two $\Theta$-strata $\cS_\pm$ such that $\cZ_+$ and $\cZ_-$ are identified via the involution of $\Grad(\cX)$ induced by the inversion homomorphism $\Gm \to \Gm$, i.e., $\cZ_+$ and $\cZ_-$ classify the same points $x \in \cX$, but with the opposite cocharacters $\lambda_\pm : \Gm \to \Aut_\cX(x)$. We say that the two stacks $\cX^{\rm ss}_\pm := \cX \setminus \cS_{\pm}$ differ by an \emph{elementary wall crossing}, generalizing \cite{BFK} and \cite{halpern2015derived}.

\begin{cor} \label{cor:wall_crossing}
Let $\cX$ be a derived algebraic stack which has affine diagonal and is a.f.p. over a locally a.f.p. algebraic $k$-stack $\cB$, and let $\cS_\pm \subset \Filt(\cX)$ be two $\Theta$-strata definining an elementary wall crossing. Assume that $\cX$ is quasi-smooth and for every finite type point in $\cZ_+$ with underlying point $x \in \cX$, $H_1(\bL_{\cX,x})$ has $\lambda_+$-weight $0$. Let $c$ be the $\lambda_+$-weight of $\det (\bL_\cX|_\cZ)$, and let $\cX^{\rm ss}_{\pm} := \cX \setminus \cS_{\pm}$.
\begin{enumerate}
\item If $c = 0$, then $\DCoh(\cX^{ss}_+) \simeq \DCoh(\cX^{ss}_-)$.
\item If $c>0$, then there is a fully faithful embedding $\DCoh(\cX^{ss}_+) \subset \DCoh(\cX^{ss}_-)$.
\item If $c<0$, then there is a fully faithful embedding $\DCoh(\cX^{ss}_-) \subset \DCoh(\cX^{ss}_+)$.
\end{enumerate}
The same holds with $\DCoh(-)$ replaced by $\Perf(-)$.
\end{cor}

\begin{rem}
As in the classical version of this statement, the semiorthogonal complement to the embeddings in (2) and (3) has a further semiorthogonal decomposition into a number of categories of the form $\DCoh(\cZ)^{w}$.
\end{rem}

\begin{proof}
The key observation is that because $H_1(\bL_{\cX,x})$ has $\lambda_+$ weight $0$ at every finite type point of $\cZ_+$ (and hence likewise for $\lambda_-$ and $\cZ_-$), then not only do $\cS_+$ and $\cS_-$ both satisfy the hypotheses of \Cref{thm:derived_Kirwan_surjectivity_quasi-smooth}, but they also satisfy the hypotheses of \Cref{prop:DKS_perfect}. So $i_\pm : \cS_{\pm} \hookrightarrow \cX$ are both regular embeddings. It follows from \Cref{L:regular_closed_immersion_general} that we have
\[
\cG_{\cS_\pm,\rm{coh}}^{w} = \left\{ F \in \DCoh(\cX) | \sigma_\pm^\ast i_\pm^\ast (F) \text{ has }\lambda_\pm \text{-weights in } [w,w+\eta_\pm) \right\},
\]
where $\eta_\pm$ is the weight of $\det(\radj{0}(\bL_{\cX}|_{\cZ_\pm}))$. It follows from the relationship of $\cZ_+$ and $\cZ_-$ that one can always choose shift parameters $w_\pm$ such that either $\cG_{\cS_-,{\rm coh}}^{w_-} \subset \cG_{\cS_+,{\rm coh}}^{w_+}$, $\cG_{\cS_+,{\rm coh}}^{w_+} = \cG_{\cS_-,{\rm coh}}^{w_-}$, or $\cG_{\cS_+,{\rm coh}}^{w_+} \subset \cG_{\cS_-,{\rm coh}}^{w_-}$ depending on whether $\eta_+ - \eta_-$ is positive, zero, or negative respectively. Finally, $\eta_+ - \eta_-$ is the  $\lambda_+$-weight of $\det(\bL_\cX|_{\cZ_+})$. The same argument applies verbatim with $\Perf(-)$ instead of $\DCoh(-)$, using \Cref{prop:DKS_perfect}.
\end{proof}

\subsubsection{Proof of \Cref{thm:derived_Kirwan_surjectivity_quasi-smooth}} 

\begin{lem} \label{L:quasi_smooth_baric_decomp}
The baric truncation functors of \Cref{P:baric_stratum} and \Cref{P:baric_decomp_supports} part (4) induce bounded baric decompositions of $\DCoh(\cS)$ and $\DCoh_\cS(\cX)$, respectively.
\end{lem}

\begin{proof}
Because $\sigma$ has finite Tor-amplitude, $\sigma^\ast F \in \DCoh(\cZ)$ for any $F \in \DCoh(\cS)$. Thus $\sigma^\ast F$ decomposes as a direct sum of objects in $\DCoh(\cZ)^w$ for finitely many $w$. By Nakayama's lemma and the fact that $\sigma^\ast$ is compatible with the baric truncation functors on $\APerf(\cS)$ and $\APerf(\cZ)$, it follows that $F$ decomposes under the baric decomposition of $\APerf(\cS)$ as an iterated extension of objects in $\APerf(\cS)^w$ for finitely many $w$. Furthermore these objects lie in $\DCoh(\cS)^w$, because $\sigma^\ast : \APerf(\cS)^w \to \APerf(\cZ)^w$ is an equivalence with inverse given by $\pi^\ast$. It follows that the baric truncation functors preserve $\DCoh(\cS)$. The boundedness of the baric decomposition of $\DCoh(\cS)$ follows again from Nakayama's lemma, the compatibility of $\sigma^\ast$ with the baric decomposition on $\DCoh(\cZ)$, and the boundedness of the latter.

Once one has the baric decomposition of $\DCoh(\cS)$, the fact that $\radj{w}$ and $\ladj{w}$ preserve $\DCoh_\cS(\cX)$ follows from the fact that $\DCoh_\cS(\cX)$ is generated by $i_\ast \DCoh(\cS)$ under shifts and cones, and the boundedness of the baric decomposition of $\DCoh_\cS(\cX)$ follows likewise.
\end{proof}

\begin{lem} \label{lem:bounded_SOD}
$\DCoh (\cX) = \bigcup_{v,w} \left( \DCoh(\cX)^{<w} \cap \DCoh(\cX)^{\geq v} \right)$.
\end{lem}

\begin{proof}
In light of \Cref{L:duality_stratum}, it suffices to show that for any $F \in \DCoh(\cX)$, $\sigma^\ast i^\ast(F) \in \APerf(\cZ)^{\geq w}$ for some $w \in \bZ$. By \Cref{T:local_structure_stratum} we can find a smooth morphism $p : X/\Gm \to \cX$, where $X$ is a quasi-smooth affine derived $k$-scheme with a $\Gm$-action, such that $p^{-1}(\cS)$ is the tautological stratum $S_+/\Gm \subset X/\Gm$ and $S_+/\Gm \to \cS$ is surjective. Because $p$ is smooth, the relative cotangent complex is locally free in homological degree $0$, so the cofiber sequence for the relative cotangent complex implies that for any $x \in X/\Gm$, the canonical map is an isomorphism
\[
H_1(\bL_{\cX,p(x)}) \xrightarrow{\cong} H_1(\bL_{X/\Gm,x}).
\]
Hence $X/\Gm$ satisfies the hypothesis \ref{eqn:obstruction_weights} as well, and it suffices to prove the claim for the stack $X/\Gm$.

We choose $x_1,\ldots, x_k \in H_0(\cO_X)$ which are homogeneous of positive weight with respect to the $\Gm$-action and which cut out $S_+$ set-theoretically, and we let $K := K_X(x_1,\ldots,x_k) \in \Perf(X/\Gm)$ be a Koszul complex as constructed in \eqref{E:koszul_complex} (see \Cref{sect:koszul}). \Cref{L:quasi_smooth_baric_decomp} implies that $K \otimes F \in \DCoh_\cS (\cX)^{\geq w}$ for some $w$, and so the weights of $\sigma^\ast i^\ast (K \otimes F) \simeq \sigma^\ast i^\ast(K) \otimes \sigma^\ast i^\ast F$ are bounded below. At the same time, $K$ admits a map $K \to \cO_X$ such that $\cofib(K \to \cO_X)|_Z$ lies in $\QC(Z/\Gm)^{<0}$, and it follows that $\cO_Z$ is the weight $0$ direct summand of $K|_{Z}$. In particular $\sigma^\ast i^\ast(F)$ is a direct summand of $\sigma^\ast i^\ast(K \otimes F)$, whose weights are bounded below.
\end{proof}

\begin{proof} [Completing the proof of \Cref{thm:derived_Kirwan_surjectivity_quasi-smooth}]
As in the proof of \Cref{P:alternate_characterization}, we may reduce to the situation where $\cB = \Spec(k)$, so we will assume this throughout.

As in the proof of \Cref{T:derived_Kirwan_surjectivity}, we may use \Cref{T:local_structure_stratum} to construct a smooth morphism $p : X/\Gm \to \cX$, where $X$ is an affine derived $k$-scheme with a $\Gm$-action, and $\cS \hookrightarrow \cX$ induces the tautological $\Theta$-stratum $S/\Gm \hookrightarrow X/\Gm$, i.e., $S \subset X$ is the derived subscheme contracted by the tautological one-parameter subgroup. Note that because $p$ is smooth, $X$ is quasi-smooth, and $X/\Gm$ satisfies \ref{eqn:obstruction_weights} (See the proof of \Cref{lem:bounded_SOD}).

To show that the semiorthogonal decompositions of $\APerf(\cX)$ in \eqref{E:main_SOD} and \eqref{E:window_SOD} induce semiorthogonal decompositions of $\DCoh(\cX)$ and $\cG_{\cS,{\rm coh}}^{[u,w)}$ respectively, it suffices to show that $\radj{w}_\cS$ and $\ladj{w}_\cS$ preserve $\DCoh(\cX)$. By \Cref{T:derived_Kirwan_surjectivity} part (6) it suffices to show this for the stack $X/\Gm$.

Let $x_1,\ldots,x_k \in H_0(\cO_X)$ be elements which cut out $S$, where $x_i$ is homogeneous of degree $d_i>0$, and let $K_n = K_X(x_1^n,\ldots,x_k^n)$ be the Koszul complex defined in \eqref{E:koszul_complex}. If we let $C_n = \cofib(K_n \to \cO_X)$ and define $D := \min(d_1,\ldots,d_k)$, then $i^\ast(C_n^\dual) \in \Perf(S/\Gm)^{\geq n D}$. For any $F \in \DCoh(X/\Gm)$, \Cref{lem:bounded_SOD} implies $F \in \DCoh(X/\Gm)^{\geq v}$ for some $v$, so \Cref{P:baric_stratum} part (6) implies that $C_n^\dual \otimes F \in \DCoh(X/\Gm)^{\geq v+nD}$. It follows that for $n \gg 0$, $\ladj{w}_\cS(C^\dual_n \otimes F) \cong 0$, and hence
\[
\ladj{w}_\cS(F) \xrightarrow{\cong} \ladj{w}_\cS(K_n^\dual \otimes F)
\]
is an isomorphism. $K_n^\dual \otimes F \in \DCoh_{S}(X/\Gm)$, and hence $\ladj{w}_\cS(F) \in \DCoh(X/\Gm)$ by \Cref{L:quasi_smooth_baric_decomp}.

We show that $\radj{w}_\cS(F) \in \DCoh(X/\Gm)$ similarly: Because $C_n$ is dualizable we have $i^{\QC,!}(C_n \otimes F) \cong i^\ast(C_n) \otimes i^{\QC,!}(F)$. If $i^{\QC,!}(F) \in \QC(S/\Gm)^{<v}$, then \Cref{P:baric_stratum} part (6) implies that  $i^{\QC,!}(C_n \otimes F) \in \QC(S/\Gm)^{<v-nD}$, because for any $G \in \QC(S/\Gm)^{\geq v-nD}$, $G \otimes i^\ast(C_n^\dual) \in \QC(S/\Gm)^{\geq v}$, and thus
\[
\RHom_{S/\Gm}(G,i^\ast(C_n) \otimes i^{\QC,!}(F)) \cong \RHom_{S/\Gm}(G \otimes i^\ast(C_n^\dual), i^{\QC,!}(F)) = 0.
\]
It follows from \Cref{lem:bounded_SOD} that for any $F \in \DCoh(X/\Gm)$ and any $w$, $\radj{w}_{S/\Gm}(C_n \otimes F) =0$ for $n\gg 0$ and hence
\[
\radj{w}_{S/\Gm}(K_n \otimes F) \to \radj{w}_{S/\Gm}(F)
\]
is an isomorphism for $n \gg 0$. \Cref{L:quasi_smooth_baric_decomp} then implies that $\radj{w}_{S/\Gm}(F) \in \DCoh(X/\Gm)$.

\medskip
\noindent \textit{Proof of (1):}
\medskip

Fully faithfulness follows from part (1) of \Cref{T:derived_Kirwan_surjectivity}, so it suffices to show that $\cG_{\cS,{\rm coh}}^w \to \DCoh(\cX \setminus \cS)$ is essentially surjective. $\DCoh(\cX) \to \DCoh(\cX \setminus \cS)$ is essentially surjective, and for any $F \in \DCoh(\cX)$, $\radj{w}(\ladj{w}(F)) \in \cG^w_{\cS,{\rm coh}}$ has the same restriction to $\cX \setminus \cS$ as $F$.

\medskip
\noindent \textit{Proof of (2):}
\medskip

The infinite semiorthogonal decomposition \eqref{eqn:main_SOD_quasi_smooth} follows from the semiorthogonal decomposition of $\cG_{\cS,{\rm coh}}^{[u,w)} := \DCoh(\cX)^{<w} \cap \DCoh(\cX)^{\geq u}$ induced by \eqref{E:window_SOD}, along with \Cref{lem:bounded_SOD}.

\medskip
\noindent \textit{Proof of (3):}
\medskip

By \Cref{T:derived_Kirwan_surjectivity} part (3) we have equivalences $i_\ast \pi^\ast : \APerf(\cZ)^{w} \to\APerf_\cS(\cX)^w$. $F \in \APerf(\cZ)^w$ has bounded homology if and only if $\pi^\ast(F) \in \APerf(\cS)^w$ does, because $\pi^\ast$ has finite Tor-amplitude by \Cref{lem:finite_tor_amplitude} and its inverse $\sigma^\ast$ has finite Tor-amplitude by \Cref{L:sigma_quasi_smooth}. Also $G := \pi^\ast(F) \in \APerf(\cS)$ has bounded homology if and only if $i_\ast(G)$ does.

\end{proof}

\section{Derived equivalences from variation of good moduli space}
\label{S:D_equivalence}

In this section, $k$ will denote a fixed ground field of characteristic $0$.

\subsection{Variation of good moduli space} \label{S:variation}

The key to our applications \Cref{T:derived_Kirwan_surjectivity} will be an intrinsic version of the theory of variation of geometric invariant theory (GIT) quotient \cite{dolgachev_hu} for stacks that admit a good moduli space in the sense of \cite{alper2013good}. We will recall the main result here. It is described in detail in \cite{halpern2014structure}*{Sect.~5.3.1}.

First we recall the notion of $\Theta$-stability from \cite{halpern2014structure} on a stack $\cX$ -- it is essentially an intrinsic version of the Hilbert-Mumford criterion. Let $L$ be an invertible sheaf on $\cX$. For any filtration $f : \Theta_{k'} \to \cX$, we can make use of the canonical identification $H^\ast(\Theta;\bQ) \cong \bQ[q]$, with $q \in H^2(\Theta;\bQ)$ the first Chern class of the unique invertible sheaf whose fiber weight at $0$ is $1$, to define a number $\frac{1}{q} f^\ast(c_1(L)) \in \bQ$.\footnote{Here our stack is locally finite type over $k$, and we are using operational Chow cohomology. However, you can avoid using cohomology by declaring $f^\ast(c_1(L))$ to be the weight of the fiber of the invertible sheaf $f^\ast(L)$ at $0$.}

\begin{defn} \label{D:theta_semistable}
A point $p \in \cX(k')$ over an algebraically closed field is \emph{unstable} if there is a filtration $f : \Theta_{k'} \to \cX$ with $f(1) \cong p$ and $q^{-1} f^\ast(c_1(L))<0$, and $p$ is \emph{semistable} otherwise.\footnote{We are using sheaves, not bundles. If this leaves any ambiguity as to sign conventions, $f$ is destabilizing if the line bundle $f^\ast(L)$ has no non-trivial sections.}
\end{defn}

In order to define generalized Harder-Narasimhan filtrations, we will need additional information. Let $b \in H^4(\cX;\bQ)$ be a cohomology class. We assume $b$ is \emph{positive definite} in the sense that $q^{-2} f^\ast(b) > 0$ for any non-trivial filtration $f : \Theta_{k'} \to \cX$, and we define
\[
|\!|f|\!|_b := \sqrt{q^{-2} f^\ast(b)} \in \bR
\]
for any filtration. In the case of a global quotient $X/G$, one can construct a positive definite class $b$ from a Weyl-group-invariant rational positive definite inner product on the space of cocharacters of the maximal torus on $T$. 

We now define the numerical invariant
\[
\mu(f) = -\frac{f^\ast(c_1(L))}{|\!|f|\!|_b} \in \bR.
\]
Note that $\mu(f)$ is locally constant in algebraic families of filtrations. When $\cX$ is the quotient of a quasi-projective scheme by the linearized action of a reductive group, this formula for $\mu$ recovers the normalized Hilbert-Mumford numerical invariant in GIT \cite{dolgachev_hu}.

More generally, one can replace the positive definite class $b \in  H^4(\cX;\bQ)$ with a norm on graded points, as defined in \cite{halpern2014structure}. This consists of an assignment to any non-degenerate morphism $\gamma : B(\Gm^n)_{k'} \to \cX$, meaning the map on automorphism groups has finite kernel, a norm $|\!|-|\!|_\gamma$ on $\bR^{n}$ that is locally constant in algebraic families and compatible with extension of field and restriction to sub-tori. If every norm $|\!|-|\!|_\gamma$ is the square root of a positive definite rational quadratic form, we say that this is a \emph{rational quadratic norm on graded points of $\cX$}.

Given a positive definite $b \in H^4(\cX;\bQ)$, we define the norm associated to any non-degenerate $\gamma : B(\Gm^n)_{k'} \to \cX$ to be $|\!|-|\!|_{\gamma} = \sqrt{\gamma^\ast(b)}$, where we have canonically identified $H^4(B(\Gm^n)_{k'};\bQ)$ with the space of rational quadratic forms on $\bQ^n$, the space of rational cocharacters of $\Gm^n$. We can then write the general form of our numerical invariant as
\[
\mu(f) = -\frac{q^{-1} f^\ast(c_1(L))}{|\!|1|\!|_{\ev_0(f)}},
\]
where the denominator refers to the norm on $\bR$ assigned to the associated graded point $\ev_0(f) : B(\Gm)_{k'} \to \cX$.

\begin{defn}
We say that $\mu$ defines a $\Theta$-stratification on $\cX$ if there is a $\Theta$-stratification whose strata $\cS \subset \Filt(\cX)$ are ordered by the values of $\mu$, and such that each $f\in \cS$ with $f(1) \cong p$ is the unique maximizer of $\mu$ (up to composition with a ramified covering $(-)^n : \Theta \to \Theta$) subject to the constraint $f(1) \cong p$. We call $f \in \cS$ the \emph{Harder-Narasimhan}(HN) filtration of $f(1)$.
\end{defn}

\begin{thm} \label{T:intrinsic_GIT}
Let $\cX$ be a classical algebraic stack locally of finite type and with affine diagonal over $k$, and let $\cX \to Y$ be a good moduli space. Then any class $\ell \in \NS(\cX)_\bQ$ and rational quadratic norm on graded points of $\cX$ (e.g., the norm coming from a positive definite $b \in H^4(\cX;\bQ)$) induces a $\Theta$-stratification
\[
\cX = \cX^{\rm{ss}}(\ell) \cup \cS_0 \cup \cdots \cup \cS_N.
\]
$\cX^{\rm{ss}}(\ell)$ admits a good moduli space $\cX^{\rm ss}(\ell) \to Y'$, and $Y'$ is projective over $Y$. The formation of this stratification commutes with base change along a map of algebraic spaces $Z \to Y$.
\end{thm}

The key to the proof of \Cref{T:intrinsic_GIT} is the final claim, that the construction is local on $Y$. The proof is an elementary consequence of the \'etale slice theorem for stacks established in \cites{alper2015luna, ahr2}, one consequence of which is the following: 
\begin{thm} \label{T:GMS}
Let $\cX \to Y$ be a good moduli space, where $\cX$ is a finite type $k$ stack with affine diagonal. Then there is an \'{e}tale surjection $\Spec(R) \to Y$ such that the base change of $\cX$ to $\Spec(R)$ is a quotient of an affine scheme by a linearly reductive $k'$-group for some \'etale extension $k'/k$.
\end{thm}
\begin{rem} \label{R:GMS_stronger}
In fact, the theorem is more constructive. If $x \in \cX(k')$ represents a closed point and $G = \Aut_\cX(x)$ is its linearly reductive stabilizer, then there is an \'etale map $\Spec(A) / G \to \cX$ taking $w \in \Spec(A)^G$ to $x$ such that $\Spec(A^G) \to Y$ is also \'etale, and $\Spec(A)/G \cong \Spec(A^G) \times_Y \cX$. (The last two conditions taken together are sometimes called ``strongly \'etale.'')
\end{rem}

To see how \Cref{T:GMS} implies \Cref{T:intrinsic_GIT}, we observe that because $Y$ is an algebraic space, any map $\Theta_{k'} \to \cX \to Y$ factors uniquely through a point $y : \Spec(k') \to Y$. This implies the value of the numerical invariant, and hence the notion of semistability and HN filtration of a point $x \in \cX(k')$ lying over $y$, is completely determined by the fiber $\cX_y$, which is a quotient of an affine scheme by a linearly reductive $k$ group. This allows one to reduce \Cref{T:intrinsic_GIT} to the case where $\cX = \Spec(A)/G$ for a linearly reductive $k$-group $G$, which is well-known.



\subsection{Local structure of stacks with self-dual cotangent complex}

Recall from \Cref{T:GMS} that any stack that admits a good moduli space is \'etale locally a quotient of an affine scheme by a reductive group. We will develop a more refined local model for stacks of Bridgeland semistable complexes on a $K3$-surface.

We consider an action of an algebraic group $G$ on a smooth affine $k$-scheme $\Spec(R)$. Let $\fg$ be the Lie algebra of $G$, and let $a : \fg \to T_R$ be the infinitesimal derivative of the $G$-action on $R$.
\begin{defn} \label{D:comoment_map}
A \emph{weak co-moment map} is a $G$-equivariant $k$-linear map $\mu : \fg \to R$ along with $G$-equivariant isomorphisms $\phi_0 : \Omega_R^1 \simeq T_R$ and $\phi_1 : R \otimes \fg \simeq R\otimes \fg$ such that the diagram
\[ \xymatrix{R \otimes \fg \ar[r]^{d\mu} \ar[d]^{\phi_1} & \Omega_R^1 \ar[d]^{\phi_0} \\ R \otimes \fg \ar[r]^{a} & T_R } \]
commutes after restricting to $(R / R \cdot \mu(\fg))^{\rm red}$.
\end{defn}

\begin{rem}
When $\phi_0$ is induced by an algebraic symplectic form, and $\phi_1$ is the identity, the map $\mu$ is called the co-moment map and is uniquely defined on each connected component of $\Spec(R)$ up to an action of the additive group $(\fg^\ast)^G$. Thus our notion is a weaker version of this more common concept.
\end{rem}

Given a weak co-moment map, one can form a commutative DGA $R[\fg[1];d \xi = \mu(\xi)]$, by which we mean the free graded-commutative $R$-algebra generated by the vector space $\fg$, where the differential on $\fg$ is given by $\mu$. This defines an affine derived scheme which also has a $G$-action. Note that $H_0(R[\fg[1];d \xi = \mu(\xi)]) \cong R/ R \cdot \mu(\fg)$.

\begin{thm} \label{thm:local_model}
Let $\cX$ be a derived algebraic stack such that $\bL_\cX \simeq \bL_\cX^\dual$, and let $\cX^{\cl} \to Y$ be a good moduli space. Let $x \in \cX(k)$ be a closed point, and let $G = \Aut_\cX(x)$ be its stabilizer, which is automatically linearly reductive.

Then there is a smooth affine $G$-scheme $\Spec(R)$, a $G$-fixed $k$-point $\tilde{x} \in \Spec(R)$, a weak co-moment map $\mu : \fg \to R$, and a map taking $\tilde{x}$ to $x$,
\[
\cX' := \Spec(R[\fg[1]; d \xi = \mu(\xi)]) / G \to \cX,
\]
that is \emph{strongly \'etale} in the sense that is \'etale and if $U = \Spec((R / R \cdot \mu(\fg))^G)$, then $(\cX')^{cl} \simeq \cX^{cl} \times_Y U$.\footnote{This is a derived interpretation of the notion of a strongly \'etale morphism of classical stacks.}
\end{thm}
\begin{rem}
Every closed point in $\cX$ is representable by some $k'$-point for a finite \'etale extension $k'/k$. Applying the theorem to $\cX_{k'}$ gives a local \'etale chart $\Spec(R[\fg[1];d\xi=\mu(\xi)])/G \to \cX$ for a linearly reductive $k'$-group $G$ and a smooth $G$-equivariant $k'$-algebra $R$, and one can use this to constuct an \'etale cover of $\cX$ by stacks of this form.
\end{rem}

As a basic input to the proof of this theorem, we observe:

\begin{lem} \label{lem:local_model}
Let $\cX$ be a derived stack such that $\cX^{\cl} \simeq \Spec(R) / G$ for some ring $R$ and reductive group $G$. Then $\cX \simeq \Spec(A) / G$ for some connective \CDGA $A$ with a $G$-equivariant isomorphism $\pi_0(A) \simeq R$.
\end{lem}
\begin{proof}
By realizing $\cX$ as the colimit of its truncations, it suffices to show that if $\cX \to \cX'$ is a square-zero extension and $\cX \simeq \Spec(A) / G$ for some $G$-equivariant \CDGA $A$, then $\cX' \simeq \Spec(A')/G$ for some square-zero extension of $G$-equivariant \CDGA's $A' \to A$.

This amounts to a deformation theory problem: showing that square-zero extensions of the stack $\cX = \Spec(A)/G$ by a coherent sheaf $M$ over $\cX^{\rm cl} = \Spec(\pi_0(A))/G$ correspond bijectively to $G$-equivariant square zero extensions of $\Spec(A)$ by this same coherent sheaf. The first are classified by the group
\[
\Hom_{\cX^{\rm cl}}(\bL_{\cX}|_{\cX^{\rm cl}},M[1]),
\]
while the second are classified by the group
\[
\Hom_{\cX^{\rm cl}}(\bL_{\cX/BG}|_{\cX^{\rm cl}},M[1]) \cong \Hom_{\pi_0(A)}(\bL_{A},M[1])^{G}.
\]
They are related by taking a square-zero extension over $BG$ and forgetting the map to $BG$, which corresponds to the homomorphism on Ext groups induced by the canonical map $\bL_{\cX} \to \bL_{\cX/BG}$, whose fiber is $\cO_{\cX} \otimes \fg[-1]$. The fact that this homomorphism is an isomorphism follows from the long exact cohomology sequence and the fact that coherent sheaves on $\Spec(\pi_0(A))/G$ have vanishing higher cohomology.
\end{proof}

The next useful observation is that Zariski locally over $\Spec(\pi_0(A)^G)$, one can describe certain derived stacks as a \emph{derived complete intersection} \cite{arinkin2015singular}. By this we mean Spec of a $G$-equivariant CDGA of the form $k[U_0,U_1;d]$, which denotes the free CDGA generated by a finite dimensional representation $U_0$ of $G$ in homological degree $0$ and a finite dimensional representation $U_1$ in homological degree $1$ with a $G$-equivariant differential defined by the $G$-equivariant linear map $d : U_1 \to k[U_0]$.

\begin{lem} \label{lem:quasi_smooth}
Let $G$ be a linearly reductive group, and let $A$ be a connective $G$-equivariant CDGA over $k$ that is quasi-smooth. Let $x \in \Spec(\pi_0(A))$ be a closed $k$-point with $\Aut(x) = G$. Then there is an $f \in \pi_0(A)^G$ which does not vanish at $x$ and a $G$-equivariant quasi-isomorphism with a derived complete intersection
\[
A_f \simeq k[U_0,U_1; d],
\]
where $A_f$ is the localization of $A$, and the maximal ideal $\mathfrak{m}_x \subset \pi_0(A)$ is generated by $U_0$.
\end{lem}

\begin{proof}
First choose some $G$-equivariant semi-free presentation $A \sim k[U_\bullet;d_A]$, where $U_\bullet = \bigoplus_{i\geq 0} U_i[i]$ -- in other words $U_i$ is placed in homological degree $i$. The construction of such a resolution begins by choosing a $G$-equivariant surjection $k[U_0] \to \pi_0(A)$, and one can arrange that $U_0$ maps to $\mathfrak{m}_x$. The fiber of the cotangent complex at $x$ has the form
\[
\bL_A|_{x} \simeq (\cdots \to U_2 \xrightarrow{d_2} U_1 \xrightarrow{d_1} U_0)
\]
Choosing a subspace $W \subset U_1$ that maps isomorphically onto the quotient $U_1 / \image(d_2)$, the inclusion of the subcomplex $(W \to U_0) \subset (\cdots \to U_2 \to U_1 \to U_0)$ is a quasi-isomorphism. It follows from the fact that $\Spec(A)$ is quasi-smooth that the inclusion of the semi-free subalgebra $A' := k[U_0,W[1];d_A] \subset k[U_\bullet;d_A]$ induces an isomorphism of cotangent complexes at $x$. Thus the closed immersion
\[
\Spec(A) / G \hookrightarrow \Spec(A') / G
\]
is \'etale at $x \in \Spec(A)$ and hence a Zariski open immersion in a neighborhood of $x$. It follows from the orbit structure of reductive group actions on affine schemes that any $G$-equivariant open neighborhood of a point with a closed orbit is contained in an open affine of the form $A_f$ for some $f \in \pi_0(A)^G$.

\end{proof}

\begin{proof}[Proof of \Cref{thm:local_model}]

Using \Cref{T:GMS}, \Cref{R:GMS_stronger}, and \Cref{lem:local_model}, we may reduce to the case where $\cX = \Spec(A) / G$ for a linearly reductive group $G$ and $G$-equivariant connective CDGA $A$. The quasi-isomorphism $\bL_\cX \simeq (\bL_\cX)^\vee$ implies that $\cX$ is quasi-smooth, so \Cref{lem:quasi_smooth} implies that we can further reduce to the case where $A = k[U_0,U_1; d]$ is a $G$-equivariant complete intersection, and $x \in \Spec(A)$ is defined by the ideal generated by $U_0$.

The cotangent complex admits an explicit presentation of the form
\[
\bL_{\Spec(A)/G} \simeq A \otimes \left(\delta U_1[1] \oplus \delta U_0 \oplus  \fg^\ast[-1] \right), 
\]
where $\delta U_i$ is the vector space of formal differentials of elements of $U_i$, and is isomorphic to $U_i$ as a $G$-representation. The differential on $\bL_{\Spec(A)/G}$ is a deformation of the differential on the free dg-$A$-module generated by $\delta U_1,\delta U_0,$ and $\fg^\ast$ in homological degree $1,0,$ and $-1$ respectively. Because $\Aut(x) = G$, the quasi-isomorphism $\psi : \bL_\cX \simeq (\bL_{\cX})^\dual$ provides an isomorphism
\[
\fg \simeq H_1(\bL_{\cX}|_{x}) \subset \delta U_1.
\]
We shall fix a $G$-equivariant splitting of this inclusion, resulting in a decomposition $\delta U_1 \simeq U_1 \simeq \fg \oplus W$ as a $G$-representation.

The $G$-equivariant complete intersection CDGA $A' := k[U_0,W; d_A] \subset A$ is smooth at $x \in \Spec(\pi_0(A))$, and hence $A'$ will be smooth and classical after inverting an element $f \in \pi_0(A)^G$. We regard $A$ as obtained from $A'$ by adjoining relations corresponding to a $k$-linear map $\mu : \fg \to k[U_0]$, and hence after inverting $f$ we will have a quasi-isomorphism
\[
A_f \simeq R[\fg[1];d\xi = \mu(\xi)],
\]
where $R := \pi_0(A'_f)$ is a smooth ring and $\mu : \fg \to k[U_0] \to R$ the induced $G$-equivariant map. This map must be a weak co-moment map by \Cref{lem:comoment} below.

\end{proof}

\begin{lem}\label{lem:comoment}
Let $G$ be a reductive group, let $R$ be a smooth $k$-algebra with a $G$-equivariant map $\mu : \fg \to R$, and let $\cX := \Spec(R[\fg[1];d\xi = \mu(\xi)]) / G$. Fix a closed $k$-point $x \in \Spec(R/R \cdot \mu(\fg))^G$ at which $\Aut(x) = G$. Then any quasi-isomorphism $\psi : \bL_\cX \simeq (\bL_\cX)^\dual$ induces, after inverting an element $f \in R^G$ with $f(x) \neq 0$, isomorphisms $\phi_1 : R\otimes \fg \to R\otimes \fg$ and $\phi_0 : \Omega_R^1 \to T_R$ giving $\mu$ the structure of a weak co-moment map.
\end{lem}

\begin{proof}
The cotangent complex, which is a dg-module over $A:= R[\fg[1];d]$, has the form
\[
\bL_{\Spec(A)/G} = A \otimes \fg[1] \oplus A \otimes_R \Omega_R^1 \oplus A \otimes \fg^\ast[-1]
\]
where the differential is determined uniquely by the Leibniz rule, i.e., the fact that it is a dg-$A$-module, and the fact that: 1) it acts on the subspace $1 \otimes \fg$ by the $k$-linear map $d\mu : \fg \to \Omega_R^1$, and 2) it acts on $1 \otimes_R \Omega_R^1$ by the base change to $A$ of the infinitesimal coaction map $a : \Omega_R^1 \to R \otimes \fg^\ast$. Concretely, if $A = \bigoplus_{n \geq 0} A_n$ is the homological grading, then because we are using a semi-free presentation of $\bL_{\Spec(A)/G}$, the quasi-isomorphism $\psi : \bL_\cX \simeq (\bL_\cX)^\dual$ corresponds to an actual homomorphism dg-$A$-modules

\[
\xymatrix{\bL_{\Spec(A)/G} \simeq \ar[d]^\psi & \cdots \ar[r] & \Omega_R^1 \oplus A_1 \otimes \fg^\ast \ar[rr]^-{a \oplus (d_A \otimes \fg^\ast)} \ar[d]^{\psi_0} & & A_0 \otimes \fg^\ast \ar[d]^{\psi_{-1}} \\
(\bL_{\Spec(A)/G})^\dual\simeq  & \cdots \ar[r] & T_R \oplus A_1 \otimes \fg^\ast \ar[rr]^-{(d\mu)^\dual \oplus (d_A \otimes \fg^\ast)} & & A_0 \otimes \fg^\ast}
\]
that induces an isomorphism on homology. Bear in mind that $A_0 = R$, $A_1 = R\otimes \fg$, $A_2 = R\otimes \bigwedge^2 \fg$, etc.

Now at the point $x \in \Spec(R / \mu(\fg)) \subset \Spec(R)$ the last differential in this complex vanishes, so $\psi_{-1}$ induces an isomorphism $R\otimes \fg^\ast \to R\otimes \fg^\ast$ in a neighborhood of $x$ in $\Spec(R)$. Likewise, as $\psi_\bullet$ is a map of dg-$A$-modules, the map $\psi_0$ maps $A_1 \otimes \fg^\ast$ to $A_1 \otimes \fg^\ast$ and is induced by $\psi_{-1}$. It thus descends to a map $\psi_0 : \Omega_R^1 \to T_R$ of $R$-modules. The resulting diagram of maps of $R$-modules
\[
\xymatrix{ \Omega_R^1 \ar[r]^{a} \ar[d]^{\psi_0} & R \otimes \fg^\ast \ar[d]^{\psi_{-1}} \\ T_R \ar[r]^{(d\mu)^\dual} & R\otimes \fg^\ast }
\]
commutes after restricting to $R/R\cdot \mu(\fg)$, and the vertical arrows are isomorphisms in a $G$-equivariant Zariski-open neighborhood of $x \in \Spec(R)$. After inverting an $f \in R^G$, we can then invert the vertical arrows and dualize this diagram, giving $\mu$ the structure of a weak co-moment map.
\end{proof}

\subsection{The magic windows theorem}
\label{S:magic_windows}

\begin{notn}
Let $\cX$ be a stack for which $\cX^{\rm cl}$ admits a good moduli space. For any representative $x \in \cX(k')$ of a closed point in $|\cX|$, let $G_x$ denote the automorphism group of $x$, which is linearly reductive. Let $M_x$ and $W_x$ denote the weight lattice and Weyl group of $G_x$ respectively. Let $\overline{\Sigma}_x \subset (M_x)_\bR$ denote the convex hull of the weights appearing in the self-dual $G_x$-representation $\bigwedge^\ast (H_0(\bL_{\cX,x}))$. Note that $M_x$, $W_x$, and $\overline{\Sigma}_x$ are independent of the choice of representative of the closed point in $|\cX|$.
\end{notn}

\begin{defn} \label{defn:generic}
We say that $\ell \in \NS(\cX)_\bR$ is \emph{generic} if for any closed finite type point $x \in \cX(k')$, the character $\ell_x$ of $G_x$ is parallel to $\overline{\Sigma}_x$ but not parallel to any face of $\overline{\Sigma}_x$. We say $\delta \in \NS(\cX)_\bR$ is \emph{lattice generic} if for any closed point $x \in \cX(k')$, $\delta_x$ is parallel to $\overline{\Sigma}_x$ but
\[
\partial(\delta_x + \frac{1}{2}\overline{\Sigma}_x) \cap M_x = \emptyset.
\]
\end{defn}

\begin{lem} \label{L:hyperplane_arrangement}
Let $\cX$ be an a.f.p. algebraic derived $k$-stack such that $\bL_\cX \cong (\bL_\cX)^\dual$, and $\cX^{\rm cl}$ admits a good moduli space. If the polytopes $\overline{\Sigma}_x$ span $(M_x)_\bR$ for any closed point $x \in |\cX|$ and there exists a generic class in $\NS(\cX)_\bR$, then there is a finite rational affine hyperplane arrangement $\{H_i \subset \NS(\cX)_\bQ\}_{i \in I}$ such that
\begin{enumerate}
\item $\delta \in \NS(\cX)_\bR$ is lattice generic if $\delta \notin \bigcup_{i \in I, \chi \in \NS(\cX)} (\chi + H_i)$, and
\item $\ell \in \NS(\cX)_\bR$ is generic if and only if $\ell$ is not parallel to any of the $H_i$.
\end{enumerate}
This hyperplane arrangement is unique up to translation by $\NS(\cX)$. If $\overline{\Sigma}_x$ does not span $(M_x)_\bR$ at every closed point, then the same holds after replacing $\NS(\cX)$ with a sublattice $\Lambda \subset \NS(\cX)$ and replacing $\NS(\cX)_\bR$ with $\Lambda_\bR$.
\end{lem}
\begin{proof}
By \Cref{thm:local_model}, we can find a strongly \'etale cover $\Spec(A) / G \to \cX$ for some split reductive group $G$. Any map $(B\Gm^n)_{k'} \to \cX$ lifts along the map $p : \Spec(A)/T \to \cX$, where $T$ is a fixed maximal torus of $G$, after passing to a suitable \'etale extension of $k'$. It follows that lattice genericity of $\delta$ is equivalent to the condition that for any finite type point $x \in \Spec(A)$ such that $p(x)$ is a closed point and the identity component ${\rm Stab}_T(x)^\circ$ maps to a maximal torus in $G_{p(x)}$, one has
\[
\partial \left( p^\ast(\delta)_x + \frac{1}{2} {\rm Hull}(\wedge^\ast H_0(p^\ast(\bL_\cX)_x)) \right) \cap \NS(B{\rm Stab}_T(x)^\circ) = \emptyset.
\]
Genericity of a class $\ell \in \NS(\cX)_\bR$ can likewise be expressed as the condition that $p^\ast(\ell)_x$ is not parallel to any face of this polytope at any such finite type point of $\Spec(A)$.

There are finitely many sub-tori $T_1,T_2,\ldots \subset T$ that arise as the identity component of stabilizers of points in $\Spec(A)$, and for any $i$ the character of $H_0(p^\ast(\bL_\cX)_x)$ is a constructible function on $\Spec(A)^{T_i}$. It follows that there are finitely many homomorphisms $\NS(\cX)_\bR \to \NS(BT_j)_\bR$ and lattice polytopes $\overline{\Sigma}_j \subset \NS(BT_j)_\bR$ such that genericity or lattice genericity of a class in $\NS(\cX)_\bR$ is equivalent to the corresponding genericity condition on the image of this class in each $\NS(BT_j)_\bR$.

Let $F \subset \NS(\cX)_\bR$ be a fundamental domain for the action of $\NS(\cX)$ on $\NS(\cX)_\bR$. We consider the hyperplane arrangement in $\NS(BT_j)_\bQ$ consisting of every hyperplane $H \subset \NS(BT_j)_\bQ$ that contains $p^\ast(\delta)$ and is parallel to a codimension $1$ face of $p^\ast(\delta) + \frac{1}{2} \overline{\Sigma}_j$ that contains a lattice point, for some $\delta \in F$. For each $j$, only finitely many hyperplanes arise in this way, and the hyperplane arrangement in $\NS(\cX)_\bQ$ in the statement of the lemma is the union of the preimage of each of these hyperplane arrangements under the maps $\NS(\cX)_\bR \to \NS(BT_j)_\bR$.

If each $\overline{\Sigma}_j$ is not full dimensional, then the genericity condition also includes the constraint that $\ell$ and $\delta$ must lie in the intersection over all $j$ of the preimage of the subspace spanned by $\overline{\Sigma}_j$. The $\Lambda \subset \NS(\cX)$ in the statement of the lemma is the lattice of integral classes in this subspace.
\end{proof}

\begin{rem}
If $\cX = \mu^{-1}(0)/G$ is the Hamiltonian reduction for a symplectic representation $V$ of $G$, then the notion of generic in \Cref{defn:generic} coincides with the notion of genericity for the action of $G$ on $V \times \fg$ introduced in \cite{SVdB}, and the condition that $\delta \in \NS(\cX)_\bR$ is lattice generic coincides with the condition on $\delta$ in \cite{halpern2016combinatorial}*{Thm.~3.2}. Using this one can construct examples of $\cX$ that do not admit generic or lattice generic classes in $\NS(\cX)_\bR$ but still have generically finite stabilizers, such as $T^\ast(\Sym^{2n}(\bC^2) / \GL_2)$.
\end{rem}

Now recall from \Cref{S:alternative_grr_condition} that for any quasi-smooth $\cX$ and any $x \in \cX(k')$, one can associate a morphism
\[
x_{\rm reg} : \Spec(k'[H_1(\bL_{\cX,x})[1]]) / G_x \to \cX.
\]
to any $G_x$-equivariant splitting of the map $H_1(\bL_{\cX,x})[1] \to \bL_{\cX,x}$, which exists if $G_x$ is linearly reductive. The morphism $x_{\rm reg}$ extends the canonical map $BG_x \to \cX$, and it is determined up to isomorphism by the condition that the fiber at $x$ of the induced map on cotangent complexes, $\bL_{\cX,x} \to H_1(\bL_{\cX,x})[1]$, is the identity on $H_1$. Note that if $\bL_\cX \cong (\bL_\cX)^\dual$, then $H_1(\bL_{\cX,x}) \cong \fg_x$, where $\fg_x$ is the lie algebra of $G_x$.

The morphism $x_{\rm reg}$ has finite Tor-amplitude, and thus $x_{\rm reg}^\ast$ preserves $\DCoh$, i.e.,
\[
x_{\rm reg}^\ast : \DCoh(\cX) \to \DCoh^{G_x}(k'[H_1(\bL_{\cX,x})[1]]).
\]

\begin{defn} \label{D:magic_windows}
Let $\cX$ be a locally a.f.p. algebraic derived $k$-stack such that $\bL_\cX \cong (\bL_\cX)^\dual$ and $\cX^{\rm cl}$ admits a good moduli space. Then for any $\delta \in \NS(\cX)_\bR$ we define $\fW_\cX(\delta) \subset \DCoh(\cX)$ to be the full subcategory of complexes $F$ such that for any $x \in \cX(k')$ representing a closed point of $|\cX|$, the character of the $G_x$-representation $H^\ast(x_{\rm reg}^\ast(F))$ lies in the polytope $\delta_x + \frac{1}{2} \overline{\Sigma}_x \subset (M_x)_\bR$.
\end{defn}

\begin{thm}[Magic windows] \label{T:magic_windows}
Let $\cX$ be a locally a.f.p. algebraic derived $k$-stack such that $\bL_\cX \simeq \bL_{\cX}^\dual$ and $\cX^{\cl}$ admits a good moduli space, and let $\cX$ be equipped with a rational quadratic norm on graded points. Suppose $\overline{\Sigma}_x \subset (M_x)_\bR$ spans for every closed point $x \in |\cX|$, and that there exists a generic class in $\NS(\cX)_\bR$.

If $\ell \in \NS(\cX)_\bR$ is generic and $\delta \in \NS(\cX)_\bR$ is lattice generic, then $\cX^{\rm{ss}}(\ell)$ is a smooth (hence classical) algebraic stack with finite automorphism groups, and the restriction functor defines an equivalence
\[
\fW_\cX(\delta) \simeq \DCoh(\cX^{\rm{ss}}(\ell)).
\]
In particular, $\DCoh(\cX^{\rm ss}(\ell)) \cong \DCoh(\cX^{\rm ss}(\ell'))$ for any two generic classes $\ell,\ell' \in \NS(\cX)_\bR$.
\end{thm}

\begin{rem} \label{R:generic_perturbation}
If $\cX$ is quasi-compact, then \Cref{T:magic_windows} holds as stated for any $\ell$ such that $\cX^{\rm ss}(\ell)$ has finite automorphism groups, because by \Cref{L:hyperplane_arrangement} $\ell$ can be perturbed so that $\ell$ becomes generic in the sense of \Cref{defn:generic} without changing $\cX^{\rm ss}(\ell)$. 

\end{rem}

The remainder of this section is dedicated to the proof of \Cref{T:magic_windows}.

\subsubsection{Reduction to a local statement}

Let $\cX$ be a locally a.f.p. algebraic derived $k$-stack such that $\bL_\cX \cong (\bL_\cX)^\dual$. We consider a $\Theta$-stratification of $\cX$, which we denote $\bS = \{\cS_\alpha \subset \Filt(\cX)\}_{\alpha \in I}$ for a totally ordered index set $I$. We will restrict our attention to maps $f : (B \Gm)_{k'} \to \cX$ classified by $k'$ points of the centers $\cZ_\alpha^{\rm{ss}}$ of the strata $\cS_\alpha$, and we call such a map $\bS$-canonical.
\begin{defn} \label{D:grr}
For any $\delta \in \NS(\cX)_\bR$, we define the full subcategory $\cG_{\bS}(\delta) \subset \DCoh(\cX)$ to consist of those complexes $F$ such that for all $\bS$-canonical maps $f : (B\Gm)_{k'} \to \cX$, $F$ satisfies the grade restriction rule,
\begin{equation} \label{E:grr_window}
\left\{ \begin{array}{rl} \minwt(f^\ast(F)) &\geq  \wt \left( \frac{1}{2}\det( (f^\ast \bL_\cX)^{<0} ) + f^\ast \delta \right), \text{ and} \\ \minwt(f^\ast (\bD_\cX(F))) &> \wt \left( \frac{1}{2} \det( (f^\ast \bL_\cX)^{<0} ) - f^\ast \delta \right) \end{array} \right.,
\end{equation}
where $(-)^{<0}$ denotes the summand on which $\Gm$ acts with negative weights, $\wt(-) : \NS(B\Gm)_\bR \to \bR$ is the isomorphism which takes a character to its weight, and $\minwt : \APerf((B\Gm)_{k'}) \to \bZ \cup \{-\infty\}$ is the function that assigns a complex $E$ to the lowest $w$ such that $E^w \ncong 0$.
\end{defn}

\begin{prop} \label{P:grr_category}
Let $\cX$ be a locally a.f.p. algebraic derived $k$-stack such that $\bL_\cX \cong (\bL_\cX)^\dual$, and let $\bS$ be a $\Theta$-stratification on $\cX$. Then for any $\delta \in \NS(\cX)_\bR$, the restriction functor induces an equivalence $\cG_\bS(\delta) \simeq \DCoh(\cX^{\rm{ss}})$.
\end{prop}
\begin{proof}
As in the proof of \Cref{T:derived_Kirwan_surjectivity_full}, one can easily reduce the claim to the case of a single closed connected $\Theta$-stratum. \Cref{L:verify_quasi-smooth_hypotheses} implies that $H_1(f^\ast(\bL_{\cX}))$ has nonnegative weights for any $\bS$-canonical graded point $f$, so \Cref{thm:derived_Kirwan_surjectivity_quasi-smooth} applies to this situation. Nakayama's lemma implies that to check if $F \in \APerf(\cZ)$ has weight $\geq w$, it suffices to verify that the fiber of $F$ over every finite type point of $\cZ$ has weight $\geq w$ with respect to the canonical $\Gm$-action. Therefore, comparing \Cref{D:grr} with \Cref{P:alternate_characterization} gives
\[
\cG_\bS(\delta) = \DCoh(\cX)^{\geq w} \cap \DCoh(\cX)^{<w},
\]
where $w = \lceil \wt(f^\ast \delta) + \frac{1}{2} a_f^\cX \rceil \in \bZ$ for any graded point $f$ of $\cX$ corresponding to a point in $\cZ$. This is equal to the category $\cG_{\cS, \rm coh}^{w}$ of \Cref{thm:derived_Kirwan_surjectivity_quasi-smooth}, and the claim follows.
\end{proof}

We now further assume that $\cX^{\rm cl}$ admits a good moduli space, and fix a rational quadratic norm on graded points of $\cX$, as in \Cref{S:variation}. \Cref{T:intrinsic_GIT} assigns a $\Theta$-stratification to any $\ell \in \NS(\cX)_\bQ$. In this case we will denote $\cG_{\bS}(\delta)$ by $\cG_\cX^\ell(\delta)$, and we also refer to an $\bS$-canonical graded point of $\cX$ as $\ell$-canonical.

Observe that by \Cref{P:grr_category}, to prove \Cref{T:magic_windows} it suffices to show that as subcategories of $\DCoh(\cX)$,
\[
\fW_\cX(\delta) = \cG_\cX^\ell(\delta),
\]
whenever $\ell$ is generic and $\delta$ is lattice generic. The key observation is that this equality can be checked \emph{locally} over the good moduli space of $\cX^{\rm cl}$. This will allow us to reduce \Cref{T:magic_windows} to a much simpler claim.

\begin{lem} \label{L:magic_containment}
If $\delta \in \NS(\cX)_\bR$ is lattice generic and $\ell$ is arbitrary, then $\fW_\cX(\delta) \subset \cG_\cX^\ell(\delta)$ as subcategories of $\DCoh(\cX)$.
\end{lem}
\begin{proof}
Consider a complex $F \in \fW_\cX(\delta)$, let $f : (B\Gm)_{k'} \to \cX$ be an $\ell$-canonical graded point, and let $x \in \cX(k')$ be the unique closed point in the closure of the image of $f$. Using \Cref{thm:local_model} one can find an \'etale neighborhood of $x$ of the form $\cX' = \Spec(R[\fg_x[1];d=\mu]) / G_x$ for a smooth $G_x$-algebra $R$ and weak co-moment map $\mu : \fg_x \to R$. $f$ lifts, after replacing $k'$ with a suitable \'etale extension, to a graded point of $\cX'$, which corresponds to a one parameter subgroup $\lambda : (\Gm)_{k'} \to G_x$ and a fixed point $y \in \mu^{-1}(0)^G(k')$. Let us change base to $k'$, so that $x$ is a closed point of $\Spec(R)$, and replace $k'$ with $k$ to simplify notation.

Let $F' \in \DCoh(\cX')$ denote the pullback of $F$, and let $E \in \DCoh^{G_x}(R)$ denote the underlying complex obtained by regarding  $F'$ as an $R$-module via the inclusion $R \subset R[\fg[1],d=\mu]$. The regularization $x_{\rm reg}$ is modeled by the surjection $R[\fg_x[1];d=\mu] \to k(x)[\fg_x[1]]$, so the $G_x$-equivariant complex of $k$-modules underlying $x_{\rm reg}^\ast(F) \in \DCoh^{G_x} (k[\fg_x[1]])$ is isomorphic to $E \otimes_R k(x)$. By hypothesis the character of $H_\ast(E \otimes_R k(x))$ lies in the polytope $\delta_x + \frac{1}{2} \overline{\Sigma}_x$.

Our goal is to verify the condition \eqref{E:grr_window}. Because $f$ lifts to $\cX'$, we can verify this weight condition for $F'$. 
\Cref{L:minweight} below then implies that
\[
\minwt(f^\ast(F')) = \minwt(E \otimes_{R} k(y)) - \wt(\det(\fg_x^{\lambda < 0})).
\]
By a theorem of Kempf \cite{kempf}, one can find a second one-parameter subgroup $\lambda' :\Gm \to G_x$ that commutes with $\lambda$, and such that $\lim_{t\to 0} \lambda'(t) \cdot y = x$. Using the semicontinuity of the fiber homology of $E$, we have
\[
\begin{array}{rl}
\minwt(f^\ast(F')) &\geq \minwt_\lambda(E \otimes_{R} k(x)) - \wt(\det(\fg_x^{\lambda<0})). \\
&> \wt(\delta_x) + \frac{1}{2} \wt(\det(H_0(\bL_{\cX,x})^{\lambda<0})) - \wt(\det(\fg_x^{\lambda<0})) \\
&= \wt(\delta_x) + \frac{1}{2} \wt(\det(\bL_{\cX,x}^{\lambda<0})).
\end{array}
\]
For the inequality on the second line above, we use the hypothesis that the character of $H_\ast(E\otimes_R k(x))$ lies in $\delta_x + \frac{1}{2} \overline{\Sigma}_x$ along with the definition of that polytope, and for the third line we have used the self-duality of $\bL_\cX$ and the fact that $\fg_x^{\lambda<0} \cong (\fg_x^\dual)^{\lambda<0}$. The second inequality holds strictly because $\delta$ is lattice generic, so it is not possible to have weights on the boundary of $\delta_x + \frac{1}{2}\overline{\Sigma}_x$.

Finally, we observe that $\wt(\det((\bL_{\cX,z})^{\lambda<0}))$ is a locally constant functor for $z \in \Spec(R[\fg_x[1];d=\mu])^{\lambda}$, so in fact we may replace this term in the final inequality with $\wt(\det((f^\ast \bL_\cX)^{<0}))$, which gives the first inequality in \eqref{E:grr_window}.

Now we consider the second inequality in \eqref{E:grr_window}. Serre duality commutes with \'etale base change, so again it suffices to bound the weights of $\bD_{\cX'}(F')$, and for this we let $E' \in \DCoh^{G_x}(R)$ denote the complex underlying $\bD_{\cX'}(F')$. The same argument as above shows that
\[
\minwt(f^\ast(\bD_{\cX'}(F'))) \geq \minwt_\lambda(E' \otimes_{R} k(x)) - \wt(\det(\fg_x^{\lambda<0})).
\]
The map $x_{\rm reg}$ is a regular closed immersion with normal bundle $T_x X \cong H_0(\bL_{\cX,x})^\ast$. It follows from Grothendieck's formula that
\[
E' \otimes_{R} k(x) \cong \bD_{k'[\fg_x[1]]}(E \otimes_R k(x)) \otimes \det(H_0(\bL_{\cX,x})^\ast)[-\dim H_0(\bL_{\cX,x})^\ast]
\]
The character $\det(H_0(\bL_{\cX,x})^\ast)$ has $\lambda$-weight $0$ by self-duality of $\bL_\cX$, so \Cref{L:weight_duality} gives
\[
\minwt(f^\ast(\bD_{\cX'}(F'))) \geq -\maxwt_\lambda(E \otimes_{R} k(x)) - \wt(\det(\fg_x^{\lambda<0})).
\]
Using the hypothesis on $E \otimes_R k(x)$ and self-duality of $\bL_\cX$, we compute
\[
\begin{array}{rl}
\maxwt_\lambda(E\otimes_R k(x)) &< \wt(\delta_x) + \frac{1}{2} \wt(\det(H_0((\bL_{\cX,x})^{\lambda>0}))) \\
&= \wt(\delta_x) - \frac{1}{2} \wt(\det(H_0((\bL_{\cX,x})^{\lambda<0}))).
\end{array}
\]
From this point, the second inequality in \eqref{E:grr_window} follows from the same argument that we used for the first inequality.
\end{proof}

The previous proof used the following, which we state in general terms for use below as well.

\begin{lem} \label{L:minweight}
Let $A = R[U[1];d]$ be a $\Gm$-equivariant \CDGA, where $R$ is a smooth $\Gm$-equivariant $k$-algebra and $U$ is a $\Gm$-representation in homological degree $1$. Let $x \in \Spec(A)$ be a $\Gm$-fixed point with residue field $k(x)$. For any $F \in \APerf^{\Gm}(\Spec(A))$, if either $U^{<0} = 0$ or $\minwt(F \otimes_A k(x)) > -\infty$, then
\[
\minwt(F \otimes_A k(x)) = \minwt(F \otimes_{R} k(x)) - \wt(\det(U^{<0})).
\]
\end{lem}
\begin{proof}
Let $A' := A \otimes_R k(x) \cong k(x)[U[1]]$, which is free on a $\Gm$-representation in homological degree $1$, and let $F' := F \otimes_A (A \otimes_R k(x))$ with its $A'$-module structure. Consider the maps of \CDGA's
\[
R \to A \to A \otimes_R k(x) \to k(x).
\]
The derived base change formula states that $F \otimes_R k(x) \cong F'$ as a graded $k(x)$-module. On the other hand, we have $F' \otimes_{A'} k(x) \cong F \otimes_A k(x)$. So it suffices to prove the claim for $A'$ itself, and for clarity we replace the residue field $k(x)$ with $k$ in our notation.

Consider the decomposition $U = U^0 \oplus U^{\neq 0}$ into summands with $0$ and non-zero $\Gm$-weights. The augmentation map factors as
\[
A' = k[U[1]] \to k[U^0[1]] \to k.
\]
The $U^{\neq 0}$-adic filtration gives a finite filtration of $A'$ as an $A'$-module whose associated graded is the $A'$-module $k[U^0[1]] \otimes_k \Sym(U^{\neq 0}[1])$, i.e., a direct sum of shifts in degree and weight of the module $k[U^0[1]]$. It follows that any $A'$-module $F \cong F \otimes_{A'} A'$ has a finite filtration whose associated graded is
\[
(F \otimes_{A'} k[U^0[1]]) \otimes_k \Sym(U^{\neq 0}[1]).
\]

Note that because $U^{\neq 0}$ has only non-zero weights, the lowest weight piece of $\Sym(U^{\neq 0}[1])$ is one dimensional, and is isomorphic to $\det(U^{<0})[\dim(U^{<0})]$. It follows that if $\minwt(F \otimes_{A'} k[U^0[1]]) = w > -\infty$, then there is a unique step in the filtration of $F$ whose graded piece has minimum weight $w+\wt(\det(U^{<0}))$ and all other graded pieces have higher minumum weight, which implies that
\[
\minwt(F) = \minwt(F \otimes_{A'} k[U^0[1]]) + \wt(\det(U^{<0})).
\]

Because $k[U^0[1]]$ has weight $0$ with respect to $\Gm$, $F \otimes_{A'} k[U^0[1]]$ already decomposes as a direct sum of constant-weight $k[U^0[1]]$-modules. Nakayama's lemma then implies that $F \otimes_{A'} k[U^0[1]]$ and $F \otimes_{A'} k$ are non-vanishing in precisely the same weights. This establishes the lemma in the case where the non-vanishing weights of $F \otimes_{A'} k[U^0[1]]$ are bounded below.

The remaining case is when $U^{<0}=0$, so $U$ has strictly nonnegative weights, and $\minwt(F)$ is finite. Using the Koszul resolution of $k[U^0[1]]$ as a $k[U]$-module, one observes that $F \otimes_{A'} k[U^0[1]]$ has bounded below weights as well, so we are in the previous case.

\end{proof}

\Cref{L:magic_containment} shows that to prove \Cref{T:magic_windows}, it suffices to show that any $F \in \cG_\cX^\ell(\delta)$ satisfies the weight condition of \Cref{D:magic_windows}. We now formulate our main reduction.

\begin{hyp}\label{H:magic_hypotheses} We consider the following hypotheses, which are more restrictive than those of \Cref{T:magic_windows}:
\begin{enumerate}
\item $G$ is a split linearly reductive $k$-group, $X := \Spec(R)$ is a smooth affine $G$-equivariant $k$-scheme, and $\mu : \fg \to R$ is a weak co-moment map (see \Cref{D:comoment_map}), which we regard as a map $X \to \fg^\ast$;
\item $\cX$ is the hamiltonian reduction $\cX = X_0/G$, where $X_0 = \mu^{-1}(0) = \Spec(R[\fg[1];d\xi = \mu(\xi)])$ is the derived zero fiber of the weak co-moment map;
\item $\ell, \delta \in \NS(\cX)_\bQ = \NS^G(X_0)_\bQ$ are represented by the trivial bundle tensored with a rational character of $G$, with $\ell$ generic and $\delta$ lattice generic;
\item $\cX$ is equipped with the rational quadratic norm on graded points coming from a fixed choice of Weyl-invariant inner product on the coweight space of $G$; and
\item $x \in X_0^G(k)$ is a closed fixed point in $X_0$.
\end{enumerate}
\end{hyp}

\begin{prop}[Local statement] \label{P:magic_reduction}
To prove \Cref{T:magic_windows}, it suffices to show that under \Cref{H:magic_hypotheses}, for any
\[
F \in \cG_{\cX}^\ell(\delta) \subset \DCoh^G(X_0) = \DCoh^G(R[\fg[1];d=\mu]),
\]
the character of the $G$-representation $H_\ast(F \otimes_R k(x))$ lies in the polytope $\delta_x + \frac{1}{2} \overline{\Sigma}_x \subset (M_x)_\bR$, where $k(x) \cong k$ denotes the residue field of $x \in X_0^G$.
\end{prop}

\begin{proof}
By \Cref{L:magic_containment} and \Cref{P:grr_category}, the claim of \Cref{T:magic_windows} is equivalent to the claim that $\fW_\cX(\delta) \subset \cG_\cX^\ell(\delta)$ is an equality.

For a strongly \'etale cover $\cX' \to \cX$, any graded point $f : (B\Gm)_{k'} \to \cX$ lifts to $\cX'$ after a suitable \'etale extension of $k'$, and $f$ is $\ell$-canonical if and only its lift to $\cX'$ is $\ell$-canonical by \Cref{T:intrinsic_GIT}. It follows that the restriction of $F$ to $\cX'$ lies in $\cG_{\cX'}^\ell(\delta)$. Also, the regularization of any closed point $x_{\rm reg} : \Spec(k'[\fg_x[1]]) / G_x \to \cX$ lifts to $\cX'$, so $F$ lies in $\fW_\cX(\delta)$ if and only if its restriction to $F'$ lies in $\fW_{\cX'}(\delta)$. After restricting to $\cX'$, $\ell$ and $\delta$ are still generic and lattice generic respectively, so it suffices to prove $\fW_{\cX'}(\delta) = \cG_{\cX'}^\ell(\delta)$.

It therefore suffices to show that any closed point $x \in \cX(k')$ admits a strongly \'etale neighborhood satisfying the (1)-(3) of \Cref{H:magic_hypotheses}, assuming that \Cref{T:magic_windows} is known under those hypotheses. After passing to an extension of $k'$ as necessary to ensure that $G = G_x$ is split,\Cref{thm:local_model} provides a strongly \'etale neighborhood $\cX'$ satisfying (1) and (2), so we will find a strongly \'etale open substack $\cU \subset \cX'$ containing $x$ that satisfies (3).

Rescale $\ell$ so that it is integral, let $L \in \Pic(\cX')$ be a representative of the restriction of $\ell$ to $\NS(\cX')$, and let $\chi$ denote the character corresponding to the restriction of $\ell$ to $(BG)_x$. Semicontinuity of the fiber dimension of a separated group scheme implies that the kernel of the homomorphism $I_{\cX'} \to (\Gm)_{\cX'}$ induced by $L \otimes \chi^{-1}$ is finite in an open neighborhood of $x \in \cX'$, and we can assume this open neighborhood is saturated in the sense of \cite{alper2013good}*{Defn.~6.1}. For sufficiently divisible $n$, all isotropy groups act trivially on the fibers of $L' = (L \otimes \chi^{-1})^n$ over this open substack. We can therefore find an open neighborhood of $x$ of the form $\cU = \Spec(R_a[\fg[1];d])/G \subset \cX'$ for some $a \in R^G$ such that $L'|_{\cU^{\rm cl}}$ is trivializable, i.e., it admits a nowhere vanishing section. The fact that $\cU^{\rm cl}$ is cohomologically affine implies that $H^0(\cU,L') \to H^0(\cU^{\rm cl},L'|_{\cU^{\rm cl}})$ is surjective, and thus $L'$ admits a nowhere vanishing global section as well. We may repeat this procedure for $\delta$ to shrink $\cU$ so that the resulting strongly \'etale map $\cU \to \cX$ satisfies (1)-(3).

Showing that $\fW_\cX(\delta) \subset \cG_\cX^\ell(\delta)$ is an equality is equivalent, by \Cref{P:grr_category}, to showing that $\fW_{\cX}(\delta) \to \DCoh(\cX^{\rm ss}(\ell))$ is essentially surjective. This only depends on $\ell$ and does not depend on the choice of rational quadratic norm on graded points of $\cX$, so we are free to choose \emph{any} such norm (even though a priori the category $\cG_\cX^\ell(\delta)$ depends on the norm).

Therefore we have shown that it suffices to prove $\fW_\cX(\delta) = \cG_\cX^\ell(\delta)$ for stacks $\cX$ satisfying (1)-(4) of \Cref{H:magic_hypotheses}. To complete the proof, we must show that this equality can be reduced to checking the weight condition in the statement of the proposition at an individual closed point $x \in \mu^{-1}(0)$ whose stabilizer is $G$.  

Fix an $F \in \cG_\cX^\ell(\delta) \subset \DCoh^G(R[\fg[1];d=\mu])$. At any closed point $x \in X_0^G$, we have observed in the proof of \Cref{L:magic_containment} that the underlying $G$-equivariant $k'$-module of $x_{\rm reg}^\ast(F)$ is $F \otimes_R k(x)$. So the window condition of \Cref{D:magic_windows} at $x$ amounts to the character of $H_\ast(F\otimes_R k(x))$ lying in $\delta_x + \frac{1}{2} \overline{\Sigma}_x$ as claimed.

To verify that $F$ satisfies the window condition of \Cref{D:magic_windows} at a closed point $x' \in \cX$ with smaller automorphism group $G_{x'} \subset G$, we use \Cref{thm:local_model} once more to construct an \'etale neighborhood of $x'$ of the form
\[
\Spec(R'[\fg_{x'};d]) / G_{x'} \to \Spec(R[\fg;d])/G,
\]
induced by a map $\Spec(R'[\fg_{x'};d]) \to \Spec(R[\fg;d])$ that is equivariant with respect to the inclusion of groups $G_{x'} \subset G$. It suffices to verify the window condition of \Cref{D:magic_windows} for the pullback of $F$ to $\Spec(R'[\fg_{x'};d])/G_{x'}$ at the point $x'$. We observe that the restriction of $\ell$, $\delta$, and the quadratic norm on graded points continue to satisfy conditions (3) and (4), so replacing $\mu : \fg \to R$ with $\mu' : \fg_x' \to R'$, we have reduced once again to checking the condition in the proposition under \Cref{H:magic_hypotheses}.

\end{proof}

For the remainder of the proof, we will therefore assume \Cref{H:magic_hypotheses}.

\subsubsection{The Landau-Ginzburg/Calabi-Yau correspondence}
\label{S:linear_koszul}

Our proof of \Cref{T:magic_windows} in the local case will make use of the ``Landau-Ginzburg/Calabi-Yau" correspondence, so we first recall this theorem and establish some notation.

\begin{notn} \label{N:lg}
We consider the scheme $Y = X \times \fg$ along with the map $W : Y \to \bA^1$ given by $W(x,\xi) = \langle \mu(x), \xi \rangle$. $W$ is equivariant with respect to the $\Gm$-action on $Y$ that scales $\fg$ with weight $-1$ and the $\Gm$-action on $\bA^1$ with weight $-1$, i.e. $W \in \Gamma(Y,\cO_Y\langle1\rangle)^{G\times \Gm}$, where $\cO_Y\langle1\rangle$ denotes the trivial line bundle twisted by the character of $\Gm$ of weight $-1$. This defines a map of stacks
\[
W : \cY = Y/(G\times \Gm) \to \bA^1/\Gm,
\]
And we refer to the pair $(\cY,W)$ as a \emph{graded LG model}.
\end{notn}

Regarding $W$ as an element of $R \otimes \fg^\ast \subset R[\fg^\ast]$, we consider the $G$-equivariant graded CDGA over $R$,
\[
\cB = R[\fg^\ast,\epsilon[1] ; d\epsilon = -W],
\]
where $\epsilon$ has homological degree $1$ and is fixed by $G$, $\fg^\ast$ has homological degree $0$, and both have internal degree $1$ with respect to the auxiliary $\Gm$-action defining the grading. We will denote the derived zero fiber $Y_0 := W^{-1}(0) = \Spec(\cB)$, and we consider the stack $\cY_0 := Y_0 / G\times \Gm$.

The Landau-Ginzburg/Calabi-Yau (LG/CY) correspondence \cites{shipman2012geometric, isik2013equivalence} identifies the $\infty$-category $\DCoh(\cX) \cong \DCoh^{G}(X_0)$ with the graded singularity category \cite{orlov}
\[
\DSing(\cY,W) \cong \DSing^{G\times \Gm}(Y,W).
\]
The category $\DSing(\cY,W)$ is obtained from $\DCoh(\cY_0)$ by a localization that inverts a certain natural transformation
\begin{equation} \label{E:beta_trans}
\beta : F\langle1\rangle [-2] \to F \text{ for } F \in \DCoh(\cY_0).
\end{equation}
The localization annihilates perfect complexes, and it induces an equivalence with the Verdier quotient
\[
\DSing(\cY,W) \simeq \DCoh(\cY_0) / \Perf(\cY_0) \simeq \DCoh^{G\times\Gm}(Y_0) / \Perf^{G\times \Gm}(Y_0).
\]
See, for instance, \cite{halpern2015equivariant}*{Sect.~3.3} for details on the construction. Note that $\DSing(\cY,W)$ is generated by the image of $\DCoh(\cY_0)$ under shifts, cones, and direct summands.

Let us next recall the precise statement of the LG/CY correspondence from \cite{isik2013equivalence}, which uses \emph{linear Koszul duality}. We define the Koszul dual \CDGA to $\cB$ to be
\[
\cA = R[\fg[-1],\beta[-2];d\xi = \mu(\xi) \cdot \beta, d\beta = 0].
\]
where $\fg$ is in cohomological degree $1$ and $\beta$ in cohomological degree $2$, and both are in internal degree $-1$. Linear Koszul duality \cite{MR}*{Thm.~3.7.1,3.6.1} provides an $R$-linear equivalence of categories of quasi-coherent graded $dg$-$\cO_X$-modules
\[
\Psi_\mu : \DCoh^{\Gm}_X(\cB)^{\rm op} \simeq \DCoh^{\Gm}_X(\cA),
\]
where the latter denotes the full subcategory of quasi-coherent $dg$-$\cA$-modules whose homology is coherent over the graded algebra $H^\ast(\cA)$.

Any complex in $\DCoh^{\Gm}_X(\cB)$ is quasi-isomorphic to a complex $E$ of $R$-flat $\cB$ modules, and $\Psi_\mu$ is defined by
\begin{equation} \label{E:linear_koszul}
\Psi_\mu(E) = \cA \tilde{\otimes}_{R} E^\dual,
\end{equation}
where $E^\dual$ denotes the graded dual $\bigoplus_i \Hom^{\Gm}_{\cO_X}(E,\cO_X \langle i \rangle)\langle i \rangle$, the $\cA$-module structure comes from left multiplication on the left factor, and we have used the notation $\tilde{\otimes}_R$ to indicate that the differential is a deformation of the differential on the tensor product of $dg$-$R$-modules by a Koszul type differential. Let $\alpha_1,\ldots,\alpha_N \in \fg^\ast$ be a basis, and let $\alpha_i^\dual$ be the dual basis of $\fg$. Then using the natural left $\cB$-module structure on $E^\dual$ and right $\cA$-module structure on $\cA$, the differential on $\cA \tilde{\otimes}_R E^\dual$ is
\[
d_{\Psi_\mu(E)} = d_\cA \otimes 1 + 1 \otimes d_{E^\dual} + (-\cdot \beta) \otimes (\epsilon\cdot -) + \sum_i (- \cdot \alpha_i^\dual) \otimes (\alpha_i \cdot -),
\]
where we use the usual convention from graded-commutative algebra that swapping the order of two odd elements in an expression introduces a sign change. The inverse functor $\Psi_\mu^{-1}$ has the same expression, with the role of $\cA$ and $\cB$ reversed, i.e., $\Psi_\mu^{-1}(M) = \cB \tilde{\otimes}_R M^\dual$ with differential
\begin{equation} \label{E:inverse_koszul_d}
 d_\cB \otimes 1 + 1 \otimes d_{M^\dual} + (-\cdot \epsilon) \otimes (\beta\cdot -) + \sum_i (- \cdot \alpha_i) \otimes (\alpha_i^\dual \cdot -).
\end{equation}

\begin{thm}[\cite{isik2013equivalence}] \label{T:koszul}
The equivalence $\Psi_\mu$ of \eqref{E:linear_koszul} identifies the subcategory of perfect dg-$\cB$-modules with the full subcategory of dg-$\cA$-modules which are $\beta$-torsion (i.e. annihilated by $\beta^n$ for some $n$), and $\Psi_\mu$ identifies the natural transformation \eqref{E:beta_trans} on the left-hand-side with multiplication by $\beta$ on the right-hand-side. It therefore descends to an equivalence of $\beta$-localizations
\[
\xymatrix{\DSing(\cY,W)^{\rm op} \cong \DCoh^{\Gm}(\cB)^{\rm op} / \Perf^{\Gm}(\cB)^{\rm op} \ar[r]^-{\Psi_\mu}_-{\cong} & \DCoh^{\Gm}_X(\cA[\beta^{-1}]) }.
\]
Furthermore, there is a natural equivalence of $dg$-categories
\[
\DCoh^{\Gm}(\cA[\beta^{-1}]) \cong \DCoh(R[\fg[1];d\xi = \mu(\xi)])
\]
that mixes the internal grading with the homological grading of an $\cA[\beta^{-1}]$-module.
\end{thm}

Both $\Psi_\mu$ and $\Psi_\mu^{-1}$, as well as the natural transformations $\id \to \Psi_\mu \circ \Psi_\mu^{-1}$ and $\Psi_\mu^{-1} \circ \Psi_\mu \to \id$, lift naturally to functors and natural transformations on categories of $G$-equivariant modules, and the same is true for the second isomorphism of \Cref{T:koszul}. We denote the resulting isomorphism
\[
\Psi_\mu : \DCoh^{G \times \Gm}(\cB)^{\rm op} \xrightarrow{\cong} \DCoh^{G \times \Gm}(\cA),
\]
which is still given by the formula \eqref{E:linear_koszul}. We also use $\Psi_\mu^G$ to denote the resulting isomorphism $\DSing^{G \times \Gm}(\cB)^{\rm op} \cong \DCoh^G(R)$ obtained by inverting $\beta$. Composing with Serre duality gives the equivalence mentioned above,
\[
\Psi_\mu^G \bD_{\cY_0} : \DSing(\cY,W) = \DCoh^{G \times \Gm}(\cB) / \Perf^{G \times \Gm}(\cB) \xrightarrow{\cong} \DCoh(\cX).
\]

The key property of $\Psi^G_\mu$ we use is that it commutes with two natural forgetful functors from $\DSing(\cY,W)$ and $\DCoh(\cX)$ to $\DCoh(X/G) = \DCoh^G(R)$. More explicitly, consider the canonical map of \CDGA's $\cB \to R[\epsilon[1]]$ that annihilates $\fg^\ast$, which has finite Tor-amplitude, and consider the inclusion of \CDGA's $R[\beta[-2]] \subset \cA$. We can regard $R[\epsilon[1]]$ and $R[\beta[-2]]$ as arising from the same construction as $\cB$ and $\cA$, but applied to the $0$ map $0 \to R$ instead of $\mu : \fg \to R$. Therefore, linear Koszul duality gives an equivalence
\[
\Psi^G_0 : \DCoh^{G \times \Gm}(R[\epsilon[1]]) \xrightarrow{\cong} \DCoh^{G \times \Gm}(R[\beta[-2]]).
\]

\begin{lem} \label{L:koszul_forget}
The following diagram is commutative
\[
\xymatrix{
\DCoh^{G\times \Gm}(\cB)^{\rm op} \ar[r]^{\Psi_\mu^G} \ar[d]^{R[\epsilon[1]] \otimes_{\cB} - } & \DCoh^{G\times \Gm}(\cA) \ar[d]^{\text{restrict } R[\beta[-2]] \subset \cA} \\
\DCoh^{G\times\Gm}(R[\epsilon[1]])^{\rm op} \ar[r]^{\Psi_0^G} & \DCoh^{G\times \Gm}(R[\beta[-2]]),
}
\]
and it induces a commutative diagram after inverting $\beta$,
\[
\xymatrix{ \DSing(\cY,W)^{\rm op} \ar[r]^-{\Psi_\mu^G} \ar[dr]_-{\Psi_0^G(R[\epsilon[1]] \otimes_\cB -)} & \DCoh(\cX) \ar[d]^{\text{restrict }R \subset R[\fg[1];d=\mu]} \\ 
& \DCoh(X/G) }.
\]
\end{lem}
\begin{proof}

The claim is equivalent to showing that the square commutes after inverting the horizontal arrows. Given an $R$-flat complex $M \in \DCoh^{G\times \Gm}(\cA)$, we have $(\Psi^G_{\mu})^{-1}(M) = \cB \tilde{\otimes}_R M^\dual$. This is semi-free as a $\cB$-module, so we have:
\[
R[\epsilon[1]] \otimes_\cB (\Psi_\mu^G)^{-1}(M) \cong R[\epsilon] \tilde{\otimes}_R M^\dual
\]
with differential given by $1 \otimes d_{M^\dual} + (-\cdot \epsilon) \otimes (\beta \cdot -)$, which arises from the formula \eqref{E:inverse_koszul_d} after identifying all generators in $\fg^\ast$ with $0$. On the other hand, this is exactly the complex $(\Psi_0^G)^{-1}(M)$, where $M$ is regarded as a $G \times \Gm$-equivariant $R[\beta[-2]]$-module.

The identification $\DCoh^{G \times \Gm}(\cA[\beta^{-1}]) \cong \DCoh^G(R[\fg[1];d=\mu])$ is induced by the isomorphism of $G \times \Gm$-equivariant \CDGA's
\[
R[\fg[1];d=\mu] [\beta^{\pm 1}] \xrightarrow{\cong} \cA[\beta^{-1}]
\]
that acts on generators as $\beta^{\pm 1} \mapsto \beta^{\pm 1}$ and $\alpha_i^\dual \mapsto \beta^{-1} \alpha_i^\dual$, where the generators $\alpha_i^\dual$ in the former have $\Gm$-weight $0$ and homological degree $1$, followed by the evident identification
\[
\DCoh^{G \times \Gm}(R[\fg[1]][\beta^{\pm 1}]) \cong \DCoh^G(R[\fg[1]]).
\]
From this one can identify the functor $\DCoh(\cX) \to \DCoh^G(R)$ as the $\beta$-localization of the functor
\[
\DCoh^{G\times \Gm}(\cA) \to \DCoh^{G\times\Gm}(R[\beta[-2]])
\]
that restricts an $\cA$-module to $R[\beta[-2]] \subset \cA$. This establishes the second claim.
\end{proof}

\subsubsection{Proof of \Cref{T:magic_windows} in the local case}

We now return to the setting of \Cref{H:magic_hypotheses}, and use \Cref{T:koszul} to convert the statement of \Cref{P:magic_reduction} to an equivalent statement about $\DSing^{G\times\Gm}(Y,W)$.

The conditions (3) and (4) of \Cref{H:magic_hypotheses} guarantee that the linearization $\ell = [\cO_{X_0} \otimes \chi]$ and the rational quadratic norm on graded points naturally extends to $X / G$. Then, they can be pulled back to $\cY$ under the projection $\cY = X \times \fg / (G \times \Gm) \to X/G$. If we consider the projection
\[
\phi : \cY' := X \times \fg / G \to \cY,
\]
the $\Theta$-stratification of $\cY'$ induced by \Cref{T:intrinsic_GIT} is $\Gm$-equivariant, and thus descends to a $\Theta$-stratification of $\cY$, which we call $\bS$. Note that the $\bS$-canonical graded points of $\cY$ lift to $\ell$-canonical graded points of $\cY'$.


We first define two subcategories of $\DSing(\cY,W)$ analogous to \Cref{D:grr_for_mf}. For any graded point $f : (B\Gm)_{k'} \to \cY_0$ we define the integer
\[
\eta_f^\cY := \wt\left( \det( (f^\ast \bL_\cY)^{>0} ) \right),
\]
and for any $F \in \DCoh(\cY_0)$ we consider the grade restriction rule:
\begin{equation} \label{E:grr_mf}
\text{weights of } f^\ast(F) \text{ lie in the interval } \wt(f^\ast \delta) + \frac{1}{2} [-\eta_f^\cY,\eta_f^\cY].
\end{equation}

\begin{defn} \label{D:grr_for_mf}
\begin{enumerate}
\item $\fW_{\cY_0}(\delta) \subset \DCoh(\cY_0)$ denotes the full subcategory of complexes that satisfy \eqref{E:grr_mf} with respect to all graded points $f : (B\Gm)_{k'} \to \cY_0$ whose composition with the projection $\cY_0 = Y_0 / (G \times \Gm) \to B\Gm$ is trivial, and $\fW_{(\cY,W)}(\delta) \subset \DSing(\cY,W)$ denotes the subcategory generated by these objects under shifts, cones, and direct summands.
\item $\cG^\ell_{\cY_0}(\delta) \subset \DCoh(\cY_0)$ and $\cG_{(\cY,W)}^\ell(\delta) \subset \DSing(\cY,W)$ denote the same, but with the half-open interval $[-\eta_f^\cY,\eta_f^\cY)$ instead of the closed interval in \eqref{E:grr_mf}, and with this grade restriction rule holding only for $\bS$-canonical graded points $f$.
\end{enumerate}
\end{defn}

By \Cref{P:magic_reduction}, the proof of \Cref{T:magic_windows} will be complete once we verify the following:

\begin{prop} \label{P:magic_windows_mf}
Under \Cref{H:magic_hypotheses} and \Cref{N:lg}, we have
\begin{enumerate}
\item $\fW_{(\cY,W)}(\delta) = \cG_{(\cY,W)}^\ell(\delta)$,
\item the isomorphism $\Psi_\mu^G : \DSing^{G\times \Gm}(Y,W)^{\rm op} \cong \DCoh^G(X_0)$ of \Cref{T:koszul} restricts to an isomorphism $\cG_{(\cY,W)}^\ell(\delta)^{\rm op} \cong \cG_{\cX}^\ell(\delta),$ and
\item for any $E \in \fW_{(\cY,W)}(\delta)$, the character of the $G$-representation $H_\ast(\Psi_\mu^G(E) \otimes_R k(x))$ lies in the polytope $\delta_x + \frac{1}{2} \overline{\Sigma}_x \subset (M_x)_\bR$.
\end{enumerate}
\end{prop}

We will prove this at the end of the section, after gathering some preliminary lemmas.

\begin{lem} \label{L:koszul_preserve_weights}
Under \Cref{H:magic_hypotheses}, let $\lambda : \Gm \to G$ be a one parameter subgroup, and $x \in X_0^\lambda$ be a fixed point for which $H_1(\bL_{\cX,x})$ has nonnegative $\lambda$-weights. Then for any $E \in \DCoh^{G \times \Gm}(Y_0)$, if we regard it as an element of $\DSing(\cY,W)$ and let $F = \Psi_\mu^G(E) \in \DCoh^G(X_0)$, then
\begin{align*}
\minwt_\lambda(F_x) &\geq \minwt_\lambda(E_{(x,0)}) - \wt_\lambda(\det(\fg^{\lambda<0})), and\\
\maxwt_\lambda(F \otimes_R k(x)) &\leq \maxwt_\lambda(E_{(x,0)}).
\end{align*}
\end{lem}

\begin{proof}
If we let $p : X_0/G \to X/G$ be the map induced by the $G$-equivariant inclusion $R \subset R[\fg[1];d=\mu]$, let $\cB = R[\fg^\ast,\epsilon[1];d\epsilon=-W]$, and let $i : \Spec(R[\epsilon[1]]) \hookrightarrow Y_0 = \Spec(\cB)$ denote the $G$-equivariant closed immersion induced by the surjection $\cB \to R[\epsilon[1]]$. Then \Cref{L:koszul_forget} implies that $p_\ast(F)$ is quasi-isomorphic to the complex obtained by pulling back $i^\ast(E) = E \otimes_\cB R[\epsilon[1]]$ and then applying the $\beta$-localization functor
\[
\DCoh^{G\times \Gm}(R[\epsilon[1]])^{\rm op} \xrightarrow{\Psi_0^G} \DCoh^G(R[\beta[-2]]) \to \DCoh^G(R).
\]
This functor commutes with restricting to the fiber over $x \in \Spec(R)$, so we have
\begin{equation} \label{E:compute_koszul_fiber}
p_\ast(F) \otimes_R k(x) \cong \Psi_0^G(E \otimes_\cB k(x)[\epsilon[1]]) \otimes_{k(x)[\beta[-2]]} k(x)[\beta^{\pm 1}].
\end{equation}
As a $k(x)[\epsilon[1]]$-module, $E \otimes_\cB k(x)[\epsilon[1]]$ decomposes as a direct sum of complexes of constant $\lambda$-weight, because $\epsilon$ has weight $0$ with respect to $\lambda$. It follows that the Koszul dual complex $\Psi_0^G(E \otimes_\cB k(x)[\epsilon[1]]) \in \DCoh(k(x)[\beta[-2]])$ has non-vanishing homology in precisely the same $\lambda$-weights. Inverting $\beta$ can not create homology in new weights, so if $p_\ast(F) \otimes_R k(x)$ has non-vanishing homology in weight $w$, then so does $E \otimes_\cB k(x)[\epsilon[1]]$.

These observations show that $\minwt_\lambda(p_\ast(F) \otimes_R k(x)) \geq \minwt_\lambda(E \otimes_\cB k(x)[\epsilon[1]])$ and the opposite inequality for $\maxwt_\lambda(-)$. \Cref{L:minweight} implies that
\[
\minwt_\lambda(F_x) = \minwt_\lambda(p_\ast(F) \otimes_R k(x)) - \wt(\det(\fg^{\lambda <0})).
\]
The claim follows by Nakayama's lemma which implies that $E \otimes_\cB k(x)[\epsilon[1]]$ and $E_{(x,0)}$ have non-vanishing homology in exactly the same weights.
\end{proof}

\begin{lem}
Under \Cref{H:magic_hypotheses}, the isomorphism $\Psi_\mu^G : \DSing^{G \times \Gm}(Y,W) \cong \DCoh^G(X_0)$ maps $\cG_{(\cY,W)}^\ell(\delta)$ to $\cG_{\cX}^\ell(\delta)$.
\end{lem}
\begin{proof}
By \Cref{D:grr_for_mf}, it suffices to show that for any $E \in \cG_{\cY_0}^\ell(\delta)$ the complex $F:= \Psi_\mu^G(E) \in \DCoh(\cX)$ lies in $\cG_\cX^\ell(\delta)$. Consider a one-parameter-subgroup $\lambda : \Gm \to G$  and $x \in X_0^\lambda$ such that $H_1(\bL_{\cX,x})$ has nonnegative $\lambda$-weights. We let $f' : (B\Gm)_{k'} \to \cY_0$ and $f : (B\Gm)_{k'} \to \cX$ be the graded points corresponding to $\lambda$ and the fixed point $(x,0)$ and $x$ respectively. It suffices to show that for any $E \in \DCoh^{G\times \Gm}(Y_0)$, if the weights of $E_{(x,0)}$ lie in the interval
\begin{equation} \label{E:target_weight_bound}
\wt_\lambda(\delta_x) + \frac{1}{2}[-\eta_{f'}^\cY,\eta_{f'}^\cY),
\end{equation}
then $F:=\Psi_\mu^G(E) \in \DCoh(\cX)$ satisfies the weight bounds of \eqref{E:grr_window} with respect to the graded point $f$. This is what we show in the rest of the proof.

By change of base we may assume $k=k'$, to simplify notation. First let us compute the cotangent complexes:
\begin{gather*}
\bL_{\cY, (x,0)} \simeq [ 0 \to \Omega_{X,x}^1 \oplus \fg^\ast \to \fg^\ast \oplus k] , \text{ whereas}\\
\bL_{\cX,x} \simeq [ \fg \to \Omega_{X,x}^1 \to \fg^\ast].
\end{gather*}
Self duality implies that $(\bL_{\cX,x})^{\lambda<0} \simeq ((\bL_{\cX,x})^{>0})^\dual$, so one can compute
\[
\eta_{f'}^{\cY} + 2 \wt_\lambda(\det(\fg^{\lambda <0})) = - \wt(\det((\bL_{\cX,x})^{\lambda <0})).
\]
Note also that $f^\ast(\omega_{X/G}) \cong \det(\bL_{\cX,x}) \otimes \det(\fg[1])^\dual \cong k$. \Cref{L:koszul_preserve_weights} now implies
\begin{align*}
\minwt(f^\ast(F)) &\geq \minwt_\lambda(E_{(x,0)}) - \wt_\lambda(\det(\fg^{\lambda <0})) \\
&\geq \wt(f^\ast \delta) - \frac{1}{2} \eta_f^\cY - \wt(\det(\fg^{\lambda<0})) \\
&= \wt(f^\ast \delta) + \frac{1}{2} \wt_\lambda(\det(\bL_{\cX,x}^{\lambda<0})).
\end{align*}
Similarly, we observe that $p_\ast(\bD_\cX(F)) \cong \bD_{X/G} (p_\ast(F)) \cong \omega_{X/G} \otimes (p_\ast(F))^\dual$, so \Cref{L:koszul_preserve_weights} and \Cref{L:minweight} imply that
\begin{align*}
\minwt(f^\ast(\bD_{\cX}(F))) &\geq -\maxwt_\lambda (p_\ast(F) \otimes_R k(x)) - \wt_\lambda(\det(\fg^{\lambda <0})) \\
&> -\wt(f^\ast \delta) - \frac{1}{2} \eta_f^\cY - \wt_\lambda(\det(\fg^{\lambda<0})) \\
&= - \wt(f^\ast \delta) + \frac{1}{2} \wt_\lambda(\det(\bL_{\cX,x}^{\lambda<0})).
\end{align*}
These are precisely the weight bounds of \eqref{E:grr_window}.
\end{proof}

\begin{lem} \label{L:magic_windows_mf}
Under \Cref{H:magic_hypotheses}, $\fW_{(\cY,W)}(\delta) = \cG^\ell_{(\cY,W)}(\delta)$.
\end{lem}
\begin{proof}
It suffices to show that $\fW_{\cY_0}(\delta) = \cG^\ell_{\cY_0}(\delta)$ as subcategories of $\DCoh(\cY_0)$, because these categories generate the corresponding subcategories of $\DSing(\cY,W)$ by definition. Let $i : \cY_0 \hookrightarrow \cY$ denote the inclusion, and as above let $\phi : \cY' := Y/G \to \cY = Y/ (G\times \Gm)$ denote the quotient map. Let $\fW_{\cY'}(\delta)  \subset\Perf(\cY')$ denote the subcategory of complexes that satisfy \eqref{E:grr_mf} for all graded points of $\cY'$, and let $\cG^\ell_{\cY'}(\delta) \subset \Perf(\cY')$ denote the same, except only for $\ell$-canonical graded points, and using the half-open interval $\frac{1}{2} [-\eta_f^\cY,\eta_f^\cY)$ in \eqref{E:grr_mf}. 

A graded point of $\cY$ lifts to $\cY'$ if and only if the composition with the projection $\cY \to B\Gm$ is trivial, so \Cref{L:minweight} and \Cref{D:grr_for_mf} imply that $F \in \DCoh(\cY_0)$ lies in $\fW_{\cY_0}(\delta)$ or $\cG_{\cY_0}^\ell(\delta)$ if and only if $\phi^\ast(i_\ast(F))$ lies in $\fW_{\cY'}(\delta)$ or $\cG_{\cY'}^\ell(\delta)$ respectively. It therefore suffices to show that the inclusion $\fW_{\cY'}(\delta) \subset \cG^\ell_{\cY'}(\delta)$ is an equality. This equality was proved in \cite{halpern2016combinatorial} under the more restrictive hypothesis that $Y = \bA^n$ is a (quasi-symmetric) linear representation of $G$, and it is extended in \cite{HLMO} to the general case of a smooth affine quotient stack via Luna's \'{e}tale slice theorem, but we summarize the argument here:

\medskip

Let $V = T_{(x,0)} Y = T_x X \oplus \fg$. Then $T_x X$ is a self-dual $G$-representation and hence $V$ is a quasi-symmetric representation in the sense of \cite{SVdB}. Luna's slice theorem gives a strongly \'etale $G \times \Gm$-equivariant affine open neighborhood $U \subset Y$ of $(x,0)$ and a strongly \'etale $G \times \Gm$-equivariant affine morphism $\pi: U \to V$. We have used $\Gm$-equivariance to guarantee that if $(y,0)\in U$, then $U$ contains the entire fiber $\{y\} \times \fg$. It follows from \Cref{T:intrinsic_GIT} that both the inclusion $U \subset Y$ and $\pi : U \to V$ are compatible with the $\Theta$-stratifications of induced by $\ell$ (forgetting the $\Gm$-action).

In particular $U^{\rm ss}(\ell) = \pi^{-1}(V^{\rm ss}(\ell))$, so $\DCoh(U^{\rm ss}(\ell)/G)$ is generated under cones, shifts, and direct summands by the essential image of $\DCoh(V^{\rm ss}(\ell)/G)$ under $\pi^\ast$. It follows that $\cG_{U/G}^\ell(\delta)$ is generated by the essential image of $\cG_{V/G}^\ell(\delta)$ under $\pi^\ast$. By \cite{halpern2016combinatorial}*{Thm.~3.2}, the category $\cG_{V/G}^\ell(\delta)$ is generated by the complexes $\cO_V \otimes_k U$ where $U$ ranges over all $G$-representations whose character lies in $\overline{\Sigma}_x$ -- this is where the genericity hypotheses on $\ell$ and $\delta$ are used. It follows that the same is true for $\cG_{U/G}^\ell(\delta)$. In particular, $\fW_{U/G}(\delta) = \cG_{U/G}^\ell(\delta)$.

Because the inclusion $U/G \to Y/G$ is compatible with the $\Theta$-stratifications induced by $\ell$, we have shown that any complex $F \in \cG_{Y/G}^\ell(\delta)$ satisfies the condition \eqref{E:grr_mf} for all graded points in $U/G$. Using Luna's slice theorem at other closed points of $X/G$, we can apply the same argument to cover $Y$ by open subsets in which the restriction of $F$ satisfies \eqref{E:grr_mf}, and the claim follows. 

\end{proof}

\begin{lem} \label{L:comoment_critical_locus}
If $ \mu : X \to \fg^\ast$ is a weak co-moment map, inducing a function $W : X \times \fg \to \bA^1$, then
\[
\Crit(W)|_{(X \times \fg)^{\rm{ss}}} \subset X^{\rm{ss}} \times \fg
\]
\end{lem}
\begin{proof}
From the defining formula $W(x,\xi) = \langle \mu(x), \xi \rangle$, we see that $dW_{(x,\xi)} = \langle d\mu(x), \xi \rangle \oplus \mu(x) \in \Omega_{X,x}^1 \oplus \fg^\ast$. Note that for a weak co-moment map, at any point of $x$ the linear map $d\mu : T_{X,x} \to \fg^\ast$ is isomorphic to the action map $a : \Omega_{X,x}^1 \to \fg^\ast$, so $\langle d\mu(x),\xi \rangle = 0$ if and only if $\xi \in \Lie(\Stab_G(x))$. It follows that
\[
dW_{(x,\xi)} = 0 \quad \Leftrightarrow \quad \mu(x)=0 \text{ and } \xi \in \Lie(\Stab_G(x)).
\]
If such a point $(x,\xi) \in X\times \fg$ is $\ell$-semistable, then $x$ must be semistable as well. If not, then one could consider the canonical maximal destabilizing subgroup $\lambda : \Gm \to G$ for $x$, and because $\xi \in \Lie(\Stab_G(x))$ the limit $\lim_{t\to 0} \lambda(t) \cdot (x,\xi)$ would exist. Thus $\lambda$ would be destabilizing for $(x,\lambda)$ according to the Hilbert-Mumford criterion.
\end{proof}

\begin{lem} \label{L:identify_grr_categories}
The restriction of $\Psi_\mu^G$ defines an equivalence $\cG_{(\cY,W)}^\ell(\delta) \cong \cG_{\cX}^\ell(\delta)$.
\end{lem}

\begin{proof}
The equivalence of \Cref{T:koszul} clearly commutes with restriction to an open substack, so we have a commutative diagram of functors
\[
\xymatrix{
\cG_{(\cY,W)}^\ell(\delta) \ar[d]^{\Psi_\mu^G} \ar[r]^-{\res} & \DSing^{G \times \Gm}(X^{\rm ss} \times \fg, W) \ar[d]^{\Psi_\mu^G} \\
\cG_{\cX}^\ell(\delta) \ar[r]^-{\res} & \DCoh(\cX^{\rm ss}) }.
\]
The right vertical arrow is an equivalence and the left vertical arrow is fully faithful by \Cref{T:koszul}, and the bottom horizontal arrow is an equivalence by \Cref{P:grr_category}, so it suffices to show that the top horizontal arrow is essentially surjective (in which case it follows that it is an equivalence). Removing closed substacks that are disjoint from $\Crit(W)$ does not affect the singularity category \cite{orlov2}*{Prop.~1.14} \cite{halpern2016combinatorial}*{Lem.~5.5}, so by \Cref{L:comoment_critical_locus} it suffices to show that the restriction functor
\[
\res : \cG_{(\cY,W)}^\ell(\delta) \to \DSing^{G\times \Gm}((X \times \fg)^{\rm ss}, W)
\]
is essentially surjective. Because $\DSing$ is generated under shifts, cones, and direct summands by $\DCoh$, it suffices to show that
\[
\res : \cG_{\cY_0}^\ell(\delta) \to \DCoh(\cY_0^{\rm ss})
\]
is essentially surjective. \Cref{thm:derived_Kirwan_surjectivity_quasi-smooth} implies that this restriction functor is an equivalence, provided that $H_1(f^\ast(\bL_{\cY_0}))$ has nonnegative weights for any $\bS$-canonical graded point $f : (B\Gm)_{k'} \to \cY_0$. In fact, it has weight $0$, because any $\bS$-canonical graded point lifts to $\cY'$, and $Y_0$ is a hypersurface in the smooth scheme $Y$ defined by a $G$-invariant function.
\end{proof}

\begin{proof}[Proof of \Cref{P:magic_windows_mf}]
The first claim is established in \Cref{L:magic_windows_mf}, and the fact that $\Psi_\mu^G$ restricts to an equivalence $\cG_{(\cY,W)}^\ell(\delta)^{\rm op} \cong \cG_{\cX}^\ell(\delta)$ is established in \Cref{L:identify_grr_categories}.

For the claim (3), it suffices to show that for $E \in \fW_{\cY_0}(\delta)$, the image $F = \Psi_\mu^G(E)$ satisfies the weight condition. The argument for this is essentially the same as in the proof of \Cref{L:koszul_preserve_weights}. In particular, \eqref{E:compute_koszul_fiber} implies that $F \otimes_R k(x)$ is the $\beta$-localization of the koszul dual of $E \otimes_{\cB} k(x)[\epsilon[1]]$, which implies that $F \otimes_R k(x)$ satsifies any weight bounds that are satisfied by $E \otimes_{\cB} k(x)[\epsilon[1]]$. Nakayama's lemma implies that the non-zero weights of the latter are precisely the non-zero weights of $E_{(x,0)}$.

We have reduced the claim to showing that the character of the $G$-representation $H_\ast(E_{(x,0)})$ lies in $\delta_x + \frac{1}{2} \overline{\Sigma}_x \subset M_\bR$. Any cocharacter $\lambda$ defines a linear projection $M_\bR \to \bR$ onto the weight space of $\Gm$, and \Cref{D:grr_for_mf} implies that the projection of the character of $H_\ast(E_{(x,0)})$ onto $\bR$ lies in the interval $\wt_\lambda(\delta_x) + \frac{1}{2}[-\eta_\lambda,\eta_\lambda]$, where
\[
\eta_\lambda = \wt_\lambda(\det(\bL_{\cY,(x,0)}^{\lambda > 0})) = \wt_\lambda(\det((\Omega^1_{X,x})^{\lambda > 0})).
\]
This interval is the image of $\delta_x + \frac{1}{2} \overline{\Sigma}_x$. Because this holds for any cocharacter, the character of $H_\ast(E_{(x,0)})$ must lie in $\delta_x+\frac{1}{2} \overline{\Sigma}_x$.

\end{proof}


\subsection{Moduli spaces of Bridgeland semistable complexes}
\label{S:Bridgeland_moduli}

We will apply \Cref{T:intrinsic_GIT} to moduli spaces of Bridgeland semistable complexes of (twisted) coherent sheaves on a $K3$ surface, so let us recall this construction. We follow the notation of \cite{bayer2014mmp}*{Sect.~2}, and we refer the reader there for a more complete discussion.

\subsubsection{Stability conditions and the moduli functor}

Let $S$ be a $K3$ surface and let $\alpha \in {\rm Br}(S)$ be a cohomological Brauer class. We let $\cD = \DCoh(S,\alpha)$ denote the derived category of complexes of $\alpha$-twisted coherent sheaves on $S$ \cite{caldararu}. In \cite{HuybrechtsStellari}, Huybrechts and Stellari define a weight-$2$ Hodge structure on $H^\ast(S;\bZ)$, denoted $H^\ast(S,\alpha,\bZ)$. The \emph{Mukai vector} map
\[
v : K_0(\cD) \to H_{\rm alg}^\ast(S,\alpha,\bZ) := H^{1,1}(S,\alpha,\bC) \cap H^\ast(S;\bZ)
\]
is given by $v(E) = {\rm ch}^B(E) \sqrt{{\rm td}(S)}$ for a certain ``twisted'' Chern character map ${\rm ch}^{B} : K_0(\cD) \to H^\ast(S;\bQ)$. The group $H_{\rm alg}^\ast(S,\alpha,\bZ)$ has a nondegenerate symmetric bilinear form, called the \emph{Mukai pairing}, that satisfies $- (v(E),v(F)) = \chi(E,F)$ for any $E,F \in \DCoh(S,\alpha)$. $v$ identifies $H_{\rm alg}^\ast(S,\alpha,\bZ)$ with the numerical $K$-theory of $\cD$, i.e., the quotient of $K_0(\cD)$ by the kernel of this pairing (surjectivity follows from \cite{HuybrechtsStellari}*{Prop.~1.4}).

A \emph{numerical stability condition} $\sigma$ on $\cD$ consists of the heart of a bounded $t$-structure $\cA_\sigma \subset \cD$ and a group homomorphism $Z_\sigma : H_{\rm alg}^\ast(S,\alpha,\bZ) \to \bC$ called the \emph{central charge} that satisfy certain conditions. It is required that for any non-zero $E \in \cA$, $\Im(Z(E)):=\Im(Z(E)) \geq 0$ and if equality holds then $\Re(Z(E))<0$. We define the \emph{phase} $\phi(E) \in (0,1]$ as the unique value for which $Z(E) = |Z(E)| e^{\phi(E) \pi i}$. An object $E \in \cA$ is \emph{$\sigma$-semistable} if there is no sub-object with larger phase. Furthermore, it is required that every object $E \in \cA$ admits a finite filtration $0 \subsetneq E_p \subsetneq \cdots \subsetneq E_0 = E$, called the Harder-Narasimhan filtration, such that $E_i / E_{i+1}$ is semistable and $\phi(E_i/E_{i+1})$ is strictly increasing in $i$.

The main result of \cite{bridgeland} is that the set of numerical stability conditions that satisfy a ``support condition'' naturally has the structure of a complex manifold, called $\cStab(\cD)$, where the map $\cStab(\cD) \to \Hom(H_{\rm alg}^\ast(S,\alpha,\bZ) ,\bC)$ taking $(\cA,Z) \mapsto Z$ is locally a homeomorphism onto its image. In \cite{bridgelandK3}, Bridgeland describes a particular connected component $\cStab^\dagger (\cD)$ of the space of numerical stability conditions in the case $\alpha=0$, and this was extended to twisted $K3$ surfaces in \cite{HMS}.

\begin{defn}[Moduli of Bridgeland semistable complexes] \label{D:moduli_complexes}
For $\sigma \in \cStab^\dagger(\cD)$ and Mukai vector $v \in H_{\rm alg}^\ast(S,\alpha,\bZ)$, we let $\cM_\sigma(v)$ denote the (classical) moduli stack of $\sigma$-semistable complexes $E \in \cA_\sigma$ with $v(E) = \pm v$,\footnote{Only one of the vectors $\pm v \in H_{\rm alg}^\ast(S,\alpha,\bZ)$ can be the image of an object in $A_\sigma$. It might happen that as the $t$-structure varies $E \in \cA_\sigma$ no longer lies in the heart, but $E[1]$ does. So while the objects in $M_\sigma(v)$ change, it is natural to identify the two moduli functors, but this will change the sign of the Mukai vector.} Regard $E \in \DCoh(T \times S,\alpha)$ as a family of twisted complexes over a finite type $k$-scheme $T$, and let $E_{k(t)} \in \DCoh(S_{k(t)},\alpha)$ denote the derived fiber over $t \in T$. If $t \in T$ is a finite type point, we can regard $E_{k(t)}$ as an object of $\DCoh(S,\alpha)$ via pushforward along the finite map $S_{k(t)} \to S$, and the rational numerical $K$-theory class $\frac{1}{\deg (k(t)/k)} v(E_{k(t)})$ is locally constant \cite{halpern2014structure}*{Lem.~6.0.2.1}. The moduli functor $\cM_\sigma(v)$ assigns
\[
T \mapsto \left\{ E \in \DCoh(T \times S, \alpha) \left| \begin{array}{c} \forall \text{ finite type } t \in T, E_{k(t)} \in \cA_\sigma \text{ and is }\\ \sigma\text{-semistable of class } \pm \deg(k(t)/k) v \end{array} \right. \right\},
\]
It is shown in \cite{toda_moduli} that $\cM_\sigma(v)$ is an algebraic stack of finite type with affine diagonal over $\Spec(k)$ for $\sigma \in \cStab^\dagger(\cD)$.
\end{defn}

If one fixes $v \in H_{\rm alg}^\ast(S,\alpha,\bZ)$ and lets $\sigma \in \cStab^\dagger(\cD)$ vary, then $\cM_\sigma(v)$ is constant for $\sigma$ outside of a collection of real codimension $1$ walls. The connected components of the complement of these walls are called cells with respect to $v$, and we say that $\sigma$ is generic with respect to $v$ if it does not lie on a wall. For any coherent sheaf $E$ on $S$ and any polarization $H$ on $S$ that is generic with respect to $E$, there is a stability condition $\sigma \in \cStab^\dagger(\cD)$ that is generic with respect to $v(E)$ such that $\cM_\sigma(v(E))$ is the stack of Gieseker semistable coherent sheaves on $S$ in the numerical $K$-theory class of $E$ \cite{bridgelandK3}*{Prop.~14.1} \cite{toda_moduli}*{Sect.~6}.

In general, when $v$ is primitive, i.e., not divisible, and $v^2>0$, then for generic $\sigma$ the stack $\cM_\sigma(v)$ is a $\Gm$-gerbe over a smooth projective hyperk\"{a}hler manifold $M_\sigma(v)$ of dimension $v^2 + 2$. In fact, the minimal model program for these varieties is completely controlled by this moduli problem:
\begin{thm} \cite{bayer2014mmp}*{Thm.~1.2} \label{T:bayer_macri}
Let $S$ be a $K3$ surface, and let $\cD = \DCoh(S,\alpha)$ be the derived category of twisted coherent sheaves with respect to a cohomological Brauer class $\alpha \in {\rm Br}(S)$. If $v \in H_{\rm alg}^\ast(S,\alpha,\bZ)$ is a primitive Mukai vector and $\sigma \in \cStab^\dagger(\cD)$ is generic with respect to $v$, then any $K$-trivial birational model of $M_\sigma(v)$ is isomorphic to $M_{\sigma'}(v)$ for some $\sigma' \in \cStab^\dagger(\cD)$ lying in a unique cell of the chamber decomposition with respect to $v$.
\end{thm}

In order to apply \Cref{T:intrinsic_GIT}, we will also need the following result for arbitrary $v$ and $\sigma$:
\begin{thm} \label{thm:good_moduli} \cite{AHLH}*{Thm.~7.25}
The stack $\cM_{\sigma}(v)$ admits a proper good moduli space, which we denote $M_\sigma(v)$, for any $\sigma \in \cStab^\dagger(\cD)$ and any $v \in H_{\rm alg}^\ast(S,\alpha,\bZ)$.
\end{thm}

For any non-zero family of complexes $E \in \DCoh(T \times S,\alpha)$ over a base $T$, there is a canonical sub-group-scheme $(\Gm)_T \hookrightarrow \Aut_T(E)$ of automorphisms that act by scaling. This defines a sub-group-sheaf
\[
(\Gm)_{\cM_\sigma(v)} \hookrightarrow I_{\cM_\sigma(v)},
\]
where the latter denotes the inertia stack. We can use this to construct a stack $\cM^{\rm rig}_{\sigma}(v) := \cM_\sigma(v) \thickslash \Gm$, called the \emph{rigidification}, such that $\cM_\sigma(v) \to \cM^{\rm rig}_{\sigma}(v)$ is a $\Gm$-gerbe, and the kernel of the homomorphism $I_{\cM_\sigma(v)} \to I_{\cM^{\rm rig}_\sigma(v)}$ is the subgroup of scaling automorphisms (see \cite{abramovich2008tame} for a description of this construction). It is immediate that $M_\sigma(v)$ is also the good moduli space of $\cM_\sigma^{\rm rig}(v)$.

\subsubsection{Variation of stability}

\Cref{T:bayer_macri} implies that if $v$ is primitive and $\sigma$ is generic with respect to $v$, then one can compare $M_\sigma(v)$ to any one of its $K$-trivial birational models $X$ as follows: Choose a generic $\sigma_1$ in the cell of $\cStab^\dagger(\cD)$ corresponding to $X$, then choose a generic path $\sigma_t \in \cStab^\dagger(\cD)$ for $t \in [0,1]$ from $\sigma = \sigma_0$ to $\sigma_1$. As $t$ varies, $M_{\sigma_t}(v)$ will undergo a finite sequence of modifications at critical values $t_1,t_2,\ldots,t_N \in (0,1)$, corresponding to stability conditions for which there exist semistable objects of class $v$ that are not stable.

\begin{rem}
The substacks $\cM_{\sigma_{t_{i}-\epsilon}}(v)$ and $\cM_{\sigma_{t_i+\epsilon}}(v)$ do not necessarily have any common points. This happens when $\cM_{\sigma_{t_i}}(v)$ has multiple irreducible components, one of which contains semistable objects for $t>t_i$ and one of which contains semistable objects for $t<t_i$. This fact complicates the story in \cite{bayer2014mmp} slightly, because for such wall-crossings one must use a canonical derived (anti-)autoequivalence $\Phi: \cD \to \cD$ to show that $M_{\sigma_{t_i-\epsilon}}(v)$ and $M_{\sigma_{t_i+\epsilon}}(v)$ are birationally equivalent. The results and methods of this paper, however, do not actually need $\cM_{\sigma_{t_{i} \pm \epsilon}}(v)$ to have a point in common, so we will not use this technique.
\end{rem}

For any $\sigma \in \cStab^\dagger(\cD)$ and $v \in H^\ast_{\rm alg}(S,\alpha,\bZ)$, let $E_{\rm un} \in \DCoh(\cM_{\sigma}(v) \times S, \alpha)$ denote the universal complex, and let $p_1 : \cM_{\sigma}(v) \times S \to \cM_\sigma(v)$ and $p_2 : \cM_\sigma(v) \times S \to S$ denote the projections. Then the Mukai homomorphism is the map
\[
\Phi(-) := \det \left((p_1)_\ast (p_2^\ast(-)^\dual \otimes E_{\rm un}) \right)^\dual : K_0^{\rm num}(\cD) \to \NS(\cM_\sigma(v)).
\]
The weight of the action of the central $\Gm$ on $\Phi(E)$ is given by the Mukai pairing $-(v(E),v)$, so using $v(-)$ to identify $K_0^{\rm num}(\cD) \cong H_{\rm alg}^\ast(S,\alpha,\bZ)$, $\Phi$ restricts to a homomorphism
\begin{equation} \label{E:reduced_mukai}
\Phi :  v^\perp \subset H_{\rm alg}^\ast(S,\alpha,\bZ)  \to \NS(\cM_{\sigma}^{\rm{rig}}(v)),
\end{equation}
where $v^\perp$ is the subgroup orthogonal to $v$ under the Mukai pairing.

Now let $\Omega_\sigma \in H_{\rm alg}^\ast(S,\alpha,\bZ) \otimes \bC$ be the unique class such that $Z_\sigma(-) = (\Omega_\sigma,-)$. Then we consider the class
\[
\ell_\sigma := \Phi(\Im(\frac{-1}{Z_\sigma(v)} \Omega_\sigma)) \in \NS(\cM^{\rm rig}_\sigma(v))_{\bR}.
\]
In \cite{BMprojective}*{Thm.~4.1, Rem.~4.6} it is shown that if $\sigma$ is generic with respect to $v$, then $\ell_\sigma$ is ample on $M_\sigma(v)$, and the map $\sigma \mapsto \ell_\sigma$ identifies the ample cone of $M_\sigma(v)$ with the cell in $\cStab^\dagger(\cD)$ that contains $v$ \cite{bayer2014mmp}. At the time of this writing, it is not known if $\ell_\sigma$ is ample for $\sigma$ on a wall. We can also define a class $b \in H^4(\cM_\sigma(v);\bQ)$ by the formula
\[
b_\sigma := 2 \Im \left({\rm ch}_2 \left( (p_2)_\ast (E_{\rm un} \otimes p_1^\ast(\frac{i}{Z_\sigma(v)}\Omega_\sigma^\dual)) \right) \right).
\]

Let us define the rank and degree homomorphisms $K_0^{\rm num} \to \bZ$ corresponding to $\sigma'$ and $v$ by the formula 
\[
\frac{Z_{\sigma'}(-)}{Z_{\sigma'}(v)} = \rank_{\sigma',v}(-)+i \deg_{\sigma',v}(-).
\]
Then for any map $f : \Theta_{k'} \to \cM_\sigma(v)$ corresponding to a filtration $0 \neq E_p \subsetneq \cdots \subsetneq E_0$ with weights $w_0 < \cdots < w_p$ in $\bZ$, we have $q^{-1} f^\ast(\ell_{\sigma'}) = - \sum_j w_j \deg_{\sigma',v}(E_j/E_{j+1})$ and $q^{-2} f^\ast(b_{\sigma'}) = \sum_i w_j^2 \rank_{\sigma',v}(E_j/E_{j+1})$, and thus we compute the numerical invariant associated to $\ell_{\sigma'}$ and $b_{\sigma'}$
\begin{equation} \label{E:numerical_invariant_Bridgeland}
\mu_{\sigma'}(f) = \frac{\sum w_j \deg_{\sigma',v}(E_j/E_{j+1})}{\sqrt{\sum w_j^2 \rank_{\sigma',v}(E_j/E_{j+1})}}.
\end{equation}
The main fact we use is the following:

\begin{prop} \label{P:bridgeland_compare_1}
For any $\sigma \in \cStab^\dagger(\cD)$, there is small open neighborhood $U$ of $\sigma$ such that $\forall \sigma' \in U$:
\begin{enumerate}
\item $\cM_{\sigma'}(v) \subset \cM_\sigma(v)$ is the substack of $\Theta$-semistable points with respect to $\ell_{\sigma'}$, in the sense of \Cref{D:theta_semistable}; and
\item if $\Omega_{\sigma'} \in H_{\rm alg}^\ast(S,\alpha,\bZ) \otimes \bC$ is rational, then the HN filtration of any $\sigma'$-unstable point $[E] \in \cM_\sigma(v)$ with respect $\mu_{\sigma'}$ is the $\sigma'$ Harder-Narasimhan filtration $E_p \subsetneq \cdots \subsetneq E_0 = E$ with weights $w_0 < \cdots < w_p$ proportional to the slopes, i.e., for some $c>0$
\[
w_j = c \deg_{\sigma',v}(E_j/E_{j+1}) / \rank_{\sigma',v}(E_j/E_{j+1}).
\]
\end{enumerate}
\end{prop}

\begin{proof}
When $\alpha = 0$, this is \cite{halpern2014structure}*{Thm.~6.0.2.11} applied to the stability condition obtained from $\sigma$ by rescaling so that $Z_\sigma(v)=i$. The argument in the twisted case is identical, and follows solely from the formula \eqref{E:numerical_invariant_Bridgeland}. There is a slight difference in notation, because the class $\omega_Z \in K_0^{\rm num}(\DCoh(S))$ of \cite{halpern2014structure}*{Sect.~6.0.2} is the linear dual of our $\Omega_Z$, but with this change, the classes $\ell_{\sigma'}$ and $b_{\sigma'}$ agree with those of \cite{halpern2014structure}*{Sect.~6.0.2}.
\end{proof}

We now extend our discussion to $\cM_\sigma^{\rm rig}(v)$. Any $\bZ^n$-graded point of $\cM^{\rm rig}_\sigma(v)$ lifts, after passing to a field extension, to a $\bZ^n$-graded point of $\cM_\sigma(v)$, which corresponds to an $E \in \cM_\sigma(v)$ along with a direct sum decomposition $E = \bigoplus_{\chi \in \Lambda} E_\chi$, where
\[
\Lambda = \Hom_{\rm gp}(\Gm^n,\Gm)
\]
denotes the character lattice of $\Gm^n$. Furthermore, two $\bZ^n$-graded points of $\cM_\sigma(v)$, corresponding to $E = \bigoplus E_\chi = \bigoplus E'_\chi$, give the same $\bZ^n$-graded point of $\cM_\sigma^{\rm rig}(v)$ if and only if $E'_\chi = E_{\chi+c}$ for some $c \in \Lambda$. A $\bZ^n$-graded point of $\cM_\sigma(v)$ gives a non-degenerate $\bZ^n$-graded point of $\cM_\sigma^{\rm rig}(v)$ if the composition
\[
(\Gm^n)_{k'} \to \Aut(E) \to \Aut(E) / (\Gm)_{\rm scaling}
\]
has finite kernel. In other words, there is no non-trivial cocharacter of $\Gm^n$ that acts via scaling on $E$. This is equivalent to saying the subset $\{\chi \in \Lambda | E_\chi \neq 0\}$ is not contained in any hyperplane (allowing shifted hyperplanes).

Let $\gamma : (B\Gm^n)_{k'} \to \cM_\sigma(v)$ be a lift of a non-degenerate $\bZ^n$-graded point of $\cM_\sigma^{\rm rig}(v)$. Regarding $\chi \in \Lambda$ as a linear function on $\Lambda_\bQ^\dual$, we have
\[
\gamma^\ast(b_{\sigma'}) = \sum_{\chi \in \Lambda} \chi^2 \rank_{\sigma',v}(E_\chi)
\]
is a positive definite rational quadratic form on $\Lambda_\bQ^\dual$, making explicit the norm on graded points associated to $b_{\sigma'} \in H^4(\cM_\sigma(v);\bQ)$. Note that this quadratic form is \emph{not} preserved by the translation $E_\chi \mapsto E_{\chi+c}$, so this does not give a well-defined norm for graded points of $\cM_\sigma^{\rm rig}(v)$. On the other hand we can define the function on $\bR^n$
\[
\prescript{\rm rig}{}{|\!|x|\!|}^2_\gamma := \min_{c \in \Lambda_\bR} \left( \sum_{\chi \in \Lambda} (\chi(x)+c(x))^2 \rank_{\sigma',v}(E_\chi) \right)
\]
This minimum occurs when $c(x) = - \sum \chi(x) \rank_{\sigma',v}(E_\chi)$, hence one can compute
\begin{equation} \label{E:reduced_norm}
\prescript{\rm rig}{}{|\!|x|\!|}^2_\gamma= \sum_{\chi \in \Lambda} (\chi - \sum_{\eta \in \Lambda} \eta r_\eta)^2 r_\chi,
\end{equation}
where $r_\chi := \rank_{\sigma',v}(E_\chi)$. If the set $\{\chi \in \Lambda | E_\chi \neq 0\}$ is not contained in a shifted hyperplane, then the same is true for the shifted characters $\chi + c$, and $\prescript{\rm rig}{}{|\!|x|\!|}^2_\gamma$ is a positive definite rational quadratic form. Now the construction is invariant under the translation $E_\chi \mapsto E_{\chi + c}$, so by the discussion above, $\prescript{\rm rig}{}{|\!|x|\!|}^2_\gamma$ gives a well-defined norm on graded points of $\cM^{\rm rig}_\sigma(v)$.
\begin{prop} \label{P:bridgeland_compare_2}
In the context of \Cref{P:bridgeland_compare_1}, if $\Omega_{\sigma'}$ is rational, the numerical invariant on $\cM_\sigma^{\rm rig}(v)$ associated to $\ell_{\sigma'}$ and the norm on graded points \eqref{E:reduced_norm},
\[
\mu^{\rm rig}_{\sigma'}(f) = - \frac{q^{-1} f^\ast(\ell_{\sigma'})}{\prescript{\rm rig}{}{|\!|1|\!|}^2_{\ev_0(f)}},
\]
defines a $\Theta$-stratification of $\cM^{\rm rig}_\sigma(v)$. The map
\[
\Filt(\cM_\sigma(v)) \to \Filt(\cM_\sigma^{\rm rig}(v))
\]
maps the $\Theta$-strata of \Cref{P:bridgeland_compare_1} to the $\Theta$-strata on $\cM_\sigma^{\rm rig}(v)$, and it induces a bijection between strata of $\cM_\sigma(v)$ and $\cM_\sigma^{\rm rig}(v)$.
\end{prop}
\begin{proof}
Because \eqref{E:reduced_norm} is a norm that is defined by a positive definite rational quadratic form, the numerical invariant associated to $\ell_{\sigma'}$ and $|\!|-|\!|^2$ satisfies condition (R) of \cite{halpern2014structure}, so it induces a $\Theta$-stratification of $\cM_{\sigma}(v)$ by \cite{halpern2014structure}*{Thm.~5.0.1.7}.

Composition with $\cM_\sigma(v) \to \cM_\sigma^{\rm rig}(v)$ induces a bijection between filtrations in $\cM_\sigma^{\rm rig}(v)$ over an algebraically closed field $k'$ and equivalence classes filtered points $\cdots E_{w+1} \subset E_w \subset \cdots $ of $\cM_\sigma(v)$ over $k'$, where we say that two filtrations are equivalent, $f \sim f'$, if they differ by a shift of weights, $E'_w = E_{w+c}$ for some $c \in \bZ$. The value of $q^{-1} f^\ast(\ell_{\sigma'}) \in \bQ$ does not change under this equivalence relation, because $\ell_{\sigma'}$ descends to $\cM_\sigma^{\rm rig}(v)$, and it follows that a point in $\cM_\sigma(v)$ is semistable if and only if its image in $\cM^{\rm rig}_\sigma(v)$ is.

If $f : \Theta_{k'} \to \cM_\sigma(v)$ is a filtration such that $q^{-1} f^\ast(\ell_{\sigma'}) <0$, i.e., $f$ is destabilizing, then by the definition of $\prescript{\rm rig}{}{|\!|x|\!|}^2_\gamma$, we have
\[
\mu_{\sigma'}(f)  \leq \max \left\{ \mu_{\sigma'}(f') \left| f' \sim f \right. \right\} = \mu^{\rm rig}_{\sigma'}(f)
\]

This implies that any HN filtration in $\cM_\sigma(v)$ maps to an HN filtration in $\cM^{\rm rig}_\sigma(v)$, and the uniqueness of HN filtrations implies that any HN filtration in $\cM_\sigma^{\rm rig}(v)$ is the image of one in $\cM_\sigma(v)$. The bijection between strata is a consequence of the smoothness of $\Filt(\cM_\sigma(v)) \to \Filt(\cM_\sigma^{\rm rig}(v))$ \cite{halpern2014structure}*{Sect.~1.0.2}.
\end{proof}

\subsubsection{Derived structure on the rigidified moduli stacks}
\label{S:derived_moduli}

Note that although \Cref{D:moduli_complexes} only defines $\cM_\sigma(v)$ as a classical stack, the definition can be interpreted verbatim to define a derived stack. For the remainder of this section, we will let $\cM_\sigma(v)$ denote this derived stack - see \cite{toen_vaquie} for a more thorough discussion of this stack in the derived setting.

The fiber of the cotangent complex at a point $[E] \in \cM_\sigma(v)$, corresponding to a complex $E \in \DCoh(S,\alpha)$, is \cite{toen_vaquie}*{Cor.~3.17}
\[
\bL_{\cM_\sigma(v)}|_{[E]} \simeq \RHom(E,E[1])^\dual.
\]
More generally, given a morphism $\xi : T \to \cM_\sigma(v)$, corresponding to a complex $E \in \DCoh(T \times S,\alpha)$, we have
\[
\xi^\ast(\bL_{\cM_\sigma(v)}) \simeq \pi_\ast(\inner{\RHom}^\otimes_{T \times S} (E,E[1]))^\dual,
\]
where $\pi : T \times S \to T$ is the projection. From this we see that Serre duality on $S$ gives an isomorphism $\phi : \bL_{\cM_\sigma(v)} \cong (\bL_{\cM_\sigma(v)})^\dual$.

The action of $\Gm$ on any family by scaling defines a map
\[
\alpha : \bL_{\cM_\sigma(v)} \to \cO_{\cM_\sigma(v)}[-1]
\]
which on the fiber at a point $[E] \in \cM_\sigma(v)$ is dual to the identity map $k[1] \to \RHom_S(E,E[1])$. There is also a homomorphism $\tau : \cO_{\cM_\sigma(v)}[-1] \to \bL_{\cM_\sigma(v)}$ that on fibers is dual to the trace map
\[
\RHom_S(E,E[1]) \cong R\Gamma(S,E\otimes E^\dual [1]) \xrightarrow{\ev : E \otimes E^\dual \to \cO_S} R\Gamma(S,\cO_S[1]) \to k[1],
\]
where the last map is induced by a fixed choice of splitting $R\Gamma(S,\cO_S) \cong k \oplus k[-2]$. The composition $\alpha \circ \tau$ is just multiplication by the rank $r \geq 0$ of $E$, so if $E$ has rank $>0$, $\tau$ is a splitting for $\alpha$. Likewise, Serre duality gives maps
\[
\alpha^\dual : \cO_{\cM_\sigma(v)}[1] \to \bL_{\cM_\sigma(v)}, \quad \tau^\dual : \bL_{\cM_\sigma(v)} \to \cO_{\cM_\sigma(v)}[1]
\]
that define a trivial summand of $\bL_{\cM_\sigma(v)}$ in homological degree $1$ when $r>0$. In particular, when $v$ has positive rank we have a decomposition
\[
\bL_{\cM_\sigma(v)} \simeq \cO_{\cM_\sigma(v)}[1] \oplus (\bL_{\cM_\sigma(v)})^{\rm{red}} \oplus \cO_{\cM_\sigma(v)}[-1],
\]
where $(\bL_{\cM_\sigma(v)})^{\rm{red}}$ is again self-dual. Note that the derived stack $\cM_\sigma(v)$ is not smooth at any point.

\begin{prop} \label{P:derived_reduction}
For any $v \in H^\ast_{\rm}(S,\alpha,\bZ)$ and any $\sigma \in \cStab^\dagger(\cD)$, there are derived algebraic stacks $\cM_\sigma(v) \thickslash \Gm$ and $\cM_\sigma^{\rm{red}}(v)$ with maps
\[
\xymatrix{ \cM_\sigma(v) \ar[rr]^-{\Gm \txt{-gerbe}}_-{q} & & \cM_\sigma(v) \thickslash \Gm & & & \cM_\sigma^{\rm{red}}(v) \ar[lll]^p_-{\txt{surjective\\closed immersion}}}
\]
along with direct sum decompositions
\[
\begin{array}{c}
\bL_{\cM_\sigma(v)} \cong \cO_{\cM_\sigma(v)}[-1] \oplus q^\ast \bL_{\cM_\sigma(v) \thickslash \Gm} \\
p^\ast \bL_{\cM_\sigma(v) \thickslash \Gm} \cong \bL_{\cM^{\rm red}_\sigma(v)} \oplus \cO_{\cM_\sigma(v)}[1]
\end{array}
\]
such that the inclusion of the summand $q^\ast \bL_{\cM_\sigma(v) \thickslash \Gm}$ and the projection onto the summand $\bL_{\cM^{\rm red}_\sigma(v)}$ are the canonical derivative maps of $q$ and $p$ respectively. On underlying classical stacks,
\[
(\cM_\sigma(v) \thickslash \Gm)^{\rm cl} \cong \cM_\sigma^{\rm rig}(v),
\]
and the morphism $q$ induces the canonical map $\cM_\sigma(v)^{\rm cl} \to \cM_\sigma^{\rm rig}(v)$ of the previous section.
\end{prop}

We prove this proposition after some preliminary results. We will use the following, whose proof was suggested by Arend Bayer, to reduce to the positive rank case:
\begin{lem} \label{L:fourier_transform}
Let $(S,\alpha)$ be a twisted $K3$ surface, and let $v \in H^\ast_{\rm alg}(S,\alpha,\bZ)$ be a Mukai vector of rank $0$. Then there is another twisted $K3$ surface $(S',\alpha')$ and a Fourier-Mukai equivalence $\Phi : \DCoh(S,\alpha) \to \DCoh(S',\alpha')$ such that $\Phi(v)$ has positive rank.
\end{lem}
\begin{proof}
Let us use the notation $v = (r,c,s)$ for a Mukai vector, where $r \in H^0(S;\bZ) \cong \bZ$, $c \in H^2(S;\bZ)$, and $s \in H^4(S;\bZ) \cong \bZ$. The Mukai pairing is then given by the formula $((r,c,s),(r',c',s')) = c \cdot c' - rs'-r's$. Let $L$ be an invertible sheaf on $S$ with Chern class $H$. Note that, as observed in the proof of \cite{BZ}*{Thm.~1.2}, $L^m \otimes(-)$ gives an equivalence $\DCoh(S,\alpha) \to \DCoh(S,\alpha)$ that maps objects with Mukai vector $(r,c,s)$ to those with Mukai vector
\[
e^{mH} v = (r,c+rmH,s+m (c \cdot H)+\frac{1}{2} r m^2 H^2).
\]
Therefore, if $L$ is ample and $m \gg 0$, the equivalence $L^m \otimes(-)$ maps the Mukai vector $v$ to a vector whose degree $4$ component is positive.

Let $R$ be the index of the class $\alpha$, and let $w = (R,0,0)$. Then $S' = \cM^{\rm rig}_\sigma(w)$ is a twisted $K3$ surface, with Brauer class $\alpha'$ on $S'$ determined by the $\Gm$-gerbe $\cM_\sigma(w) \to \cM^{\rm rig}_\sigma(w)$. There is a Fourier-Mukai equivalence $\Psi : \DCoh(S',\alpha') \to \DCoh(S,\alpha)$ that maps point sheaves to twisted sheaves with Mukai vector $w$. Any derived equivalence preserves the Mukai pairing, so for any Mukai vector $v = (r,c,s)$ on $S$, we have
\[
\rank(\Psi^{-1}_\ast(v)) = -(\Psi^{-1}_\ast(v),(0,0,1)) = -(v,(R,0,0)) = R s.
\]
In particular, the composition
\[
\DCoh(S,\alpha) \xrightarrow{L^m \otimes(-)} \DCoh(S,\alpha) \xrightarrow{\Psi^{-1}} \DCoh(S',\alpha')
\]
maps complexes of class $v$ to complexes of positive rank for $m \gg 0$.
\end{proof}

For the following construction, let $\cX$ be a derived algebraic stack with a weak action of the monoid $B\Gm$ 
(see \Cref{D:weak_theta_action}), and let $a : B\Gm \times \cX \to \cX$ be the action map. The map $a$ induces a map of group sheaves $(\Gm)_{B\Gm \times \cX} \subset I_{B\Gm \times \cX} \to a^{-1} (I_{\cX})$ over $B\Gm \times \cX$. Restricting this map further along the identity section $\cX \to B\Gm \times \cX$ gives a map of group sheaves $(\Gm)_\cX \to I_{\cX}$ over $\cX$.

\begin{lem} \label{L:derived_rigidification}
Let $\cX$ be a derived algebraic stack with a weak action of the monoid $B\Gm$ such that the induced map of group sheaves $(\Gm)_\cX \to I_{\cX}$ is a closed immersion. Then there exists a derived algebraic stack $\cX \thickslash \Gm$ and a morphism $q : \cX \to \cX \thickslash \Gm$ that is a $\Gm$-gerbe such that the map $\bL_\cX \to \bL_{\cX/ (\cX \thickslash \Gm)}$ is isomorphic to the map $\bL_{\cX} \to \cO_{\cX}[-1]$ induced by the embedding $(\Gm)_\cX \hookrightarrow I_\cX$, and $\cX^{\rm cl} \to (\cX \thickslash \Gm)^{\rm cl}$ is the classical rigidification.
\end{lem}
\begin{proof}
Recall from \Cref{L:center_grading} that a weak $B\Gm$-action on any stack induces a weight decomposition into a direct sum
\[
\QC(\cX) \cong \bigoplus_{w \in \bZ} \QC(\cX)^w,
\]
where $F \in \QC(\cX)^w$ if and only if $a^\ast(F) \in \QC(B\Gm \times \cX)$ lies in weight $w$ with respect to the $B\Gm$ factor. For any extension field $k'/k$ and $p\in \cX(k')$, the $B\Gm$-action defines a homomorphism of $k'$-groups $(\Gm)_{k'} \to \Aut(p)$, and Nakayama's lemma implies that $F \in \APerf(\cX)$ lies in the weight $w$ summand if and only if $(\Gm)_{k'}$ acts with weight $w$ on the fiber $F_p$ for every $p \in \cX(k')$.

For any $n \geq 0$, let $\cX_n = \tau_{\leq n} \cX \hookrightarrow \cX$ be the $n^{th}$ truncation, e.g., $\cX_0 = \cX^{\rm cl}$. Then $B\Gm \times \cX_n$ is $n$-truncated as well, so the action map restricted to $\cX_n$ factors uniquely through a map $B\Gm \times \cX_n \to \cX_n$ that defines a weak action of $B\Gm$ on $\cX_n$. The inclusion $\cX_n \hookrightarrow \cX$ is equivariant with respect to this $B\Gm$-action, and hence the pullback functor along this map is compatible with the weight decompositions of $\QC(-)$ induced by these actions.

Now consider the classical $\Gm$-gerbe $\cX^{\cl} \to \cX^{\rm cl}\thickslash \Gm$. The pullback functor $\APerf(\cX^{\rm cl}\thickslash \Gm) \to \APerf(\cX^{\cl})$ is $t$-exact and fully faithful. A complex in $\APerf(\cX^{\cl})$ descends to $\APerf(\cX^{\rm cl} \thickslash \Gm)$ if and only if $\Gm$ acts with weight $0$ in every fiber, so pullback defines an equivalence
\[
\APerf(\cX^{\rm cl}\thickslash \Gm) \cong \APerf(\cX^{\rm cl})^0.
\]

We now claim that one can inductively construct a sequence of square-zero extensions of derived algebraic stacks $\cX'_0 \hookrightarrow \cX'_1 \hookrightarrow \cX'_2 \hookrightarrow \cdots$ fitting into a cartesian diagram
\begin{equation} \label{E:gerbe_ext_diagram}
\xymatrix{\cX_n \ar@{^{(}->}[r] \ar[d]^{q_n} & \cX_{n+1} \ar[d]^{q_{n+1}} \\
\cX'_n \ar@{^{(}->}[r] & \cX'_{n+1} },
\end{equation}
starting with $\cX'_0 = \cX^{\rm cl} \thickslash \Gm$ and $q_0$ the canonical map $\cX^{\rm cl} \to \cX^{\rm cl} \thickslash \Gm$.

The map $\cX_n \hookrightarrow \cX_{n+1}$ is an extension by the pushforward of $M_{n+1} := H_{n+1}(\cO_{\cX})[n] \in \QC(\cX_0)$, and is thus classified by a class
\[
\eta_n \in \Hom_{\cX_0} \left(\bL_{\cX_n}|_{\cX_0}, M_{n+1}[1] \right)
\]
One can use the fact that $a : B\Gm \times \cX \to \cX$ is smooth (hence flat) to show that $H_\ast(\cO_{\cX}) \in \QC(\cX_0)^0$, and thus $M_{n+1}$ descends uniquely to an object in $\QC(\cX'_0)$, which we denote $M'_{n+1}$. It follows that the image of $\eta_n$ under the homomorphism
\[
\Hom_{\cX_0} \left(\bL_{\cX_n}|_{\cX_0}, M_{n+1}[1] \right) \to \Hom_{\cX_0} \left(q_n^\ast(\bL_{\cX'_n})|_{\cX_0}, M_{n+1}[1] \right)
\]
is the image under $q_n^\ast$ of a unique class $\eta'_n \in \Hom_{\cX'_0}(\bL_{\cX'_n}|_{\cX'_0},M'_{n+1}[1])$. This class defines an extension $\cX'_n \to \cX'_{n+1}$ that fits into the cartesian diagram \eqref{E:gerbe_ext_diagram}.

Note that $\cX \cong \colim_n \cX_n$. Now that we have constructed the diagram \eqref{E:gerbe_ext_diagram} for all $n$, we define $\cX \thickslash \Gm = \colim_n \cX'_n$.

Using deformation theory, as in \cite{HAGII}*{Lem.~C.0.11}, one can show that given an \'etale cover $U \to \cX'_0$ and an isomorphism $U \times_{\cX'_0} \cX_0 \cong (B\Gm)_U$ over $U$, one can extend this data to an \'etale cover $U' \to \cX'$ and an isomorphism $U' \times_{\cX'} \cX \cong (B\Gm)_{\cX'}$. Also, at every level, the map $\bL_{\cX_n} \to \bL_{\cX_n/\cX'_n}$ can be identified with the map $\bL_{\cX_n} \to \cO_{\cX_n}[-1]$ induced by the embedding $(\Gm)_{\cX_n} \hookrightarrow I_{\cX_n}$, and thus the same holds for the colimit.
\end{proof}

\begin{proof}[Proof of \Cref{P:derived_reduction}]
Consider the map $B\Gm \times \cM_\sigma(v) \to \cM_\sigma(v)$ that acts on $T$-points for any derived test scheme $T$ by $(L,E) \mapsto \pi^\ast(L) \otimes E$, where $L \in \Pic(T)$, $\pi : T \times S \to T$ is the projection, and $E \in \DCoh(T \times S,\alpha)$. This map satisfies the identity and associativity axioms up to homotopy, and thus defines a weak action of the monoid $B\Gm$ on the derived stack $\cM_\sigma(v)$. 

If $v$ does not have positive rank, use \Cref{L:fourier_transform} to construct a twisted $K3$ surface $(S',\alpha')$ and a Fourier-Mukai equivalence $\Phi : \DCoh(S,\alpha) \cong \DCoh(S',\alpha')$ such that $v':=\Phi_\ast(v) \in H^\ast_{\rm alg}(S',\alpha',\bZ)$ has positive rank. For any derived $k$-scheme $T$, $\Phi$ extends to an equivalence $\Phi_T : \DCoh(T \times S,\alpha) \to \DCoh(T\times S',\alpha')$ that maps families of $\sigma$-semistable complexes to families of $\Phi(\sigma)$-semistable complexes. Furthermore, we have a natural isomorphism $\Phi_T(\pi^\ast(L) \otimes E) \cong \pi^\ast(L) \otimes \Phi_T(E)$ for $L \in \Pic(T)$ and $E \in \DCoh(T\times S,\alpha')$. Therefore $\Phi$ defines an isomorphism of stacks $\cM_\sigma(v) \cong \cM_{\Phi(\sigma)}(v')$ that is equivariant for the canonical weak $B\Gm$-action on each. We now replace $(S,\alpha)$ with $(S',\alpha')$, $\sigma$ with $\Phi_\ast(\sigma)$, and $v$ with $v'$, so we will assume for the remainder of the proof that $v$ has positive rank.

We use \Cref{L:derived_rigidification} to construct the algebraic derived stack $\cM_\sigma(v) \thickslash \Gm$ and the morphism $q : \cM_\sigma(v) \to \cM_\sigma(v) \thickslash \Gm$ satisfying the desired properties. In particular, we have an exact triangle
\[
q^\ast \bL_{\cM_\sigma(v) \thickslash \Gm} \to \bL_{\cM_\sigma(v)} \to \cO_{\cM_\sigma(v)}[-1] \to,
\]
that is split by the trace map $\tau$, which we defined in the discussion leading up to \Cref{P:derived_reduction}, because $v$ has positive rank.

To define $\cM_\sigma^{\rm red}(v)$, we will use a trick to remove this extraneous summand in the cotangent complex by considering the determinant map (see \cite{schurg2015derived})
\[
\det : \cM_\sigma(v) \to \Pic(S)_c
\]
where $\Pic(S)_c$ is the derived Picard stack of $S$ of invertible sheaves of numerical class $c \in \NS(S)$. Because $H^1(S,\cO_S) = 0$, we have $\Pic(S)_c \simeq \Spec(k[\epsilon[1]]) / \Gm$, where $\Gm$ acts trivially on $k[\epsilon[1]]$. The induced map on cotangent complexes at a point $E \in \cM_\sigma(v)$ is dual to the evaluation map
\[
\xymatrix{ \RHom_S(\det(E),\det(E)[1])^\ast \ar@{}|-{\rotatebox[origin=c]{90}{$\cong$}}[d] & \RHom_S(E,E[1])^\ast \ar@{}|-{\rotatebox[origin=c]{90}{$\cong$}}[d] \\
R\Gamma(S, \cO_S[1])^\ast \ar[r]^-{\ev: E \otimes E^\dual \to \cO_S} & R\Gamma(S, E \otimes E^\dual[1])^\ast }.
\]
As in the discussion leading up to \Cref{P:derived_reduction}, let $\tau_E \in H_{-1}(\RHom_S(E,E[1])^\ast)$ be the class that is linearly dual to the trace map, and let $\alpha_E^\dual \in H_{1}(\RHom_S(E,E[1])^\ast)$ be the class that is Serre dual to the identity map. Then the differential of the map $\det$ maps $\tau_E \mapsto r \tau_{\det(E)}$ and $\alpha_E^\dual \mapsto r \alpha_{\det(E)}^\dual$.

The determinant map descends to a map
\[
\det : \cM_\sigma(v) \thickslash \Gm \to \Pic(S)_c \thickslash \Gm = \Spec(k[\epsilon[1]]).
\]
The cotangent complex of the latter is free on the generator $\epsilon = \alpha^\dual_{\det(E)}$ in homological degree $1$, for any $E \in \cM_\sigma(v)$. Therefore, if we define
\[
\cM^{\rm{red}}_\sigma(v) := (\cM_\sigma(v) \thickslash \Gm) \times_{\Spec(k[\epsilon[1]])} \Spec(k),
\]
we have an exact triangle
\[
\cO_{\cM_\sigma^{\rm red}(v)}[1] \xrightarrow{p^\ast(\alpha^\dual)} p^\ast \bL_{\cM_\sigma(v) \thickslash \Gm} \to \bL_{\cM_\sigma^{\rm red}(v)} \to.
\]
Because the rank of $v$ is positive, the is extension is split by $p^\ast(\tau^\dual)$.

\end{proof}

\begin{rem}
In \cite{schurg2015derived}, the reduced stack is defined as $\cM_\sigma(v) \times_{\Pic(S)_c} \Spec(k)$, where $\Spec(k) \to \Pic(S)_c$ is the unique point. If $v$ has rank $r>0$, the resulting stack has the reduced cotangent complex $(\bL_{\cM_\sigma(v)})^{\rm red}$, but if $r>1$ the generic point will still have a canonical action of the multiplicative group $\mu_r$, so this does not fully reduce the gerbe.
\end{rem}


\subsection{\texorpdfstring{$D$}{D}-equivalence conjecture for moduli of sheaves on a \texorpdfstring{$K3$}{K3}}

We now combine the results of the previous sections to the $D$-equivalence conjecture. As before, we let $(S,\alpha)$ be a twisted $K3$ surface, and we let $\cD = \DCoh(S,\alpha)$ denote the derived category of complexes of twisted coherent sheaves.

In \Cref{P:derived_reduction} we have constructed a derived stack $\cM_\sigma^{\rm red}(v)$, parameterizing $\sigma$-semistable complexes with Mukai vector $v$, for any $\sigma \in \cStab^\dagger(\cD)$. The underlying classical stack of $\cM^{\rm red}_{\sigma}(v)$ is the rigidification of the usual stack of $\sigma$-semistable complexes \cite{toda_moduli}, and it admits a good moduli space $M_\sigma(v)$. When $\sigma$ is generic with respect to $v$, $M_\sigma(v) = \cM_\sigma^{\rm red}(v)$ is a smooth projective Calabi-Yau variety.

\begin{thm} \label{T:D_equivalence}
 $\sigma \in \cStab^\dagger(\cD)$ be a stability condition that is generic with respect to a primitive class $v \in H^\ast_{\rm alg}(S,\alpha,\bZ)$ with $(v,v) > 0$. Then for any smooth projective variety $X$ that is birationally equivalent to $M_\sigma(v)$ and has $K_X \cong \cO_X$, there is an equivalence of derived categories
\[
\DCoh(X) \cong \DCoh(M_\sigma(v)).
\]
\end{thm}

\begin{ex}
For any $v \in H^\ast(S,\bZ)$ that is the Mukai vector of a coherent sheaf and any ample class $H \in \NS(S)$, there is a $\sigma \in \cStab^\dagger(\cD)$ such that $\cM_\sigma(v)$ is the moduli of Gieseker $H$-semistable sheaves with Mukai vector $v$ \cite{bridgelandK3}*{Prop.~14.2}. In particular, \Cref{T:D_equivalence} establishes the D-equivalence conjecture for any birational equivalence class of Calabi-Yau manifolds that contains a smooth moduli space of coherent sheaves on a K3 surface.
\end{ex}

We will prove \Cref{T:D_equivalence} at the end of the section, after studying genericity for classes in $\NS(\cM^{\rm red}_\sigma(v))_\bR$.

A point in $\cM^{\rm red}_\sigma(v)$, which classifies a complex $E \in \cD_{k'} = \DCoh(S_{k'},\alpha)$ for some extension field $k'$ of $k$, is closed if $E$ is polystable, i.e., decomposes as
\begin{equation} \label{E:polystable}
E = \bigoplus_{i\in I} E_i \otimes V_i,
\end{equation}
where $E_i$ are non-isomorphic simple objects of $\cD_{k'}$, and the $V_i$ are certain multiplicity vector spaces. We will often assume, by extending the base field as necessary, that $E \in \cD$.

Because $\cM^{\rm red}_\sigma(v)^{\rm cl} \cong \cM_\sigma(v)^{\rm cl} \thickslash \Gm$, if $E$ is polystable we have
\[
G_E := \Aut(E) =  \left( \prod_i \GL(V_i)  \right) / (\Gm)_{\rm diag},
\]
where $(\Gm)_{\rm diag}$ acts diagonally by scaling on each $V_i$. From the discussion in \Cref{S:derived_moduli}, we compute
\[
H_0(\bL_{\cM^{\rm{red}}_{\sigma}(v),[E]}) = H_0(\bL_{\cM_{\sigma}(v),[E]}) \simeq \bigoplus_{i,j} \Hom(E_i,E_j[1])^\dual \otimes  \Hom(V_i,V_j),
\]
Serre duality provides an equivalence $\bL_{\cM^{\rm red}_\sigma(v)} \cong (\bL_{\cM^{\rm red}_\sigma(v)})^\dual$, which on the fiber above identifies $\Hom(E_i,E_j[1]) \simeq \Hom(E_j,E_i[1])^\dual$. The identification is symplectic when $i=j$, so $\Hom(E_i,E_i[1])$ has even dimension.

For $E$ as in \eqref{E:polystable}, we define the dimension vector $d_E = (\dim(V_1),\ldots,\dim(V_n)) \in \bZ^n_{>0}$, and for any $a \in v^\perp_\bR$ we define
\[
\zeta_E(a) := \left( (a,v(E_1)),\ldots,(a,v(E_n)) \right) \in \bR^n.
\]
Note that $\zeta_E(a) \cdot d_E = 0$. In the following, we say $(w_1,\ldots,w_n) \leq (v_1,\ldots,v_n)$ if $w_i \leq v_i$ for all $i$, and $w < v$ means that $w \leq v$ and $w \neq v$. 

The reduced Mukai homomorphism \eqref{E:reduced_mukai} extends via the same definition to the derived context, and extending this $\bR$-linearly gives a homomorphism
\[
\Phi : v^\perp_{\bR} \subset H^\ast_{\rm alg}(S,\alpha,\bZ) \otimes \bR \to \NS(\cM_{\sigma}^{\rm red}(v))_\bR.
\]

\begin{lem} \label{L:bridgeland_generic}
For $a \in v^\perp_\bR$, $\Phi(a) \in \NS(\cM^{\rm red}_\sigma(v))_\bR$ is generic in the sense of \Cref{defn:generic} if for any polystable $E = \bigoplus_{i=1}^n E_i \otimes V_i \in \cM^{\rm red}_\sigma(v)$ and any $w \in \bZ^n$ with $0 < w < d_E$ one has $\zeta_E(a) \cdot w \neq 0$. If $v$ is primitive, $\NS(\cM^{\rm red}_\sigma(v))_\bR$ contains a generic class. 
\end{lem}
\begin{proof}
To any polystable $E \in \cM^{\rm red}_\sigma(v)$ with decomposition into simple objects \eqref{E:polystable}, we associate a quiver $Q_E$, i.e., a directed graph, with:
\begin{itemize}
\item one vertex for each index $i$,
\item $\dim(\Hom(E_i,E_j[1]))$ arrows from vertex $i$ to vertex $j$ for $i < j$, and
\item $\dim(\Hom(E_i,E_i[1]))/2$ arrows from vertex $i$ to itself.
\end{itemize}
Our discussion above identifies
\[
H_0(\bL_{\cM^{\rm{red}}_{\sigma}(v),[E]}) \cong T^\ast \Rep_{d_E}(Q_E),
\]
where the latter denotes the cotangent space of the space of representations with dimension vector $d_E$. This is equivalent to the space of representations of the ``doubled'' quiver which for each arrow in $Q_E$ adds another arrow with the opposite orientation.

By \Cref{defn:generic}, the class $\Phi(a) \in \NS(\cM^{\rm red}_\sigma(v))_\bR$ is generic if for any finite type polystable point $E \in \cM^{\rm red}_{\sigma}(v)$ the real character $\Phi(\ell)|_{[E]}$ of $G_E$ is generic in the sense of \cite{halpern2016combinatorial}*{Sect.~2} with respect to the representation $T^\ast \Rep_{d_E}(Q_E) \oplus \fg_E$, where $\fg_E$ denotes the adjoint representation of $G_E$. As mentioned above, we replace $k$ with a finite extension as necessary so that we may assume $E$ is a $k$-point.

We identify the character lattice of $G_E$ with the set of vectors $\zeta = (\zeta_1,\ldots,\zeta_n) \in \bZ^n$ that satisfy $\zeta \cdot d_E = 0$, where $\zeta$ corresponds to the character $\det(V_1)^{\zeta_1} \otimes \cdots \otimes \det(V_n)^{\zeta_n}$. In this case, one computes from the definition that
\[
\Phi(a)_{[E]} \cong \det(R\Gamma(S,v^{-1}(a)^\dual \otimes E))^\dual,
\]
where $v^{-1}(a) \in K_0^{\rm num}(\cD)_\bR$ is the class whose mukai vector is $a$. Using \eqref{E:polystable} we compute
\[
\Phi(a)_{[E]} \cong \det(\sum_i \chi(v^{-1}(a),E_i) V_i)^\dual,
\]
which upon identifying $\chi(v^{-1}(a),E_i) = -(a,v(E_i))$ gives $\Phi(a)|_{[E]} = \zeta_E(a)$ as defined above.

The analysis of genericity for representations of the form $T^\ast \Rep_{d_E}(Q_E) \oplus \fg_E$ was carried out in \cite{halpern2016combinatorial}*{Prop.~5.6}. The analysis there applies to the space of framed representations with non-zero framing dimension vector, but the framing is only necessary so that the generic stabilizer with respect to a maximal torus is trivial. In our setting, we are considering the space of representations with framing dimension vector $0$, but because we are working with the reduced stack $\cM_\sigma^{\rm red}$, the automorphism group is quotiented by $(\Gm)_{\rm scaling}$, so the generic stabilizer for the action of a maximal torus in $G_E$ on $T^\ast \Rep_{d_E}(Q_E) \times \fg_E$ is trivial.

The rest of the proof of \cite{halpern2016combinatorial}*{Prop.~5.6} applies verbatim to show that $\zeta_E(a)$ is generic if for every $0 < w < d_E$,
\begin{equation} \label{E:bridgeland_genericity}
\zeta_E(a) \cdot w \neq 0.
\end{equation}
Let us sketch the argument: one fixes a maximal torus $T \subset G_E$ and shows that genericity is equivalent to the condition that $T$ acts with finite stabilizers on the $T$-semistable locus $(T^\ast \Rep_{d_E}(Q_E) \times \fg_E)^{T\rm{-ss}}(\zeta_E(a))$. This torus $T$ can be realized as the gauge group (modulo scaling) for a new quiver $Q_E'$ obtained by splitting each vertex $i$ into $\dim(V_i)$ vertices, and $T^\ast \Rep_{d_E}(Q_E) \oplus \fg_E$ is isomorphic as a $T$-representation to a space of representations of $Q'_E$ with dimension vector $1$ at every vertex. One can then use the results of \cite{Nakajima}*{Sect.~2} to arrive at the criterion \eqref{E:bridgeland_genericity} for when $T$ acts with finite stabilizers on the semistable locus. We refer the reader to \cite{halpern2016combinatorial}*{Prop.~5.6} for the details of this argument.

There are finitely many polystable $E \in \cM_\sigma^{\rm red}(v)$ up to algebraic equivalence, so if each of the conditions \eqref{E:bridgeland_genericity} on $a$ is non-vacuous, this provides a finite hyperplane arrangement in $v^\perp_\bR$ such that $\Phi(a)$ is generic if $a$ lies outside these hyperplanes. Thus to verify that there exists an $a$ such that $\Phi(a)$ is generic, it suffices to show that for any polystable $E$ and any $0 < w < d_E$, there is \emph{some} $a$ such that
\[
\zeta_E(a) \cdot w  = (a,v(E_1^{\oplus w_1} \oplus \cdots \oplus E_n^{\oplus w_n})) \neq 0.
\]
Because the Mukai pairing is non-degenerate, if $\zeta_E(a) \cdot w = 0$ for all $a \in v^\perp_\bR$, then $v(E_1^{w_1} \oplus \cdots \oplus E_n^{w_n}) = r v$ for some $0 < r<1$. But $v(E_1^{w_1} \oplus \cdots \oplus E_n^{w_n}) \in H_{\rm alg}^\ast(S,\alpha,\bZ)$, so this can not happen if $v$ is primitive.
\end{proof}

\begin{proof}[Proof of \Cref{T:D_equivalence}]
By \Cref{T:bayer_macri} there is a stability condition $\sigma' \in \cStab^\dagger(\cD)$ that is generic with respect to $v$ such that $X \cong M_{\sigma'}(v)$. Choose a continuous path $\sigma_t \in \cStab^\dagger(\cD)$ for $t \in [0,1]$ with $\sigma_0 = \sigma$ and $\sigma_1=\sigma'$. Using \Cref{P:bridgeland_compare_1} and \Cref{P:bridgeland_compare_2}, one can deform the path so that $\sigma_t$ is generic with respect to $v$ for all but finitely many values of $t$, $0<t_1 < \cdots < t_N < 1$, and $\cM^{\rm red}_\sigma(v)$ is constant for $t \in (0,1) \setminus \{t_1,\ldots,t_N\}$. It therefore suffices to construct an equivalence
\[
\DCoh(\cM^{\rm red}_{\sigma_{t_i-\epsilon}}(v)) \cong \DCoh(\cM^{\rm red}_{\sigma_{t_i+\epsilon}}(v))
\]
for all $i$. Letting $\cX := \cM_{\sigma_{t_i}}^{\rm red}(v)$, \Cref{P:bridgeland_compare_1} and \Cref{P:bridgeland_compare_2} identify 
\[
\cM_{\sigma_t}^{\rm red}(v) = \cX^{\rm ss}(\ell_{\sigma_t})
\]
for all $t$ close to $t_i$, where the latter denotes the semistable locus in the sense of $\Theta$-stability associated to the class $\ell_{\sigma_t} \in \NS(\cM_{\sigma_{t_i}}^{\rm red}(v))$.

For any $\delta \in \NS(\cX)_\bR$, one can choose a generic $\beta \in \NS(\cX)_\bR$ by \Cref{L:bridgeland_generic}, and \Cref{L:hyperplane_arrangement} implies that $\delta+\epsilon \beta$ will be lattice generic for $\epsilon$ sufficiently small. So fix a choice of lattice generic $\delta \in \NS(\cX)_\bR$. \Cref{T:magic_windows} and \Cref{R:generic_perturbation} imply that restriction defines equivalences
\[
\xymatrix{\DCoh(\cX^{\rm ss}(\ell_{\sigma_{t_i-\epsilon}})) & \fW_{\cX}(\delta) \ar[r]^-{\res}_-{\cong} \ar[l]_-{\res}^-{\cong} & \DCoh(\cX^{\rm ss}(\ell_{\sigma_{t_i+\epsilon}}))}
\]
for all $0<\epsilon\ll 1$, where $\fW_\cX(\delta) \subset \DCoh(\cM^{\rm red}_{\sigma_{t_i}}(v))$ is the subcategory introduced in \Cref{D:magic_windows}. This establishes the desired equivalence.
\end{proof}

\appendix

\section{Derived deformation to the normal cone}

We construct a version of deformation to the normal cone for an almost finitely presented closed immersion of algebraic derived stacks. This was established in \cite{GR2}*{Chap.~II.9} for a much larger class of stacks over a field of characteristic $0$, and morphisms which need not be closed immersions. A version of derived deformation to the normal cone was constructed for algebraic stacks without the characteristic $0$ hypotheses in \cite{khan2019virtual}, but only for regular closed immersions. 

Let $\sAlg_R$ denote the simplicial model category of simplicial commutative $R$-algebras, for a fixed commutative ring $R$, in which fibrations and weak equivalences are defined to be those on underlying simplicial sets. Recall that the formal completion of an algebraic derived stack $\cX$ along a closed substack $\cS \hookrightarrow \cX$ is the full substack $\cX^{\wedge}_\cS \subset \cX$ parameterizing those maps $\Spec(A) \to \cX$ for which $\Spec(\pi_0(A)^{\rm red}) \to \cX$ factors through $\cS$.

\begin{thm} [Derived infinitesimal neighborhoods]  \label{T:deformation_to_normal_cone}
Let $\cX$ be a quasi-compact quasi-separated a algebraic derived stack over a commutative base ring $R$, and let $i : \cS \hookrightarrow \cX$ be an almost finitely presented closed immersion. Then there is an $\bN$-indexed system of closed immersions
\[
\cS = \cS^{(0)} \hookrightarrow \cS^{(1)} \hookrightarrow \cS^{(2)} \hookrightarrow \cdots \hookrightarrow \cX
\]
such that each $\cS^{(n)} \to \cS^{(n+1)}$ is surjective, each $\cO_{\cS^{(n)}} \in \QC(\cX)$ is almost perfect, and there is a canonical isomorphism for all $n>0$
\[
\fib(\cO_{\cS^{(n)}} \to \cO_{\cS^{(n-1)}}) \cong i_\ast(\Sym_{\cS}^n(\bL_{\cS/\cX}[-1])) \in \QC(\cX),
\]
where $\Sym_{\cS}(-)$ denotes the derived symmetric power functor on $\QC(\cS)$. Furthermore:
\begin{enumerate}
\item As a sheaf of $\pi_0(\cO_\cX)$-algebras $\pi_0(\cO_{\cS^{(n)}}) \cong \pi_0(\cO_\cX)/I^n$, where $I = \ker(\pi_0(\cO_\cX) \to \pi_0(\cO_\cS))$ is the ideal of definition of $\cS^{\rm cl} \hookrightarrow \cX^{\rm cl}$. \\
\item The canonical map to the formal completion $\colim_n \cS^{(n)} \to \cX^\wedge_\cS$ is an isomorphism of prestacks on the full $\infty$-subcategory $\sAlg_R^{<\infty} \subset \sAlg_R$ of algebras that are $d$-truncated for some $d$; \\
\item For any $F \in \AAPerf_\cS(\cX)$ and any $d \in \bZ$, the $\bN^{\rm op}$-indexed system $\{\tau_{\leq d}(\cO_{\cS^{(n)}} \otimes F)\}_{n \geq 1}$ in $\AAPerf_\cS(\cX)$ is eventually constant.\\
\item For any $F \in \QC(\cX)$ that is $d$-truncated for some $d\in \bZ$, the canonical map is an isomorphism
\[
\colim_{n \to \infty} \inner{\RHom}^\otimes_\cX(\cO_{\cS^{(n)}}, F) \xrightarrow{\cong} R\Gamma_\cS(F).
\]
In general, one has
\[
\lim_{d \to \infty} \left( \colim_{n \to \infty} \inner{\RHom}^\otimes_\cX(\cO_{\cS^{(n)}},\tau_{\leq d} F) \right) \xrightarrow{\cong} R\Gamma_\cS(F).
\]
\end{enumerate}
\end{thm}

The proof of this theorem is essentially a construction, which in the affine case is \Cref{C:normal_cone} below. First we observe a useful consequence:
\begin{cor} \label{C:filtration_push_pull}
Let $i : \cS \to \cX$ be an almost finitely presented closed immersion of quasi-compact quasi-separated algebraic derived stacks. Then for any $E \in \QC(\cS)_{>-\infty}$, there is a functorial tower in $\QC(\cS)$
\[
i^\ast (i_\ast(E)) \to \cdots \to E_{n+1} \to E_n \to \cdots \to E_0,
\]
such that 1) for any $d \in \bZ$, $\tau_{\leq d} (i^\ast (i_\ast(E))) \to \tau_{\leq d}(E_n)$ is an isomorphism for $n\gg 0$, in particular $i^\ast(i_\ast(E)) \cong \varprojlim_n E_n$, and 2) there is a canonical isomorphism for all $n \geq 1$,
\[
\fib(E_n \to E_{n-1}) \cong \Sym^n(\bL_{\cS/\cX}) \otimes E.
\]
\end{cor}

\begin{proof}
Consider the cartesian diagram
\[
\xymatrix{\cS \times_\cS \cS \ar[r]_{p_2} \ar[d]^{p_1} & \cS \ar[d]_i \ar@/_/[l]_{e} \\ \cS \ar[r]_i & \cX},
\]
where $e$ is the identity section. Let $\cS \to \cdots \cdot \cS^{(n)}_e \to \cS^{(n+1)}_e \to \cdots \to \cS \times_\cX \cS$ denote the filtered system of infinitesimal neighborhoods associated by \Cref{T:deformation_to_normal_cone} to the map $e$, which is a surjective, almost finitely presented, closed immersion. The canonical fiber sequence for the cotangent complex gives a canonical isomorphism $\bL_{e} \cong \bL_{\cS/\cX}[1]$.

The derived base change formula functorially identifies $i^\ast(i_\ast(E)) \cong (p_2)_\ast(p_1^\ast(E))$. We then define
\[
E_n = (p_2)_\ast(p_1^\ast(E) \otimes_{\cO_{\cS \times_\cX \cS}} \cO_{\cS^{(n)}_e}),
\]
from which claim (2) of the corollary follows immediately from the fact that by construction $\fib(\cO_{\cS_e^{(n)}} \to \cO_{\cS_e^{(n-1)}}) \cong e_\ast(\Sym^n(\bL_e[-1]))$ and from the projection formula
\[
(p_2)_\ast(p_1^\ast(E) \otimes e_\ast(-)) \cong (p_2)_\ast e_\ast((-) \otimes e^\ast p_1^\ast(E)) \cong (-) \otimes E.
\]
Claim (1) of the corollary follows from the fact that $\cO_{\cS \times_\cX \cS}$ is automatically set-theoretically supported on $\cS$, so part (3) of \Cref{T:deformation_to_normal_cone} implies that $\forall d \in \bZ$,
\[
\tau_{\leq d}(\cO_{\cS \times_\cX \cS}) \cong \tau_{\leq d} (\cO_{\cS^{(n)}_e}) \text{ for all }n\gg 0.
\]
\end{proof}

The key observation behind \Cref{T:deformation_to_normal_cone} is that given a simplicial commutative $R$-algebra $A \in \sAlg_R$, there are several equivalent ways in which an $\bN^{\rm op}$-indexed system of $A$-algebras can represent the derived completion of $A$ along a finitely generated ideal $I \subset \pi_0(A)$, the first of which is easy to check in practice. We will use the fact that for any $A \in \sAlg_R$ and any $d \in \bZ$, the $\infty$-category $A\Mod_{\leq d}$ is compactly generated by the objects that are perfect to order $d+1$ \cite{DAGVIII}*{Defn.~2.6.1}.

\begin{lem}\label{L:characterize_completion}
Let $A \in \sAlg_R$, and let $\cdots A_{n+1} \to A_n \to \cdots \to A_0$ be a tower of $A$-algebras such that: i) each $A_n$ is almost-perfect as an $A$-module, ii) each map $\pi_0(A) \to \pi_0(A_n)$ is surjective, and iii) $\pi_0(A_{n+1}) \to \pi_0(A_n)$ has nilpotent kernel for every $n$. Consider the finitely generated ideal $I := \ker(\pi_0(A) \to \pi_0(A_0))$. Then the following are equivalent:
\begin{enumerate}
\item For any $d \geq 0$, the canonical map $\pi_0(A_0) \to \tau_{\leq d}(A_n \otimes_A \pi_0(A_0))$ is an isomorphism for $n \gg 0$.
\item For any $d \geq 0$ and any $M \in A\Mod_{\leq d}$ that is perfect to order $d+1$ and is $I$-nilpotent \cite{DAGXII}*{Defn.~4.1.3}, the canonical map $M \to \tau_{\leq d}(A_n \otimes_A M)$ is an isomorphism for $n \gg 0$.
\item For any $M \in A\Mod_{<\infty}$, the canonical map $\colim_n \inner{\RHom}^\otimes_A(A_n,M) \to R\Gamma_I(M)$ is an isomorphism.
\item The map $\colim_n \Spec(A_n) \to \Spec(A)^\wedge_I$ is an isomorphism of presheaves on the full $\infty$-subcategory $\sAlg^{<\infty}_R \subset \sAlg_R$ that consists of objects which are eventually truncated.
\end{enumerate}
Furthermore, if these conditions hold, then $\varprojlim_n (A_n \otimes_A M)$ is $I$-complete \cite{DAGXII}*{Defn.~4.2.1} for any $M \in A\Mod_{>-\infty}$, and the canonical morphism from the $I$-completion $M^\wedge_I \to \varprojlim_n (A_n \otimes_A M)$ is an isomorphism.
\end{lem}

\begin{rem}
For any $A \in \sAlg_R$ and finitely generated ideal $I \subset \pi_0(A)$, the existence of a tower $\{A_n\}_{n \geq 0}$ as in the lemma is established in \cite{DAGXII}*{Lem.~5.1.5}, which is stated for $E_\infty$-algebras but whose proof applies essentially verbatim for simplicial commutative algebras. In fact, in $\sAlg_R$ one can even arrange that each $A_n$ is perfect as an $A$-module \cite{halpern2014mapping}*{Prop.~2.1.2}. If one starts with an almost finitely presented surjection $A \to A_0$, then one canonical choice is $A_n = \Tot^{\leq n}(\Cech(A \to A_0))$, which results in the theory of Adams completion. Below we will use another canonical choice, resulting in the theory of derived $I$-adic completion. The lemma above is useful for comparing these different versions of derived completion.
\end{rem}

\begin{proof}[Proof of \Cref{L:characterize_completion}]
$(1) \Rightarrow (2):$ $M$ is of bounded homological amplitude, and if the claim holds for two complexes, then it holds for any extension of those complexes. It therefore suffices to prove the claim for $H_i(M)[i]$. We know that $I^n H_i(M)=0$ for some $n$ -- in the noetherian case it follows from the fact that $H_i(M)$ is finitely generated, and in general if follows from the proof of \cite{DAGXII}*{Prop.~4.1.15}, which provides a perfect complex $Q \in A\Mod$ which satisfies this condition and for which $\tau_{\leq d} (Q[k])$ generates $A\Mod_{\leq d}$. It therefore suffices to check the claim for each associated graded piece of the corresponding finite $I$-adic filtration of $H_i(M)[i]$. We may thus assume that $M = N[i]$ for some discrete $\pi_0(A_0)$-module $N$ regarded as an $A$-module via the map $A \to \pi_0(A_0)$. The projection formula then gives
\[
\tau_{\leq d}(A_n \otimes_A M) \cong \tau_{\leq d}\left(\tau_{\leq d-i}(A_n \otimes_A \pi_0(A_0)) \otimes_{\pi_0(A_0)} N[i]\right),
\]
so $(2)$ holds if $\tau_{\leq d-i}(A_n \otimes_A \pi_0(A_0))$ is eventually $\pi_0(A_0)$.

\medskip
$(2)\Leftrightarrow (3):$ If $M \in A\Mod_{\leq d}$, then $\inner{\RHom}^\otimes_{A}(A_n,M)$ and $R\Gamma_I(M)$ lie in $A\Mod_{\leq d}$ as well. The claim $(3)$ is equivalent to the claim that for any $d \geq 0$, any compact object $N \in A\Mod_{\leq d}$ that is $I$-nilpotent, and any $M \in A\Mod_{\leq d}$ the canonical map
\[
\Map_A(N,\colim_n \inner{\RHom}^\otimes_A(A_n,M)) \to \Map_A(N,M)
\]
is an isomorphism. Using compactness of $N$, the definition of the inner Hom, and the fact that $M$ is $d$-truncated, this is equivalent to the canonical homomorphism being an isomorphism
\[
\colim_n \Map_A(\tau_{\leq d}(A_n \otimes_A N),M) \to \Map_A(N,M).
\]
This is clearly true if $(2)$ holds.

On the other hand consider a fixed $N$, and assume this condition holds for all $M \in A\Mod_{\leq d}$. Applying the condition to $M = N$ we find that there exists an $k \geq 0$ and a map $\tau_{\leq d}(A_{k} \otimes_A N) \to N$ such that the composition $N \to \tau_{\leq d} (A_{k} \otimes_A N) \to N$ is homotopic to the identity, and applying the condition to $M=\tau_{\leq d}(A_k \otimes_A N)$, there is a map $\tau_{\leq d}(A_k \otimes_A N) \to N$ such that the composition $\tau_{\leq d}(A_k \otimes_A N) \to N \to \tau_{\leq d}(A_k \otimes_A N)$ is the identity. It follows that the canonical map $N \to \tau_{\leq d}(A_n \otimes_A N)$ is an isomorphism for all $n \geq k$.

\medskip
$(2) \Rightarrow (4):$ We know from \cite{DAGXII}*{Lem.~5.1.5} that there is \emph{some} tower $\cdots \to A'_{n+1} \to A'_n \to \cdots \to A'_1 \to A'_0$ of $A$-algebras with each $A_n$ almost perfrect as an $A$-module, $\pi_0(A'_0) = \pi_0(A_0)$, and $\colim_n \Spec(A_n) \to \Spec(A)^\wedge_I$ is an isomorphism of presheaves of $\sAlg_R$. Consider the filtered diagram of derived schemes
\[
\xymatrix{\Spec(A_0 \otimes_A A_0') \ar[r] \ar[d] & \Spec(A_0 \otimes_A A_1') \ar[r] \ar[d] & \cdots \\
\Spec(A_1 \otimes_A A_0') \ar[d] \ar[r] & \Spec(A_1 \otimes_A A_1') \ar[r] \ar[d] & \cdots \\
\vdots & \vdots & },
\]
regarded as presheaves on the category $\sAlg_R^{<\infty}$. Using the fact that fiber products commute with filtered colimits and the definition of the formal completion, the colimit of the $i^{th}$ row is canonically isomorphic to $\Spec(A_i) \times_{\Spec(A)} \Spec(A)^\wedge_I \cong \Spec(A_i)$, so the colimit of the diagram is canonically isomorphic to $\colim_n \Spec(A_n)$. On the other hand, for any $d$-truncated $A$-algebra $B$,
\[
\colim_i \Map_A(A_i \otimes_A A'_j, B) \cong \colim_i \Map_A(\tau_{\leq d}(A_i \otimes_A A'_j),B) \cong \Map_A(A'_j,B),
\]
where the last equivalence uses $(2)$. Thus the colimit of the $j^{th}$ column is canonically isomorphic to $\Spec(A'_j)$ as a presheaf on $\sAlg_R^{< \infty}$. It follows that the colimit of the diagram is canonically isomorphic to $\Spec(A)^\wedge_I$ by hypothesis.

\medskip
$(4) \Rightarrow (1):$ Using the fact that fiber products of presheaves commute with filtered colimits, and the definition of the formal completion, we have that
\[
\colim_n \Spec(A_n \otimes_A \pi_0(A_0)) \to \Spec(A)^\wedge_I \times_{\Spec(A)} \Spec(\pi_0(A_0)) \cong \Spec(\pi_0(A_0))
\]
as presheaves on $\sAlg_R^{<\infty}$. In other words, for any $d$-truncated $B \in \sAlg_A$, the canonical map
\[
\colim_n \Map_A(\tau_{\leq d}(A_n \otimes_A \pi_0(A_0)),B) \to \Map_A(\pi_0(A_0),B)
\]
is an isomorphism. The same argument as we used in the case of modules to show that $(3) \Rightarrow (2)$ implies that $\pi_0(A_0) \to \tau_{\leq d}(A_n \otimes_A \pi_0(A_0))$ must be an isomorphism of algebras.

\medskip
\textit{Computing the completion:} The functor $A \mapsto A\Mod_{>-\infty}$ is nil-complete, meaning that $A\Mod \to \lim_d (\tau_{\leq d}(A)\Mod)$ is an equivalence. In particular the functor $\QC(-)_{>-\infty}$ on the $\infty$-category of presheaves on $\sAlg_R$ is Kan extended from presheaves on $\sAlg_R^{<\infty}$. The claim (4) then implies that the canonical map of $\infty$-categories obtained by level-wise pullback
\[
\QC(\Spec(A)^\wedge_I)_{>-\infty} \to \lim_n \QC(\Spec(A_n))_{>-\infty} \cong \lim_n (A_n\Mod_{>-\infty})
\]
is an equivalence. If $i : \Spec(A)^\wedge_I \to \Spec(A)$ is the inclusion, then under the isomorphism above, the pullback functor $i^\ast : \QC(\Spec(A))_{>-\infty} \to \QC(\Spec(A)^\wedge_I)_{>-\infty}$ is identified with the functor $A\Mod_{>-\infty} \to \lim_n (A_n\Mod_{>-\infty})$ given level-wise by $(-) \otimes_A A_n$. The right adjoint of this functor takes a compatible system $\{M_n \in A_n\Mod\}_{n\geq 0}$ to $\varprojlim_n M_n \in A\Mod_{>-\infty}$. The formula for the completion functor given in the lemma follows from the fact that $(-)^\wedge_I \cong i_\ast(i^\ast(-))$.
\end{proof}

We now discuss the local case of the construction in \Cref{T:deformation_to_normal_cone}. Let $\cP$ denote the category of pairs $(A,I)$ with $A \in \sAlg_R$ and $I \subset A$ and maps $(A,I) \to (B,J)$ consisting of maps $\phi : A \to B$ in $\sAlg_R$ such that $\phi(I) \subset J$. We will make use of a simplicial model structure on $\cP$ discussed in \cite{derived_derham}*{Cor.~4.14}. A map $(A,I) \to (B,J)$ is a (trivial) fibration if and only if the underlying maps on simplicial sets $I \to J$ and $A \to B$ are so. The cofibrant objects are pairs $(A,I)$ such that $A_n$ is a free $R$-algebra on a subset $X_n \subset A_n$, the ideal $I_n$ is generated by a subset of these generators $Y_n \subset X_n$, and both $X_n$ and $Y_n$ are preserved by the degeneracy maps.

\begin{lem} \label{L:adic_weak_equivalence}
If $(A,I) \to (B,J)$ is a weak-equivalence of cofibrant objects in $\cP$, then $A/I^n \to B/J^n$ is a weak equivalence for every $n \geq 1$.
\end{lem}
\begin{proof}
The induced map $A/I \to B/J$ is a weak equivalence by the $5$-lemma. The $I$-adic filtration of the simplicial commutative $A$-algebra $A/I^n$ is finite, and because $I$ is level-wise regular the canonical map is a weak-equivalence
\[
\Sym_{A/I}^{\leq n-1}(I/I^2) \xrightarrow{\cong} {\rm gr}(A/I^n).
\]
The same holds for the pair $(B,J)$, so it suffices to show that the map $\Sym_{A/I} (I/I^2) \to \Sym_{B/J}(J/J^2)$ is a weak-equivalence. Both $A/I$ and $B/J$ are level-wise polynomial, and $I/I^2$ and $J/J^2$ are level-wise free models for a shift of the relative cotangent complex of $A \to A/I$ and $B \to B/J$ respectively. The claim follows from the fact that the cotangent complex is intrinsic to the maps $A \to A/I$.
\end{proof}

\begin{lem} \label{L:adic_cocartesian}
Let $(A,I) \to (B,J)$ be a morphism in $\cP$ such that $I$ and $J$ are level-wise regular ideals. If the canonical map $B \otimes_A^L (A/I) \to B/J$ is a weak-equivalence, then so is the map $B \otimes^L_A (A/I^n) \to B/J^n$ for all $n$.
\end{lem}

\begin{proof}
The argument is essentially the same as the previous lemma. The canonical map of $A$-modules $A/I^n \to B/J^n$ preserves the adic filtration, so it suffices to show that the induced map on associated graded pieces
\[
B \otimes_A^L \Sym_{A/I}(I/I^2) \cong (B \otimes^L_A A/I) \otimes^L_{A/I} \Sym_{A/I}(I/I^2) \to \Sym_{B/J}(J/J^2)
\]
is a weak-equivalence. $I/I^2$ and $J/J^2$ are level-wise free, so they model the shifted relative cotangent complex of $A \to A/I$ and $B \to B/J$ respectively. The claim then follows from the fact that the relevant cotangent complex is preserved by base change.
\end{proof}

\begin{const}[Adic completion] \label{C:normal_cone}
Given a surjective map $A \to B$, we let $A \to A' \to B$ be the functorial factorization as an acyclic cofibration followed by a fibration for the usual simplicial model structure on simplicial commutative algebras. Then the fact that $\pi_0(A') \to \pi_0(B)$ is surjective implies that $A' \to B$ is surjective, so $B \cong A' / I'$ for some simplicial ideal $I' \subset A'$. We then let $(A'',I'')$ be a functorial cofibrant replacement for $(A',I') \in \cP$ in the simplicial model structure discussed above. We define
\[
\Cpl^{\rm adic}(A \to B) := \left( A'' \to \cdots \to A''/(I'')^{n+1} \to A''/(I'')^n \to \cdots \to A''/I'' \right),
\]
where the right-hand-side is regarded as an element of the diagram category $\Fun((\bN \cup \{\infty\})^{\rm op},\sAlg_R)$. This functor is well-defined, up to weak-equivalence of the diagram $(A \to B)$, by \Cref{L:adic_weak_equivalence} above.
\end{const}

Note that there is a forgetful functor between diagram categories $F : (A_\infty \to \cdots \to A_{n+1} \to A_n \to \cdots \to A_2 \to A_0) \mapsto (A_\infty \to A_0)$, and $F \circ \Cpl^{\rm adic}(A \to B) = A'' \to A''/I''$ is canonically weakly equivalent to $A \to B$.

\begin{proof}[Proof of \Cref{T:deformation_to_normal_cone}]
$\cX$ is quasi-compact and quasi-separated, so we may present $\cX$ as the geometric realization of a simplicial object $X_\bullet$ in the full $\infty$-subcategory of affine derived $R$-schemes ${\rm Aff}_R$. We have a defining equivalence of $\infty$-categories
\[
{\rm Aff}_R^{\rm op} = N(\sAlg^\circ_R),
\]
where the right-hand-side is underlying $\infty$-category of $\sAlg_R$, i.e., the simplicial nerve of the simplicial category of cofibrant and fibrant objects. We can therefore identify $X_\bullet$ with a cosimplicial object $A^\bullet \in \Fun({\mathbf \Delta}, N(\sAlg^\circ_R))$. By \cite{HTT}*{Prop.~4.2.4.4}, we have an equivalence
\[
N(((\sAlg_R)^{\mathbf \Delta})^\circ) \cong \Fun({\mathbf \Delta}, N(\sAlg^\circ_R)),
\]
where the left side denotes the underlying $\infty$-category of the simplicial model category of cosimplicial simplicial commutative $R$-algebras, with its projective model structure.

So up to equivalence, $A^\bullet$ comes from an actual cosimplicial object in $\sAlg_R$. Likewise, the closed immersion $\cS \hookrightarrow \cX$ comes from an actual map of cosimplicial objects $A^\bullet \to B^\bullet$ in $\sAlg_R$ that is level-wise surjective. $A^\bullet \to B^\bullet$ is also cocartesian in the sense that for any map $A_n \to A_m$ coming from the cosimplicial structure, the canonical map $A_{[m]} \otimes^L_{A_{[n]}} B_{[n]} \to B_{[m]}$ is a weak equivalence.

We now apply \Cref{C:normal_cone} level-wise to obtain a diagram of cosimplicial objects in $\sAlg_R$
\[
\Cpl^{\rm adic}(A^\bullet \to B^\bullet) = \left( (A'')^\bullet \to \cdots \to A_{n+1}^\bullet \to A_n^\bullet \to \cdots \to A_0^\bullet\right),
\]
along with a canonical weak-equivalence if diagrams
\[
(A^\bullet \to B^\bullet) \sim ((A'')^\bullet \to A_0^\bullet).
\]
The latter identifies the geometric realization of the simplicial object $|\Spec((A'')^\bullet)|$ with $\cX$. \Cref{L:adic_cocartesian} implies that each map $(A'')^\bullet \to A_n^\bullet$ is cocartesian in the above sense, and thus the closed immersion of simplicial objects $\Spec(A_n^\bullet) \to \Spec((A'')^\bullet)$ descends to a closed substack of the geometric realization. We define $\cS^{(n)} \hookrightarrow \cX$ to be the corresponding closed substack under the same identification $\cX \cong |\Spec((A'')^\bullet)|$. Thus we have our $\bN$-indexed system of closed derived substacks
\[
\cS \cong \cS^{(0)} \hookrightarrow \cS^{(1)} \hookrightarrow \cdots \hookrightarrow \cX.
\]
Note that because the relative cotangent complex satisfies smooth descent over $\cX$, the canonical isomorphisms $\fib(A_{n}^\bullet \to A_{n-1}^\bullet) \cong \Sym_{A_0^\bullet}(\bL_{A_0/A}[-1])$ descend to the stated isomorphisms in the theorem. Because $\cS_0 \to \cX$ is an almost finitely presented closed immersion, $\bL_{\cS_0/\cX} \in \AAPerf(\cS_0)$ and $i_\ast(-)$ preserves $\AAPerf$, and it follows by induction that $\cO_{\cS^{(n)}} \in \AAPerf(\cX)$ for all $n$.

The claims (1)-(4) of the theorem are smooth-local over $\cX$, so it suffices to verify them in the case where $\cX =\Spec(A)$ is affine. The claim (1) in the affine case is immediate from \Cref{C:normal_cone}. The remaining claims are exactly the conditions (2)-(4) of \Cref{L:characterize_completion} (using the fact that for an almost perfect complex $M$, $\tau_{\leq d}(M)$ is perfect to order $d+1$ for any $d$, and any $d$-truncated complex that is perfect to order $d+1$ is the truncation of some almost perfect complex).

So, to complete the proof, we must show that for a surjective map $A \to B$, the system $\Cpl^{\rm adic}(A\to B)$ of \Cref{C:normal_cone} satisfies the equivalent conditions of \Cref{L:characterize_completion}. So, if $(A,I)$ is a cofibrant object of $\cP$, then we must show that for any $d \geq 0$, the canonical map
\[
\pi_0(A/I) \to \tau_{\leq d}((A/I^n) \otimes^L_A \pi_0(A/I))
\]
is a weak equivalence for all $n\gg 0$. By \Cref{L:adic_cocartesian}, it suffices to replace $A \to A/I$ with the morphism $\pi_0(A/I) \to \pi_0(A/I) \otimes_A^L (A/I)$, and to show that if we denote
\[
\Cpl^{\rm adic}(\pi_0(A/I) \to \pi_0(A/I) \otimes_A^L (A/I)) = \left(B \to \cdots \to B / J^{n+1} \to B/ J^n \to \cdots \to B / J \right),
\]
then $\tau_{\leq d} (B) \to \tau_{\leq d}(B / J^n)$ is a weak-equivalence for $n \gg 0$. Note that $\pi_0(B) \to \pi_0(B/J)$ is an isomorphism, and $\pi_1(B/J)=0$, so $\pi_0(J)=0$. Quillen's theorem, as formulated in \cite{derived_derham}*{Prop.~4.11}, says that $B \cong R\lim_n B / J^n$ in this case. Although the statement refers to the limit, the proof in fact shows that $J^{n+1}$ is $n$-connected for all $n$, which implies that $\tau_{\leq n} (B) \to \tau_{\leq d}(B/J^n)$ is a weak-equivalence for all $n\gg 0$.
\end{proof}


\section{The canonical complex of a quasi-smooth stack}
\label{A:grothendieck}

In this appendix we prove \Cref{P:canonical_complex_classical} after establishing several preliminary results. We first discuss the two (graded) line bundles we wish to compare more carefully.

\subsubsection*{The graded determinant of the cotangent complex}

Let $\cX$ be a derived algebraic stack, and let $\Pic^{\bZ}(\cX) \subset \Perf(\cX)^\otimes$ be the group-like $E_\infty$-monoid of invertible objects and isomorphisms between them, with monoidal structure given by tensor product. $\pi_0(\Pic^{\bZ}(\cX))$ consists of isomorphism classes of perfect complexes that are homological shifts of invertible sheaves, where the shift can differ on different connected components of $\cX$. We recall the construction from \cite{heleodoro2019determinant} of a map of group-like $E_\infty$-monoids
\begin{equation} \label{E:graded_determinant}
\detz : K(\Perf(\cX)) \to \Pic^{\bZ}(\cX),
\end{equation}
where $K(-)$ denotes the connective algebriac $K$-theory of an $\infty$-category \cite{barwick}.

Over a field of characteristic $0$, $B\GL_n$ is the left Kahn extension of its classical counterpart, so the map $\bigwedge^n : B\GL_n \to B\bG_m$ in the classical setting extends to the derived setting as well. For any derived stack $\cX$ we let $({\rm Vect}(\cX)^{\cong},\oplus)$ denote the symmetric monoidal $\infty$-groupoid of locally free sheaves and isomorphisms with symmetric monoidal structure given by direct sum. The formula
\[
\detz(E) := \det(E)[E] = \bigwedge{}^{\rank(E)}(E) [E]
\]
defines a symmetric monoidal functor of symmetric monoidal $\infty$-groupoids $\detz : ({\rm Vect}(\cX)^{\cong},\oplus) \to (\Pic^\bZ(\cX),\otimes)$.

To construct the map $\detz$ of \eqref{E:graded_determinant}, we consider a diagram of symmetric monoidal $\infty$-groupoids
\[
\xymatrix{
K(\Perf(\cX)) &  ({\rm Vect}(\cX),\oplus)^{\rm gp} \ar[r]^-{\detz} \ar[l]_{i} & \Pic^\bZ(\cX)
},
\]
where $(-)^{\rm gp}$ denotes the group-like completion of a symmetric monoidal $\infty$-groupoid. The right map is the extension of $\detz$ from ${\rm Vect}(\cX)^{\cong}$, which exists and is unique up to contractible choices because $\Pic^\bZ(\cX)$ is group-like. The map $i$ is constructed in \cite{heleodoro2019determinant}*{Sect.~1.1.2}.

The maps $i$ and $\detz$ are functorial in $\cX$, and thus may be regarded as maps of presheaves of group-like symmetric monoidal $\infty$-groupoids on $\cX$ (or equivalently, a presheaf of connective spectra). The key fact is that $i$ is an isomorphism for affine derived schemes \cite{heleodoro2019determinant}*{Thm.~1} and thus gives an isomorphism after smooth sheafification, and $\Pic^\bZ(-)$ is already a smooth sheaf, so we define the determinant as the composition
\[
K(\Perf(\cX)) \to \Gamma(\cX, K(\Perf(-))^{\rm sh}) \xrightarrow{\detz \circ (i^{\rm sh})^{-1}} \Gamma(\cX,\Pic^\bZ(-)) \cong \Pic^\bZ(\cX).
\]
\begin{ex}
If $E$ is a locally free sheaf on $\cX$, regarded as a class in $\pi_0(K(\Perf(\cX)))$, then $\detz(E) = \bigwedge^{\rank(E)} (E) [\rank(E)]$.
\end{ex}

The fact that $\detz$ is a map of symmetric monoidal $\infty$-groupoids implies that for any morphism $A \to B$ in $\Perf(\cX)$, one has an isomorphism $\detz(B) \cong \detz(A) \otimes \detz(\cofib(A \to B))$.

More generally, we consider a filtered object of $\Perf(\cX)$, which we regard as a sequence of maps $A_1 \to \cdots \to A_n$, although it is more precisely defined as an $n$-gapped object \cite{HA}*{Defn.~1.2.2.2}. One can consider the cofiber sequence $A_i \to A_n \to \cofib(A_i \to A_n)$ for any $i$, and then iterate this procedure using the induced filtrations on $A_i$ and $\cofib(A_i \to A_n)$ until one has constructed $A_n$ as an iterated sequence of extensions starting with the associated graded objects $B_i := \cofib(A_{i-1} \to A_{i})$, where $A_{-1}:=0$. This results in a sequence of isomorphisms
\begin{align*}
\detz(A_n) &\cong \detz(A_i) \otimes \detz(\cofib(A_i \to A_n)) \\
&\cong \cdots \cong \detz(B_1) \otimes \cdots \otimes \detz(B_n).
\end{align*}
The fact that $\detz$ is a symmetric monoidal map of symmetric monoidal $\infty$-groupoids implies that the homotopy class of the resulting isomorphism $\detz(A_n) \cong \detz(B_1) \otimes \cdots \otimes \detz(B_n)$ is independent of the choice of how to realize $A_n$ as an iterated extension.

For an arbitrary morphism $f : \cX \to \cY$ of quasi-smooth algebraic derived stacks, the relative cotangent complex $\bL_{\cX/\cY}$ is perfect, so we can consider the graded line bundle $\detz(\bL_{\cX/\cY})$. For any composition $\cW \to \cX \to \cY$, the additivity of $\detz$ applied to the cofiber sequence $f^\ast(\bL_{\cX/\cY}) \to \bL_{\cW/\cY} \to \bL_{\cW/\cX}$ induces a canonical isomorphism
\begin{equation} \label{E:shriek_vs_star_pullback2}
\detz(\bL_{\cW/\cY}) \cong \detz(\bL_{\cW/\cX}) \otimes \detz(\bL_{\cX/\cY})|_{\cW}.
\end{equation}
A longer composition of morphisms $\cW \to \cX \to \cY \to \cZ$ induces a filtration $\bL_{\cY/\cZ}|_{\cW} \to \bL_{\cX/\cZ}|_{\cW} \to \bL_{\cW/\cZ}$, and applying \eqref{E:shriek_vs_star_pullback2} in different ways results in a diagram of isomorphisms that commutes up to homotopy
\begin{equation} \label{E:commute_order_1}
\xymatrix{ \detz(\bL_{\cW/\cZ}) \ar[r] \ar[d] & \detz(\bL_{\cW/\cX}) \otimes \detz(\cX/\cZ)|_{\cW} \ar[d] \\
 \detz(\bL_{\cW/\cY}) \otimes \detz(\bL_{\cY/\cZ})|_{\cW} \ar[r] & \detz(\bL_{\cW/\cX}) \otimes \detz(\bL_{\cX/\cY})|_{\cW} \otimes \detz(\bL_{\cY/\cZ})|_{\cW}}.
\end{equation}
Finally, given a cartesian diagram of quasi-smooth algebraic derived stacks
\[
\xymatrix{\cX' \ar[d]^g \ar[r]^f & \cX \ar[d] \\ \cY' \ar[r] & \cY},
\]
the unique vertical arrows that make the following diagram commute
\[
\xymatrix{f^\ast(\bL_{\cX/\cY}) \ar[r] \ar[d] & \bL_{\cX'/\cY} \ar@{=}[d] \ar[r] & \bL_{\cX'/\cX} \\
\bL_{\cX'/\cY'} & \bL_{\cX'/\cY} \ar[l] & g^\ast(\bL_{\cY'/\cY}) \ar[l] \ar[u] }
\]
are isomorphisms. In particular, we both cofiber sequences in the top and bottom row are canonically split, and applying $\detz(-)$ gives an isomorphism
\[
\detz(\bL_{\cX'/\cY}) \cong f^\ast(\detz(\bL_{\cX/\cY})) \otimes g^\ast(\detz(\bL_{\cY'/\cY}))
\]
that is independent up to homotopy of which of the two cofiber sequences was used.

\subsubsection*{The second canonical graded line bundle}
For any eventually coconnective morphism of algebraic derived stacks $f : \cX \to \cY$ pullback of quasi-coherent complexes preserves $\DCoh$. Using this one can define a pullback $f^{\IC,\ast} : \IC(\cY) \to \IC(\cX)$ that is locally the ind-completion of the pullback $\DCoh(\cY) \to \DCoh(\cX)$. Also, $f^{\IC,\ast}$ agrees with the usual pullback $f^\ast : \QC(\cY)_{<\infty} \to \QC(\cX)_{<\infty}$ under the isomorphism $\Psi_{(-)} : \IC(-)_{<\infty} \cong \QC(-)_{<\infty}$ \cite{gaitsgory2013ind}*{Sect.~11.3}.

Any morphism between quasi-smooth algebraic derived stacks is Gorenstein, meaning it is eventually coconnective and the relative canonical complex $f^!(\cO_\cY) \in \IC(\cX)$ is a graded line bundle in $\DCoh(\cX)$. If $f : \cX \to \cY$ is representable and Gorenstein, there is an isomorphism
\[
f^!(-) \cong \cK_{\cX/\cY} \otimes f^{\IC,\ast}(-),
\]
for the canonically defined graded line bundle $\cK_{\cX/\cY} := f^{\QC,!}(\cO_\cY)$, as discussed in the proof of \cite{GR2}*{Prop.~II.9.7.3.4}. As in the case of the cotangent complex, we let $\cK_{\cX}$ denote $\cK_{\cX/\Spec(k)}$. We note that $\cK_{\cX/\cY} = \Psi_\cX^{-1}(\omega_\cX) \otimes f^\ast(\Psi_\cX^{-1}(\omega_\cY))$, where $\omega_\cX = (\cX \to \Spec(k))^!(k) \in \DCoh(\cX) \subset \IC(\cX)$ is the canonical complex.

Given a second representable Gorenstein morphism $g: \cY \to \cZ$, the $\QC(\cY)^\otimes$-linear structure of $f^{\IC,\ast}(-)$ gives a canonical isomorphism of functors
\[
(g\circ f)^!(-) \cong \cK_{\cX/\cY} \otimes f^\ast(\cK_{\cY/\cZ}) \otimes f^{\IC,\ast}(g^{\IC,\ast}(-)),
\]
which induces a canonical isomorphism of invertible sheaves
\begin{equation} \label{E:canonical_composition}
\cK_{\cX/\cZ} \cong \cK_{\cX/\cY} \otimes f^\ast(\cK_{\cY/\cZ}).
\end{equation}

For a composition $\cW \to \cX \to \cY \to \cZ$, the resulting diagram of canonical isomorphisms commutes up to homotopy:
\begin{equation} \label{E:independent_order}
\xymatrix{ \cK_{\cW/\cZ} \ar[r] \ar[d] & \cK_{\cW/\cX} \otimes \cK_{\cX/\cZ}|_{\cW} \ar[d] \\
\cK_{\cW/\cY} \otimes \cK_{\cY/\cZ}|_{\cW} \ar[r] & \cK_{\cW/\cX} \otimes \cK_{\cX/\cY}|_{\cW} \otimes \cK_{\cY/\cZ}|_{\cW} }
\end{equation}

Finally, given a cartesian diagram of quasi-smooth stacks $\cX' \cong \cX \times_{\cY} \cY'$, we have canonical isomorphisms $\cK_{\cX'/\cY'} \cong \cK_{\cX/\cY}|_{\cX'}$ and $\cK_{\cX'/\cX} \cong \cK_{\cY'/\cY}|_{\cX'}$. This gives an isomorphism
\begin{equation}\label{E:product_formula}
\cK_{\cX'/\cY} \cong \cK_{\cX/\cY}|_{\cX'} \otimes \cK_{\cY'/\cY}|_{\cY'},
\end{equation}
that does not depend, up to homotopy, on which of the two compositions $\cX' \to \cX \to \cY$ or $\cX' \to \cY' \to \cY$ is used to define it.

\subsubsection*{Canonical isomorphisms}

\begin{lem} \label{L:grothendieck_formula_1}
For any regular closed immersion of algebraic derived stacks $f : \cX \to \cY$, there is a canonical isomorphism of line bundles $\phi_f : \cK_{\cX/\cY} \cong \detz(\bL_{\cX/\cY})$ such that:
\begin{enumerate}
\item For any other regular closed immersion $g : \cY \to \cZ$, under the isomorphisms \eqref{E:canonical_composition} and \eqref{E:shriek_vs_star_pullback2} there is a homotopy $\phi_{g\circ f} \sim \phi_f \otimes f^\ast(\phi_g)$, i.e., the following diagram commutes up to homotopy
\begin{equation} \label{E:homotopy_commute}
\xymatrix@C=35pt{ \cK_{\cX/\cZ} \ar[r]^{\phi_{g\circ f}} \ar[d]_\cong^{\eqref{E:canonical_composition}} & \detz(\bL_{\cX/\cZ}) \ar[d]_\cong^{\eqref{E:shriek_vs_star_pullback2}} \\
\cK_{\cX/\cY} \otimes f^\ast(\cK_{\cY/\cZ}) \ar[r]^-{\phi_f \otimes f^\ast(\phi_g)} & \detz(\bL_{\cX/\cY}) \otimes f^\ast(\detz(\bL_{\cY/\cZ})) }.
\end{equation}
\item For any cartesian square of the form
\[
\xymatrix{ \cX' \ar[d]^{g'} \ar[r]^{f'} & \cY' \ar[d]^g \\
\cX \ar[r]^f & \cY},
\]
under the canonical isomorphisms $\cK_{\cX'/\cY'} \cong (g')^\ast(\cK_{\cX/\cY})$ and $\detz(\bL_{\cX'/\cY'}) \cong \detz(\bL_{\cX/\cY})$, $\phi_{f'}$ is homotopic to $(g')^\ast(\phi_f)$.
\end{enumerate}

\end{lem}
\begin{proof}
The isomorphism $\phi_f$ is constructed in \cite{GR2}*{Cor.~II.9.7.3.2}, so here we verify the compatibility conditions (1) and (2).

\medskip
\noindent \textit{Proof of (1):}
\medskip

Let $\varphi_1$ denote the right vertical arrow in \eqref{E:homotopy_commute}, and $\varphi_2: \detz(\bL_{\cX/\cZ}) \cong \detz(\bL_{\cX/\cY}) \otimes f^\ast(\detz(\bL_{\cY/\cZ}))$ the isomorphism $(\phi_f \otimes f^\ast(\phi_g)) \circ \phi_{g \circ f}^{-1}$. We must show that $\varphi_1 \sim \varphi_2$.

It suffices to replace $\cY$ and $\cZ$ with their formal completions along $\cX$, by the proof of \cite{GR2}*{Cor.~II.9.7.3.2}. The construction of deformation to the normal cone \cite{GR2}*{Sect.~II.2.4} is functorial for maps between formal stacks under $\cX$, so deformation to the normal cone gives a sequence of regular closed immersions of formal stacks under $\cX \times \Theta$ and over $\Theta$
\[
\xymatrix{
\cX \times \Theta \ar@{^{(}->}[r]^-{\tilde{f}} & \tilde{\cY} \ar@{^{(}->}[r]^-{\tilde{g}} & \tilde{\cZ}.
}
\]
Both $\varphi_1$ and $\varphi_2$ extend to isomorphisms over $\cX \times \Theta$
\[
\tilde{\varphi}_1,\tilde{\varphi}_2 : \detz(\bL_{\cX\times \Theta/\tilde{\cZ}}) \cong \detz(\bL_{\cX\times \Theta/\tilde{\cY}}) \otimes \detz(\tilde{f}^\ast \bL_{\tilde{\cY}/\tilde{\cZ}}).
\]
The composition $\tilde{\varphi}_1^{-1} \circ \tilde{\varphi}_2$ is an automorphism of $\detz(\bL_{\cX\times \Theta/\tilde{\cZ}})$, i.e., a non-vanishing section of $\cO_{\cX \times \Theta}$. We observe that $\Gamma(\cX \times \Theta, \cO_{\cX \times \Theta}) \cong \Gamma(\cX,\cO_\cX)$, so if $\tilde{\varphi}_1^{-1} \circ \tilde{\varphi}_2$ is homotopic to the identity when restricted to the fiber over $0 \in \Theta$, then $\varphi_1 \sim \varphi_2$ as well.

We have therefore reduced the claim to the case where $\cY$ and $\cZ$ are the total spaces vector bundles over $\cX$, $f$ is the zero section, and $g$ is the inclusion of $\cY$ as a sub-bundle of $\cZ$. Applying the same argument as above to the deformation to the normal cone of $g$, we may further assume that this inclusion is split, i.e., $\cZ \cong \cY \times_{\cX} \cY'$ for another vector bundle $\cY'$. $\varphi_1$ and $\varphi_2$ are homotopic because they are both pulled back from the canonical isomorphism $\det(\bA^m \oplus \bA^n)[-m-n] \cong \det(\bA^m)[-m] \otimes \det(\bA^n)[-n]$ on the universal example $\cY = (\bA^m / \GL_m) \times B\GL_n$ and $\cZ = \bA^m \oplus \bA^n / \GL_m \times \GL_n$, where $m$ and $n$ are the ranks of $\cY$ and $\cY'$ respectively.

\medskip
\noindent \textit{Proof of (2):}
\medskip

As for the previous claim, we may replace $\cY$ with its formal completion along the closed substack $\cX$. The construction of the deformation to the normal cone is compatible with base change, and using the same argument as in the previous claim, this allows us to reduce to the case where $\cY \to \cX$ is the total space of a vector bundle, $f : \cX \to \cY$ is the zero section, and $\cY'$ is the pullback of the vector bundle $\cY \to \cX$ along the map $g' : \cX' \to \cX$. In this case both $\phi_f$ and $\phi_{f'}$ are pulled back from the universal example of the zero section $B\GL_n \hookrightarrow \bA^n / \GL_n$, so the claim follows from the compatibility of the base change isomorphisms $\cK_{\cX'/\cY'} \cong (g')^\ast(\cK_{\cX/\cY})$ and $\detz(\bL_{\cX'/\cY'}) \cong (g')^\ast(\detz(\bL_{\cX/\cY}))$ with composition of cartesian squares.
\end{proof}

\begin{lem} \label{L:grothendieck_formula_2}
For any separated and smooth representable morphism of algebraic derived stacks $f : \cX \to \cY$, there is a canonical isomorphism of line bundles $\phi_f : \cK_{\cX/\cY} \cong \detz(\bL_{\cX/\cY})$ such that: 1) claim (1) of \Cref{L:grothendieck_formula_1} holds for any composition of representable morphisms $\cX \to \cY \to \cZ$ where $\cX$ and $\cY$ are smooth and separated over $\cZ$, and $\cX \to \cY$ is either a regular closed immersion or smooth and separated; and 2) claim (2) of that lemma holds for any cartesian diagram in which $f$ is smooth, representable, and separated.
\end{lem}
\begin{proof}
The construction is given in \cite{GR2}*{Prop.~II.9.7.3.4} (whose proof appears to require a separatedness hypothesis that is not included in the statement). We recall it here so as to verify the compatibility with base change and composition. Given a smooth morphism $f : \cX \to \cY$, we consider the canonical identification $\cK_{\cX \times_\cY \cX/\cY} \cong p_1^\ast(\cK_{\cX/\cY}) \otimes p_2^\ast(\cK_{\cX/\cY})$, where $p_1,p_2 : \cX \times_\cY \cX$ denote the left and right projection respectively. If $\Delta_f : \cX \to \cX \times_\cY \cX$ is the diagonal
\begin{align*}
\cK_{\cX/\cX \times_\cY \cX} &\cong \cK_{\cX/\cY} \otimes \Delta_f^\ast(\cK_{\cX \times_\cY \cX/\cY})^\dual \\
&\cong \cK_{\cX/\cY} \otimes \Delta_f^\ast(p_1^\ast(\cK_{\cX/\cY}))^\dual \otimes \Delta_f^\ast(p_2^\ast(\cK_{\cX/\cY}))^\dual \\
&\cong \Delta_f^\ast(p_2^\ast(\cK_{\cX/\cY}))^\dual \\
&\cong \cK_{\cX/\cY}^\dual,
\end{align*}
and there is an analogous sequence of isomorphisms with $\detz(\bL_{-/-})$ in place of $\cK_{-/-}$. Under these isomorphisms, we define $\phi_f := (\phi_{\Delta_f}^{-1})^\dual$, where $\phi_{\Delta_f} : \cK_{\cX/\cX \times_\cY \cX} \cong \detz(\bL_{\cX/\cX \times_\cY \cX})$ denotes the isomorphism associated to the regular closed immersion $\Delta_f$ in \Cref{L:grothendieck_formula_1}.

\medskip
\noindent \textit{Proof of (1):}
\medskip

Consider a composition as in (1) in which both $f$ and $g$ are smooth, and consider the following diagram
\[
\xymatrix{
\cX \ar@{^{(}->}[r]^{\Delta_f} \ar[dr]^f \ar@/^15pt/[rr]^{\Delta_{g\circ f}}& \cX \times_\cY \cX \ar@{^{(}->}[r] \ar[d]^\pi & \cX \times_\cZ \cX \ar[d] \\
& \cY \ar@{^{(}->}[r]^{\Delta_g} & \cY \times_\cZ \cY},
\]
where the square is cartesian and the horizontal arrows are regular closed immersions. The claims (1) and (2) of \Cref{L:grothendieck_formula_1} imply that $\phi_{\Delta_{g\circ f}} \sim \phi_{\Delta_f} \otimes f^\ast(\phi_{\Delta_g})$ under the canonical isomorphism
\[
\cK_{\cX/\cX \times_\cZ \cX} \cong \cK_{\cX/\cX \times_\cY \cX} \otimes f^\ast(\cK_{\cY / \cY \times_\cZ \cY})
\]
and the analogous isomorphism for $\detz(\bL_{-/-})$. To conclude, we observe that this isomorphism agrees with the canonical isomorphism $\cK_{\cX/\cZ} \cong \cK_{\cX/\cY} \otimes f^\ast(\cK_{\cY/\cZ})$ under the identifications $\cK_{\cX/\cZ} \cong \cK_{\cX/ \cX \times_\cZ \cX}^\dual$, $\cK_{\cX/\cY} \cong \cK_{\cX/\cX \times_\cY \cX}^\dual$, and $\cK_{\cY/\cZ} \cong \cK_{\cY/\cY\times_\cZ \cY}^\dual$ discussed above, and the same holds for $\detz(\bL_{-/-})$.

If $f$ is a regular closed immersion, $g$ is smooth, and $g \circ f$ is smooth, then instead we consider the following diagram, in which the outer rectangle is cartesian,
\[
\xymatrix{\cX \ar@/^15pt/[rr]^{{\rm Graph}(f)} \ar@{^{(}->}[r]^{\Delta_{g\circ f}} \ar[d]^f & \cX \times_\cZ \cX \ar@{^{(}->}[r]^{\id \times f} & \cX \times_\cZ \cY \ar[d]^{f\times \id} \\
\cY \ar@{^{(}->}[rr]^{\Delta_g} & & \cY \times_\cZ \cY }.
\]
Then \Cref{L:grothendieck_formula_1} implies that $f^\ast(\phi_{\Delta_g}) \sim \phi_{\Delta_{g\circ f}} \otimes \phi_f$ under the canonical isomorphisms
\begin{align*}
f^\ast(\cK_{\cY / \cY \times_\cZ \cY}) &\cong \cK_{\cX / \cX \times_\cZ \cY} \\
&\cong \cK_{\cX / \cX \times_\cZ \cX} \otimes \Delta_{g\circ f}^\ast(\cO_\cX \boxtimes \cK_{\cX/\cY}) \\
&\cong \cK_{\cX / \cX \times_\cZ \cX} \otimes \cK_{\cX/\cY}
\end{align*}
and the analogous isomorphisms for $\detz(\bL_{-/-})$. We conclude by observing that this isomorphism corresponds to the canonical isomorphism $\cK_{\cX/\cZ} \cong \cK_{\cX / \cY} \otimes f^\ast(\cK_{\cY/\cZ})$ under the identifications $\cK_{\cX/\cX \times_\cZ \cX} \cong \cK_{\cX/\cZ}^\dual$ and $\cK_{\cY/\cY\times_\cZ \cY} \cong \cK_{\cY/\cZ}^\dual$ discussed above, and likewise for $\detz(\bL_{-/-})$.

\medskip
\noindent \textit{Proof of (2):}
\medskip

This follows immediately from claim (2) of \Cref{L:grothendieck_formula_1}, because $\Delta_{f'} : \cX' \to \cX' \times_{\cY'} \cX'$ is the base change of $\Delta_f : \cX \to \cX \times_\cY \cX$ along the map $\cY' \to \cY$. The same holds for the canonical isomorphisms $\cK_{\cX'/\cX' \times_{\cY'} \cX'} \cong \cK_{\cX'/\cY'}^\dual$ and $\detz(\bL_{\cX'/\cX' \times_{\cY'} \cX'}) \cong \detz(\bL_{\cX'/\cY'})^\dual$.

\end{proof}

\begin{lem} \label{L:quasi-smooth_embedding}
If $S$ is a quasi-smooth affine derived $k$-scheme, there is a regular closed embedding $S \hookrightarrow \bA^n_k$. Given a second embedding $S \hookrightarrow \bA^{n'}_k$, there are embeddings $\bA^n_k, \bA^{n'}_k \hookrightarrow \bA^m$ such that the following diagram commutes:
\[
\xymatrix{S \ar[r] \ar[d] & \bA^n_k \ar[d] \\
\bA^{n'}_k \ar[r] & \bA^m_k}
\]
\end{lem}
\begin{proof}
Say $S = \Spec(A)$ for some a.f.p. semi-free \CDGA $A$. The embeddings $S \hookrightarrow \bA^n$ and $S \hookrightarrow \bA^{n'}$ correspond to maps $\psi : k[x_1,\dots,x_n] \to A$ and $\phi : k[y_1,\ldots,y_{n'}] \to A$ inducing surjections onto $H_0(A)$. These are automatically regular closed immersions if $S$ is quasi-smooth. Consider the tensor product map
\[
\psi \otimes \phi : k[x_1,\ldots,x_n,y_1,\ldots,y_{n'}] \to A.
\]
This is also surjective on $H_0$. By hypothesis, the image of every $y_j$ is homotopic to something in the image of $k[x_i]$, i.e. $y_j \sim f_j(x_1,\ldots,x_n)$ for some $f_j \in k[x_i]$. A choice of such $f_j$ defines a surjection $k[x_i,y_j] \to k[x_i]$ such that the composition with $\psi$ is homotopic to $\psi \otimes \phi$. In the same way, we construct a surjection $k[x_i,y_j] \to k[y_j]$ such that the composition with $\phi$ is homotopic to $\psi \otimes \phi$. The claim follows, where $m=n+n'$.
\end{proof}

\begin{lem} \label{L:grothendieck_formula_3}
For any morphism of quasi-smooth affine derived $k$-schemes $f : S' \to S$, there is a canonical isomorphism $\phi_f : \cK_{S'/S} \cong \detz(\bL_{S'/S})$ such that:
\begin{enumerate}
\item For any other morphism $g : S \to T$ to a quasi-smooth affine derived $k$-scheme, the composition claim (1) of \Cref{L:grothendieck_formula_1} holds.\\ 
\item For any quasi-smooth morphism $g : T \to S$ from an affine derived $k$-scheme, the base change claim (2) of \Cref{L:grothendieck_formula_1} holds.\\
\item $\phi_f$ agrees with the isomorphism of \Cref{L:grothendieck_formula_1} when $f$ is a regular closed immersion, and it agrees with the isomorphism of \Cref{L:grothendieck_formula_3} when $f$ is smooth.

\end{enumerate}
\end{lem}
\begin{proof}
We use \Cref{L:quasi-smooth_embedding} to choose closed embeddings $S \hookrightarrow \bA_k^n$ and $S' \hookrightarrow \bA^m_k$, and consider the commutative square
\[
\xymatrix{S' \ar@{^{(}->}[r]^{j'} \ar[d]^f & \bA^n \times \bA^m \ar[d]^\pi \\ S \ar@{^{(}->}[r]^j & \bA^n}.
\]
We then define $\phi_f : \cK_{S'/S} \cong \detz(\bL_{S'/S})$ to be the unique homotopy class of isomorphism that makes the following diagram commute
\begin{equation} \label{E:commutativity_1}
\xymatrix{\cK_{S'/S} \otimes f^\ast(\cK_{S/\bA^n}) \ar[r]^-{\cong} \ar[d]^{\phi_f \otimes f^\ast(\phi_j)} & \cK_{S'/\bA^{n+m}} \otimes (j')^\ast(\cK_{\bA^{n+m} / \bA^n}) \ar[d]^{\phi_{j'} \otimes (j')^\ast(\phi_{\pi})} \\ \detz(\bL_{S'/S}) \otimes f^\ast(\detz(\bL_{S/\bA^n})) \ar[r]^-{\cong} & \detz(\bL_{S'/\bA^{n+m}}) \otimes (j')^\ast(\detz(\bL_{\bA^{n+m} / \bA^n}))},
\end{equation}
where $\phi_j$ and $\phi_{j'}$ are the isomorphisms of \Cref{L:grothendieck_formula_1}, and $\phi_\pi$ is the isomorphism of \Cref{L:grothendieck_formula_2}.

To complete the proof, we must show that the homotopy class of $\phi_f$ is independent of the choices in this definition, and verify claims (1)-(3).

\medskip
\noindent \textit{Independence of choice of embedding:}
\medskip

Consider further embeddings $i : \bA^n \hookrightarrow \bA^{n'}$ and $i' : \bA^m \hookrightarrow \bA^{m'}$, and the corresponding commutative diagram:
\begin{equation} \label{E:composition_diagram}
\xymatrix{S' \ar@{^{(}->}[r]^{j'} \ar[d]^f \ar[dr]^g \ar@/^20pt/[rr]^{j''} & \bA^n \times \bA^m \ar[d]^\pi \ar@{^{(}->}[r]^{i \times i'} & \bA^{n'} \times \bA^{m'} \ar[d]^{\pi'}\\
 S \ar@{^{(}->}[r]^j & \bA^n \ar@{^{(}->}[r]^i & \bA^{n'} }.
\end{equation}
We have the following sequence of canonical isomorphisms:
\begin{equation} \label{E:composition_isomorphisms}
\begin{array}{ll}
\cK_{S'/S} \otimes f^\ast(\cK_{S/\bA^{n'}})  & \\
\;\;\; \cong \cK_{S'/S} \otimes f^\ast(\cK_{S/\bA^n}) \otimes g^\ast(\cK_{\bA^n/\bA^{n'}}), & \Leftarrow \eqref{E:canonical_composition} \\
\;\;\; \cong \cK_{S'/\bA^{n+m}} \otimes (j')^\ast(\cK_{\bA^{n+m} / \bA^n}) \otimes g^\ast(\cK_{\bA^n/\bA^{n'}}), & \Leftarrow j \circ f \cong \pi \circ j' \\
\;\;\; \cong \cK_{S'/\bA^{n+m}} \otimes (j')^\ast(\cK_{\bA^{n+m} / \bA^{n'+m'}}) \otimes (j'')^\ast(\cK_{\bA^{n'+m'}/\bA^{n'}}), & \Leftarrow  i\circ \pi \cong \pi' \circ (i \times i') \\
\;\;\; \cong \cK_{S'/\bA^{n'+m'}} \otimes (j'')^\ast(\cK_{\bA^{n'+m'}/\bA^{n'}}), & \Leftarrow \eqref{E:canonical_composition}.
\end{array}
\end{equation}
Using the commutativity of \eqref{E:independent_order}, one can verify that the composition of these isomorphisms is homotopic to the isomorphism $\cK_{S'/S} \otimes f^\ast(\cK_{S/\bA^{n'}}) \cong \cK_{S'/\bA^{n'+m'}} \otimes (j'')^\ast(\cK_{\bA^{n'+m'}/\bA^{n'}})$ induced by the isomorphism $\pi' \circ j'' \cong (i \circ j) \circ f$.

We have another sequence of isomorphisms identical to \eqref{E:composition_isomorphisms}, but with $\detz(\bL_{-/-})$ in place of $\cK_{-/-}$ and using \eqref{E:shriek_vs_star_pullback2} in place of \eqref{E:canonical_composition}. Once again the composition gives an isomorphism
\[
\detz(\bL_{S'/S}) \otimes f^\ast(\detz(\bL_{S/\bA^{n'}})) \cong \detz(\bL_{S'/\bA^{n'+m'}}) \otimes (j'')^\ast(\detz(\bL_{\bA^{n'+m'}/\bA^{n'}})).
\]
that is homotopic to that induced by the isomorphism $\pi' \circ j'' \cong (i \circ j) \circ f$, by the commutativity of \eqref{E:commute_order_1}.

Assume for the moment that either $i = \id : \bA^n \to \bA^n$ or $i' = \id : \bA^m \to \bA^m$. Under the sequence of isomorphisms \eqref{E:composition_isomorphisms}, and the corresponding sequence for $\detz(\bL_{-/-})$, we have homotopies
\[
\begin{array}{rll}
\phi_f \otimes f^\ast(\phi_{i \circ j}) &\sim \phi_f \otimes f^\ast(\phi_j) \otimes g^\ast(\phi_i) &\Leftarrow \Cref{L:grothendieck_formula_1}\\
&\sim \phi_{j'} \otimes (j')^\ast(\phi_{\pi}) \otimes g^\ast(\phi_i) &\Leftarrow \eqref{E:commutativity_1} \\
&\sim \phi_{j'} \otimes (j')^\ast(\phi_{i \times i'}) \otimes (j'')^\ast(\phi_{\pi'}) & \Leftarrow \Cref{L:grothendieck_formula_2} / \Cref{L:grothendieck_formula_1} \\
&\sim \phi_{j''} \otimes (j'')^\ast(\phi_{\pi'}) &\Leftarrow \Cref{L:grothendieck_formula_1}
\end{array}
\]
For the third homotopy, if $i=\id$ we have applied claim (1) of \Cref{L:grothendieck_formula_2}, and if $i'=\id$ we have used the fact that the right square is cartesian and \Cref{L:grothendieck_formula_1}. This shows that the isomorphism $\phi_f$ defined by the left commuting square in \eqref{E:composition_diagram} is homotopic to the isomorphism defined by the outer commuting rectangle in \eqref{E:composition_diagram}. It now follows from \Cref{L:quasi-smooth_embedding} that the homotopy class of $\phi_f$ is independent of both of the choices of embeddings $j$ and $j'$.

\medskip
\noindent \textit{Proof of (1):}
\medskip

Given a second morphism $g : S \to T$, we choose an embedding $T \hookrightarrow \bA^{p}$, and consider the commutative diagram
\[
\xymatrix{ S' \ar[r]^f \ar[d]^{j'} & S \ar[r]^g \ar[d]^{j} & T \ar[d]^{i} \\
\bA^{n+m+p} \ar[r]^{\pi'} & \bA^{n+p} \ar[r]^{\pi} & \bA^{p} }
\]
We wish to show that under the isomorphisms $\cK_{S'/T} \cong \cK_{S'/S} \otimes f^\ast(\cK_{S/T})$ and $\detz(\bL_{S'/T}) \cong \detz(\bL_{S'/S}) \otimes f^\ast( \detz(\bL_{S/T}) )$ the isomorphism $\phi_{g \circ f}$ is homotopic to $\phi_f \otimes f^\ast(\phi_g)$. As before we have a sequence of homotopies of isomorphisms $\cK_{S'/\bA^p} \cong \detz(\bL_{S'/\bA^p})$ (leaving the evident identification of the source and target between one line and the next implicit for brevity):
\[
\begin{array}{ll}
\phi_f \otimes f^\ast(\phi_g) \otimes (g \circ f)^\ast(\phi_i) & \\
\;\;\; \sim \phi_f \otimes g^\ast(\phi_g \otimes g^\ast(\phi_i)) & \\
\;\;\; \sim \phi_f \otimes f^\ast(\phi_{j} \otimes j^\ast(\phi_\pi))  & \Leftarrow \text{definition of }\phi_g \\
\;\;\; \sim \phi_{j'} \otimes (j')^\ast(\phi_{\pi'}) \otimes (j \circ f)^\ast(\phi_\pi) & \Leftarrow \text{definition of }\phi_f\\
\;\;\; \sim \phi_{j'} \otimes (j')^\ast(\phi_{\pi'} \otimes (\pi')^\ast(\phi_\pi)) & \\
\;\;\; \sim \phi_{j'} \otimes (j')^\ast(\phi_{\pi \circ \pi'}) & \Leftarrow \Cref{L:grothendieck_formula_2}
\end{array}
\]
Therefore, $\phi_g \otimes g^\ast(\phi_f)$ satisfies the defining property of $\phi_{f \circ g}$, so the two isomorphisms are homotopic.

\medskip
\noindent \textit{Proof of (2):}
\medskip

We first observe that given any cartesian diagram
\begin{equation} \label{E:grothendieck_formula_base_change}
\xymatrix{ T' \ar[r]^{g'} \ar[d]^{f'} & S' \ar[d]^f \\ T\ar[r]^g & S},
\end{equation}
the claim that $\phi_{f'} \sim (g')^\ast(\phi_{f})$ under the canonical isomorphisms $\cK_{T'/T} \cong (g')^\ast(\cK_{S'/S})$ and $\detz(\bL_{T'/T}) \cong (g')^\ast(\detz(\bL_{S'/S}))$ is equivalent to the claim that $\phi_{g'} \sim (f')^\ast(\phi_g)$ under the analogous isomorphisms. To see this, note that the composition of canonical of isomorphisms
\[
\begin{array}{rl}
(g')^\ast(\cK_{S'/S}) \otimes (f')^\ast(\cK_{T/S}) &\cong \cK_{T' / T} \otimes (f')^\ast(\cK_{T/S}) \\
&\cong \cK_{T' / S} \\
&\cong \cK_{T'/S'} \otimes (g')^\ast(\cK_{S'/S}) \\
&\cong (f')^\ast(\cK_{T/S}) \otimes (g')^\ast(\cK_{S'/S})
\end{array}
\]
is homotopic to the isomorphism that switches the two tensor factors, and the same is true for the analogous isomorphisms with $\detz(\bL_{-/-})$ instead of $\cK_{-/-}$. The claim that $\phi_{f'}$ is homotopic to $(g')^\ast(\phi_f)$ is equivalent to the claim that $(g')^\ast(\phi_f) \otimes (f')^\ast(\phi_g) \sim \phi_{f'} \otimes (f')^\ast(\phi_g)$. By claim (1) of this lemma, the latter is homotopic to $\phi_{g'} \otimes (g')^\ast(\phi_f)$. Hence $\phi_{f'} \sim (g')^\ast(\phi_f)$ if and only if $\phi_{g'} \otimes (g')^\ast(\phi_f) \sim (f')^\ast(\phi_{g}) \otimes (g')^\ast(\phi_f)$, which is in turn equivalent to $\phi_{g'} \sim (f')^\ast(\phi_g)$.

We can factor $g$ as a composition $T \hookrightarrow \bA^n \times Z \to Z$, and the hypothesis that $g$ is quasi-smooth implies that the first morphism is a regular closed immersion. Thus we can factor our cartesian square as a composition of cartesian squares
\begin{equation} \label{E:factor_base_change}
\xymatrix{ T' \ar@{^{(}->}[r] \ar[d]^{f'} & \bA^n \times S' \ar[r]^{\pi'} \ar[d]^h & S' \ar[d]^f \\ T\ar@{^{(}->}[r] & \bA^n \times S \ar[r]^\pi & S}.
\end{equation}
By the compatibility of the canonical isomorphisms $\cK_{T'/T} \cong (g')^\ast(\cK_{S'/S})$ and $\detz(\bL_{T'/T}) \cong (g')^\ast (\detz(\bL_{S'/S}))$ with composition, it suffices to prove the claim for each of these cartesian squares. For the left square, vertical base change follows from \Cref{L:grothendieck_formula_1}, and so horizontal base change follows as well.

For the right cartesian square, it again suffices to show vertical base change. Both maps $\pi$ and $\pi'$ in \eqref{E:factor_base_change} are the base change of the projection map $\bA^n \to \Spec(k)$. It therefore suffices to show that for any quasi-smooth affine $k$-scheme $S$, if $\pi_{\bA^n} :\bA^n \times S \to \bA^n$ and $\pi_S : \bA^n \times S \to S$ are the projections, then $\phi_{\pi_S} \sim \pi_{\bA^n}^\ast(\phi_{\bA^n \to \Spec(k)})$ under the isomorphisms $\cK_{\bA^n \times S/S} \cong \pi_{\bA^n}^\ast(\cK_{\bA^n/\Spec(k)})$ and $\detz(\bL_{\bA^n \times S / S}) \cong \pi_{\bA^n}^\ast(\detz(\bL_{\bA^n/\Spec(k)})$.

We choose an embedding $j : S \hookrightarrow \bA^m$. By definition $\phi_{\pi_S}$ is defined via the left square of the following commutative diagram
\[
\xymatrix{\bA^n \times S \ar[d]^{\pi_S} \ar@{^{(}->}[r]^{j'=\id \times j} & \bA^n \times \bA^m \ar[r] \ar[d]^{\pi_{\bA^m}} & \bA^n \ar[d] \\
S \ar@{^{(}->}[r]^j & \bA^m \ar[r] & \Spec(k) },
\]
in the sense that $\phi_{\pi_S} \otimes \pi_S^\ast(\phi_j) \sim \phi_{j'} \otimes (j')^\ast(\phi_{\pi_{\bA^m}})$, where $\phi_{\pi_{\bA^m}}$ is the isomorphism of \Cref{L:grothendieck_formula_2}, where both squares are cartesian. \Cref{L:grothendieck_formula_1} implies vertical base change for the left square, and hence horizontal base change holds as well, and \Cref{L:grothendieck_formula_2} implies horizontal base change for the right square. Hence the outer rectangle satisfies horizontal base change.

\medskip
\noindent \textit{Proof of (3):}
\medskip

If $f : S' \to S$ is a regular closed immersion, we can consider a closed embedding $j : S \hookrightarrow \bA^n$ and consider the commutative square
\[
\xymatrix{S' \ar[r]^{j \circ f} \ar[d]^f & \bA^n \ar@{=}[d] \\ S \ar[r]^j & \bA^n }.
\]
$\phi_f$ is defined in this lemma as the unique homotopy class of isomorphism for which $\phi_f \otimes f^\ast(\phi_j) \sim \phi_{j \circ f}$, but by claim (1) of \Cref{L:grothendieck_formula_1}, the previous definition of $\phi_f$ for the regular closed immersion $f$ already satisfies this condition.

If $f : S' \to S$ is smooth, then we consider the following diagram, in which the square is cartesian:
\[
\xymatrix{S' \ar[dr]^{\id} \ar@{^{(}->}[r]^{\Delta_f} & S' \times_S S' \ar[r]^{p_2} \ar[d]^{p_1} & S' \ar[d]^f \\
& S' \ar[r] & S}
\]
We have already established that with $\phi_f$ as defined in this lemma, $\id \sim \phi_{\Delta_f} \otimes \Delta_f^\ast(\phi_{p_1}) \sim \phi_{\Delta_f} \otimes \Delta_f^\ast(p_2^\ast(\phi_{f})) \sim \phi_{\Delta_f} \otimes \phi_f$, with $\phi_{\Delta_f}$ as defined in \Cref{L:grothendieck_formula_1}. However, this characterizes the $\phi_f$ for the smooth morphism $f$ as defined in \Cref{L:grothendieck_formula_2}, so the two definitions agree.
\end{proof}

\begin{proof}[Proof of \Cref{P:canonical_complex_classical}]

Using the canonical isomorphisms \eqref{E:shriek_vs_star_pullback2} and \eqref{E:canonical_composition}, it suffices to establish an isomorphism in the absolute setting, i.e., when $\cY = \Spec(k)$. We know that $\cK_\cX^\dual \otimes \detz(\bL_\cX)$ is the image of a graded line bundle, which we denote $\cL_\cX \in \Perf(\cX)$, under the embedding $\Perf(\cX) \subset \DCoh(\cX) \subset \IC(\cX)$. Our goal is to construct, for any quasi-smooth locally a.f.p. algebraic derived $k$-stack $\cX$, a canonical isomorphism $\cO_{\cX^{\rm cl}} \cong \cL_\cX|_{\cX^{\rm cl}}$. We will use smooth descent. 

Any smooth morphism from an affine derived scheme $f : S \to \cX$ is separated and representable, so one can combine the canonical isomorphisms \eqref{E:shriek_vs_star_pullback2} and \eqref{E:canonical_composition} to obtain a canonical isomorphism
\begin{equation} \label{E:ratio_line_bundle}
f^\ast(\cL_\cX) \cong \cK_S^\dual \otimes \detz(\bL_S) \otimes \cK_{S/\cX} \otimes \detz(\bL_{S/\cX})^\dual \cong \cL_S \otimes \cL_{S/\cX}^\dual.
\end{equation}
The isomorphism $(\phi_f^{-1})^\dual$ of \Cref{L:grothendieck_formula_2} provides a trivialization $\cO_S \cong \cL_{S/\cX}^\dual$, and the isomorphism $\phi_{S \to \Spec(k)}$ of \Cref{L:grothendieck_formula_3} provides a trivialization $\cO_S \cong \cL_S$. So $f^\ast(\cL_\cX)$ has no homological shift, and under the above isomorphism we can regard $\phi_{S \to \Spec(k)} \otimes (\phi_{f}^{-1})^\dual$ as a nowhere vanishing section of $f^\ast(\cL_\cX)$.

Consider a second smooth morphism from an affine derived scheme $g : S' \to S$. Then as above we have canonical isomorphisms
\begin{align*}
g^\ast(\cL_S) &\cong \cL_{S'} \otimes \cL_{S'/S}^\dual, \text{ and}\\
g^\ast(\cL_{S/\cX}) &\cong \cL_{S'/\cX} \otimes \cL_{S'/S}^\dual.
\end{align*}
Combining these with the isomorphism \eqref{E:ratio_line_bundle} gives an isomorphism
\[
g^\ast(f^\ast(\cL_\cX)) \cong  \cL_{S'} \otimes \cL_{S'/S}^\dual \otimes \cL_{S'/S} \otimes \cL_{S'/\cX}^\dual \cong \cL_{S'} \otimes \cL_{S'/\cX}^\dual
\]
that is homotopic to the canonical isomorphism $g^\ast(f^\ast(\cL_\cX)) \cong (f\circ g)^\ast(\cL_\cX) \cong \cL_{S'} \otimes \cL_{S'/\cX}$ of \eqref{E:ratio_line_bundle} (see the proof of independence of embedding in \Cref{L:grothendieck_formula_3} for an explanation). It follows from \Cref{L:grothendieck_formula_2} and \Cref{L:grothendieck_formula_3} that under these isomorphisms
\[
g^\ast(\phi_{S \to \Spec(k)} \otimes (\phi_f^{-1})^\dual) \sim \phi_{S' \to \Spec(k)} \otimes (\phi_{f\circ g}^{-1})^\dual
\]
as sections of $(f \circ g)^\ast(\cL_\cX) \cong g^\ast(f^\ast(\cL_\cX))$.

We have thus constructed a nowhere vanishing section $\sigma_f := \phi_{S \to \Spec(k)} \otimes (\phi_{f}^{-1})^\dual \in \Gamma(S,f^\ast(\cL_\cX))$ for any smooth morphism from an affine derived scheme $f : S \to \cX$ such that for any smooth map $g : S' \to S$, $g^\ast(\sigma_f)$ is \emph{homotopic} to $\sigma_{f\circ g}$ in $\Gamma(S,(f\circ g)^\ast(\cL_\cX))$. This does not descend to a section of $\cL_\cX$ over $\cX$ without specifying higher coherence data. However, under the canonical map $\cL_\cX \to H_0(\cL_\cX)$, we can regard $\sigma_f$ as a section of $f^\ast(H_0(\cL_\cX))$. The spaces $\Gamma(S,f^\ast(H_0(\cL_\cX)))$ are discrete because $f^\ast(H_0(\cL_\cX))$ is truncated, so it makes sense to say that $g^\ast(\sigma_f)$ is \emph{equal} to $\sigma_{f \circ g}$ in $\Gamma(S,f^\ast(H_0(\cL_\cX)))$, and hence the $\sigma_f$ descend to a nowhere vanishing section of $H_0(\cL_\cX) \cong i_\ast(i^\ast(\cL_\cX))$, where $i : \cX^{\rm cl} \hookrightarrow \cX$ is the inclusion. This is equivalent to giving an isomorphism $\cO_{\cX^{\rm cl}} \cong i^\ast(\cL_\cX)$.
\end{proof}

\bibliographystyle{plain}
\bibliography{theta_stratifications}{}


\end{document}